\documentclass[11pt,a4paper]{amsart}
\usepackage[a4paper,margin=28mm]{geometry}
\usepackage{microtype}
\usepackage{amssymb,mathtools,mathrsfs}
\usepackage{booktabs,array,longtable}
\usepackage{graphicx}
\usepackage{enumitem}
\usepackage{aliascnt}
\usepackage{needspace}
\usepackage{placeins}
\usepackage{tikz}
\usetikzlibrary{arrows.meta}
\usepackage{xcolor}

\usepackage{listings}
\usepackage{upquote}
\lstdefinestyle{paperIIWolfram}{
  language=Mathematica,
  basicstyle=\ttfamily\fontsize{8.5}{10.2}\selectfont,
  keywordstyle=\color{blue!45!black},
  commentstyle=\color{black!60},
  stringstyle=\color{black},
  columns=fullflexible,keepspaces=true,showstringspaces=false,
  upquote=true,breaklines=true,breakatwhitespace=false,
  tabsize=2,numbers=none,aboveskip=6pt,belowskip=6pt,
  frame=single,rulecolor=\color{black!18},framerule=0.3pt,
  xleftmargin=4pt,xrightmargin=4pt,framesep=4pt
}

\lstdefinelanguage{WolframLanguage}{
  sensitive=true,
  morekeywords={Module,Block,ClearAll,If,Do,While,Return,Break,Table,Total,
    Flatten,Factor,Together,PowerExpand,Expand,D,Det,Cancel,Exponent,
    Coefficient,Binomial,PolynomialRemainder,TrueQ,Print,Select,Length,
    AppendTo,Range,Times,Min,Floor,Sign,InputForm,ToString,Join,
    ConstantArray,Delete,True,False,Infinity},
  morecomment=[s]{(*}{*)},
  morestring=[b]"
}
\usepackage[colorlinks=true,linkcolor=blue!55!black,
  citecolor=blue!55!black,urlcolor=blue!55!black]{hyperref}
\usepackage[nameinlink,capitalise,noabbrev]{cleveref}
\makeatletter
\def\l@section{\@tocline{1}{3pt plus1pt}{0pt}{}{\bfseries}}
\def\l@subsection{\@tocline{2}{0pt}{1.8em}{3.5em}{}}
\def\l@subsubsection{\@tocline{3}{0pt}{3.6em}{4.5em}{}}
\makeatother
\allowdisplaybreaks
\numberwithin{equation}{section}

\newtheorem{maintheorem}{Theorem}

\crefname{maintheorem}{Theorem}{Theorems}
\Crefname{maintheorem}{Theorem}{Theorems}
\newtheorem{theorem}{Theorem}[section]
\newaliascnt{proposition}{theorem}
\newtheorem{proposition}[proposition]{Proposition}
\aliascntresetthe{proposition}
\crefname{proposition}{Proposition}{Propositions}
\Crefname{proposition}{Proposition}{Propositions}
\newaliascnt{lemma}{theorem}
\newtheorem{lemma}[lemma]{Lemma}
\aliascntresetthe{lemma}
\crefname{lemma}{Lemma}{Lemmas}
\Crefname{lemma}{Lemma}{Lemmas}
\newaliascnt{corollary}{theorem}
\newtheorem{corollary}[corollary]{Corollary}
\aliascntresetthe{corollary}
\crefname{corollary}{Corollary}{Corollaries}
\Crefname{corollary}{Corollary}{Corollaries}
\newaliascnt{problem}{theorem}

\aliascntresetthe{problem}
\crefname{problem}{Problem}{Problems}
\Crefname{problem}{Problem}{Problems}
\theoremstyle{definition}
\newaliascnt{definition}{theorem}
\newtheorem{definition}[definition]{Definition}
\aliascntresetthe{definition}
\crefname{definition}{Definition}{Definitions}
\Crefname{definition}{Definition}{Definitions}
\newaliascnt{example}{theorem}
\newtheorem{example}[example]{Example}
\aliascntresetthe{example}
\crefname{example}{Example}{Examples}
\Crefname{example}{Example}{Examples}
\theoremstyle{remark}
\newaliascnt{remark}{theorem}
\newtheorem{remark}[remark]{Remark}
\aliascntresetthe{remark}
\crefname{remark}{Remark}{Remarks}
\Crefname{remark}{Remark}{Remarks}

\newcommand{\R}{\mathbb R}
\newcommand{\norm}[1]{\lvert #1\rvert}
\newcommand{\ip}[2]{\langle #1,#2\rangle}
\newcommand{\Sph}{\mathbb S}
\newcommand{\cH}{\mathcal H}
\newcommand{\cL}{\mathcal L}
\newcommand{\cK}{\mathcal K}
\newcommand{\cA}{\mathcal A}
\newcommand{\cB}{\mathcal B}
\newcommand{\cC}{\mathcal C}

\newcommand{\reg}{\mathrm{reg}}
\newcommand{\Sing}{\operatorname{Sing}}
\newcommand{\Per}{\operatorname{Per}}
\newcommand{\diver}{\operatorname{div}}
\newcommand{\dist}{\operatorname{dist}}
\newcommand{\Hess}{\operatorname{Hess}}
\newcommand{\Id}{\operatorname{Id}}
\newcommand{\tr}{\operatorname{tr}}
\newcommand{\dd}{\,\mathrm d}
\newcommand{\one}{\mathbf 1}
\newcommand{\eps}{\varepsilon}
\newcommand{\llbracket}{[\![}
\newcommand{\rrbracket}{]\!]}

\begin{document}

\title[Laplacian Algebras and Subcalibrations]
{Laplacian Algebras and Subcalibrations for Minimal Hypercones}
\author{Hongbin Cui}
\address{Wu Wen-Tsun Key Laboratory of Mathematics, USTC,
Chinese Academy of Sciences, School of Mathematical Sciences,
University of Science and Technology of China, 96 Jinzhai Road,
Hefei, Anhui 230026, China}
\email{cuihongbin@ustc.edu.cn}
\dedicatory{Dedicated to Anguo Cui}
\thanks{This work was supported by the National Natural Science Foundation
of China (No.~12301068), the project of Stable Support for
Youth Team in Basic Research Field, CAS (YSBR-001), and the Xiaomi Young
Scholar Fund.}
\date{}
\hypersetup{
  bookmarksdepth=3,
  pdftitle={Laplacian Algebras and Subcalibrations for Minimal Hypercones},
  pdfauthor={Hongbin Cui},
  pdfsubject={Subcalibration theory, strict stability, and strict area minimization},
  pdfkeywords={subcalibration, Laplacian algebra, Jacobi operator,
  algebraic realization, complex determinant, isoparametric cone}}

\begin{abstract}
We develop an algebraic mechanism based on Laplacian algebras for
constructing minimal hypercones and determining their variational properties.
Generator counts and degrees organize the constructions; homogeneous
multipliers depending on the generators yield subcalibration fields,
stability bounds and quantitative area-minimizing inequalities.
Two-generator isoparametric algebras give polynomial subcalibrations
for their minimizing hypercones.
In ambient dimension at least three, strict gradient-square closure
forces nonlinear odd degree to be cubic. We then study algebras of degrees
$(2,2,g)$ and quadratic compositions, placing OT--FKM splitting quartics
and Riedler's pullbacks in a common quadratic-composition framework.
It yields two further infinite quartic families with a sharp minimizing
threshold, as well as harmonic sextics from complex Gram coordinates.
For an irreducible harmonic minimal cubic on its essential space, we
prove a sharp bound of five generators for its smallest Laplacian algebra.
We determine the minimizing range of Clifford cubics and construct
irreducible nonradial harmonic examples in ambient dimensions $8,14,26$;
the last is strictly area minimizing. Finally, spectral families
whose smallest algebras have arbitrarily many generators provide examples
in every degree at least two: real parts of octonionic products, complex Pfaffians and complex
determinants. Explicit potentials establish strict area minimality
and stability bounds, including across nonisolated singularities.
The algebraic identities and certificates are supported by reproducible
exact symbolic computations.
\end{abstract}
\maketitle

\setcounter{tocdepth}{2}
\tableofcontents
\clearpage
\section{Introduction}

Area-minimizing hypercones arise as tangent cones at singularities of
area-minimizing hypersurfaces and play a central role in their regularity
theory \cite{HardtSimon85,Simon94,Szekelyhidi26,FiresterTsiamisWang26}.
Recent work of Chodosh and collaborators \cite{CMS26,CMSW25}
establishes generic regularity in ambient dimensions nine, ten, and eleven
for Plateau solutions and homology minimizers. Their arguments rely on
the analysis of minimizing cones and Jacobi fields. In particular, the
treatment of cylindrical singularities in \cite{CMSW25} uses the strict
stability and strict area minimality of minimizing quadratic cone factors.
We study three closely connected aspects of algebraic minimal hypercones:
their variational properties, the structure of their associated Laplacian
algebras, and their construction through polynomial realizations of
algebraic closure relations. Laplacian algebras provide a common framework
for these questions: their closure identities encode both the minimality
equation for a defining polynomial and the differential expressions
governing its subcalibrations.

The construction of algebraic minimal hypercones and their classification
in low degree have a long history. In 1967, Hsiang~\cite{Hsiang67}
classified the quadratic case and initiated the study of cubic minimal
hypercones. To formulate these problems algebraically, write
\(N=n+1\), \(n\geq2\), \(r=|z|\), and \(R=r^2\), and set
\begin{equation}\label{eq:L-definition}
 \cL(H)=|\nabla H|^2\Delta H
 -\frac12\left\langle\nabla|\nabla H|^2,\nabla H\right\rangle,
\end{equation}
for \(H\in C^2\). If \(F\) is a square-free homogeneous polynomial
whose real irreducible factors have nonsingular real zeros,
the regular part of \(\{F=0\}\) is minimal exactly when
\begin{equation}\label{eq:intro-minimality}
 \cL(F)=\lambda_FF,
\end{equation}
for a polynomial \(\lambda_F\); see \cref{prop:divisibility}.

Without additional structural assumptions, both classification and
explicit construction remain difficult. For example, the
Cartan--M\"unzner identities describe
isoparametric families whose minimal members yield algebraic hypercones
\cite{Muenzer80,Muenzer81}. For radial minimal cubics, where
\(\lambda_F\) is proportional to \(R\), Peng--Xiao~\cite{PengXiao93} and, independently,
Tkachev~\cite{TkachevClassification10,TkachevClifford10}
developed a systematic classification theory.

Once a minimal hypercone is given, a further question is whether it
minimizes area. The classical minimizing quadratic cones
\cite{BDGG69,Lawson72,Simoes73} provide a model for an algebraic
approach to this question. Building on earlier calibration and
subsolution arguments, De Philippis and Paolini
\cite{DePhilippisPaolini} gave a short proof of the area minimality
of Simons cones using subcalibrations. Suitable divergence inequalities
on the two sides replace the divergence-free condition and give direct
perimeter comparisons; see \cref{rem:subcalibration-history}.

Motivated by these quadratic constructions, we seek potentials
\(G=aF\), where \(a\) is a positive homogeneous multiplier on an
appropriate conical domain. Its normalized gradient satisfies
\begin{equation}\label{eq:div-normalized-gradient}
 X_G=\frac{\nabla G}{|\nabla G|},\qquad
 \diver X_G=\frac{\cL(G)}{|\nabla G|^3},
\end{equation}
where \(\nabla G\ne0\). On the regular zero set in this domain,
\(\nabla G=a\nabla F\), so \(X_G\) agrees
with the exterior normal to \(\{F<0\}\). The following definition
includes the removable singular sets needed in our applications.

\begin{definition}[Homogeneous subcalibration potential]
\label{def:homogeneous-subcalibration-potential}
Let \(F\) be a nonzero homogeneous polynomial on \(\R^N\).
A continuous function \(G\) on \(\R^N\) is a
\emph{homogeneous subcalibration potential} for \(F\) if there are
a closed conical set \(Z\) with \(\cH^{N-1}(Z)=0\) and a positive
homogeneous function \(a\in C^2(\R^N\setminus Z)\) such that
\(G=aF\) on \(\R^N\setminus Z\), \(G\) is homogeneous of degree
\(\ell>0\), and
\begin{equation}\label{eq:master-sign}
 \nabla G\ne0,\qquad G\cL(G)\ge0
 \quad\text{on }\R^N\setminus Z.
\end{equation}
\end{definition}

Since polynomial sublevel sets have locally finite perimeter,
\cref{cor:gradient-subcalibration} gives perimeter comparison
for \(\{F<0\}\) across \(Z\).

For the Simons cone, with \(x,y\in\R^m\), \(m\ge4\), the potential
in \cite[Section~2]{DePhilippisPaolini} is, up to a positive constant,
\[
 F=|x|^2-|y|^2,\qquad
 G=|x|^4-|y|^4=(|x|^2+|y|^2)F=RF.
\]
For Lawson cones in \(\R^k\times\R^h\), \(k,h\ge2\), put
\(U=|x|^2/(k-1)\), \(V=|y|^2/(h-1)\), and \(F=U-V\).
In the dimension ranges of
\cite[Lemma~4.1, equation~(4.7)]{DePhilippisMaggi14}, the quartic
subcalibration potential factors as
\begin{equation}\label{eq:intro-Lawson-multiplier}
 G=\frac{U^2-V^2}{4}=aF,\qquad a=\frac{U+V}{4}.
\end{equation}
Phase-dependent alternatives are discussed in \cite{Liu19} and
\cref{rem:phase-local-fields}.
In the two polynomial examples above, both \(F\) and \(a\) lie in
\(\R[|x|^2,|y|^2]=\R[R,F]\). This algebra is closed under
\(\Delta\) and the gradient pairing
\(\Gamma(u,v)=\langle\nabla u,\nabla v\rangle\), so the
subcalibration calculation stays in the same two quadratic variables.

More generally, a subalgebra \(\cA\subset\R[z_1,\ldots,z_N]\) is
a \emph{Laplacian algebra} if \(R\in\cA\) and
\(\Delta\cA\subset\cA\).
It is graded and finitely generated
\cite[Corollary~22 and Lemma~24]{MendesRadeschi}, and the product
rule gives closure under \(\Gamma\).
For homogeneous generators \(I=(I_1,\ldots,I_k)\), closure is
recorded by polynomial identities
\[
 \Delta I_i=b_i(I),\qquad
 \Gamma(I_i,I_j)=c_{ij}(I),
\]
which form the \emph{closure table}
\cite[Proposition~37(a)]{MendesRadeschi}.
For \(F\in\cA\) and a homogeneous multiplier \(a=\alpha\circ I\),
with \(\alpha\) smooth on an open set \(W\subset\R^k\), the chain
rule expresses \(\cL(F)\), \(\cL(aF)\), and \(|\nabla(aF)|^2\)
in the generator variables on \(I^{-1}(W)\).
These closure tables may therefore guide the construction and
recognition of minimal algebraic hypercones; see
\cref{sec:gradient-square-closure}.

We use the structural theory of Mendes--Radeschi
\cite{MendesRadeschi,MendesRadeschiQuadratic,MendesRadeschiMaximality}
to study cubic and quartic models and develop new polynomial realizations.
We begin with algebras generated by the radius squared and one
defining polynomial.

\subsection{Two generators and radial multipliers}

In this two-generator case, the Mendes--Radeschi characterization
\cite[Proposition~37(c)]{MendesRadeschi} leads to an isoparametric
family. We construct polynomial subcalibrations for all of its
nonflat minimizing cones.

\begin{maintheorem}[Polynomial subcalibrations for isoparametric hypercones]
\label{main:isoparametric}
Let $C\subset\R^N$ be a nonflat area-minimizing isoparametric hypercone,
and let $F$ be its centered Cartan--M\"unzner defining polynomial, as
in \eqref{eq:iso-F-z}. There is an explicit homogeneous polynomial
$a_F$, positive on $\R^N\setminus\{0\}$, such that $G=a_FF$ satisfies
\[
 \deg G\le10,\qquad \nabla G\ne0
 \quad\text{on }\R^N\setminus\{0\},
\]
and
\[
 \operatorname{sgn}(G)\diver X_G>0\qquad\text{off }C.
\]
The field $X_G$ is a subcalibration of $C$, and its divergence proves
strict area minimality.
\end{maintheorem}
Lawlor \cite{Lawlor91}, Ferus--Karcher \cite{FerusKarcher85} and
Wang \cite{Wang94} established the minimizing results used here;
see \cref{sec:isoparametric-reduction} for their precise scope.
Our explicit potentials have a dimension-independent degree bound
and recover strict area minimality. The complete list is in
\cref{thm:iso-polynomial-table}.
For harmonic defining polynomials, a radial criterion applies in
arbitrary degree.

\begin{maintheorem}[A radial criterion for harmonic polynomials]
\label{main:radial-criterion}
Let $F\not\equiv0$ be a homogeneous polynomial of degree $e\ge2$
on $\R^N$, $N\ge3$, and suppose
\[
 \Delta F=0,\qquad S=|\nabla F|^2,\qquad \Gamma(F,S)=2qF
\]
for a polynomial $q$. If, for a constant $\kappa\ge0$,
\[
 Rq\le\kappa S\quad\text{on }\R^N,\qquad
 N\ge2e+1+2\sqrt\kappa,
\]
then $G=r^{(N-1-2e)/2}F$ generates a two-sided subcalibration.
The cone $C=\{F=0\}$, oriented as the boundary of $\{F<0\}$,
is strictly area minimizing, and
\[
 Q_C(u)\ge\left[\frac{(N-1-2e)^2}{4}-\kappa\right]
       \int_Cr^{-2}u^2,\qquad u\in C_c^\infty(C\setminus\{0\}).
\]
\end{maintheorem}
The stability estimate follows from the multiplier--Jacobi identity
and the homogeneous ground-state bound
(\cref{thm:Jacobi-trace,thm:ground-state-gap}).
A finite $\kappa$ exists exactly when $\nabla F$ vanishes only at
the origin. Its optimal value is at least $e(e-1)$, with equality
precisely for normalized harmonic Cartan--M\"unzner polynomials
(\cref{prop:harmonic-global-comparison,prop:harmonic-sharp-comparison}).

\begin{remark}[Realizability and regular algebraic cones]
\label{rem:regular-algebraic-realizability}
This regularity requirement connects \cref{main:radial-criterion}
with Perdomo's question about non-isoparametric regular algebraic
minimal hypersurfaces of spheres
\cite[Section~1 and Main theorem]{PerdomoRegular2006}; his nonexistence
theorem concerns irreducible cubics in \(\Sph^3\).
We conjecture that, for \(N\ge3\), a nonconstant real homogeneous
irreducible polynomial \(F\) whose zero link is nonempty and minimal,
with \(\nabla F\ne0\) everywhere on that link, defines an
isoparametric hypersurface. In the harmonic case, we expect \(F\)
itself, up to scaling, to be a harmonic Cartan--M\"unzner polynomial.
\Cref{prop:regular-radial-weight-rigidity} proves this harmonic
conclusion when \(q=\lambda R^{e-2}\), an additional condition on
the minimality weight, distinct from choosing the radial multiplier
\(a=r^\sigma\) in \cref{main:radial-criterion}.
\end{remark}

\subsection{Gradient-square closure and quadratic compositions}

Adjoining $S=|\nabla F|^2$ closes the algebra in a distinguished
class. For $N\ge3$, odd degree at least three then forces
$\deg F=3$. If the critical set consists only of the origin, $F$
is a Cartan cubic; the other cubic members are described in
\cref{prop:cubic-gradient-classification}.
In even degree, we record the closure identities and a short
isoparametric consequence on the regular link, followed by a matrix
quartic example.

For three-generator degrees $(2,2,g)$, $g>2$, subtracting a polynomial
in the two block radii from the higher-degree generator makes it
separately harmonic and bihomogeneous. The motivating OT--FKM splitting
quartics, developed from \cite{CuiFKM25}, are quadratic compositions
of Clifford coordinates, as are Riedler's quartic pullbacks.
\Cref{sec:two-quadratic} develops their general orthogonal-block criterion.
These compositions have generator degrees $(2,4,4)$, distinguished
from strict gradient-square closure $(2,4,6)$.

\begin{maintheorem}[Subcalibrations from orthogonal quartic blocks]
\label{main:bihomogeneous}\label{main:quartic-table}
Let $E_0,E_1,E_2$ be nonzero homogeneous quartic polynomials on
$\R^N$, with $R=|x|^2$, satisfying
\[
 E_0+E_1+E_2=R^2,\qquad
 \Gamma(E_i,E_j)=16RE_i\delta_{ij},\qquad
 \Delta E_i=8r_iR.
\]
Assume $r_i>0$, $r_1,r_2\ge2$, and either
\[
 r_1+r_2\ge9,
 \qquad\text{or}\qquad
 r_1=r_2=4\ \text{ and }\ r_0\ge5.
\]
Set $F=(r_2-1)E_1-(r_1-1)E_2$ and
$\Sigma=\{E_1=E_2=0\}$. Then there is a homogeneous polynomial
potential $G=AF$ of degree at most ten, with $A>0$ off $\Sigma$
and $\operatorname{Crit}G=\Sigma$, such that
\[
 \operatorname{sgn}(G)\diver X_G>0\qquad\text{where }F\ne0.
\]
The exceptional set $\Sigma$ is removable for this field. The
oriented boundary of $\{F<0\}$ has support $\{F=0\}$ and is
strictly area minimizing and strictly stable.
\end{maintheorem}
Together with the instability results, the theorem gives the minimizing
classification for the splitting family and recovers Riedler's
Clifford--Simons range. The same table
also treats the octonionic Hopf pairing and the balanced $(3,7,7)$
realization. Their weights, potentials and ranges are collected in
\cref{tab:quartic-weight-realizations}; their different singular
strata distinguish the geometric realizations. In particular,
\cref{thm:orthq-infinite-families} constructs two further infinite
families in $\R^{8k}$ with an unused block of weight two; they are
area minimizing exactly for $k\ge3$. Quadratic Jordan and Gram
coordinates also give the degree-independent profile equation of
\cref{prop:three-quadratic-profile-criterion} and the harmonic
sextic family in \cref{thm:complex-gram-cycle-sextic}.

\subsection{Cubic structure and a common subcalibration}

The historical and algebraic background is recalled in
\cref{sec:cubic-algebraic-theory}. For harmonic minimal cubics, we
obtain a sharp bound on the smallest algebra.
For \(\cA=\R[I_1,\ldots,I_k]\), its \emph{spherical quotient dimension}
is \(\dim I(\Sph^{N-1})\), where \(I=(I_1,\ldots,I_k)\).

\begin{maintheorem}[Five generators for harmonic minimal cubics]
\label{main:cubic-generator-bound}
Let $F$ be a nonzero irreducible homogeneous harmonic cubic on its
essential Euclidean space, with minimal regular zero set. The smallest
Laplacian algebra $\cA_R(F)$ has at most five homogeneous generators.
For a radial cubic at most four suffice. For a nonradial cubic the
algebra is polynomial on five algebraically independent generators of
degrees $(2,2,3,3,4)$, with spherical quotient dimension four;
its ambient dimension is $8$, $14$, or $26$.
\end{maintheorem}
Here radial means that $\cL(F)=-qF$ has $q$ proportional to $R$.
The intrinsic closure recovers a Cartan polynomial and forces the
dimension list $8,14,26$. Explicit Cartan deformations realize all
three dimensions: the models in $\R^8$ and $\R^{14}$ are unstable,
whereas the model in $\R^{26}$ is strictly area minimizing.
For radial cubics, the following criterion is read from $\Delta S$.

\begin{maintheorem}[A cubic trace criterion]
\label{main:cubic-trace}
Let $F\not\equiv0$ be a homogeneous harmonic cubic on $\R^N$,
$N\ge8$, satisfying $\Gamma(F,S)=4RF$. If
\[
 \Delta S\ge\max\left\{16,\frac{2(40-N)}3\right\}R,
\]
then there exists a continuous homogeneous potential $G=AF$ of
degree six, with $A>0$ on $\{S>0\}$ and $G$ smooth there, whose
gradient is nonzero there and whose normalized gradient has the
strict divergence sign of $F$ off its zero set. The critical set
$\{S=0\}$ is removable. Hence $\{F<0\}$ is locally perimeter
minimizing, and its oriented boundary has support $\{F=0\}$ and
is strictly area minimizing.
\end{maintheorem}
The stronger form, \cref{thm:cubic-unified-sextic}, also covers the
low-dimensional Clifford cases outside this sufficient trace bound.
It yields the common degree-six potential $F\sqrt{RS}$ for rank-four
trace cubics, complexified rank-three Jordan norms and the $G_2$
trilinear cubic. The nonradial model in $\R^{26}$ uses the same
potential with its five-generator defect calculation.

For the Clifford family, let $(A_0,\ldots,A_q)$ be a symmetric
Clifford system on $\R^{2m}$, $q\ge1$, and put
\[
 F(y,z)=\sum_{i=0}^qz_i\langle A_i y,y\rangle.
\]

\begin{maintheorem}[Clifford cubic classification]
\label{main:clifford-classification}
The oriented boundary of $\{F<0\}$ is area minimizing exactly when
$m\ge4$, $q\ge3$, and $(m,q)\ne(4,3)$. In that range it is
strictly area minimizing, and, after replacing $F$ by $F/2$,
its Laplacian data satisfy \cref{thm:cubic-unified-sextic}.
For $(m,q)=(4,3)$ the cone
is unstable although every tangent cone at a nonzero point is area
minimizing.
\end{maintheorem}
The classification uses the common positivity criterion,
singular-stratum obstructions and an explicit negative variation for
the exceptional pair. The examples and their criteria are summarized
in \cref{tab:intro-family-summary}.

\begin{table}[htbp]
\centering
\small
\setlength{\tabcolsep}{4pt}
\renewcommand{\arraystretch}{1.2}
\caption{Selected families and the criteria used for their minimizing
members. The third column counts generators of the smallest Laplacian
algebra $\cA_R(F)$; $\le4$ is an upper bound. Explicit potentials
are constructed in the cited results.}
\label{tab:intro-family-summary}
\begin{tabular}{@{}>{\raggedright\arraybackslash}p{.24\linewidth}p{.07\linewidth}p{.09\linewidth}>{\raggedright\arraybackslash}p{.25\linewidth}>{\raggedright\arraybackslash}p{.25\linewidth}@{}}
\toprule
Family & $\deg F$ & Number & Criterion & Minimizing range \\
\midrule
Isoparametric & $g$ & $2$ & Two-generator classification; \cref{main:isoparametric}
 & Full list in \cref{thm:iso-polynomial-table} \\
Orthogonal quartic families & $4$ & $3$ & Three-weight criterion; \cref{main:quartic-table}
 & Splitting, Hopf, pullback and balanced ranges in \cref{tab:quartic-weight-realizations} \\
Rank-four trace & $3$ & $3$ & Cubic trace bound; \cref{main:cubic-trace}
 & $N=27$; \cref{cor:quaternionic-trace-cubic} \\
Complexified rank-three Jordan norms & $3$ & $4$ & Cubic trace bound; \cref{main:cubic-trace}
 & $N=18,30,54$; \cref{thm:exc-jordan-minimizing} \\
$G_2$ trilinear & $3$ & $4$ & Cubic trace bound; \cref{main:cubic-trace}
 & $N=21$; \cref{thm:exc-g2} \\
Clifford cubics & $3$ & $\le4$ & Full cubic criterion; \cref{thm:cubic-unified-sextic}
 & $m\ge4$, $q\ge3$, $(m,q)\ne(4,3)$ \\
Nonradial Cartan deformation & $3$ & $5$ & Five-generator positivity; \cref{thm:nonradial-minimizing}
 & $N=26$ \\
Octonionic products; complex Pfaffians & $d\ge3$ & $d+1$ & Spectral threshold; \cref{main:spectral-criterion}
 & Every $d\ge3$; \cref{main:spectral-families} \\
Complex determinants & $m\ge3$ & $m+1$ & Minor-norm inequalities; \cref{sec:complex-determinants}
 & Every $m\ge3$; \cref{main:complex-determinants} \\
\bottomrule
\end{tabular}
\end{table}
\FloatBarrier
\subsection{Pullbacks and arbitrary-degree spectral families}

Riedler's pullback construction and question
\cite[Theorem~5.2.5 and Question~5.3.3(2)]{RiedlerThesis25}
motivate the transfer theorem in \cref{thm:clifford-area-transfer}.
A homogeneous target subcalibration $H$ with
$H\Delta H\ge-M|\nabla H|^2$ transfers through a quadratic Clifford
map $P$ to $R^aH(P)$ in sufficiently large admissible dimensions,
under the singular-set hypotheses of that theorem.
The corresponding smallest algebra has an exact description.

\begin{maintheorem}[The smallest algebra under Clifford pullback]
\label{main:clifford-smallest-algebra}
Let $A_1,\ldots,A_m$ be symmetric endomorphisms of $\R^N$, $m\ge2$,
with $A_iA_j+A_jA_i=2\delta_{ij}\Id$, and set
\[
 P(x)=(\langle A_1x,x\rangle,\ldots,\langle A_mx,x\rangle),
 \qquad R=|x|^2,\quad q(y)=|y|^2,\quad Q=q(P).
\]
For every nonzero essential homogeneous polynomial $f$ of degree
$s\ge2$ on $\R^m$,
\[
 \cA_R(f\circ P)=\R[R,P^*\cA_q(f)].
\]
If $Q<R^2$ somewhere, this algebra is isomorphic to the polynomial
extension $\cA_q(f)[t]$, with target degrees doubled and $\deg t=2$.
In this case, its minimal number of homogeneous generators is
exactly one more than that of $\cA_q(f)$.
\end{maintheorem}
Iteration doubles the defining degree at each step and, under the
nondegeneracy hypothesis, adds exactly one minimal generator per step
(\cref{cor:iterated-clifford-transfer}). Cubic applications appear in
\cref{cor:common-cubic-pullback,thm:cartan-pullback}.

For arbitrary degree, nonnegative spectral variables give a second
uniform construction. Write $e_k$ for elementary symmetric polynomials,
with $e_0=1$.

\begin{maintheorem}[A spectral criterion and its octonionic and Pfaffian applications]
\label{main:spectral-criterion}\label{main:spectral-families}
Let $F,E_1,\ldots,E_d$ be real homogeneous polynomials on $\R^N$
of degrees $d,2,\ldots,2d$, where $d\ge2$, $F\not\equiv0$, and
$E_1=R$. Assume that at every point $E_k=e_k(r_1,\ldots,r_d)$ for
some nonnegative $r_i$, and that
\[
 \begin{gathered}
 \Delta F=0,\qquad S:=\Gamma(F,F)=E_{d-1},\\
 \Gamma(F,E_k)=2(d-k+1)E_{k-1}F,\\
 \Gamma(E_i,E_j)=4\sum_{\alpha=1}^d
       r_\alpha\,\partial_\alpha e_i(r)\partial_\alpha e_j(r).
 \end{gathered}
\]
If
\[
 \Delta S\ge\left(24-\frac{16}{d}\right)E_{d-2}
 \qquad\text{on }\{S>0\},
\]
there exists a homogeneous potential $G=AF$ of degree $2d-1$,
with $A>0$ and $\nabla G\ne0$ there, satisfying
$G\cL(G)>0$ wherever $F\ne0$. The hypotheses force $N\ge4d$.
The critical set $\{S=0\}$ is removable, and the normalized gradient
of $G$ gives a two-sided subcalibration. The phase and boundary
current are locally minimizing. The cone is strictly area minimizing,
and its regular part satisfies
\[
 Q_C(u)\ge\frac{(N-4d+1)^2}{4}\int_C |x|^{-2}u^2,
 \qquad u\in C_c^\infty(C_{\reg}\setminus\{0\}).
\]

In particular, for every $d\ge2$, each of the following harmonic degree-$d$ cones
is the support of a strictly area-minimizing integral boundary
and has strictly stable regular part:
\begin{enumerate}[label=\textup{(\roman*)}]
\item $\{\operatorname{Re}(x_1\cdots x_d)=0\}\subset\mathbb O^d$,
for every fixed parenthesization;
\item $\{\operatorname{Re}\operatorname{Pf} Z=0\}$ on the complex
skew-symmetric $2d\times2d$ matrices, with Euclidean norm
$\sum_{i<j}|z_{ij}|^2$.
\end{enumerate}
Their critical sets have real codimension sixteen and twelve, respectively.
The Hardy lower bounds are $(4d+1)^2/4$ and
$(4d^2-6d+1)^2/4$, respectively.
\end{maintheorem}
The proof uses Newton inequalities and monotonicity in reciprocal
spectral variables. Octonionic products and Pfaffians realize the
same gradient-pairing table with Laplacian coefficients $32$ and $24$.
For their earlier harmonic-morphism and regular-minimality results, see
\cite[Section~4.5.2]{RiedlerThesis25} and
\cite[Corollary~5.2 and Remark~5.5]{HoppeTkachev19}.

\subsection{Complex determinants}
\label{subsec:intro-determinants}

Real parts of complex determinants have this gradient table with the
smaller Laplacian coefficient $16$. The preceding multiplier fails in
orders at least three. A product of proper minor norms supplies a
different certificate in every order, including across nonisolated
singularities. Here $E_j$ is the sum of squared absolute values of all
$j\times j$ minors and $F_m=\operatorname{Re}\det_{\mathbb C}Z$.

\begin{maintheorem}[Real parts of complex determinants]
\label{main:complex-determinants}
For every $m\ge1$, the boundary
$T_m=\partial\llbracket\{F_m<0\}\rrbracket$ is locally area
minimizing in $M_m(\mathbb C)\cong\R^{2m^2}$. For $m\ge2$,
\[
 G_m=F_m\left(\prod_{j=1}^{m-1}E_j\right)^{1/2}
\]
gives a subcalibration smooth off the rank-at-most-$m-2$ locus,
which has real codimension eight. For $m\ge2$, the cone is strictly area minimizing,
and
\begin{equation}\label{eq:main-determinant-stability}
 Q_{C_m}(u)\ge\frac{(m^2-m-1)^2}{4}\int_{C_m}r^{-2}u^2,
 \qquad u\in C_c^\infty((C_m)_{\reg}\setminus\{0\}).
\end{equation}
\end{maintheorem}
Regular minimality was proved in
\cite[Corollary~5.2 and Example~5.4]{HoppeTkachev19}.
For $m\ge3$, the smallest algebra has exactly $m+1$ independent
generators of maximal degree $2m$. The proof of global positivity
is uniform in $m$; for $m=3$ its potential is $F\sqrt{RS}$,
linking the spectral construction to the cubic one.

\Cref{sec:preliminaries,sec:laplacian-calculus,sec:algebraic-jacobi}
develop singular-set comparison, finite algebraic calculus and the
multiplier--Jacobi identity; \cref{sec:quartic-application} develops
pullback transfer. Structural restrictions in
\cref{sec:low-rank,sec:gradient-square-closure,sec:cubic-algebraic-theory}
support the quartic criteria and cubic minimizing classification in
\cref{sec:two-quadratic,sec:cubic-potentials,sec:nonradial-examples}.
\Cref{sec:spectral-positivity,sec:norm-products-pfaffians,sec:complex-determinants}
extend the method to arbitrary degree and unbounded generator number.

\begin{remark}[Smooth minimizing hypersurfaces]
These new area-minimizing algebraic hypercones also give smooth
minimizing hypersurfaces. Wang's extension
\cite[Theorem~1.1]{WangSmoothing24} of the Hardt--Simon theorem
\cite{HardtSimon85} provides, on each side of such a cone, a properly
embedded smooth area-minimizing hypersurface having that cone as its
tangent cone at infinity. This applies also to the examples with
nonisolated singularities.
\end{remark}

\section[Geometric Comparison and Subcalibrations]{Geometric Comparison and Subcalibrations across Singular Sets}
\label{sec:preliminaries}

All ambient metrics in this section are Euclidean. For background on
minimal surfaces, functions of bounded variation, and minimizing
boundaries, see \cite{Giusti84}.

\subsection{Stationarity and stability on the regular part}
\label{sec:prelim-stability}

A smooth submanifold is \emph{minimal} when its mean curvature vector
vanishes. For singular objects we use integral currents, with the
boundary convention \(\partial T(\eta)=T(d\eta)\), mass measure
\(\|T\|\), and restriction \(T\llcorner A\) \cite{Federer69}.
The associated varifold is \emph{stationary} when
\begin{equation}\label{eq:stationarity-definition}
 \delta|S|(Y)=\int\diver_{T_xS}Y\,d\|S\|(x)=0,
 \qquad Y\in C_c^1(\Omega\setminus\operatorname{spt}\partial S;\R^N),
\end{equation}
where \(T_xS\) is the approximate tangent plane and its divergence
is the trace of the ambient derivative
\cite[Section~16]{Simon83}.
We write \(S_{\reg}\) for the regular part
and \(\Sing S\) for the remaining support, away from the boundary.

For a smooth minimal \(M\), let \(A\) be its second fundamental form
and \(\nabla^\perp\) its normal connection.  The second variation in a
compactly supported normal direction \(V\) is
\begin{equation}\label{eq:normal-second-variation}
 Q_M^\perp(V)=\int_M\left(|\nabla^\perp V|^2
       -\sum_{i,j=1}^n\langle A(e_i,e_j),V\rangle^2\right),
\end{equation}
where \(e_1,\ldots,e_n\) is a local orthonormal tangent frame.
With the Laplacian convention \(\int u\Delta u=-\int|\nabla u|^2\),
the normal Jacobi operator is
\begin{equation}\label{eq:normal-Jacobi-definition}
 J_M^\perp V=\Delta_M^\perp V+
       \sum_{i,j}\langle A(e_i,e_j),V\rangle A(e_i,e_j),
 \qquad Q_M^\perp(V)=-\int_M\langle V,J_M^\perp V\rangle.
\end{equation}
Stability means \(Q_M^\perp(V)\geq0\) for every such \(V\).
For a stationary singular current the tests are compactly supported
smooth normal fields on \(S_{\reg}\), away from its boundary;
see \cite{Simons68} and
\cite[Lemma~9.7 and Remark~9.8]{Simon83}.
Each test has compact support away from the singular set, but no common
positive distance is prescribed for the whole test class.
Extension across it uses intrinsic energy cutoffs, whereas divergence
comparison below uses ambient \(L^1\) cutoffs. For cones, strict
stability means the scale-invariant gap defined below.

\subsection{Area minimization and the strict notions for cones}
\label{sec:prelim-minimization}

An integral current \(S\) is \emph{locally mass minimizing} if
\[
 \mathbf M_U(S)\leq\mathbf M_U(S+\partial R)
\]
whenever \(U\Subset\Omega\) and \(R\) is an integral current compactly
supported in \(U\) of one higher dimension.  Here
\(\mathbf M_U(S)=\|S\|(U)\).  For multiplicity-one
submanifolds we also call this \emph{area minimization}.
Such minimizers are stationary and stable, by first and second variation.

A measurable set \(E\subset\Omega\) has \emph{locally finite perimeter}
if \(\one_E\in BV_{\rm loc}(\Omega)\): its distributional derivative
is a vector-valued Radon measure of locally finite mass. For
\(U\Subset\Omega\),
\begin{equation}\label{eq:finite-perimeter-definition}
 \Per(E;U)=|D\one_E|(U),\qquad
 D\one_E=-\nu_E\,\cH^{N-1}\llcorner\partial^*E,
\end{equation}
where \(\partial^*E\) is the reduced boundary and \(\nu_E\) its
measure-theoretic exterior unit normal; see \cite{Maggi12}.

The current \(\llbracket E\rrbracket\) denotes integration over \(E\)
with the standard orientation, and
\(\mathbf M_U(\partial\llbracket E\rrbracket)=\Per(E;U)\).
Its boundary has multiplicity one on \(\partial^*E\), with support
\(\overline{\partial^*E}\cap\Omega\). For polynomial hypersurfaces we take
\(C=\partial\llbracket\{F<0\}\rrbracket\).
Its support equals the full zero set when both signs occur near every
zero, in particular for nonzero harmonic \(F\);
see \cref{rem:polynomial-finite-perimeter}.

A set \(E\) is \emph{locally perimeter minimizing} if
\(\Per(E;U)\leq\Per(H;U)\) for every locally finite-perimeter set
\(H\) with \(H\triangle E\Subset U\Subset\Omega\).
Local perimeter minimization of \(E\) is equivalent to local mass
minimization of \(\partial\llbracket E\rrbracket\).  This is a direct
consequence of the codimension-one decomposition theorem
\cite[Theorem~27.6 and Corollary~27.8]{Simon83}; see also
\cite[Proposition~2.15]{DeLellis15}.

For the cone statements let \(C\) be a stationary multiplicity-one
integral cone of real dimension \(n\geq2\), with \(\partial C=0\).
Put \(r=|x|\) and \(C_1=C\llcorner B_1\).
We also use \(C\) for its support. By slicing and homogeneity,
\(\partial C_1=\langle C,r,1\rangle\) is the oriented spherical slice
on \(\partial B_1\) \cite[Lemma~28.5(2)]{Simon83}.

\begin{definition}[Minimality and stability, with their strict forms]
\label{def:four-variational-notions}
The cone is \emph{area minimizing} if \(C_1\) has least mass among
integral \(n\)-currents with boundary \(\partial C_1\).  It is
\emph{strictly area minimizing} in the vertex-avoidance sense if it is
area minimizing and there is \(\Theta>0\) such that
\begin{equation}\label{eq:strict-minimality-definition}
 \mathbf M(T)-\mathbf M(C_1)\geq\Theta\eps^n
 \quad\text{whenever}\quad
 \partial T=\partial C_1,\quad
 \operatorname{spt}T\cap B_\eps=\varnothing,\quad 0<\eps<1.
\end{equation}
The competitors are integral currents of finite mass and arbitrary
integer multiplicity in the ambient Euclidean space.
For \(C=\partial\llbracket E\rrbracket\) with \(E\)
conical, area minimization is equivalent to local perimeter minimization.

For a two-sided hypersurface cone and \(V=u\nu_C\),
\eqref{eq:normal-second-variation} becomes
\begin{equation}\label{eq:Q-definition}
 Q_C(u)=\int_{C_{\reg}}\bigl(|\nabla_Cu|^2-|A_C|^2u^2\bigr),
 \qquad u\in C_c^\infty(C_{\reg}\setminus\{0\}).
\end{equation}
The cone is \emph{stable} if \(Q_C(u)\geq0\) for every such \(u\), and
\emph{strictly stable} if there is \(\delta>0\) such that
\begin{equation}\label{eq:strict-stability-definition}
 Q_C(u)\geq\delta\int_{C_{\reg}}r^{-2}u^2
 \qquad\text{for every }u\in C_c^\infty(C_{\reg}\setminus\{0\}).
\end{equation}
In arbitrary codimension the same definition uses \(Q_C^\perp(V)\)
and \(|V|^2\) in place of \(Q_C(u)\) and \(u^2\).
\end{definition}

The vertex-avoidance convention is that of Hardt--Simon
\cite[Section~3, (3.1)]{HardtSimon85}, stated for integral currents in
\cite[p.~210]{Lin87}. These strict notions enter constructions with
prescribed conical singularities \cite{CaffarelliHardtSimon84},
metric deformations \cite{WangDeformations20}, and Simon's Liouville
theorem \cite[equations~(1.4), (1.12) and Corollary~1.13]{Simon21}.

Strict stability does not imply area minimization,
as \cref{prop:iso-stability,thm:classical-iso-classification} illustrate;
\cref{prop:complex-zero-gap} gives calibrated area-minimizing cones
of real codimension two without a positive Hardy gap.
For a singular link
\(\Lambda=C\cap\Sph^{N-1}\), angular tests are compactly supported in
\(\Lambda_{\reg}\); the corresponding spectral formulation is recorded
in \eqref{eq:exact-Hardy-gap}.

\subsection{Calibrations and the Stokes comparison}
\label{sec:prelim-calibrations}

A \emph{calibration} is a closed differential \(n\)-form whose
comass, the supremum of its absolute value on orthonormal
\(n\)-frames, is at most one. It calibrates an integral current \(S\) if
\(\omega(\tau_S)=1\) at \(\|S\|\)-almost every point, where \(\tau_S\)
is the oriented unit tangent \(n\)-vector
\cite{HarveyLawson82a}.

If \(\omega\) calibrates \(S\) and \(T-S=\partial R\) with
\(\operatorname{spt}R\Subset U\), Stokes' theorem gives
\begin{equation}\label{eq:calibration-Stokes-comparison}
 \mathbf M_U(T)-\mathbf M_U(S)
 \geq (T-S)(\omega)=\partial R(\omega)=R(d\omega)=0.
\end{equation}
The common contributions outside \(U\) cancel; a cutoff equal to one
near \(\operatorname{spt}R\) makes the evaluation of the difference
compactly supported.  Thus a calibrated current is mass minimizing.

\subsection{Hypersurface forms and vector-field subcalibrations}
\label{sec:prelim-subcalibrations}

In codimension one, the identification \(\omega_X=\iota_XdV\) gives
\begin{equation}\label{eq:form-vector-equivalence}
 \|\omega_X\|_*=|X|,\qquad
 d\omega_X=(\diver X)dV,\qquad
 \omega_X|_{\partial E}=(X\cdot\nu_E)\,d\mu_{\partial E}.
\end{equation}
Evaluation on an oriented orthonormal frame and the coordinate
formula for divergence give these identities. A hypersurface
calibration is equivalently a field with \(|X|\leq1\),
\(\diver X=0\), and \(X\cdot\nu_E=1\) on the boundary.

\begin{definition}[Two-sided subcalibration, smooth formulation]
\label{def:subcalibration}
Let \(E\subset\Omega\subset\R^N\) have smooth boundary.  A
\(C^1\) field \(X\) is a two-sided subcalibration for \(E\) if
\begin{equation}\label{eq:smooth-subcalibration-signs}
 |X|\leq1,\qquad X\cdot\nu_E=1\text{ on }\partial E,\qquad
 \diver X\leq0\text{ in }E,\quad
 \diver X\geq0\text{ in }\Omega\setminus E.
\end{equation}
The equivalent form conditions follow from \eqref{eq:form-vector-equivalence}.
\end{definition}

These phase signs give the comparison used by De~Philippis and Paolini
\cite{DePhilippisPaolini}. Applying the Stokes argument above to a smooth
competitor \(H\) with \(H\triangle E\Subset U\Subset\Omega\) gives
\begin{align}
 \Per(H;U)-\Per(E;U)
 &\geq\int_U(\one_H-\one_E)\diver X\,dx\notag\\
 &=\int_{H\triangle E}q_X\,dx\geq0,
 \qquad q_X=(1-2\one_E)\diver X.
 \label{eq:smooth-subcalibration-Stokes}
\end{align}
The boundary flux is saturated for \(E\) and bounded by the perimeter
for \(H\). Below we extend this comparison to singular and
phase-dependent fields.

\begin{remark}[Historical background on quadratic cones]
\label{rem:subcalibration-history}
Fleming's regularity theorem in $\R^3$ \cite{Fleming62} and
Almgren's theorem in $\R^4$ \cite{Almgren66} exclude nonflat
area-minimizing hypercones in those dimensions.
Simons established stability of the balanced quadratic cones in
\(\R^{2m}\), \(m\geq4\), \cite{Simons68}; Bombieri--De~Giorgi--Giusti
proved their area minimality \cite[Sections~II--III]{BDGG69}.
The homogeneous barriers in \cite[Section~IV, p.~260]{BDGG69}
occur in their construction of entire minimal graphs, separately from
the cone-minimization proof. Lawson reduced the equivariant Plateau
problem to weighted orbit spaces and constructed exact calibrating
one-forms through explicit homogeneous trial functions
\cite[Section~5, pp.~244--247]{Lawson72}.
In his 1973 Berkeley thesis, Sim\~oes completed the remaining
quadratic cases: the cone over the minimal product
$S^1\times S^5$ is not area minimizing, whereas the cone over
the minimal product $S^2\times S^4$ is area minimizing
\cite{Simoes73}. Here the factors have the radii making their
product minimal in the unit sphere.
Miranda and collaborators developed new proofs through the
nonparametric minimal surface equation. In particular,
Massari--Miranda used graphical subsolutions and cylinder limits
\cite{MassariMiranda83}; the quartic function
\(G=(|x|^4-|y|^4)/4\) already appears in this approach
\cite[pp.~146--147]{Miranda08}.
Benarros--Miranda gave two one-sided homogeneous functions
\cite[Section~2, pp.~46--47]{BenarrosMiranda95}.
After orienting both normalized gradient fields consistently, their
construction gives potentials \(G_\pm=a_\pm F\) with positive,
phase-dependent multipliers.
De~Philippis--Paolini combined the earlier approaches in a short
subcalibration proof, using divergence inequalities directly in the
ambient space \cite[pp.~234--235]{DePhilippisPaolini}.
De~Philippis--Maggi used the exact perimeter-deficit identity
\cite[Proposition~4.1, equation~(4.5)]{DePhilippisMaggi14}
to obtain quantitative stability estimates for Lawson cones.
Liu completed their six exceptional parameter pairs using different
subcalibrations \cite[Theorems~1--2 and Section~2.1]{Liu19},
explicitly acknowledging the inspiration of Benarros--Miranda
\cite[p.~1002]{Liu19}.
\end{remark}

\subsection{Comparison across removable exceptional sets}
\label{sec:comparison}

The Stokes comparison \eqref{eq:smooth-subcalibration-Stokes}
extends across locally \(\cH^{N-1}\)-null singular sets,
also for different phase potentials.

\subsection*{Boundary flux away from an exceptional set}

For a locally finite-perimeter set \(E\), a relatively closed set
\(Z\), and \(X\in C^1(\Omega\setminus Z;\R^N)\), the boundary
condition is \(X\cdot\nu_E=1\) almost everywhere on
\(\partial^*E\setminus Z\). Extending \(\eta X\) by zero for
\(\eta\in C_c^1(\Omega\setminus Z)\), Gauss--Green gives
\[
 \int_E\diver(\eta X)\,dx
 =\int_{\partial^*E\setminus Z}\eta\,X\cdot\nu_E\,d\cH^{N-1}.
\]
To extend this calculation across \(Z\), we use cutoffs below.
For singular subcalibration fields in a Riemannian metric, see
\cite[Fact~2.3 and Section~7]{LiuSimon26}.

\subsection*{Cutoffs near the singular set}

\begin{lemma}[Cutoffs near sets of zero Hausdorff measure]
\label{lem:null-set-cutoffs}
Let \(Z\subset\Omega\subset\R^N\) be relatively closed and satisfy
\begin{equation}\label{eq:H-null}
 \cH^{N-1}(Z\cap K)=0
 \quad\text{for every compact }K\Subset\Omega.
\end{equation}
For every compact \(K\Subset U\Subset\Omega\), every open set
\(O\) with \(Z\cap K\subset O\Subset U\), and every \(\eps>0\),
there is \(\psi\in C_c^1(O)\) such that \(0\leq\psi\leq1\),
\(\psi=1\) on a neighborhood of \(Z\cap K\), and
\begin{equation}\label{eq:null-set-cutoff}
 \int_U\bigl(\psi+|\nabla\psi|\bigr)<\eps.
\end{equation}
\end{lemma}

\begin{proof}
Cover the compact set \(Z\cap K\) by finitely many balls
\(B_{r_i}(x_i)\) whose doubled balls lie in \(O\), with
\(r_i<1\) and \(\sum_i r_i^{N-1}<\eps/(2C_N)\).
Choose smooth cutoffs \(\psi_i\) supported in the doubled balls,
equal to one on the covering balls, with \(0\leq\psi_i\leq1\)
and \(|\nabla\psi_i|\leq C/r_i\).
Then \(\psi=1-\prod_i(1-\psi_i)\) equals one near \(Z\cap K\),
and satisfies \(\psi\leq\sum_i\psi_i\) and
\(|\nabla\psi|\leq\sum_i|\nabla\psi_i|\).  Hence
\[
 \int_U\bigl(\psi+|\nabla\psi|\bigr)
 \leq C_N\sum_i(r_i^N+r_i^{N-1})<\eps.\qedhere
\]
\end{proof}

We extend the quantitative subcalibration identity of De Philippis--Maggi
\cite[Proposition~4.1, equation~(4.5)]{DePhilippisMaggi14}, also used in
\cite[Section~2, equation~(3)]{Liu19}, to bounded fields that are \(C^1\)
outside a relatively closed, locally \(\cH^{N-1}\)-null set.

\Needspace{7\baselineskip}
\begin{theorem}[Two-sided comparison]
\label{thm:removable-subcalibration}
Let \(E\subset\Omega\) have locally finite perimeter, and let
\(Z\subset\Omega\) be relatively closed with
\(\cH^{N-1}(Z\cap K)=0\) for every compact \(K\Subset\Omega\).
Suppose
\(X\in C^1(\Omega\setminus Z;\R^N)\) satisfies
\begin{enumerate}[label=\textup{(\roman*)}]
 \item \(|X|\leq1\);
 \item \(X\cdot\nu_E=1\) for \(\cH^{N-1}\)-almost every point of
 \(\partial^*E\setminus Z\);
 \item \(\diver X\leq0\) almost everywhere in \(E\setminus Z\), and
 \(\diver X\geq0\) almost everywhere in \((\Omega\setminus E)\setminus Z\).
\end{enumerate}
Define \(q_X=(1-2\one_E)\diver X\geq0\) off \(Z\) and set it to zero on
\(Z\).  Then \(q_X\in L^1_{\rm loc}(\Omega)\).  For every
locally finite-perimeter competitor \(H\) with \(H\triangle E\Subset U\Subset\Omega\),
\begin{align}
 \Per(H;U)-\Per(E;U)
 &=\int_{\partial^*H\cap U}(1-X\cdot\nu_H)\,d\cH^{N-1}
   +\int_{H\triangle E}q_X\,dx.\label{eq:quantitative-comparison-identity}
\end{align}
The values of \(X\) on \(Z\) do not affect the boundary integral.
In particular,
\begin{equation}\label{eq:quantitative-comparison}
 \Per(H;U)-\Per(E;U)\geq\int_{H\triangle E}q_X\,dx.
\end{equation}
Thus \(E\) is locally perimeter minimizing and
\(\partial\llbracket E\rrbracket\) is locally mass minimizing.
If \(q_X>0\) almost everywhere, equality in perimeter comparison forces
\(H=E\) up to a null set.  When \(|X|=1\), the first integral in
\eqref{eq:quantitative-comparison-identity} is
\(\frac12\int_{\partial^*H\cap U}|\nu_H-X|^2\,d\cH^{N-1}\).
\end{theorem}

\begin{proof}
Choose \(K_0\Subset K_1\Subset U\) containing the support of
\(\one_H-\one_E\) in \(K_0\), and put
\(\mu=|D\one_E|+|D\one_H|\).
The null-set assumption and \eqref{eq:finite-perimeter-definition}
give \(\mu(Z\cap K_1)=0\).
Take open neighborhoods \(O_j\Subset U\) of \(Z\cap K_1\) with
\(\mu(O_j)\to0\), and use \cref{lem:null-set-cutoffs} to choose
\(\psi_j\in C_c^1(O_j)\) with
\(\int_U(\psi_j+|\nabla\psi_j|)\to0\).
Pass to a subsequence so that \(\psi_j\to0\) almost everywhere, and set
\(\eta_j=1-\psi_j\).  Fix \(\chi\in C_c^1(U)\) equal to one near
\(K_0\), with support in \(K_1\).
The field \(\chi\eta_jX\), extended by zero near \(Z\), is an admissible
\(C^1\) field.  The distributional derivative identity
\eqref{eq:finite-perimeter-definition}, applied to
\(\one_H-\one_E\), gives
\begin{align}
 \int_{H\triangle E}\eta_jq_X
 &=\int_{\partial^*H}\chi\eta_j X\cdot\nu_H
   -\int_{\partial^*E}\chi\eta_j
   -\int_U(\one_H-\one_E)X\cdot\nabla\eta_j.
 \label{eq:BV-cutoff-comparison}
\end{align}
We have used \(\chi=1\) on the difference set and
\((\one_H-\one_E)\diver X=q_X\) there.  The final integral is bounded in
absolute value by \(\int_U|\nabla\psi_j|\).
The boundary integrals converge because their cutoff errors are at most
\(\mu(O_j)\). Thus the right-hand side of
\eqref{eq:BV-cutoff-comparison} has a finite limit. Since
\(0\leq\eta_jq_X\to q_X\) almost everywhere, Fatou's lemma gives
\[
 \int_{H\triangle E}q_X
 \leq\liminf_j\int_{H\triangle E}\eta_jq_X<\infty.
\]
Now \(0\leq\eta_jq_X\leq q_X\), so dominated convergence gives
equality in the limit of \eqref{eq:BV-cutoff-comparison}.
Since \(H=E\) outside \(K_0\), the boundary fluxes cancel there.
Consequently
\[
 \int_{H\triangle E}q_X
 =\int_{\partial^*H\cap U}X\cdot\nu_H-\Per(E;U),
\]
which proves \eqref{eq:quantitative-comparison-identity}.
For almost every ball \(B\Subset\Omega\), both
\(H=E\setminus B\) and \(H=E\cup B\) have locally finite perimeter.
The integrability just proved for these two comparisons gives integrability
of \(q_X\) on the two parts of \(B\), hence local integrability in
\(\Omega\).  The current conclusion follows from the equivalence in
\cref{sec:prelim-minimization}; the other conclusions follow from
nonnegativity.
\end{proof}

\begin{remark}[Semialgebraic singular sets]
A closed semialgebraic set of dimension at most \(N-2\) satisfies
\eqref{eq:H-null}: its finite smooth stratification
\cite[Sections~2.3 and~2.8]{BochnakCosteRoy98} has strata of zero
\(\cH^{N-1}\)-measure.
\end{remark}

\begin{remark}[Finite perimeter of polynomial sublevel sets]
\label{rem:polynomial-finite-perimeter}
Let \(F\) be a nonzero real polynomial of degree \(d\), and put
\(E=\{F<0\}\) and \(Q_L=(-L,L)^N\).
On each coordinate line, the restriction of \(F\) is either zero
or a nonzero univariate polynomial of degree at most \(d\).
Thus the restriction of \(\one_E\) has one-dimensional variation
at most \(d\). For \(\varphi\in C_c^1(Q_L)\), Fubini and
one-dimensional integration by parts give
\[
 \left|\int_{Q_L}\one_E\,\partial_i\varphi\,dx\right|
 \leq d(2L)^{N-1}\|\varphi\|_\infty.
\]
Each distributional derivative \(D_i\one_E\) is therefore a finite
measure on \(Q_L\), and
\[
 \Per(E;Q_L)\leq\sum_{i=1}^N|D_i\one_E|(Q_L)
 \leq Nd(2L)^{N-1}.
\]
In particular, polynomial sublevel sets have locally finite perimeter.
This also applies to \(\{G<0\}\) whenever \(G\) has the same sign
almost everywhere as a polynomial.

We also recall that \(\{F=0\}\) has Lebesgue measure zero.
Write \(F\) as a polynomial in one coordinate. Outside the zero
set of a nonzero coefficient polynomial, its fibers are finite.
Induction on \(N\) and Fubini prove the assertion.

If both signs of \(F\) occur in every neighborhood of each zero, then
\[
 \partial E=\operatorname{spt}\partial\llbracket E\rrbracket
 =\{F=0\}.
\]
Indeed, both phases have positive volume in every ball centered at
a zero. The perimeter in such a ball cannot vanish, since otherwise
\(\one_E\) would be constant almost everywhere there. Away from
the zero set, \(\one_E\) is locally constant.
For nonzero harmonic \(F\), the strong maximum principle gives this
sign condition.
\end{remark}

\begin{corollary}[Gradient subcalibration criterion]
\label{cor:gradient-subcalibration}
Let \(G\in C^2(\Omega\setminus Z)\cap C^0(\Omega)\), assume that
\(E=\{G<0\}\) has locally finite perimeter, and suppose that \(Z\) contains
the critical set of \(G\) and the singular part of \(\{G=0\}\).
If \(Z\) is relatively closed, \(\cH^{N-1}(Z\cap K)=0\) for every
compact \(K\Subset\Omega\), and
\[
 G\cL(G)\geq0\qquad\text{on }\Omega\setminus Z,
\]
then \(E\) is locally perimeter minimizing.
\end{corollary}

\begin{proof}
Use \(X=\nabla G/|\nabla G|\) and
\eqref{eq:div-normalized-gradient}.  On \(G<0\) its divergence is
nonpositive, and on \(G>0\) it is nonnegative.  Along the regular zero set it
is the exterior unit normal.
\end{proof}

We call \(G\) a \emph{sub-calibration potential} for \(E=\{G<0\}\)
when its normalized gradient satisfies \cref{def:subcalibration}, or
\cref{thm:removable-subcalibration} when an exceptional set is present.
\Cref{def:homogeneous-subcalibration-potential} specifies the homogeneous
potentials used for polynomial defining equations.

\paragraph{Applying the criterion to homogeneous potentials.}
In the algebraic applications, \(G=aF\) has positive degree
\(\ell\), and \(a\) is smooth and positive outside a closed conical
exceptional set \(Z\) containing the singular zeros of \(F\).
Once \(G\) extends continuously with the signs of \(F\), its negative
phase has locally finite perimeter by
\cref{rem:polynomial-finite-perimeter}. Moreover, outside \(Z\),
\[
 F\neq0\ \Longrightarrow\
 \langle x,\nabla G\rangle=\ell G\neq0,
 \qquad
 F=0\ \Longrightarrow\ \nabla G=a\nabla F\neq0.
\]
It remains to check the null-set condition on \(Z\) and the global
sign \(G\cL(G)\geq0\). The exceptional sets and their codimensions
are collected in \cref{app:exceptional-set-summary}.

\subsection*{One-sided comparison and two different fields}

A set \(E\) is \emph{inward minimizing} if perimeter comparison holds
for competitors contained in \(E\), and \emph{outward minimizing} if
it holds for competitors containing \(E\).  All changes are compactly
supported. These are the subminimality and complementary subminimality
of \cite[Definition~1.1]{DePhilippisPaolini}. The paired comparison
retains both divergence remainders.

\begin{theorem}[One-sided and paired comparison]
\label{thm:paired-subcalibration}
Let \(E\subset\Omega\) have locally finite perimeter, and let
\(Z\subset\Omega\) be relatively closed with
\(\cH^{N-1}(Z\cap K)=0\) for every compact \(K\Subset\Omega\).
Suppose \(X_-,X_+\in C^1(\Omega\setminus Z;\R^N)\) satisfy
\[
 |X_\pm|\leq1,\qquad X_\pm\cdot\nu_E=1
 \quad\cH^{N-1}\text{-a.e. on }\partial^*E\setminus Z.
\]
Assume only the respective phase signs
\[
 \diver X_-\leq0\text{ in }E\setminus Z,\qquad
 \diver X_+\geq0\text{ in }(\Omega\setminus E)\setminus Z.
\]
Define, with zero values on \(Z\),
\[
 q_-=-\one_E\diver X_-,\qquad
 q_+=\one_{\Omega\setminus E}\diver X_+.
\]
Then \(q_-,q_+\in L^1_{\rm loc}(\Omega)\), and the following
conclusions hold for changes compactly supported in \(U\Subset\Omega\).
\begin{enumerate}[label=\textup{(\roman*)}]
 \item The assumptions on \(X_-\) alone imply
 \[
  \Per(H;U)-\Per(E;U)\geq\int_{E\setminus H}q_-
  \qquad(H\subset E).
 \]
 The assumptions on \(X_+\) alone imply the corresponding estimate
 \[
  \Per(H;U)-\Per(E;U)\geq\int_{H\setminus E}q_+
  \qquad(H\supset E).
 \]
 \item Together they imply, for every locally finite-perimeter \(H\),
 \begin{equation}\label{eq:paired-weighted-comparison}
  \Per(H;U)-\Per(E;U)
  \geq\int_{E\setminus H}q_-+\int_{H\setminus E}q_+.
 \end{equation}
 Thus \(E\) is locally perimeter minimizing, and its oriented boundary
 is locally mass minimizing among integral currents.  If each weight
 is positive almost everywhere on its phase, equality implies
 \(H=E\) up to a null set.
\end{enumerate}
\end{theorem}

\begin{proof}
For \(H\subset E\), the cutoff proof of
\cref{thm:removable-subcalibration} uses the divergence sign only on
\(E\setminus H\). With \(X_-\), it gives
\[
 \Per(H;U)-\Per(E;U)
 =\int_{\partial^*H\cap U}(1-X_-\cdot\nu_H)
  +\int_{E\setminus H}q_-.
\]
The analogous identity with \(X_+\) holds for \(H\supset E\).
Their nonnegative boundary terms prove \textup{(i)}; taking
\(H=E\setminus B\) and \(H=E\cup B\) for almost every relatively
compact ball proves local integrability of the respective weights.

For an unrestricted competitor, apply these inequalities to
\(H\cap E\) and \(H\cup E\).  The perimeter inequality for
unions and intersections \cite{Maggi12} gives
\[
 \Per(H\cap E;U)+\Per(H\cup E;U)
 \leq\Per(H;U)+\Per(E;U).
\]
Adding the two one-sided estimates proves
\eqref{eq:paired-weighted-comparison}.  The current conclusion follows
from the equivalence in \cref{sec:prelim-minimization}, and positivity
gives the equality assertion.
\end{proof}

\begin{remark}[Fields defined on their respective sides]
\label{rem:phase-local-fields}
When \(E\) is open and \(\partial E\setminus Z\) is smooth, it is
enough to construct \(X_-\) on an open neighborhood of
\(\overline E\setminus Z\), and \(X_+\) on an open neighborhood of
\(\overline{\Omega\setminus E}\setminus Z\), relative to
\(\Omega\setminus Z\).  A smooth cutoff equal to one on the relevant
closed phase extends each bounded field by zero to the unused side.
This preserves its norm, boundary values, and the divergence on its
own phase.  In a conical setting the cutoff may be chosen on the unit
sphere, so homogeneity is preserved.  No matching of the two potentials
is required; their normalized gradients both equal \(\nu_E\) on the
regular boundary.

Liu's subcalibrations for the six exceptional Lawson-cone parameter pairs
\cite[Sections~2.1 and~3]{Liu19} use powers of different block radii,
and sometimes different homogeneous degrees, on the two sides.
The positive multipliers \(G_\pm=a_\pm F\) give the common normal
\(\nu_E\) along the regular cone, so
\cref{thm:paired-subcalibration} applies.

Phase-dependent multipliers also appear in Liu's application of
Simon's analytic estimates to hypersurfaces with prescribed closed
singular sets \cite[Theorem~1 and Definition~1.3]{LiuSimon26}.
\end{remark}

\subsection*{Strict minimality and distance-weighted comparison}

\begin{theorem}[Strict minimality from two-sided divergence integrals]
\label{thm:strict-minimality-flux}
Let \(E\subset\R^{n+1}\), \(n\geq2\), be conical, and suppose the hypotheses
of \cref{thm:removable-subcalibration} hold globally with a conical
exceptional set \(Z\) and a zero-homogeneous field \(X\).
Suppose
\begin{equation}\label{eq:two-phase-flux}
 A_-:=\int_{B_1\cap E}q_X\,dx>0,
 \qquad
 A_+:=\int_{B_1\setminus E}q_X\,dx>0.
\end{equation}
Then \(C=\partial\llbracket E\rrbracket\) is strictly area minimizing in
\cref{def:four-variational-notions}, with
\begin{equation}\label{eq:strict-minimality-constant}
 \Theta=\min\{A_-,A_+\}.
\end{equation}
In particular the assertion permits nonisolated singularities.
\end{theorem}

\begin{proof}
Both integrals are finite by \cref{thm:removable-subcalibration}.
Homogeneity gives \(q_X(rx)=r^{-1}q_X(x)\), so their values over
\(B_\eps\) are \(\eps^nA_-\) and \(\eps^nA_+\).
If a set competitor \(H\) has no perimeter in the connected ball
\(B_\eps\), then \(\one_H\) is constant there.  Accordingly,
\(H\triangle E\) contains either \(B_\eps\cap E\) or
\(B_\eps\setminus E\), up to a null set.
Equation \eqref{eq:quantitative-comparison} proves the desired mass gap
for set-boundary competitors.

For an arbitrary integral-current competitor, consider the nearest-point
projection onto the closed unit ball,
\[
 \pi(x)=\frac{x}{\max\{1,|x|\}}.
\]
This map is \(1\)-Lipschitz, fixes \(\partial C_1\), and maps
\(\R^{n+1}\setminus B_\eps\) into itself.
Although \(\pi\) is not proper, the integral representation of the
finite-mass current \(T\) defines \(\pi_\#T\), since the
\(n\)-Jacobian of \(\pi\) is bounded by one.
The pushforward mass estimate and the identity
\(\partial(\pi_\#T)=\pi_\#(\partial T)\)
give
\[
 \mathbf M(\pi_\#T)\leq\mathbf M(T),
 \qquad \partial(\pi_\#T)=\partial C_1;
\]
see \cite[Remark~26.21(1), Lemma~26.25, and
Remark~27.2(3)]{Simon83}.  For finite-mass currents with noncompact
support, these statements follow from the compactly supported case by
cutoffs: the additional boundary term is bounded by
\(O(L^{-1})\mathbf M(T)\) for a cutoff at radius \(L\), and tends to
zero.  The same projection is used in
\cite[proof of Proposition~1.2]{DeLellisSpadaroSpolaor17}.
We may therefore replace \(T\) by \(\pi_\#T\) and assume
\(\operatorname{spt}T\subset\overline B_1\).

Since \(\partial(T-C_1)=0\), coning \(T-C_1\) to the origin gives an
integral \((n+1)\)-current \(R\) with
\(\partial R=T-C_1\) and support in \(\overline B_1\)
\cite[equation~(26.26) and Remark~27.2(2)--(3)]{Simon83}.
A top-dimensional integral current is integration against an
integer-valued multiplicity.  Since its boundary has finite mass, this
multiplicity has bounded variation
\cite[Remark~26.28]{Simon83}.  Thus
\[
 R=\llbracket k\rrbracket,
 \qquad k\in BV_c(\R^{n+1};\mathbb Z),
 \qquad v:=\one_E+k\in BV_{\mathrm{loc}}(\R^{n+1};\mathbb Z).
\]
We have \(\partial\llbracket v\rrbracket=T+C-C_1\).
In \(B_\eps\), the right-hand side vanishes because \(T\) avoids
this ball and \(C=C_1\) there.  The Constancy Theorem
\cite[Theorem~26.27]{Simon83}, applied on the connected
open set \(B_\eps\), therefore gives \(v=m\) almost everywhere
there for some \(m\in\mathbb Z\).

Set \(H=\{v\geq1\}\). The integer-level decomposition
\cite[Theorem~27.6]{Simon83} gives, for every relatively compact
open set \(U\subset\R^{n+1}\),
\[
 |Dv|(U)=\sum_{j\in\mathbb Z}
       \Per(\{v\geq j\};U)
 \geq\Per(H;U).
\]
Thus \(H\) has locally finite perimeter. Moreover,
\(H=E\) outside \(\overline B_1\), and \(\one_H\) is constant
in \(B_\eps\).
Since \(\|C\|(\partial B_1)=0\), the mass measures of \(T\) and
\(C-C_1\) are mutually singular.  Hence, for \(\rho>1\),
\[
 \Per(H;B_\rho)\leq |Dv|(B_\rho)
 =\mathbf M(T)+\Per(E;B_\rho)-\mathbf M(C_1).
\]
Apply the set comparison in \(B_\rho\) and cancel the common exterior
perimeter.  This yields
\eqref{eq:strict-minimality-definition} with
\eqref{eq:strict-minimality-constant}.
\end{proof}

\begin{corollary}[Strict mass gap from paired fields]
\label{cor:paired-strict-minimality}
Suppose \cref{thm:paired-subcalibration} holds globally for a conical
set \(E\subset\R^{n+1}\), \(n\geq2\), a conical exceptional set \(Z\),
and fields that are zero-homogeneous on their respective phases.
If
\[
 A_-:=\int_{B_1\cap E}q_->0,\qquad
 A_+:=\int_{B_1\setminus E}q_+>0,
\]
then every integral-current competitor in
\eqref{eq:strict-minimality-definition} satisfies
\[
 \mathbf M(T)-\mathbf M(C_1)
 \geq\min\{A_-,A_+\}\eps^n.
\]
\end{corollary}

\begin{proof}
The degree-\(-1\) weights are locally integrable by
\cref{thm:paired-subcalibration}. A set competitor avoiding \(B_\eps\)
fills or empties it, so \eqref{eq:paired-weighted-comparison} gives the
gap. The current-to-set reduction in the proof of
\cref{thm:strict-minimality-flux} preserves ball avoidance and does
not increase the mass excess.
\end{proof}

\begin{corollary}[Polynomial potentials with nonzero divergence]
\label{cor:polynomial-strict-minimality}
Let \(N\geq3\).  Suppose a homogeneous polynomial \(G\) satisfies the global hypotheses of
\cref{cor:gradient-subcalibration}.  If both \(\{G<0\}\) and \(\{G>0\}\)
are nonempty and \(\cL(G)\not\equiv0\), then
\(\partial\llbracket\{G<0\}\rrbracket\) is strictly area minimizing.
\end{corollary}

\begin{proof}
The normalized gradient is zero-homogeneous.  Its critical set is
conical and satisfies \eqref{eq:H-null}.  Off this set,
\[
 q_X=\frac{|\cL(G)|}{|\nabla G|^3}.
\]
A nonzero polynomial has a Lebesgue-null zero set
by \cref{rem:polynomial-finite-perimeter}. Thus \(q_X>0\) almost everywhere on both open
phases.  Equation \eqref{eq:two-phase-flux} follows, and
\cref{thm:strict-minimality-flux} applies.
\end{proof}

For nonpolynomial potentials, verify positive flux in both phases and
apply \cref{thm:strict-minimality-flux}, as for the radical potentials in
\cref{ex:Riedler-correct-power,sec:complex-determinants}.

For regular area-minimizing hypercones, Niu \cite[Theorem~3]{Niu26}
characterizes simultaneous strict stability and strict minimality by
\eqref{eq:distance-weighted-comparison}. We prove sufficiency across
the singular sets of \cref{thm:removable-subcalibration}, using the
quantitative subcalibration method of De Philippis--Maggi
\cite[Section~4]{DePhilippisMaggi14} and Liu
\cite[Section~2.2]{Liu19}.

\begin{proposition}[Distance-weighted sufficient criterion]
\label{prop:distance-weighted-comparison}
Under the geometric and comparison hypotheses of
\cref{thm:strict-minimality-flux}, suppose the cone is two-sided and minimal on its regular part and, for some \(c>0\),
\begin{equation}\label{eq:linear-distance-divergence}
 q_X(x)\geq c\,\frac{\dist(x,C)}{|x|^2}
 \quad\text{for almost every }x\notin C.
\end{equation}
Assume each phase contains points at positive distance from \(C\).
Then the cone is strictly area minimizing and strictly stable, and
\begin{equation}\label{eq:distance-weighted-comparison}
 \Per(H;U)-\Per(E;U)
 \geq c\int_{H\triangle E}\frac{\dist(x,C)}{|x|^2}\,dx.
\end{equation}
More precisely, \(Q_C(u)\geq c\int_Cr^{-2}u^2\) for the regular-part
test class in \cref{def:four-variational-notions}.
\end{proposition}

\begin{proof}
The weighted comparison follows from
\eqref{eq:quantitative-comparison}. The two phase integrals are
positive, so \cref{thm:strict-minimality-flux} gives strict minimality.
For a smooth compactly supported regular-part function \(u\), perturb the
boundary by the normal graph \(tu\) in a fixed regular tubular neighborhood.
Minimality and tubular coordinates give
\begin{align*}
 \Per(E_{tu};U)-\Per(E;U)&=\tfrac12t^2Q_C(u)+o(t^2),\\
 \int_{E_{tu}\triangle E}\frac{\dist(x,C)}{|x|^2}\,dx
 &=\tfrac12t^2\int_Cr^{-2}u^2+o(t^2).
\end{align*}
Indeed, on each normal segment the leading integral is
\(r^{-2}\int_0^{|tu|}s\,ds\); its tubular Jacobian and the radial weight
contribute only higher-order terms on this compact support.
Divide \eqref{eq:distance-weighted-comparison} by \(t^2/2\) and let
\(t\to0\).
\end{proof}

\subsection{Algebraic tangent cones and local instability}
\label{sec:tangent-instability-transfer}

Let $C=\{F=0\}$ for a nonzero polynomial $F$, and let $p\in C$.
Writing $F(p+z)=H_m(z)+\cdots+H_d(z)$ with $H_m\not\equiv0$
homogeneous of degree $m$, we have
\[
 r^{-m}F(p+rz)\longrightarrow H_m(z)
 \quad\text{in }C^\infty_{\mathrm{loc}}\quad(r\downarrow0).
\]
The implicit function theorem gives smooth graphical convergence of
$(C-p)/r$ near compact subsets of $\{H_m=0,\ \nabla H_m\ne0\}$.
If $H_m$ changes sign near every zero, then $\{H_m=0\}$ is the
unique local Hausdorff tangent: uniform convergence gives one inclusion,
and sign change gives the other. This condition holds when $F$ is harmonic
or $H_m$ is an indefinite quadratic form. Thus an indefinite Hessian at
a critical point gives the quadratic tangent
$\{z:\langle\Hess F(p)z,z\rangle=0\}$.
For homogeneous $F$ and $p\ne0$, Euler's identity gives $D_pH_m=0$,
so this leading zero set splits off the line $\R p$.
The cubic Hessian criterion in \cref{prop:cubic-general-singular-spectrum}
identifies these quadratic factors.

\begin{lemma}[Instability under a regular tangent limit]
\label{lem:tangent-instability-transfer}
Let $C$ be an $n$-dimensional hypersurface, minimal on its two-sided
regular part. Suppose that $C_j=(C-p)/r_j$, $r_j\downarrow0$,
converges locally graphically in $C^2$ on compact regular subsets of
$\Sigma=Q\times\R^k$, where $Q$ is a two-sided minimal cone.
If $Q$ is unstable on its regular part, then $C$ admits a compactly
supported regular-part variation with negative second variation
in every neighborhood of $p$.
\end{lemma}

\begin{proof}
Write $Q_\Sigma(\phi)=\int_\Sigma
(|\nabla_\Sigma\phi|^2-|A_\Sigma|^2\phi^2)$.
Choose $\phi\in C_c^\infty(Q_{\reg})$ with $Q_Q(\phi)<0$.
For $k>0$, take a nonzero $\chi\in C_c^\infty(\R^k)$ and
set $\chi_T(w)=\chi(w/T)$. The product metric gives
\[
 Q_{Q\times\R^k}(\phi\chi_T)
 =T^k\left(\int_{\R^k}\chi^2\right)Q_Q(\phi)
 +T^{k-2}\left(\int_{\R^k}|\nabla\chi|^2\right)
             \left(\int_Q\phi^2\right)<0
\]
for sufficiently large $T$; for $k=0$ use $\phi$ itself.
Fix the resulting test function $\psi$ and a relatively compact regular
neighborhood of its support. Transport $\psi$ by the graphical
parametrizations to $\psi_j$ on $C_j$, extending by zero.
The induced metrics and second fundamental forms converge, so
$Q_{C_j}(\psi_j)\to Q_\Sigma(\psi)<0$.
For $u_j(x)=\psi_j((x-p)/r_j)$,
\[
 Q_C(u_j)=r_j^{\,n-2}Q_{C_j}(\psi_j)<0.
\]
Its support lies in the regular part and shrinks to $p$.
\end{proof}

A leading form $a(|u|^2-|v|^2)$ in orthogonal coordinates,
$a\ne0$, with $u,v\in\R^s$ and $s=2,3$, gives an unstable
quadratic factor: $|A|^2=(2s-2)/\rho^2$, where
$\rho^2=|u|^2+|v|^2$, and the sharp radial Hardy constant is
$(2s-3)^2/4<2s-2$.
For example, the minimal cubic $F=\operatorname{Re}(z_1z_2z_3)$
on $\mathbb C^3$ has leading term $\operatorname{Re}(z_2z_3)$
at $p=(1,0,0)$, giving the tangent
$\R^2\times\{(u,v)\in\R^2\oplus\R^2:|u|=|v|\}$.
The lemma proves local instability at $p$; see
\cref{sec:tangents} for the Clifford and Hurwitz cubic models.

\section{Laplacian Algebras and Quotient Calculus}
\label{sec:laplacian-calculus}

\subsection{Laplacian algebras: differential closure and finite generation}
\label{sec:laplacian-foundations}

Following Mendes--Radeschi \cite[Section~5]{MendesRadeschi}, we develop
closure and finite generation to reduce differential calculations to
finite tables and identify the smallest algebra containing a defining polynomial.

Let \(V=\R^N\) have its Euclidean inner product, let \(\R[V]\) be its
ring of real polynomial functions, and put \(R(x)=|x|^2\).
For smooth functions define
\[
 \Gamma(f,g)=\langle\nabla f,\nabla g\rangle.
\]
All polynomial algebras contain the constants; their generators need
not be algebraically independent.

\begin{definition}[Mendes--Radeschi]
\label{def:laplacian-algebra}
A subalgebra \(\cA\subset\R[V]\) is a \emph{Laplacian algebra} if
\(R\in\cA\) and \(\Delta f\in\cA\) for every \(f\in\cA\).
\end{definition}

This definition is orthogonally invariant and requires no group action.

\begin{lemma}[Gradient closure and homogeneity]
\label{lem:laplacian-grading}
A Laplacian algebra is closed under \(\Gamma\) and contains every
homogeneous component of each of its elements. In particular,
\(\cA=\bigoplus_{d\geq0}\cA_d\), where
\(\cA_d=\cA\cap\R[V]_d\) and \(\R[V]_d\) denotes the homogeneous
polynomials of degree \(d\).
\end{lemma}

\begin{proof}
The product rule gives
\begin{equation}\label{eq:Gamma-polarization}
 2\Gamma(f,g)=\Delta(fg)-f\Delta g-g\Delta f.
\end{equation}
Thus \(\Gamma(f,g)\in\cA\) for \(f,g\in\cA\). The Euler operator
\(E=x\cdot\nabla\) consequently preserves \(\cA\), since
\(Ef=\Gamma(R,f)/2\). If \(f=\sum_{j=0}^D f_j\) is its homogeneous
decomposition, then \(Ef_j=jf_j\). For each \(0\leq d\leq D\), choose
the interpolation polynomial \(p_d\) satisfying
\(p_d(j)=\delta_{dj}\) for \(j=0,\ldots,D\). Then
\(f_d=p_d(E)f\in\cA\).
\end{proof}

\begin{example}[Orthogonal invariants]
\label{ex:orthogonal-laplacian}
If a group \(H\subset O(V)\) acts orthogonally, its invariant algebra
\(\R[V]^H\) is Laplacian. Indeed, \(R\) is invariant and
\(\Delta(f\circ h)=(\Delta f)\circ h\) for \(h\in H\).
For compact \(H\), normalized Haar averaging gives
\[
 \Pi_Hg(x)=\int_H g(hx)\,dh,
 \qquad \Pi_H(fg)=f\Pi_H(g)\quad(f\in\R[V]^H).
\]
The projection below extends this module property to every Laplacian algebra.
\end{example}

\begin{example}[Two orthogonal blocks]
\label{ex:two-block-laplacian-background}
Let \(V=\R^p\oplus\R^q\), where \(p,q\geq1\), and write
\(x\in\R^p\), \(y\in\R^q\), \(u=|x|^2\), \(v=|y|^2\). Then \(R=u+v\), and
\[
 \Delta u=2p,\quad \Delta v=2q,\qquad
 \Gamma(u,u)=4u,\quad \Gamma(u,v)=0,\quad \Gamma(v,v)=4v.
\]
For a polynomial \(h(u,v)\), direct differentiation therefore gives \(\Delta h=4u h_{uu}+4v h_{vv}+2p h_u+2q h_v\).
Thus \(\R[u,v]\) is Laplacian, with image \([0,\infty)^2\).
\end{example}

\subsubsection{The Reynolds projection}

We recall the properties needed for finite generation, following
\cite[Lemma~21, Corollary~22, and
Theorem~23]{MendesRadeschi}.

For polynomials \(f,g\) and an integer \(j\geq0\), set
\[
 B_j(f,g)=\sum_{i_1,\ldots,i_j=1}^N
 (\partial_{i_1}\cdots\partial_{i_j}f)
 (\partial_{i_1}\cdots\partial_{i_j}g),
 \qquad B_0(f,g)=fg.
\]
The product rule and commutation of derivatives give
\begin{equation}\label{eq:MR-higher-contraction}
 2B_{j+1}(f,g)=\Delta B_j(f,g)-B_j(\Delta f,g)-B_j(f,\Delta g).
\end{equation}
Induction shows that \(B_j(f,g)\in\cA\) whenever \(f,g\in\cA\).
If \(f\) is homogeneous of degree \(d\), write
\(f=\sum_{|\alpha|=d}f_\alpha x^\alpha\) and define the
constant-coefficient operator
\(f(\partial)=\sum_{|\alpha|=d}f_\alpha\partial^\alpha\).
Counting ordered derivatives gives
\begin{equation}\label{eq:MR-dual-operator}
 f(\partial)g=\frac1{d!}B_d(f,g).
\end{equation}
In particular, \(f(\partial)\) takes \(\cA\) into \(\cA\).

On the finite-dimensional space \(\R[V]_d\), use the Fischer inner
product
\[
 (f,g)_d=f(\partial)g
 =\sum_{|\alpha|=d}\alpha! f_\alpha g_\alpha.
\]
For degrees \(a,b,a+b\), the coefficient formula also gives
\begin{equation}\label{eq:MR-Fischer-adjoint}
 (fg,h)_{a+b}=(g,f(\partial)h)_b.
\end{equation}

\begin{proposition}[Reynolds projection; Mendes--Radeschi]
\label{prop:MR-Reynolds}
Let \(\Pi_d:\R[V]_d\to\cA_d\) be orthogonal projection for the
Fischer inner product and let \(\Pi=\bigoplus_d\Pi_d\). Then
\(\Pi|_{\cA}=\Id\), \(\Pi\) preserves degree, and
\begin{equation}\label{eq:MR-Reynolds-property}
 \Pi(fg)=f\Pi(g)\qquad(f\in\cA,\ g\in\R[V]).
\end{equation}
\end{proposition}

\begin{proof}
For homogeneous \(f\in\cA_a\) and \(g\in\R[V]_b\), write
\(g=\Pi_b g+g^\perp\).
For \(h\in\cA_{a+b}\),
\eqref{eq:MR-dual-operator}--\eqref{eq:MR-Fischer-adjoint} give
\[
 (fg^\perp,h)_{a+b}=(g^\perp,f(\partial)h)_b=0,
\]
since \(f(\partial)h\in\cA_b\). Thus \(\Pi(fg^\perp)=0\) and
\(\Pi(f\Pi_b g)=f\Pi_b g\); sum over degrees.
\end{proof}

\begin{proposition}[Finite generation and polynomial quotients]
\label{prop:MR-finite-and-division}
Every Laplacian algebra \(\cA\) has finitely many homogeneous
generators. Moreover,
\begin{equation}\label{eq:MR-polynomial-quotients}
 \cA=\operatorname{Frac}(\cA)\cap\R[V].
\end{equation}
Equivalently, if \(f,g\in\cA\), \(g\neq0\), and \(f/g\) is a
polynomial on \(V\), then \(f/g\in\cA\).
\end{proposition}

\begin{proof}[Proof {\normalfont\cite[Lemma~24(a)--(b)]{MendesRadeschi}}]
Put \(\cA_+=\bigoplus_{d>0}\cA_d\).
Hilbert's basis theorem
\cite[Chapter~I, Section~2, Proposition~2.3 and Corollary~2.4, p.~11]{Kunz2013}
gives homogeneous \(I_1,\ldots,I_k\in\cA_+\) generating the ideal
\((\cA_+)\R[V]\). For homogeneous \(f\in\cA_d\), \(d>0\), write
\(f=\sum_i a_iI_i\) with \(\deg a_i=d-\deg I_i<d\).
The Reynolds property gives
\[
 f=\sum_i\Pi(a_i)I_i.
\]
Degree induction proves algebraic generation. If \(h=f/g\) is a
polynomial, applying \(\Pi\) to \(f=gh\) gives
\(f=g\Pi(h)\), hence \(h=\Pi(h)\in\cA\).
\end{proof}

\subsubsection{Finite closure tables and the smallest algebra}

\begin{proposition}[Checking closure on generators]
\label{prop:MR-generator-test}
Let \(I_1,\ldots,I_k\) be homogeneous polynomials and suppose that
\(R\in\cA=\R[I_1,\ldots,I_k]\). Then \(\cA\) is Laplacian if
and only if
\[
 \Delta I_i\in\cA,\qquad
 \Gamma(I_i,I_j)\in\cA\quad\text{for every }i,j.
\]
\end{proposition}

\begin{proof}
Necessity follows from \cref{lem:laplacian-grading}. For sufficiency,
the chain rule for a polynomial \(p\) gives
\[
 \Delta(p\circ I)=\sum_i p_i(I)\Delta I_i
       +\sum_{i,j}p_{ij}(I)\Gamma(I_i,I_j)\in\cA.
\]
\end{proof}

For a prescribed homogeneous polynomial \(F\), define
\begin{equation}\label{eq:minimal-Laplacian-algebra}
 \cA_R(F)=\bigcap\{\cA\subset\R[V]:\cA\text{ is Laplacian and }
                                          F\in\cA\}.
\end{equation}
The collection contains \(\R[V]\); its Laplacian intersection
\(\cA_R(F)\) is therefore the smallest Laplacian algebra containing \(R,F\).
Finite generation is uniform in fixed dimension and degree.

\begin{proposition}[Uniform finite generation]
\label{prop:uniform-finite-generation}
For every \(N,d\ge1\), there is an integer \(B=B(N,d)\) such that
every homogeneous polynomial \(F\in\R[x_1,\ldots,x_N]\) of degree
\(d\) has
\[
 \cA_R(F)=\R[I_1,\ldots,I_k],\qquad
 1\le\deg I_i\le B,\qquad
 k\le\binom{N+B}{N}-1,
\]
with homogeneous \(I_i\). In particular, these bounds apply uniformly
to irreducible degree-\(d\) equations of minimal cones in \(\R^N\),
and to the subclass of harmonic equations.
\end{proposition}

\begin{proof}
We follow \cite[proof of Lemma~24(a)]{MendesRadeschi}, using all
coefficients at once.
Put \(T=\R[t_\alpha:|\alpha|=d]\),
\(\mathscr F=\sum_{|\alpha|=d}t_\alpha x^\alpha\), and let
\(\mathscr B\subset T[x_1,\ldots,x_N]\) be the smallest
\(T\)-algebra containing \(R,\mathscr F\) and closed under
\(\Delta_x\). It is graded by \(x\)-degree: its elements are sums
of expressions obtained from these homogeneous inputs by multiplication
and \(\Delta_x\). Hilbert's basis theorem
\cite[Chapter~I, Section~2, Proposition~2.3 and Corollary~2.4, p.~11]{Kunz2013} therefore
gives positive-degree, \(x\)-homogeneous elements
\(b_1,\ldots,b_s\in\mathscr B\) generating the ideal
\(J=(\mathscr B_+)T[x]\).

For every real coefficient vector \(c\), evaluation \(t=c\)
commutes with multiplication and \(\Delta_x\). Consequently
\(\mathscr B(c)=\cA_R(F_c)=:\cA_c\): every evaluated expression
lies in \(\cA_c\), and the evaluated algebra is itself Laplacian
and contains \(R,F_c\). Evaluation also preserves \(x\)-degree,
so the image ideal is \(J(c)=(\cA_c)_+\R[x]=(b_1(c),\ldots,b_s(c))\).
For homogeneous \(f\in\cA_c\) of positive degree, write
\(f=\sum_i h_i b_i(c)\) with homogeneous coefficients of degree
\(\deg f-\deg_x b_i<\deg f\), omitting zero terms.
Apply the graded Reynolds projection of this real Laplacian algebra,
\cref{prop:MR-Reynolds}, or \cite[Theorem~23]{MendesRadeschi}:
\[
 f=\sum_i\Pi_c(h_i)b_i(c).
\]
Induction on degree shows that the nonzero \(b_i(c)\) generate
\(\cA_c\). The bound \(B=\max_i\deg_x b_i\)
therefore depends only on \(N,d\).
Finally, bases of the homogeneous pieces of \(\cA_c\) in degrees
\(1,\ldots,B\) give at most
\(\sum_{j=1}^B\binom{N+j-1}{j}=\binom{N+B}{N}-1\) generators.
\end{proof}

The uniform bound \(B(N,d)\) is nonexplicit.

\phantomsection\label{par:essential-convention}
Generator counts use the essential Euclidean space
\[
 V_{\mathrm{ess}}(F)=V_{\mathrm{cyl}}(F)^\perp,\qquad
 V_{\mathrm{cyl}}(F)=\{v:D_vF\equiv0\},
\]
unless a cylindrical extension is specified; \(R\) and \(\cA_R(F)\)
refer to this space.

Harmonic minimal equations already require generators of increasing
degree: for \(F=\operatorname{Re}\det Z\) on \(M_d(\mathbb C)\),
\(d\ge3\), the smallest algebra needs \(d+1\) generators and its
least possible maximal homogeneous generator degree is \(2d\).
Thus \(B(2d^2,d)\ge2d\). The proof is given with the determinant
calculations in \cref{ex:determinant-generator-degrees}.

Since \(S=|\nabla F|^2\in\cA_R(F)\), the generator test proves
\(\cA_R(F)=\R[R,F,S]\) whenever this algebra is Laplacian.
Otherwise quantities such as \(\Delta S\) or \(\Gamma(S,S)\) may
force additional generators.

\subsection{The quotient geometry of a Laplacian algebra}
\label{sec:laplacian-quotient-background}

Choose homogeneous generators \(I_1,\ldots,I_k\) for \(\cA\),
including \(I_1=R\), and write \(I=(I_1,\ldots,I_k)\).
Its fibers \(I^{-1}(z)\) lie on spheres and are the common levels
of all elements of \(\cA\).

The images
\[
 K=I(V),\qquad K_{\Sph}=I(\Sph^{N-1})
\]
are semialgebraic by the Tarski--Seidenberg theorem
\cite[Section~2.2]{BochnakCosteRoy98}.
The induced map identifies the topological quotient of the sphere
with \(K_{\Sph}\), since it is a continuous bijection from a compact
space to a Hausdorff space. Changing generators gives mutually inverse
polynomial maps between the images. Hence \(d_{\Sph}(\cA):=\dim K_{\Sph}\),
the \emph{spherical quotient dimension}, is intrinsic; it need not
equal \(k-1\).

By Mendes--Radeschi
\cite[Theorem~A and Theorem~25]{MendesRadeschi},
\cite[Theorem~A and Corollary~B]{MendesRadeschiMaximality}, these
fibers define a \emph{manifold submetry}: the quotient projection maps
each closed ball onto the ball of the same radius about its image, and
its fibers are smooth embedded submanifolds, possibly disconnected and
of different dimensions. Moreover, \(\cA\) contains every polynomial
constant on its fibers. The quotient metric, given by distance between
fibers, generally differs from the Euclidean metric inherited by the
coordinate image \(K\subset\R^k\).

\subsubsection{The metric and the induced differential operator}

Write a closure table as
\(\Delta I_i=b_i(I)\) and
\(\Gamma(I_i,I_j)=g_{ij}(I)\). At every \(x\in V\),
\begin{equation}\label{eq:background-quotient-Gram}
 \bigl(g_{ij}(I(x))\bigr)=dI_x\,dI_x^T.
\end{equation}
Thus \(\mathbf g\ge0\) and \(\operatorname{rank}\mathbf g
=\operatorname{rank}dI\). On a regular quotient chart, an independent
subset of the \(I_i\) has inverse metric given by its Gram submatrix:
its gradients span the orthogonal complement of the fiber. This is
the Riemannian-submersion description of
\cite[Proposition~26]{MendesRadeschi}.

The induced Laplacian has second-order coefficients \(g_{ij}\) and
first-order coefficients \(b_i\), as in \cref{thm:finite-reduction}.
It need not be the quotient Laplace--Beltrami operator: the \(b_i\)
also record fiber-volume variation.

\begin{example}[The interval quotient and its weight]
\label{ex:two-block-quotient-background}
Return to \cref{ex:two-block-laplacian-background}, with
\(N=p+q\) and \(p,q\geq2\). On the unit sphere, \(u+v=1\),
so \(t=u\in[0,1]\) is the single independent invariant. Euler's
identity and the polar-coordinate expression for \(\Delta\) give
\[
 |\nabla_{\Sph}t|^2=4t(1-t),\qquad
 \Delta_{\Sph}t=2p-2Nt.
\]
The quotient metric is therefore
\[
 ds^2=\frac{dt^2}{4t(1-t)}.
\]
The substitution \(t=\cos^2\theta\) turns it into \(d\theta^2\),
with \(0\leq\theta\leq\pi/2\). For a function \(h(t)\), however,
\[
 \Delta_{\Sph}h(t)=4t(1-t)h''(t)+(2p-2Nt)h'(t).
\]
Equivalently this operator is
\[
 \frac1{w(t)}\frac{d}{dt}
       \left(4t(1-t)w(t)\frac{dh}{dt}\right),
 \qquad w(t)=t^{p/2-1}(1-t)^{q/2-1},
\]
where \(w(t)\,dt\) is proportional to the pushforward of spherical
volume. The rank drops at the endpoints \(x=0\) and \(y=0\).
\end{example}

For the chosen generators, the finite closure table has the form
\begin{equation}\label{eq:closure-table}
 \Delta I_i=b_i(I),\qquad
 \Gamma(I_i,I_j)=g_{ij}(I).
\end{equation}
For smooth functions of \(z=(z_1,\ldots,z_k)\) on a neighborhood of
the image, write \(p_i=\partial p/\partial z_i\) and
\(p_{ij}=\partial^2p/\partial z_i\partial z_j\), and define
\[
 \widehat\Delta p=\sum_i b_i p_i+\sum_{i,j}g_{ij}p_{ij},
 \qquad
 \widehat\Gamma(p,q)=\sum_{i,j}g_{ij}p_iq_j.
\]

\begin{theorem}[Exact finite-dimensional reduction]
\label{thm:finite-reduction}
If \eqref{eq:closure-table} holds, then for \(H=p\circ I\) and
\(K=q\circ I\),
\[
 \Delta H=(\widehat\Delta p)\circ I,
 \qquad
 \Gamma(H,K)=\widehat\Gamma(p,q)\circ I.
\]
Writing \(Q_p=\widehat\Gamma(p,p)\), one has
\begin{equation}\label{eq:reduced-L}
 \boxed{
 \cL(p\circ I)=
 \left(Q_p\widehat\Delta p-
 \frac12\widehat\Gamma(p,Q_p)\right)\circ I.}
\end{equation}
Thus the curvature inequality is an exact differential inequality on the
semialgebraic image \(I(\R^N)\).
\end{theorem}

\begin{proof}
The first two identities follow from the chain rule and
\eqref{eq:closure-table}.  In particular, \(|\nabla H|^2=Q_p\circ I\), so
\[
 \Hess H(\nabla H,\nabla H)
 =\tfrac12\Gamma(H,|\nabla H|^2)
 =\tfrac12\widehat\Gamma(p,Q_p)\circ I.
\]
Together these formulas give \eqref{eq:reduced-L}.
\end{proof}

The sign is tested on the actual image \(K=I(\R^N)\). These identities
hold also at rank-drop points; for \(G=p\circ I\), \(X_G\) is defined
wherever \(Q_p\circ I>0\). The comparison theorem imposes no bound
on quotient dimension.

For polarization of these contractions, see
\cref{sec:higher-metrized-operations,sec:cubic-algebraic-theory}.
\Cref{sec:algebraic-jacobi} reduces the divergence sign to an inequality
on \(K\).

\subsection{Area minimization and irreducible defining equations}
\label{sec:irreducible-equations}

The full zero set \(C=\{F=0\}\) may differ from the support of
\(\partial\llbracket\{F<0\}\rrbracket\). For a square-free equation,
area minimization permits only one irreducible factor with a nonsingular
real zero, hence only one contributing a smooth hypersurface patch.
We use this effective equation and write
\(C_{\reg}=\{F=0,\ \nabla F\ne0\}\).

Recall that a degree-\(d\) homogeneous harmonic polynomial restricts
to a \emph{spherical harmonic} \(f=F|_{\Sph^{N-1}}\), with
\(F(r\theta)=r^df(\theta)\)
\cite[Chapter~5]{AxlerBourdonRamey01}.
The link of its zero cone is the nodal set \(\{f=0\}\) on the sphere;
\cref{fig:spherical-harmonic} illustrates this restriction.

\begin{lemma}[Square-free equations and their critical sets]
\label{lem:square-free-critical-set}
Let \(F\) be a nonzero homogeneous polynomial of positive degree on
\(\R^N\), \(N\ge2\). If \(F\) is square-free, then
\(Z=\{\nabla F=0\}\) has dimension at most \(N-2\), and hence
\(\cH^{N-1}(Z\cap K)=0\) for every compact \(K\).
If \(F\) is harmonic, it has no nonconstant nonnegative polynomial
factor; in particular, it is square-free and each of its real
irreducible factors has a nonsingular real zero.
\end{lemma}

For the harmonic square-freeness and critical-set statements, see
\cite[Proposition~2.2 and Theorem~2.1]{HardtHoffmannNadirashvili99}.
The following argument also identifies the possible real factors.

\begin{proof}
If the critical set had a component of codimension one, its
irreducible defining polynomial \(P\) would divide every partial
derivative of \(F\). Euler's identity then gives \(P\mid F\).
Writing \(F=PH\), square-freeness gives \(P\nmid H\), so
\(P\mid\partial_iF\) forces \(P\mid\partial_iP\) for every
\(i\), contradicting the degree of a nonzero derivative of \(P\).
The Hausdorff-measure assertion follows from the semialgebraic
stratification used in \cref{sec:comparison}.

For the harmonic assertion, factors of a homogeneous polynomial are
homogeneous: the lowest and highest degrees in a product add, so their
equality forces equality for each factor. Suppose \(F=PQ\), with
\(P\ge0\) nonconstant. Then \(\deg Q<\deg F\), and spherical
harmonic orthogonality gives
\[
 0=\int_{\Sph^{N-1}}FQ
   =\int_{\Sph^{N-1}}PQ^2>0,
\]
a contradiction; the orthogonality permits arbitrary lower-degree
\(Q\), not necessarily harmonic
\cite[Proposition~5.9, pp.~79--80]{AxlerBourdonRamey01}.
Strict positivity follows because \(PQ^2\) is a nonzero nonnegative
homogeneous polynomial, positive on a nonempty open part of the sphere.
Taking \(P=H^2\) excludes repeated factors. Every real irreducible
factor must also change sign, since otherwise it or its negative would
be a nonnegative factor. Such a factor has a nonsingular real zero by
\cite[Theorem~4.5.1]{BochnakCosteRoy98}.
\end{proof}

\begin{figure}[tbp]
\centering
\includegraphics[width=0.97\textwidth]{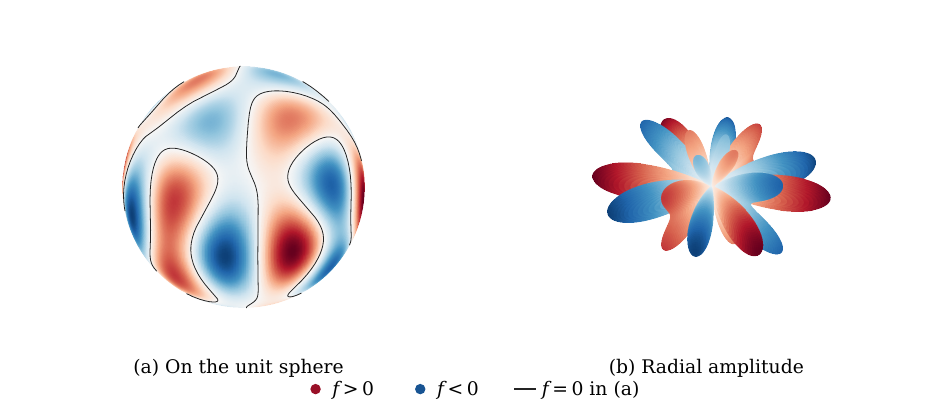}
\caption{A degree-six spherical harmonic \(f=\Phi|_{\Sph^2}\), where
\(\Phi=\operatorname{Re}((x+iy)^6+\tfrac12(y+iz)^6+
\tfrac13(z+ix)^6)\). Left: values on the unit sphere, with the nodal
set in black. Right: the radial amplitude plot
\(\theta\mapsto c^{-1}|f(\theta)|\theta\), with a positive display
normalization \(c\). Red and blue indicate positive and negative values.
The link of the zero cone is the nodal set in the left panel.
Visualization inspired by Matthew de Courcy-Ireland's illustration in
Logunov's lecture slides \cite[slide~4]{Logunov2020}.}
\label{fig:spherical-harmonic}
\end{figure}

\begin{proposition}[Irreducible effective equations for minimizing cones]
\label{prop:minimizing-irreducibility}
Let \(F\in\R[x_1,\ldots,x_N]\), \(N\ge2\), be nonconstant,
homogeneous and square-free. Suppose \(E=\{F<0\}\) is locally
perimeter minimizing in \(\R^N\) and
\(\partial\llbracket E\rrbracket\ne0\).
Then exactly one real irreducible factor of \(F\) has a nonsingular
real zero. After absorbing a nonzero constant and signs into this
factor, one can write
\[
\begin{gathered}
 F=PQ,\qquad P\text{ real irreducible},\qquad Q\ge0,\\
 \{F<0\}=\{P<0\}\quad\text{almost everywhere}.
\end{gathered}
\]
In particular, if every real irreducible factor has a nonsingular
real zero, then \(F\) is irreducible. Under the same minimizing
hypotheses, harmonic \(F\) is therefore irreducible.
\end{proposition}

\begin{proof}
Let \(V=|\partial\llbracket E\rrbracket|\) and
\(M=\operatorname{reg}V\). Its nonempty support is a closed cone,
hence connected. Regularity of minimizing boundaries makes \(M\)
real analytic and gives \(\cH^{N-2}(\operatorname{sing}V)=0\)
\cite[Part~III, introductory theorem, pp.~275--277]{Maggi12}.
Thus \(M\) is connected: for \(N=2\) this follows from smoothness,
and for \(N\ge3\) from
\cite[Section~5, proof of Theorem~1.1; Theorem~3.2]{WickramasekeraMaximum14}.
The singular set here belongs to the boundary, not to its polynomial
representation.

We use the following elementary divisibility fact. If a polynomial
\(H\) vanishes on a nonempty smooth open patch of the zero set of
a real irreducible polynomial \(P_j\), then \(P_j\mid H\).
Indeed, express the patch as a graph \(t=\varphi(x')\).
If \(P_j\nmid H\), Gauss' lemma
\cite[Chapter~VIII, Section~5]{Knapp16} makes \(P_j,H\) coprime over
\(\R(x')\), so their resultant in \(t\) is a nonzero polynomial
in \(x'\). The common root \(\varphi(x')\) makes it vanish on
an open set, a contradiction.

Write \(F=c\prod_jP_j\) with distinct irreducible factors.
For each factor with a nonsingular real zero, the divisibility fact
and coprimality give a nonempty smooth patch
\(U_j\subset\{P_j=0\}\) avoiding all other factors. On \(U_j\),
\(\nabla F\ne0\) and \(F\) changes sign, so \(U_j\subset M\).
If two such factors \(P_1,P_2\) existed, the analytic function
\(P_1|_M\), zero on \(U_1\), would vanish on connected \(M\),
contradicting \(P_1\ne0\) on \(U_2\).

For an irreducible real polynomial, having a nonsingular real zero
is equivalent to changing sign
\cite[Theorem~4.5.1]{BochnakCosteRoy98}. There must be one such
factor, since otherwise \(F\) is of fixed sign and
\(\partial\llbracket E\rrbracket=0\). Choose the signs of the
remaining factors to make their product \(Q\ge0\).
Its zero set is Lebesgue null, proving equality of the phases almost
everywhere. The harmonic assertion follows from
\cref{lem:square-free-critical-set}.
\end{proof}

\begin{example}[A redundant line and two intersecting hyperplanes]
\label{rem:reducible-minimal-equations}
For \(F=x(y^2+z^2)\) in \(\R^3\), the zero set is the plane
\(P=\{x=0\}\) together with the line \(\ell=\{y=z=0\}\).
The negative phase agrees almost everywhere with \(\{x<0\}\);
its boundary current is the plane calibrated by \(dy\wedge dz\),
with support \(P\). The extra line contributes no two-dimensional mass.
For \(F=xy\), instead, \(\Delta F=0\) and \(\cL(F)=-2F\),
so the regular zero set is minimal, but both factors contribute
hypersurface branches. Their boundary is not area minimizing by
\cref{prop:minimizing-irreducibility}. For the elementary crossing-lines
comparison and its product with a line, see
\cite[Figures~10.1.3 and~10.7.1, Lemma~10.6]{Morgan16}.
See also the regularity
obstruction in \cite[Theorem~28.1]{Maggi12} and
\cite[Section~1 and Theorem~1]{Federer70}.
\end{example}

\begin{figure}[htbp]
\centering
\begingroup
\centering
\resizebox{\linewidth}{!}{%
\begin{tikzpicture}[x=1cm,y=1cm,line cap=round,line join=round,
  font=\footnotesize]
  \definecolor{coneplane}{RGB}{42,104,149}
  \definecolor{coneline}{RGB}{173,78,39}
  \tikzset{plane edge/.style={draw=coneplane,line width=.65pt},
    second edge/.style={draw=coneline,line width=.65pt}}
  \path[use as bounding box] (-3.55,-2.95) rectangle (10.75,2.65);

  \begin{scope}
    \node[font=\footnotesize\bfseries] at (0,2.4)
      {(a) A plane and a transverse line};
    \filldraw[plane edge,fill=coneplane!12]
      (-1.45,-1.7575)--(1.45,-.7425)--(1.45,1.7575)--(-1.45,.7425)--cycle;
    \draw[coneline,line width=1.05pt] (-2.5175,.901)--(-1.45,.51895);
    \draw[coneline,line width=1.05pt,dash pattern=on 3pt off 2.4pt]
      (-1.45,.51895)--(0,0);
    \draw[coneline,line width=1.05pt] (0,0)--(2.5175,-.901);
    \fill[black!75] (0,0) circle (1.1pt);
    \node[coneplane,anchor=south east,fill=white,inner sep=1.6pt]
      at (1.35,1.52) {$P=\{x=0\}$};
    \node[coneline,anchor=north east] at (2.65,-.98)
      {$\ell=\{y=z=0\}$};
    \node at (0,-2.1) {$F=x(y^2+z^2)$};
    \node[font=\footnotesize\bfseries] at (0,-2.5)
      {Area-minimizing current: $[\![P]\!]$};
    \node[font=\scriptsize] at (0,-2.87)
      {The added line has zero two-dimensional mass.};
  \end{scope}

  \begin{scope}[xshift=7.2cm]
    \node[font=\footnotesize\bfseries] at (0,2.4)
      {(b) Two intersecting hyperplanes};
    \filldraw[plane edge,fill=coneplane!9]
      (0,-1.15)--(1.65,-.5725)--(1.65,1.7275)--(0,1.15)--cycle;
    \filldraw[second edge,fill=coneline!9]
      (-1.5675,-.589)--(0,-1.15)--(0,1.15)--(-1.5675,1.711)--cycle;
    \filldraw[plane edge,fill=coneplane!17]
      (-1.65,-1.7275)--(0,-1.15)--(0,1.15)--(-1.65,.5725)--cycle;
    \filldraw[second edge,fill=coneline!17]
      (0,-1.15)--(1.5675,-1.711)--(1.5675,.589)--(0,1.15)--cycle;
    \draw[black!75,line width=.9pt] (0,-1.15)--(0,1.15);
    \node[coneplane,anchor=east] at (-1.83,-.52) {$P_1$};
    \node[coneline,anchor=west] at (1.75,-.6) {$P_2$};
    \node[anchor=south,font=\scriptsize] at (0,1.2) {$P_1\cap P_2$};
    \node at (0,-2.1) {$F=xy$};
    \node[font=\footnotesize\bfseries] at (0,-2.5)
      {Not area minimizing};
    \node[font=\scriptsize] at (0,-2.87)
      {Both planes contribute two-dimensional mass.};
  \end{scope}
\end{tikzpicture}
}
\par
\endgroup
\caption{Left: \(\{x(y^2+z^2)=0\}=P\cup\ell\), but the minimizing
boundary current has support \(P\). Right: both intersecting
hyperplanes in \(\{xy=0\}\) carry area; their union is stationary
as a varifold but the associated boundary current is not area minimizing.
Both pictures are illustrated in \(\R^3\).}
\label{fig:reducible-cones}
\end{figure}

\phantomsection\label{par:irreducible-convention}
For structural results, unless stated otherwise, we choose a nonconstant
real homogeneous irreducible defining polynomial with \(C_{\reg}\ne\varnothing\).
Removing a nonnegative factor preserves the boundary current, but may
change the full zero set, the degree and \(\cA_R(F)\); all degree
and generator statements refer to the chosen effective equation.

\Needspace{12\baselineskip}
The classical algebraic minimality criterion now takes the following
form; see \cite[Section~6.2.1, equation~(6.2.2)]{NadirashviliTkachevVladut14}.

\begin{proposition}[Minimality and polynomial divisibility]
\label{prop:divisibility}
Let \(F\) be real irreducible and homogeneous of degree \(e\ge1\),
with \(C_{\reg}\ne\varnothing\). Then \(C_{\reg}\) is minimal
if and only if
\begin{equation}\label{eq:L-divisibility}
 \cL(F)=F\lambda_F,
\end{equation}
where \(\lambda_F\) is homogeneous of degree \(2e-4\), possibly
zero, for \(e\ge2\), and \(\lambda_F=0\) for \(e=1\).
The same equivalence holds for a square-free homogeneous \(F\)
when every real irreducible factor has a nonsingular real zero.
\end{proposition}

\begin{proof}
By \eqref{eq:div-normalized-gradient}, minimality means
\(\cL(F)=0\) on \(C_{\reg}\). The patch divisibility argument
in \cref{prop:minimizing-irreducibility} gives \(F\mid\cL(F)\).
For the square-free extension, take for each factor a nonsingular
patch avoiding the others, apply the same argument, and use
pairwise coprimality. Since \(\deg\cL(F)=3e-4\) when it is
nonzero, the degree follows. For \(e=1\), all second derivatives
vanish. The converse follows from the same curvature formula.
\end{proof}

The factor condition cannot be omitted in the square-free extension:
for \(t=y^2+z^2\) and \(F=x(2x^2+t)\), the regular zero set is
planar, but
\(\cL(F)=8(9x^2-t)F+4xt^2\) is not divisible by \(F\).
Its effective equation is simply \(x\).

Under \eqref{eq:L-divisibility}, differential closure and polynomial
division \cite[Lemma~24(b)]{MendesRadeschi}, or
\cref{prop:MR-finite-and-division}, give \(\lambda_F\in\cA_R(F)\).
In the harmonic case this becomes
\begin{equation}\label{eq:harmonic-minimal-weight}
 \Delta F=0,\qquad S=|\nabla F|^2,\qquad
 \Gamma(F,S)=2qF,\qquad q=-\lambda_F,
\end{equation}
with \(S,q\in\cA_R(F)\). These identities simplify the closure
calculation, while \cref{prop:uniform-finite-generation} supplies the
uniform generator bound.

\section[Multipliers and Global Positivity]{Multipliers, Jacobi Identities, and Global Positivity}
\label{sec:algebraic-jacobi}

The multiplier expansion relates global divergence to its Jacobi trace;
\cref{sec:laplacian-calculus} reduces the sign test to the invariant image.

\subsection{The global multiplier defect}

\begin{proposition}[Odd-power invariance and regular defining potentials]
\label{prop:potential-normal-form}
Let \(\Omega\subset\R^N\setminus\{0\}\) be open and conical, and
let \(F\) be smooth and homogeneous with \(\nabla F\neq0\) on
\(C\cap\Omega=\{F=0\}\cap\Omega\).
\begin{enumerate}[label=\textup{(\roman*)}]
 \item Let \(a>0\) be smooth and homogeneous on \(\Omega\), and
 let \(m\geq1\) be an odd integer.  Set
 \[
  \widetilde F=F^m,\qquad G=aF,\qquad
  \widetilde G=a^m\widetilde F=G^m.
 \]
 Then
 \begin{equation}\label{eq:odd-power-field}
  \nabla\widetilde G=mG^{m-1}\nabla G,\qquad
  \cL(\widetilde G)=m^3G^{3m-3}\cL(G).
 \end{equation}
 Thus \(X_{\widetilde G}=X_G\) wherever \(F\neq0\) and
 \(\nabla G\neq0\), and \(X_G\) gives its smooth extension across
 \(C\cap\Omega\).  Conversely, every positive smooth homogeneous
 multiplier \(b\) for \(\widetilde F\) is obtained by setting
 \(a=b^{1/m}\).  Hence the two defining equations admit the same
 multiplier subcalibration fields, with the same regular boundary
 extensions and the same remaining exceptional sets.
 \item If a smooth homogeneous function \(G\) has the same zero set
 and signs as \(F\) on \(\Omega\), and \(\nabla G\neq0\) on
 \(C\cap\Omega\), then \(G=bF\) for a positive smooth homogeneous
 function \(b\) on \(\Omega\).
\end{enumerate}
\end{proposition}

\begin{proof}
For \textup{(i)}, the chain rule gives \(\cL(\psi(G))=(\psi'(G))^3\cL(G)\).
Take \(\psi(t)=t^m\).
Its derivative is positive off \(t=0\), which proves the field
identity.  On the regular cone, \(\nabla G=a\nabla F\neq0\),
so \(X_G\) provides the exterior normal extension.  Off the cone,
the two gradients vanish at exactly the same points.  The inverse
correspondence follows from \(bF^m=(b^{1/m}F)^m\); the positive
root is smooth and homogeneous.

For \textup{(ii)}, use local coordinates \((y,t)\) with \(t=F\).
Since \(G(y,0)=0\),
\[
 G(y,t)=t\int_0^1\partial_tG(y,st)\,ds.
\]
The integral extends \(b=G/F\) smoothly, with
\(b|_C=|\nabla G|/|\nabla F|>0\) by the common phase signs and
regularity.  These extensions agree on overlaps; the quotient also
gives \(\deg b=\deg G-\deg F\).
\end{proof}

Part \textup{(i)} preserves the divergence and its phase integrals.
For \(m>1\), \(\nabla\widetilde G=0\) on the cone; the
Jacobi--Hardy formulas use its regular field extension \(X_G\), the degree
\(\ell=\deg\widetilde G/m\), and \(w=|\nabla G|^{-1}|_C\).

The displayed algebra \(\R[R,F^m]\) need not itself be Laplacian closed, even when
\(\R[R,F]\) is: for \(F=x_1\),
\(\Delta(F^3)=6F\notin\R[R,F^3]\).
The radial and gradient powers used below are discussed in
\cref{app:power-identities}.

Let \(\Omega\subset\R^N\) be open and \(F,a\in C^\infty(\Omega)\),
with \(a>0\). Suppose \(\nabla F\ne0\) on
\(C=\{F=0\}\cap\Omega\) and \(C\) is minimal. Put \(G=aF\).
The product rule gives \(\cL(G)|_C=a^3\cL(F)|_C=0\).
Smooth division by \(F\), as in
\cref{prop:potential-normal-form}\textup{(ii)}, therefore defines
\[
 \lambda_F=\frac{\cL(F)}F,\qquad \cK_F(a)=\frac{\cL(aF)}F.
\]
Then \(G\cL(G)=aF^2\cK_F(a)\). Polynomial divisibility gives
the same \(\lambda_F\) as \eqref{eq:L-divisibility}. Write
\[
 P=\nabla F,\quad Q=\nabla a,\qquad
 S=|P|^2,\quad T=\langle P,Q\rangle,\quad V=|Q|^2.
\]

\begin{theorem}[Exact multiplier expansion]
\label{thm:exact-multiplier}
Under the preceding smooth hypotheses, the multiplier defect equals
\begin{align}
\cK_F(a)={}&a^3\lambda_F\notag\\
&+a^2\bigl(S\Delta a+2T\Delta F-2\Hess F(P,Q)-\Hess a(P,P)\bigr)\notag\\
&+2a(T^2-SV)\notag\\
&+Fa\bigl(2T\Delta a+V\Delta F-\Hess F(Q,Q)
              -2\Hess a(P,Q)\bigr)\notag\\
&+F^2\cL(a).                                           \label{eq:exact-multiplier}
\end{align}
If, in addition, \(F,a\) are polynomials in a Laplacian algebra
\(\cA\) and \(F\mid\cL(F)\) in \(\R[x_1,\ldots,x_N]\),
then every term in \eqref{eq:exact-multiplier} belongs to \(\cA\).
\Cref{prop:divisibility} gives sufficient geometric conditions for
this polynomial divisibility.
\end{theorem}

\begin{proof}
Expand \(\nabla(aF)=aP+FQ\),
\(\Delta(aF)=a\Delta F+F\Delta a+2T\), and
\[
 \Hess(aF)=a\Hess F+F\Hess a+dF\otimes da+da\otimes dF.
\]
The constant term in \(\cL(aF)\) is
\(a^3\cL(F)=a^3F\lambda_F\); collecting the remaining powers of \(F\)
gives the formula.  The identity
\[
 2\Hess u(\nabla v,\nabla w)
 =\Gamma(v,\Gamma(u,w))+\Gamma(w,\Gamma(u,v))
  -\Gamma(u,\Gamma(v,w))
\]
expresses the Hessian terms through gradient inner products.

By closure, \(\cL(F)=S\Delta F-\Gamma(F,S)/2\in\cA\).
Polynomial division, \cref{prop:MR-finite-and-division}, then gives
\(\lambda_F\in\cA\), proving the algebra assertion.
\end{proof}

\begin{lemma}[The defect at a critical zero]
\label{lem:critical-multiplier-defect}
Let \(F\) be a nonzero polynomial satisfying
\(\cL(F)=F\lambda_F\), with polynomial \(\lambda_F\), and let
\(a\) be a \(C^2\) function near a point \(p\) with
\(F(p)=0\) and \(\nabla F(p)=0\).
The right-hand side of \eqref{eq:exact-multiplier} defines a continuous
extension of \(\cL(aF)/F\), and
\begin{equation}\label{eq:critical-multiplier-defect}
 \cK_F(a)(p)=a(p)^3\lambda_F(p).
\end{equation}
If \(\lambda_F(p)<0\), no such multiplier with \(a(p)>0\)
can satisfy \((aF)\cL(aF)\ge0\) throughout a punctured neighborhood
of \(p\). In particular, a nonnegative \(C^2\) multiplier satisfying
this sign condition must vanish at \(p\).
\end{lemma}

\begin{proof}
The expansion is a product-rule identity wherever \(F\ne0\), and
every term on its right-hand side is continuous. At \(p\), one has
\(F=0\), \(P=0\), \(S=0\), and \(T=0\), so only
\(a^3\lambda_F\) remains. If \(a(p)>0\) and
\(\lambda_F(p)<0\), continuity gives \(a>0\) and
\(\cK_F(a)<0\) near \(p\). A nonzero polynomial has a zero set
with empty interior, so there are nearby points with \(F\ne0\), where
\[
 (aF)\cL(aF)=aF^2\cK_F(a)<0.
\]
This contradicts the required sign.
\end{proof}

\paragraph{Logarithmic differentiation of a positive multiplier.}
On the open domain where \(a>0\), including the regular cone,
\(h=\log a\) is smooth and satisfies
\begin{equation}\label{eq:log-multiplier-derivatives}
 \nabla h=\frac{\nabla a}{a},\qquad
 \Delta h=\frac{\Delta a}{a}-\frac{|\nabla a|^2}{a^2},\qquad
 \Gamma(h,q)=\frac{\Gamma(a,q)}a.
\end{equation}
If $a$ has degree $d$, then $h(tx)=h(x)+d\log t$ for $t>0$.

\subsection{The Jacobi boundary condition}
\label{sec:jacobi-boundary-condition}

Let \(J_{C_{\rm reg}}=\Delta_{C_{\rm reg}}+|A_C|^2\)
be the Jacobi operator associated with the second variation; see
\cite{Simons68} and \cite[Remark~9.8]{Simon83}. With scalar mean curvature
defined by \(H=\diver_{C_{\rm reg}}\nu\), a variation with normal
speed \(w\) has linearization \(-J_Cw\).

\begin{theorem}[Exact Jacobi trace]
\label{thm:Jacobi-trace}
For the smooth \(F,a\) of \cref{thm:exact-multiplier}, along
\(C_{\reg}\cap\Omega\),
\begin{equation}\label{eq:Jacobi-trace}
 \boxed{
 \left.\cK_F(a)\right|_{C_{\rm reg}}
 =-a^4|\nabla F|^3J_{C_{\rm reg}}
 \left(\frac1{a|\nabla F|}\right).}
\end{equation}
Consequently, \(\cK_F(a)\geq0\) near this hypersurface requires the positive
Jacobi supersolution
\begin{equation}\label{eq:positive-Jacobi}
 w=(a|\nabla F|)^{-1}>0,
 \qquad J_Cw\leq0.
\end{equation}
It is a Jacobi field exactly when the boundary defect vanishes.
\end{theorem}

\begin{proof}
The flow of \(\nabla G/|\nabla G|^2\) parametrizes the level sets
\(M_t=\{G=t\}\), with normal speed
\(w=|\nabla G|^{-1}=(a|\nabla F|)^{-1}\) at \(t=0\).
The mean-curvature variation at the minimal zero level is
\(\frac{d}{dt}H_{M_t}|_0=-J_Cw\).
On the other hand, \eqref{eq:div-normalized-gradient} gives, along this flow,
\[
 H_{M_t}=\frac{F\cK_F(a)}{|\nabla G|^3},\qquad
 \frac{dF}{dt}\bigg|_0
 =\left.\frac{\langle\nabla F,\nabla G\rangle}{|\nabla G|^2}\right|_C
 =\frac1a.
\]
Since \(F=0\) on \(C\), differentiation of the other factors contributes
zero at \(t=0\).  Using \(|\nabla G|^3=a^3|\nabla F|^3\) there yields
\[
 -J_Cw=\frac{d}{dt}H_{M_t}\bigg|_0
 =\left.\frac{\cK_F(a)}{a^4|\nabla F|^3}\right|_C,
\]
which is \eqref{eq:Jacobi-trace}.
\end{proof}

\begin{proposition}[Harmonic power multipliers on the boundary]
\label{prop:harmonic-power-Jacobi}
Let \(F\not\equiv0\) be a homogeneous polynomial of degree \(e\ge2\) satisfying
\eqref{eq:harmonic-minimal-weight}. Put \(R=r^2\),
\(K=\Delta S=2|\Hess F|^2\), and \(T=|\nabla S|^2\).
For real \(\sigma,\gamma\), define \(G=r^\sigma S^\gamma F\) on
\(\{S>0\}\). Its boundary speed is \(w=r^{-\sigma}S^{-\gamma-1/2}\).
On \(C_{\reg}\),
\begin{equation}\label{eq:harmonic-power-Jacobi}
 -\frac{J_Cw}{w}
 =\frac{\sigma\bigl(N-1-2e-\sigma-4\gamma(e-1)\bigr)}{R}
  +\frac{\gamma K-(2\gamma+1)q}{S}
  -\frac{\gamma(2\gamma+3)}2\frac{T}{S^2}.
\end{equation}
In particular, a two-sided subcalibration generated by \(G\) requires
the right-hand side to be nonnegative on \(C_{\reg}\).
\end{proposition}

\begin{proof}
Set \(w_0=S^{-1/2}\). Applying \cref{thm:Jacobi-trace} with
\(a=1\) and \(\cL(F)=-qF\) gives \(J_Cw_0/w_0=q/S\).
On \(C_{\reg}\), equation \eqref{eq:harmonic-minimal-weight} implies
\(\nabla S\perp\nabla F\) and \(\Hess F(\nu_C,\nu_C)=0\).
Since \(\Hess F\,\nu_C=\nabla S/(2\sqrt S)\), decomposition into
normal and tangential components gives
\[
 |A_C|^2=\frac{K}{2S}-\frac{T}{2S^2},\qquad
 |\nabla_C\log w_0|^2=\frac{T}{4S^2}.
\]
For \(p=2\gamma+1\), the chain rule reads
\[
 \frac{J_C(w_0^p)}{w_0^p}
 =p\frac{J_Cw_0}{w_0}+(1-p)|A_C|^2
       +p(p-1)|\nabla_C\log w_0|^2.
\]
Finally, \(C\) has dimension \(N-1\), so
\(\Delta_Cr^{-\sigma}/r^{-\sigma}
 =\sigma(\sigma-N+3)/R\).
The homogeneity of \(w_0^p\), of degree \(-p(e-1)\), gives the
cross term \(2\sigma p(e-1)/R\). Substitution proves the formula.
\end{proof}

\begin{theorem}[Positive-solution identity and homogeneous Hardy gap]
\label{thm:ground-state-gap}
Let \(C^n\subset\R^{n+1}\), \(n\geq2\), be a cone whose regular part
is two-sided and minimal.  Let \(G\) be a smooth homogeneous defining function of degree
\(\ell\) on a conical neighborhood of \(C_{\reg}\setminus\{0\}\), with
\(\nabla G\neq0\) there.  Set
\begin{equation}\label{eq:boundary-Jacobi-potential}
 w=|\nabla G|^{-1}\big|_C>0,\qquad
 V_G=-\frac{J_Cw}{w}=r^{-2}\beta_G(\theta).
\end{equation}
For every \(u\in C_c^\infty(C_{\reg}\setminus\{0\})\),
\begin{equation}\label{eq:ground-state-identity}
 Q_C(u)=\int_Cw^2\left|\nabla_C\left(\frac uw\right)\right|^2
             +\int_CV_Gu^2.
\end{equation}
If \(\beta_G\geq\beta_0\) on the regular link for a finite constant \(\beta_0\), then
\begin{equation}\label{eq:homogeneous-strict-stability}
 Q_C(u)\geq\left[\left(\ell-\frac n2\right)^2+\beta_0\right]
                  \int_Cr^{-2}u^2.
\end{equation}
For \(G=aF\) in \cref{thm:Jacobi-trace},
\begin{equation}\label{eq:boundary-potential-algebraic}
 V_G=\left.\frac{\cK_F(a)}{a^3|\nabla F|^2}\right|_C
     =\left.\partial_{\nu_C}\diver
          \left(\frac{\nabla G}{|\nabla G|}\right)\right|_C.
\end{equation}
Consequently a nonnegative Jacobi boundary defect implies stability;
if also \(\ell\neq n/2\), it implies strict stability with constant at
least \((\ell-n/2)^2\).
\end{theorem}

\begin{proof}
Substituting \(u=wv\) and integrating by parts gives
\eqref{eq:ground-state-identity}.
Homogeneity gives \(w=r^{1-\ell}\phi(\theta)\), \(\phi>0\).
Writing \(\Lambda_{\reg}=C_{\reg}\cap\Sph^n\), the cone volume
element \(r^{n-1}\,dr\,d\mu_\Lambda\) gives
\[
 \int_Cw^2|\nabla_Cv|^2
 \geq\int_{\Lambda_{\reg}}\phi^2
       \int_0^\infty r^{n+1-2\ell}|\partial_rv|^2\,dr\,d\mu_\Lambda.
\]
For every real \(p\) and \(f\in C_c^\infty(0,\infty)\), completion of the
square proves
\[
 \int_0^\infty r^p(f')^2\,dr
 \geq\frac{(p-1)^2}{4}\int_0^\infty r^{p-2}f^2\,dr.
\]
Apply this with \(p=n+1-2\ell\) for each angular point.
The resulting term is \((\ell-n/2)^2\int_Cr^{-2}u^2\).
Adding \(\int_CV_Gu^2\geq\beta_0\int_Cr^{-2}u^2\) proves
\eqref{eq:homogeneous-strict-stability}.
The first equality in \eqref{eq:boundary-potential-algebraic} follows by
dividing \eqref{eq:Jacobi-trace} by \(a^4|\nabla F|^3w\).
For the second, differentiate
\(\diver X_G=F\cK_F(a)/|\nabla G|^3\) in the unit normal direction on
\(F=0\), where \(\partial_\nu F=|\nabla F|\).
\end{proof}

\begin{remark}[Vanishing trace and the Hardy gap]
\label{rem:vanishing-trace-strict-stability}
If \(J_Cw=0\), formula \eqref{eq:ground-state-identity} still contains
the Hardy gap \((\ell-n/2)^2\).  This lower bound vanishes only at the central degree
\(\ell=n/2\).
Failure of strict stability requires a sequence with Hardy quotient
tending to zero; see \cref{prop:complex-zero-gap}.
\end{remark}

\begin{remark}[Subcalibrations as paired geometric barriers]
\label{rem:paired-barriers}
For \(G=aF\), \(a>0\), and \(\cK_F(a)\geq0\), the regular levels
\(M_t=\{G=t\}\), oriented by \(X_G\), satisfy
\[
 H_{M_t}=\frac{F\cK_F(a)}{|\nabla G|^3},\qquad
 tH_{M_t}\geq0.
\]
Thus \(G\) is a supersolution on \(F<0\) and a subsolution on
\(F>0\) of the level-set equation \(\diver X_G=0\).
The level-set variation in the proof of \cref{thm:Jacobi-trace}
has normal graph height \(s_t=tw+O(t^2)\), with
\(w=(a|\nabla F|)^{-1}\), and gives
\begin{equation}\label{eq:paired-barrier-expansion}
 H_{M_t}=-tJ_Cw+O(t^2).
\end{equation}
Thus \(J_Cw<0\) gives strict local barriers on both sides.
\end{remark}

The sign of \(\cK_F(a)\) away from the cone requires a separate
calculation, as the stable non-minimizing examples in
\cref{sec:isoparametric-reduction} illustrate.

\subsubsection*{The admissible degree}

Let \(C^{N-1}\subset\R^N\), \(N\geq3\), be a minimal cone, let
\(\Sigma=C\cap\Sph^{N-1}\), and suppose a homogeneous potential \(G=aF\)
has total degree \(\ell\).  Then \(w=(a|\nabla F|)^{-1}\) has degree
\(1-\ell\).  If the entire link \(\Sigma\) is smooth, it is compact
and connected; see the maximum-principle argument of Hardt--Simon
\cite[p.~104]{HardtSimon85}.
Let \(\lambda_1(\Sigma)\) be the first eigenvalue of
\(-\Delta_\Sigma-|A_\Sigma|^2\).  Existence of
a positive homogeneous Jacobi supersolution requires
\begin{equation}\label{eq:degree-window}
 \ell\in I_F:=
 \left[\frac{N-1}{2}-\beta_F,
       \frac{N-1}{2}+\beta_F\right],
 \qquad
 \beta_F^2=\frac{(N-3)^2}{4}+\lambda_1(\Sigma).
\end{equation}
If the radicand is negative, no positive homogeneous Jacobi supersolution
with positive angular part exists.  More precisely, if \(\phi_1>0\) is the
first eigenfunction and
\(w=r^{1-\ell}\phi_1\), then
\[
 J_Cw=r^{-1-\ell}\left[
 (1-\ell)(N-2-\ell)-\lambda_1(\Sigma)
 \right]\phi_1.
\]
For any positive angular supersolution, integrating its inequality
against \(\phi_1\) gives the same necessary degree condition.
Using \(\phi_1\), interior degrees give \(J_Cw<0\) and endpoints
give \(J_Cw=0\). A prescribed angular multiplier requires its own sign test.

For a smooth compact connected link, the zero-gap case is rigid.
Write $n=N-1$, $H_\Sigma=-\Delta_\Sigma-|A_\Sigma|^2$, and
$\delta_C=(n-2)^2/4+\lambda_1(\Sigma)$, the exact Hardy gap in
\eqref{eq:exact-Hardy-gap}. If $w=r^{1-\ell}\phi>0$ and
$-J_Cw/w=r^{-2}\beta\ge0$, then
\begin{equation}\label{eq:critical-boundary-rigidity}
 \delta_C=0\quad\Longrightarrow\quad
 \ell=n/2,\qquad \beta\equiv0,\qquad
 \phi\text{ is a positive first eigenfunction of }H_\Sigma.
\end{equation}
Indeed, \eqref{eq:degree-window} forces $\ell=n/2$.
Then $(H_\Sigma+(n-2)^2/4)\phi=\beta\phi$.
Pairing with a positive first eigenfunction $\psi$ gives
$0=\int_\Sigma\beta\phi\psi$, hence $\beta=0$.
Simplicity of the first eigenvalue now gives the last assertion.
In particular, a nonnegative boundary defect that is positive at one
point suffices for strict stability on such a link.

\begin{remark}[Two admissible power multipliers]
\label{cor:two-power-strict-stability}
Let $F$ be a homogeneous polynomial of degree $e\ge2$ on
$\R^{n+1}$, $n\ge3$, whose zero set $C$ is minimal, with
$\nabla F\ne0$ on $C\setminus\{0\}$ and nonempty smooth compact
connected link $\Sigma$. Put $S=|\nabla F|^2$, and suppose that
$s=S|_\Sigma$ is nonconstant. If two distinct real parameter pairs
$(\alpha_i,b_i)$ give potentials
\[
 G_i=r^{\alpha_i}S^{b_i/2}F,\qquad
 \operatorname{sgn}(F)\diver X_{G_i}\ge0\quad\text{where }F\ne0
\]
in a two-sided neighborhood of $C\setminus\{0\}$ on which both
fields are defined, then $C$ is strictly stable.
Indeed, by \cref{thm:Jacobi-trace}, the positive boundary speeds
$w_i=r^{-\alpha_i}S^{-(b_i+1)/2}$ are Jacobi supersolutions.
\Cref{thm:ground-state-gap} first gives stability.
Their potential degrees are $\ell_i=e+\alpha_i+b_i(e-1)$.
If the Hardy gap were zero, \eqref{eq:critical-boundary-rigidity}
would force $\ell_i=n/2$ and both angular parts
$s^{-(b_i+1)/2}$ to be proportional to the same first eigenfunction.
Thus $s^{-(b_1-b_2)/2}$ would be constant. Since $s$ is nonconstant,
$b_1=b_2$, and equality of the degrees then gives $\alpha_1=\alpha_2$,
a contradiction.
\end{remark}

\paragraph{Example: admissible degrees on the Simons cone.}
\label{par:Simons-degree-example}
Let
\[
 F(x,y)=|x|^2-|y|^2,\qquad (x,y)\in\R^4\times\R^4,\qquad R=r^2.
\]
Its zero link is
\(\Sigma=\Sph^3(1/\sqrt2)\times\Sph^3(1/\sqrt2)\), with
\(|A_\Sigma|^2=6\).  The constant function attains
\(\lambda_1(\Sigma)=-6\), so \eqref{eq:degree-window} gives
\[
 \beta_F^2=\frac{25}{4}-6=\frac14,\qquad I_F=[3,4].
\]
The quadratic \(F\) itself lies outside this interval.  Choose instead
\(G=RF=|x|^4-|y|^4\), of degree \(\ell=4\), as in
\cite[Section~2]{DePhilippisPaolini}; see
\cref{rem:subcalibration-history} for its historical antecedents.
Direct calculation gives
\begin{align*}
 |\nabla G|^2&=4R(R^2+3F^2),& \Delta G&=24F,\\
 \cL(G)&=192RF^3,& \cK_F(R)&=192RF^2,\\
 G\cL(G)&=192R^2F^4\geq0.
\end{align*}
This is the strictly minimizing \(g=2\), \(m_1=m_2=3\) case of
\cref{sec:isoparametric-reduction}.

On the other hand,
\[
 w=|\nabla G|^{-1}\big|_C=\frac1{2r^3},\qquad
 J_Cw=0,\qquad \cK_F(R)\big|_C=0.
\]
Nevertheless \cref{thm:ground-state-gap} gives
\[
 Q_C(u)\geq\left(4-\frac72\right)^2\int_Cr^{-2}u^2
 =\frac14\int_Cr^{-2}u^2,
\]
and \eqref{eq:exact-Hardy-gap} shows this constant is optimal.
Thus an endpoint degree and zero Jacobi trace coexist with both
strict area minimality and strict stability.
For \(G_\eps=r^{-\eps}G\), \(0<\eps<1\), the boundary coefficient
is \(\beta_{G_\eps}=\eps-\eps^2>0\).
For all sufficiently small \(\eps>0\), \(G_\eps\) retains the
subcalibration sign and its normalized gradient satisfies the linear
distance bound; see \cref{cor:four-degenerate-certificates}.

\subsubsection*{The spectral obstruction for a singular link}
For a singular link, we use the angular Rayleigh infimum, without
assuming connectedness of \(\Sigma_{\reg}\) or existence of a first
eigenfunction.\footnote{If \(\Sigma\) is the connected support of a stationary
integral \((N-2)\)-varifold in \(\Sph^{N-1}\), with
\(\cH^{N-3}(\Sing\Sigma)=0\), regular-part connectedness follows from
\cite[Lemma~2.2]{Zhu18}.}
Define
\begin{equation}\label{eq:Friedrichs-bottom}
 \lambda_{\rm F}(\Sigma)=
 \inf_{0\neq\phi\in C_c^\infty(\Sigma_{\rm reg})}
 \frac{\int_{\Sigma_{\rm reg}}
 (|\nabla\phi|^2-|A_\Sigma|^2\phi^2)}
 {\int_{\Sigma_{\rm reg}}\phi^2}.
\end{equation}
This infimum may equal \(-\infty\); when finite, it is the bottom of
the Friedrichs extension and need not be attained.
Zhu's spherical stability convention
\cite[Section~2.1.5, equation~(2.8)]{Zhu18} includes the Ricci term:
\[
 \lambda_1^{Z}(\Sigma)=\lambda_{\rm F}(\Sigma)-(N-2).
\]
The exact cone stability quotient is
\begin{equation}\label{eq:exact-Hardy-gap}
\begin{aligned}
 \inf_{0\neq u\in C_c^\infty(C_{\rm reg}\setminus\{0\})}
 \frac{Q_C(u)}{\int_Cr^{-2}u^2}
 &=\frac{(N-3)^2}{4}+\lambda_{\rm F}(\Sigma)\\
 &=\frac{(N-1)^2}{4}+\lambda_1^{Z}(\Sigma).
\end{aligned}
\end{equation}
By radial--angular separation \cite{Simons68,CabrePoggesi18},
the sharp radial Hardy inequality and the angular Rayleigh bound give
the lower bound. For a product
test \(u(r,\theta)=\eta(r)\varphi(\theta)\), the cone quotient is
the sum of the radial Hardy quotient and the angular Rayleigh
quotient. Independent minimizing sequences give equality, including
\(\lambda_{\rm F}(\Sigma)=-\infty\).
If the right-hand side is negative, no fields satisfying the
comparison criterion can exist.

\begin{remark}[What the boundary trace determines]
Equality \(a_1|_{C_{\rm reg}}=a_2|_{C_{\rm reg}}\) is equivalent to
local smooth divisibility of \(a_1-a_2\) by \(F\). By
\eqref{eq:Jacobi-trace}, it implies
\begin{equation}\label{eq:boundary-ideal-class}
 \left.\cK_F(a_1)\right|_{C_{\rm reg}}
 =\left.\cK_F(a_2)\right|_{C_{\rm reg}}.
\end{equation}
Thus the boundary defect depends only on \(a|_C\).
\end{remark}

For calibrated complex cones with zero Hardy gap, see the
Jacobi-field construction in \cref{prop:complex-zero-gap}.

\subsection{A multiplier criterion for harmonic defining functions}
\label{sec:harmonic-multiplier-criterion}

The following criterion permits multipliers that degenerate at the
singular set. Its spectral application appears in
\cref{sec:spectral-positivity}.

Let \(F\not\equiv0\) be a homogeneous polynomial of positive degree on
\(\R^N\), \(N\ge2\), satisfying \eqref{eq:harmonic-minimal-weight}, and set
\(U=\{S>0\}\).  Let \(h\) be smooth on \(U\), put \(a=e^h\),
and suppose that a smooth function \(\tau\) on \(U\) satisfies
\begin{equation}\label{eq:harmonic-log-compatibility}
 \Gamma(F,h)=\tau F.
\end{equation}
Thus \(\nabla h\) is tangent to the regular zero set. Introduce
\begin{equation}\label{eq:harmonic-log-auxiliary}
 H_h=\Gamma(h,h),\qquad Q_h=H_h+2\tau,\qquad
 P_h=Q_h\Delta h-\tfrac12\Gamma(h,Q_h),
\end{equation}
and the boundary coefficient
\begin{equation}\label{eq:harmonic-log-D0}
 D_h=S(\Delta h-\tau-H_h)-q-\tfrac12\Gamma(h,S).
\end{equation}
These functions are defined on \(U\) without division by \(F\).

\begin{lemma}[Harmonic multiplier identity]
\label{lem:harmonic-log-multiplier}
With the preceding notation, \(G=aF\) satisfies
\begin{align}
 |\nabla G|^2&=a^2(S+F^2Q_h),
 \label{eq:harmonic-log-gradient}\\
 \cL(G)&=a^3F\left(D_h+F^2P_h-
                    \tfrac F2\Gamma(F,Q_h)\right).
 \label{eq:harmonic-log-defect}
\end{align}
Along \(C_{\reg}=\{F=0,S>0\}\), one has
\begin{equation}\label{eq:harmonic-log-Jacobi}
 \frac{D_h}{S}=-\frac{J_Cw}{w},\qquad
 w=(a\sqrt S)^{-1}.
\end{equation}
If in addition \(\Gamma(F,Q_h)=\kappa F\) for a smooth function
\(\kappa\), then \eqref{eq:harmonic-log-defect} becomes
\begin{equation}\label{eq:harmonic-log-two-coefficients}
 \cL(G)=a^3F(D_h+F^2D_{1,h}),\qquad
 D_{1,h}=Q_h\Delta h-\tfrac12\kappa-
                         \tfrac12\Gamma(h,Q_h).
\end{equation}
\end{lemma}

\begin{proof}
The product and chain rules give
\[
 \nabla G=a(\nabla F+F\nabla h),\qquad
 \Delta G=aF(\Delta h+Q_h).
\]
The first formula gives \eqref{eq:harmonic-log-gradient}.  Set
\(V=S+F^2Q_h\).  Differentiating \(|\nabla G|^2=a^2V\) yields
\[
 a^{-3}\Gamma(G,|\nabla G|^2)
 =2FV(\tau+H_h)+\Gamma(F,V)+F\Gamma(h,V).
\]
Here
\[
 \begin{aligned}
 \Gamma(F,V)&=2qF+2FSQ_h+F^2\Gamma(F,Q_h),\\
 \Gamma(h,V)&=\Gamma(h,S)+2\tau F^2Q_h+
                                      F^2\Gamma(h,Q_h).
 \end{aligned}
\]
Substitute these identities in
\(\cL(G)=|\nabla G|^2\Delta G-
\tfrac12\Gamma(G,|\nabla G|^2)\), and use
\(Q_h=H_h+2\tau\).  The terms proportional to \(F\) and
\(F^3\) give \eqref{eq:harmonic-log-defect}.
The additional divisibility assumption gives
\eqref{eq:harmonic-log-two-coefficients}.

On the regular zero set, \(\cK_F(a)=a^3D_h\), so
\eqref{eq:boundary-potential-algebraic} gives
\eqref{eq:harmonic-log-Jacobi}.
\end{proof}

\begin{theorem}[A monotonicity criterion for harmonic multipliers]
\label{thm:harmonic-monotonicity-criterion}
In the setting of \cref{lem:harmonic-log-multiplier}, suppose on
\(U\) that
\begin{equation}\label{eq:harmonic-log-monotonicity}
 Q_h\ge0,\qquad D_h\ge0,\qquad \Delta h\ge0,\qquad
 F\Gamma(F,Q_h)\le0,\qquad \Gamma(h,Q_h)\le0.
\end{equation}
Then \(\nabla(aF)\ne0\) on \(U\), and
\(F\cL(aF)\ge0\) there. With
\(Z=\R^N\setminus U=\{\nabla F=0\}\), the field
\[
 X=\frac{\nabla(aF)}{|\nabla(aF)|}\quad\text{on }U
\]
is a two-sided subcalibration of \(E=\{F<0\}\) across \(Z\).
Consequently \(E\) is locally perimeter minimizing and
\(\partial\llbracket E\rrbracket\) is locally mass minimizing.
No regularity of the spherical link is assumed.
\end{theorem}

\begin{proof}
By \eqref{eq:harmonic-log-gradient},
\(|\nabla(aF)|^2\ge a^2S>0\).  In the bracket in
\eqref{eq:harmonic-log-defect}, each of
\(D_h\), \(F^2Q_h\Delta h\),
\(-F^2\Gamma(h,Q_h)/2\), and
\(-F\Gamma(F,Q_h)/2\) is nonnegative.  Thus
\(F\cL(aF)\ge0\), and
\eqref{eq:div-normalized-gradient} gives the required divergence
signs.  At every regular zero, \(\nabla(aF)=a\nabla F\), so
\(X\) is the exterior unit normal to \(E\).

By \cref{rem:polynomial-finite-perimeter,lem:square-free-critical-set},
\(E\) has locally finite perimeter and its algebraic critical set
\(Z\subset\{F=0\}\) has locally zero \(\cH^{N-1}\)-measure.
Apply \cref{thm:removable-subcalibration} to \(X\); no extension
of \(a\) across \(Z\) is needed.
\end{proof}

\begin{remark}[Alternative coefficient tests]
\label{rem:harmonic-alternative-tests}
If \(\Gamma(F,Q_h)=\kappa F\), the monotonicity conditions may be replaced by
\(Q_h\ge0\), \(D_h\ge0\), and \(D_{1,h}\ge0\).
Without this divisibility assumption, the conditions
\[
 Q_h\ge0,\quad D_h\ge0,\quad P_h\ge0,\quad
 \Gamma(F,Q_h)^2\le16D_hP_h
\]
also suffice: the last inequality gives
\(\tfrac12|F\Gamma(F,Q_h)|\le D_h+F^2P_h\).
All inequalities must hold throughout \(U\).
\end{remark}

For the natural choice \(a=\sqrt S\),
the quantities in the criterion use only \(q,S\) and their derivatives. Writing
\(K=\Delta S\) and \(T=\Gamma(S,S)\), one obtains
\begin{equation}\label{eq:harmonic-gradient-multiplier-data}
 \begin{aligned}
 h&=\tfrac12\log S,& \tau&=q/S,&
 Q_h&=\frac{T}{4S^2}+\frac{2q}{S},\\
 \Delta h&=\frac{K}{2S}-\frac{T}{2S^2},&&&
 D_h&=\frac K2-2q-\frac TS.
 \end{aligned}
\end{equation}
Under the spectral identities below, the two monotonicity signs hold,
and the remaining condition reduces to a lower bound for \(\Delta S\).

\begin{remark}[The nonharmonic correction]
\label{rem:nonharmonic-log-multiplier}
Suppose instead that \(\cL(F)=-qF\), without assuming
\(\Delta F=0\), and retain \(U=\{S>0\}\),
\eqref{eq:harmonic-log-compatibility}, and the auxiliary functions above.
The term \(2S\Delta F\) in \(\Gamma(F,S)\) changes only the defect:
\[
 \cL(aF)=a^3F\left(D_h+FQ_h\Delta F+F^2P_h
                         -\tfrac F2\Gamma(F,Q_h)\right).
\]
The gradient identity is unchanged, so the square-completion test of
\cref{rem:harmonic-alternative-tests} becomes
\[
 Q_h\ge0,\qquad D_h\ge0,\qquad P_h\ge0,\qquad
 \bigl(\Gamma(F,Q_h)-2Q_h\Delta F\bigr)^2\le16D_hP_h.
\]
It gives comparison when \(Z=\{S=0\}\) is locally
\(\cH^{N-1}\)-null, in particular for square-free homogeneous \(F\)
by \cref{lem:square-free-critical-set}. The natural choice
\(h=\frac12\log S\) must still satisfy the compatibility assumption:
now \(\Gamma(F,h)=\Delta F+qF/S\), which need not be divisible
smoothly by \(F\) across its regular zero set.
\end{remark}

\subsection{Radial multipliers in arbitrary degree}
\label{sec:general-radial-multipliers}

Let \(F\) be a nonzero homogeneous polynomial of degree \(e\ge2\),
put \(S=|\nabla F|^2\), and suppose \(\cL(F)=-qF\) for a
polynomial \(q\).
The identity in \cref{prop:radial-master-identity} gives the following
radial specialization.
For \(\sigma>0\), set
\begin{align}
 A_\sigma&=\sigma(N-1-2e-\sigma),\qquad
 c_\sigma=\sigma(\sigma+2e),\notag\\
 B_\sigma&=c_\sigma\bigl(e+\sigma(N-1)\bigr).
 \label{eq:general-radial-coefficients}
\end{align}
Then, on \(r>0\),
\begin{align}
 |\nabla(r^\sigma F)|^2
 &=r^{2\sigma}\left(S+c_\sigma\frac{F^2}{R}\right),
 \label{eq:general-radial-gradient}\\
 \cL(r^\sigma F)&=r^{3\sigma-2}F\,\mathscr H_\sigma,
 \label{eq:general-radial-defect}\\
 \mathscr H_\sigma
 &=A_\sigma S-Rq+c_\sigma F\Delta F
                  +B_\sigma\frac{F^2}{R}.
 \label{eq:general-radial-bracket}
\end{align}
The exact sign condition is \(\mathscr H_\sigma\ge0\) on \(r>0\).
A sufficient test follows.

\begin{proposition}[A radial criterion without harmonicity]
\label{prop:general-radial-existence}
Let \(F\) be a homogeneous polynomial of degree \(e\ge2\) on
\(\R^N\), \(N\ge3\), taking both signs,
and suppose \(\cL(F)=-qF\) for a polynomial \(q\).
Assume that \(Z=\{\nabla F=0\}\) has dimension at most \(N-2\),
as holds whenever \(F\) is square-free by
\cref{lem:square-free-critical-set}.
If constants \(\kappa,\eta\ge0\) and \(\sigma>0\) satisfy
\begin{equation}\label{eq:general-radial-comparisons}
 Rq\le\kappa S,\qquad F\Delta F\ge-\eta S
 \quad\hbox{on }\R^N,
\end{equation}
and
\begin{equation}\label{eq:general-radial-sufficient}
 \sigma\bigl[N-1-2e(1+\eta)-(1+\eta)\sigma\bigr]\ge\kappa,
\end{equation}
then \(G=r^\sigma F\) generates a two-sided subcalibration and
\(\partial\llbracket\{F<0\}\rrbracket\) is strictly area minimizing.
Moreover, on \(C_{\reg}=\{F=0,S>0\}\),
\begin{equation}\label{eq:general-radial-Hardy}
 Q_C(u)\ge
 \left[\frac{(N-1-2e)^2}{4}-\kappa\right]
 \int_Cr^{-2}u^2,
 \qquad u\in C_c^\infty(C_{\reg}).
\end{equation}
\end{proposition}

\begin{proof}
The assumptions give
\[
 \mathscr H_\sigma\ge
 (A_\sigma-c_\sigma\eta-\kappa)S
 +B_\sigma\frac{F^2}{R}\ge0.
\]
Since \(\sigma,c_\sigma,B_\sigma>0\), the bracket is positive off
the cone. Euler's identity and \eqref{eq:general-radial-gradient}
give \(Z\subset\{F=0\}\) and \(\nabla G\ne0\) off \(Z\).
The continuous potential has the signs of \(F\), and the algebraic
set \(Z\) is locally \(\cH^{N-1}\)-null. Thus
\cref{cor:gradient-subcalibration,thm:strict-minimality-flux}
give minimization and its strict form.

The Jacobi trace \eqref{eq:Jacobi-trace} and
\eqref{eq:general-radial-defect} give, on \(C_{\reg}\),
\begin{equation}\label{eq:general-radial-Jacobi}
 -\frac{J_Cw}{w}=\frac{A_\sigma}{R}-\frac qS
 \ge\frac{A_\sigma-\kappa}{R},
 \qquad w=r^{-\sigma}S^{-1/2}.
\end{equation}
Apply \cref{thm:ground-state-gap} with \(\ell=e+\sigma\) and use
\[
 \left(e+\sigma-\frac{N-1}{2}\right)^2+A_\sigma
 =\frac{(N-1-2e)^2}{4}
\]
to obtain \eqref{eq:general-radial-Hardy}.
\end{proof}

Writing \(b_\eta=N-1-2e(1+\eta)\), the left side of
\eqref{eq:general-radial-sufficient} is maximized, when
\(b_\eta>0\), at \(\sigma=b_\eta/[2(1+\eta)]\).
This yields the explicit sufficient dimension condition
\begin{equation}\label{eq:general-radial-dimension}
 b_\eta>0,\qquad
 N\ge1+2e(1+\eta)+2\sqrt{(1+\eta)\kappa}.
\end{equation}
Separating the terms in \eqref{eq:general-radial-bracket} can lose sharpness:

\begin{example}[Nonsharpness of the simplified criterion]
\label{ex:radial-separated-nonsharp}
Let \(F=7u-v\), where \(u=|x|^2\), \(v=|y|^2\), and
\((x,y)\in\R^2\times\R^8\). This is a classical Lawson cone;
see \cite{Liu19} for quantitative subcalibrations in this family.
Here
\[
 \Delta F=12,\qquad S=196u+4v,\qquad \cL(F)=-56F.
\]
The least nonnegative constants in
\eqref{eq:general-radial-comparisons} are
\(\kappa=14\) and \(\eta=3\), both attained at \(u=0<v\).
Consequently \eqref{eq:general-radial-sufficient} would require
\(\sigma(-7-4\sigma)\ge14\), which fails for every
\(\sigma>0\). Nevertheless, the radial polynomial potential
\(G=RF\) satisfies the exact identity
\begin{equation}\label{eq:radial-Lawson-counterexample}
 \cL(G)=64RF\bigl[(v-11u)^2+96u^2\bigr].
\end{equation}
Since \(S>0\) away from the origin, this gives a two-sided
subcalibration with strict signed divergence off the cone.
Thus \eqref{eq:general-radial-sufficient} need not hold for any
positive exponent, even with the least admissible comparison constants.
\end{example}

Finite comparison constants also need not exclude nonzero critical
points without harmonicity. For this explicitly reducible comparison,
outside the introductory convention, put \(u=|x|^2\), \(v=|y|^2\), with
\((x,y,z)\in\R^4\times\R^4\times\R^2\), the square-free
but reducible polynomial \(F=u^2-v^2\) satisfies
\[
 S=16(u^3+v^3),\qquad \Delta F=24(u-v),\qquad
 \cL(F)=192F(u-v)^2.
\]
Hence \(\kappa=\eta=0\) are admissible in
\eqref{eq:general-radial-comparisons}, and \(\sigma=1/2\)
satisfies \eqref{eq:general-radial-sufficient}. Its zero set is the
Simons cone in \(\R^8\) times \(\R^2\), with critical set
\(\{x=y=0\}\). The quartic defining equation includes the
positive factor \(u+v\) away from that set.

If \(F\) belongs to a Laplacian algebra \(\cA\) containing
\(R\), then \(S,\Delta F\in\cA\), and the division property
of \cref{prop:MR-finite-and-division} gives \(q\in\cA\).
Thus \(R\mathscr H_\sigma\in\cA\), and its sign on \(R>0\)
can be tested on the real quotient using the closure table.

\begin{corollary}[A harmonic radial criterion in arbitrary degree]
\label{cor:harmonic-radial-existence}\label{main:harmonic-radial}
Under the hypotheses of \cref{prop:harmonic-power-Jacobi}, suppose that
for a constant \(\kappa\ge0\),
\begin{equation}\label{eq:harmonic-radial-bound}
 Rq\le\kappa S\quad\hbox{on }\R^N,\qquad
 N\ge2e+1+2\sqrt\kappa.
\end{equation}
Then, with \(\sigma=(N-1-2e)/2\), the potential \(G=r^\sigma F\)
generates a two-sided subcalibration. The phase \(\{F<0\}\) is locally
perimeter minimizing, and its oriented boundary current is locally
mass minimizing with support \(C=\{F=0\}\). The cone is strictly
area minimizing, and
\begin{equation}\label{eq:harmonic-radial-Hardy}
 Q_C(u)\ge\left[\frac{(N-1-2e)^2}{4}-\kappa\right]
 \int_Cr^{-2}u^2,
 \qquad u\in C_c^\infty(C\setminus\{0\}).
\end{equation}
\end{corollary}

\begin{proof}
Write \(b=N-1-2e\). By
\cref{prop:harmonic-global-comparison} below,
\(S>0\) off the origin and \(\kappa\ge e(e-1)>0\).
The dimension assumption gives \(b>0\), and
\(\sigma=b/2\) satisfies
\(\sigma(b-\sigma)=b^2/4\ge\kappa\).
Harmonicity gives both nonempty phases and
\(F\Delta F=0\). Thus \cref{prop:general-radial-existence}
applies with \(\eta=0\) and \(Z=\{0\}\), proving the
subcalibration, strict minimality, and Hardy bound.
The support of the boundary current is all of \(\{F=0\}\)
by \cref{rem:polynomial-finite-perimeter}.
\end{proof}

Any \(\sigma\ge0\) with \(\sigma(N-1-2e-\sigma)\ge\kappa\)
works; even integers give polynomial potentials, as in the
\hyperref[par:Simons-degree-example]{Simons example}. A local radial
subcalibration already implies strict stability for regular harmonic cones:

\begin{corollary}[Strict stability from a harmonic radial potential]
\label{cor:harmonic-radial-strict-stability}
Let $F$ be a nonzero real homogeneous harmonic polynomial
of degree $e\ge2$ on $\R^{n+1}$, $n\ge3$. Suppose $C=\{F=0\}$
is minimal, $\nabla F\ne0$ on $C\setminus\{0\}$, and its link
is nonempty, smooth, compact, and connected. If for some
$\alpha\in\R$ the potential $G=r^\alpha F$ satisfies
\[
 \operatorname{sgn}(F)\diver X_G\ge0\qquad\text{where }F\ne0
\]
in a two-sided neighborhood of $C\setminus\{0\}$ on which $X_G$
is defined, then $C$ is strictly stable.
\end{corollary}

\begin{proof}
The hypotheses already imply that $F$ is real irreducible.
Indeed, \cref{lem:square-free-critical-set} shows that $F$ is
square-free and each real irreducible factor has a nonempty zero
set on the unit sphere. These closed sets are pairwise disjoint,
since a common zero of two factors would give $\nabla F=0$.
Their union is the connected link, so there can be only one factor.
Set $R=r^2$ and $S=|\nabla F|^2$.
By \cref{prop:divisibility}, $\cL(F)=-qF$ for a homogeneous
polynomial $q$ of degree $2e-4$, possibly zero.
The sign hypothesis and \cref{thm:Jacobi-trace} make
$w=r^{-\alpha}S^{-1/2}$ a positive Jacobi supersolution.
Suppose that the Hardy gap is zero. Equation
\eqref{eq:critical-boundary-rigidity} gives
$e+\alpha=n/2$ and $J_Cw=0$.
With $c=\alpha(\alpha+2e)(e+\alpha n)$, the radial identity
of \cref{prop:radial-master-identity} and $\Delta F=0$ give
\begin{equation}\label{eq:harmonic-critical-radial-defect}
 \cL(G)=r^{3\alpha-4}FB,\qquad
 B=R(\alpha^2S-Rq)+cF^2.
\end{equation}
Thus $B\ge0$ near the cone, and $J_Cw=0$ together with
\cref{thm:Jacobi-trace} gives $B|_{C\setminus\{0\}}=0$.
The divisibility argument in \cref{prop:divisibility} yields
$B=FB_1$. At a regular zero, $F$ changes sign, so $B\ge0$
on both sides forces $B_1$ to vanish there as well.
Applying the same argument again gives $F^2\mid B$.
Since $F$ has nonzero real zeros, it is coprime to $R$.
Equation \eqref{eq:harmonic-critical-radial-defect} therefore implies
$F^2\mid\alpha^2S-Rq$. Its degree is at most $2e-2<2e$, so
\begin{equation}\label{eq:harmonic-critical-radial-obstruction}
 Rq=\alpha^2S\qquad\text{identically}.
\end{equation}
Choose an extremum $p$ of $F$ on the unit sphere with $F(p)\ne0$.
Euler's identities give
$\nabla F(p)=eF(p)p$ and
$\Hess F(p)p=(e-1)\nabla F(p)$.
Using harmonicity and $\cL(F)=-qF$ gives
$q(p)=e(e-1)S(p)$, with $S(p)>0$.
Consequently \eqref{eq:harmonic-critical-radial-obstruction}
forces $\alpha^2=e(e-1)$.
Together with $\alpha=n/2-e$, this becomes $n^2=4e(n-1)$,
which is impossible for integers $n\ge3$, since it would imply
$n-1\mid1$.
\end{proof}

\begin{proposition}[Global comparison and the critical set]
\label{prop:harmonic-global-comparison}
Let \(F\not\equiv0\) be a homogeneous polynomial of degree \(e\ge2\)
on \(\R^N\), \(N\ge3\), satisfying
\(\Gamma(F,S)=2qF\), where \(S=|\nabla F|^2\) and \(q\) is a
polynomial. Then
\begin{equation}\label{eq:global-comparison-regularity}
 \exists\,\kappa<\infty:\ Rq\le\kappa S
 \quad\Longleftrightarrow\quad
 \{\nabla F=0\}=\{0\}.
\end{equation}
In this case the least comparison constant is
\begin{equation}\label{eq:harmonic-optimal-comparison}
 \kappa_F=\max_{\Sph^{N-1}}\frac qS\ge e(e-1).
\end{equation}
Consequently the hypotheses of
\cref{cor:harmonic-radial-existence} require \(N\ge4e\).
\end{proposition}

\begin{proof}
If \(S>0\) off the origin, compactness of the unit sphere and
homogeneity give the finite maximum in
\eqref{eq:harmonic-optimal-comparison}.
Conversely, suppose \(c\ne0\) and \(\nabla F(c)=0\).
Euler's identity gives \(F(c)=0\). Write
\[
 F(c+y)=P_m(y)+O(|y|^{m+1}),\qquad m\ge2,
\]
where \(P_m\not\equiv0\) is the first nonzero homogeneous Taylor
term. Choose a unit vector \(v\) maximizing \(|P_m|\), and put
\(p=P_m(v)\ne0\). Then
\(\nabla P_m(v)=mpv\) and
\(\Hess P_m(v)(v,v)=m(m-1)p\). Along \(c+tv\),
\begin{align*}
 S&=m^2p^2t^{2m-2}(1+O(t)),\\
 q=\frac{\Hess F(\nabla F,\nabla F)}{F}
  &=m^3(m-1)p^2t^{2m-4}(1+O(t)).
\end{align*}
It follows that
\begin{equation}\label{eq:critical-comparison-blowup}
 \frac{Rq}{S}(c+tv)
 =\frac{|c|^2m(m-1)}{t^2}+O(t^{-1})\longrightarrow+\infty.
\end{equation}
This rules out a finite comparison constant, including when
\(q(c)=0\).

For the lower bound, take a unit vector maximizing \(|F|\).
The same calculation, now with degree \(e\) at the origin, gives
\(q/S=e(e-1)\) at that vector. Finally, \(4e-1<2e+1+2\sqrt{e(e-1)}<4e\),
so the integer dimension in the harmonic radial criterion is at least \(4e\).
\end{proof}

\begin{proposition}[Rigidity at the sharp comparison constant]
\label{prop:harmonic-sharp-comparison}
Under the algebraic hypotheses of
\cref{prop:harmonic-global-comparison}, let
\(M=\max_{\Sph^{N-1}}|F|\). Then
\begin{equation}\label{eq:harmonic-comparison-rigidity}
 Rq\le e(e-1)S
 \quad\Longleftrightarrow\quad
 S=e^2M^2R^{e-1}.
\end{equation}
If also \(\Delta F=0\) and \(N\ge3\), the equality case is
precisely a harmonic Cartan--M\"unzner polynomial after normalization
by \(M\).
\end{proposition}

\begin{proof}
Put \(f=F|_{\Sph^{N-1}}\) and \(s=S|_{\Sph^{N-1}}\).
The restriction of \(F\) to any great circle is a real trigonometric
polynomial of degree at most \(e\). Applying the classical
Bernstein--Szeg\H{o} inequality
\cite[Section~5.1, Theorem~5.1.3]{BorweinErdelyi95} to a great circle in the
direction of \(\nabla_{\Sph}f\) gives
\begin{equation}\label{eq:spherical-Bernstein-bound}
 s=|\nabla_{\Sph}f|^2+e^2f^2\le e^2M^2.
\end{equation}
On the other hand, homogeneity and \(\Gamma(F,S)=2qF\) give
\[
 \Gamma_{\Sph}(f^2,s)
 =4f^2\bigl(q-e(e-1)s\bigr)\le0.
\]
Thus \(s\) is nonincreasing along the gradient ascent flow of
\(f^2\). The critical set of \(f^2\) is semialgebraic with
finitely many connected components, and \(f^2\) is constant on each
component. Choose \(a<M^2\) above every other critical value.
Every forward trajectory starting in the nonempty open set
\(\{f^2>a\}\) has its accumulation points in the critical set at
value \(M^2\): the value increases along the trajectory, and every
accumulation point of a gradient flow on this compact sphere is
critical. At those points \(s=e^2M^2\). Monotonicity of \(s\)
and \eqref{eq:spherical-Bernstein-bound} therefore force
\(s=e^2M^2\) throughout \(\{f^2>a\}\). Homogeneity and the
polynomial identity principle give the asserted formula for \(S\)
on all of \(\R^N\).

Conversely, differentiating that formula yields
\(q=e^3(e-1)M^2R^{e-2}\), proving the comparison with equality.
When \(F\) is harmonic, \(F/M\) satisfies both Cartan--M\"unzner
equations with equal multiplicities.
\end{proof}

\begin{lemma}[Degree of a nonnegative polynomial first integral]
\label{lem:nonnegative-first-integral-degree}
Let \(F\not\equiv0\) be a homogeneous polynomial of degree \(e\ge1\)
on \(\R^N\), \(N\ge2\), with
\(\nabla F\ne0\) on \(\{F=0\}\setminus\{0\}\).
If \(K\not\equiv0\) is a homogeneous polynomial of degree \(\ell>0\),
\(K\ge0\), and \(\Gamma(F,K)=0\), then \(\ell\ge2e\).
\end{lemma}

\begin{proof}
Choose \(p\in\Sph^{N-1}\) with
\(K(p)=B=\max_{\Sph^{N-1}}K>0\).
The constrained maximum and Euler's identity give
\(\nabla K(p)=\ell Bp\), hence
\[
 0=\Gamma(F,K)(p)=e\ell BF(p).
\]
Thus \(F(p)=0\) and \(w=\nabla F(p)\ne0\), with \(w\perp p\).
Differentiating Euler's identity gives
\(\Hess F(p)p=(e-1)w\). Differentiating
\(\Gamma(F,K)=0\) at \(p\) then gives
\[
 \Hess K(p)w=-\Hess F(p)\nabla K(p)=-\ell B(e-1)w.
\]
For \(v=w/|w|\), the function
\(T(t)=K(p\cos t+v\sin t)\) is a real trigonometric polynomial
of degree at most \(\ell\), with \(0\le T\le B\), and
\[
 T''(0)=\Hess K(p)(v,v)-\nabla K(p)\cdot p=-e\ell B.
\]
The trigonometric Bernstein inequality, applied twice to \(T-B/2\)
\cite[Section~5.1, Theorem~5.1.4]{BorweinErdelyi95}, yields
\[
 e\ell B\le\|T''\|_\infty
 \le\ell^2\|T-B/2\|_\infty\le\frac{\ell^2B}{2}.
\]
Consequently \(\ell\ge2e\).
\end{proof}

\begin{proposition}[Rigidity for a radial minimality weight]
\label{prop:regular-radial-weight-rigidity}
Let \(F\not\equiv0\) be a homogeneous polynomial of degree \(e\ge2\) on
\(\R^N\), \(N\ge3\), with \(\operatorname{Crit}(F)=\{0\}\).
Put \(S=|\nabla F|^2\) and
\(M_F=\max_{\Sph^{N-1}}|F|\). If
\[
 \Gamma(F,S)=2\lambda R^{e-2}F
\]
for a constant \(\lambda\), then
\[
 \lambda=e^3(e-1)M_F^2,\qquad
 S=e^2M_F^2R^{e-1}=\frac{\lambda}{e(e-1)}R^{e-1}.
\]
If also \(\Delta F=0\), then \(F/M_F\) is a harmonic
Cartan--M\"unzner polynomial. In particular,
\[
 \cA_R(F)=\R[R,F],\qquad e\in\{2,3,4,6\},\qquad
 \kappa_F=e(e-1).
\]
\end{proposition}

\begin{proof}
After replacing \(F\) by \(-F\) if necessary, choose
\(p\in\Sph^{N-1}\) with \(F(p)=M_F\). Then
\(\nabla F(p)=eM_Fp\) and \(S(p)=e^2M_F^2\).
Since \(S\) has degree \(2e-2\),
\[
 2\lambda M_F=\Gamma(F,S)(p)
 =eM_F(2e-2)S(p)=2e^3(e-1)M_F^3.
\]
This determines \(\lambda\). The spherical Bernstein--Szeg\H{o}
bound \eqref{eq:spherical-Bernstein-bound}, extended by homogeneity,
shows that \(K:=e^2M_F^2R^{e-1}-S\ge0\).
Using \(\Gamma(F,R)=2eF\), we obtain
\[
 \Gamma(F,K)
 =2\bigl[e^3(e-1)M_F^2-\lambda\bigr]R^{e-2}F=0.
\]
If \(K\not\equiv0\), its degree \(2e-2\) contradicts
\cref{lem:nonnegative-first-integral-degree}. Hence \(K=0\).

If \(\Delta F=0\), the normalized equations
\(|\nabla(F/M_F)|^2=e^2R^{e-1}\) and \(\Delta(F/M_F)=0\)
are the harmonic Cartan--M\"unzner equations, as in
\cref{prop:harmonic-sharp-comparison}. Together with
\(\Gamma(R,F)=2eF\) and \(\Delta R=2N\), they show that
\(\cA_R(F)=\R[R,F]\). The nonempty zero link is regular by hypothesis
and minimal since \(\cL(F)=-\lambda R^{e-2}F\); hence
\cref{thm:rank-one-classification} gives the degree restriction.
Finally, \(Rq/S=e(e-1)\) off the origin.
\end{proof}

\begin{example}[Harmonic isoparametric polynomials]
\label{ex:harmonic-isoparametric-radial}
For a normalized harmonic Cartan--M\"unzner polynomial of degree
\(e\ge2\) on \(\R^N\), \(N\ge3\),
\[
 S=e^2R^{e-1},\qquad q=e^3(e-1)R^{e-2},\qquad
 \kappa_F=e(e-1).
\]
The dimension condition of \cref{cor:harmonic-radial-existence}
is therefore equivalent to \(N\ge4e\). In this range the choice
\(\sigma=2\lfloor e/2\rfloor\in\{e-1,e\}\) satisfies
\(\sigma(N-1-2e-\sigma)\ge e(e-1)\), so
\[
 G=R^{\lfloor e/2\rfloor}F
\]
is a polynomial subcalibration potential.
The realizable equal-multiplicity parameters satisfying this threshold
are listed in \cref{sec:harmonic-isoparametric-scope}, together with
the Clifford-algebra explanation for the quartic restriction.
By \cref{prop:regular-radial-weight-rigidity}, they exhaust
\cref{cor:harmonic-radial-existence} when \(q\) is radial:
harmonic quadrics in even dimensions \(N\ge8\), and Cartan cubics
in dimensions \(14\) and \(26\).
\end{example}

\begin{proposition}[Positive multipliers at a harmonic critical zero]
\label{prop:harmonic-positive-multiplier-obstruction}
\label{rem:harmonic-radial-singular-obstruction}
Let \(F\not\equiv0\) be a harmonic polynomial, and suppose
\(F(c)=0=\nabla F(c)\). If \(a\) is \(C^2\) near \(c\) and
\(a(c)>0\), then \(G=aF\) cannot satisfy \(G\cL(G)\ge0\)
throughout any neighborhood of \(c\).
\end{proposition}

\begin{proof}
Use the first nonzero Taylor term \(P_m\), the vector \(v\), and
\(p=P_m(v)\ne0\) as in the proof of
\cref{prop:harmonic-global-comparison}. Harmonicity gives
\(\Delta(aF)=2\Gamma(a,F)+F\Delta a=O(t^{m-1})\) along
\(c+tv\). The leading Hessian term therefore yields
\[
 F(c+tv)\cL(aF)(c+tv)
 =-a(c)^3m^3(m-1)p^4t^{4m-4}+o(t^{4m-4})<0
\]
for all sufficiently small positive \(t\). Here \(a>0\) and
\(\nabla(aF)\ne0\), so the required sign fails.
\end{proof}

\begin{remark}[Quartic and quintic scope]
\label{rem:higher-degree-radial-scope}
The criterion of \cref{cor:harmonic-radial-existence} applies in every
degree but requires a regular link. In degrees four and five its
conditions read, respectively, \(N\ge16\),
\(\kappa_F\le(N-9)^2/4\), and \(N\ge20\),
\(\kappa_F\le(N-11)^2/4\). Degree five is absent from the
Cartan--M\"unzner list, so \cref{prop:harmonic-sharp-comparison}
forces \(\kappa_F>20\) in that case.
Moreover, \cref{prop:regular-radial-weight-rigidity} forces \(q\)
to be nonradial in any degree-five example satisfying the criterion.
The determinant cones in \cref{main:complex-determinants} instead
have nonzero critical points and use nonradial minor-norm multipliers.
Likewise, adding cylindrical directions cannot produce new examples
of the radial criterion, by
\cref{prop:harmonic-positive-multiplier-obstruction}.
\end{remark}

\subsection{Homogeneous potentials on a spherical quotient}
\label{sec:spherical-quotient-potentials}

Let \(I=(I_1,\ldots,I_j)\), where each invariant \(I_i\) is
homogeneous of degree \(d_i\), and write
\[
 z_i=\frac{I_i}{r^{d_i}},\qquad
 K_{\Sph}=z(\Sph^{N-1}).
\]
For an ambient closure containing \(R\), Euler's identity
\(\nabla_{\Sph}I_i=\nabla I_i-d_iI_i\theta\) and the polar
Laplacian give the spherical closure data
\begin{equation}\label{eq:spherical-closure-from-ambient}
 g^{\Sph}_{ij}=\left.\Gamma(I_i,I_j)\right|_{R=1}-d_id_jz_iz_j,
 \qquad
 b^{\Sph}_i=\left.\Delta I_i\right|_{R=1}
              -d_i(d_i+N-2)z_i.
\end{equation}
More generally, assume the spherical data close as functions of \(z\).
Let \(\widehat\Gamma_{\Sph}\), \(\widehat\Delta_{\Sph}\) be their
operators from \cref{thm:finite-reduction}, with signs tested on
the actual image \(K_{\Sph}\).

\begin{proposition}[The spherical equation for a homogeneous potential]
\label{prop:spherical-quotient-potential}
Let $\varphi$ be smooth on a neighborhood of the relevant part of
$K_{\Sph}$ in the ambient invariant-coordinate space.
For \(G=r^\ell \varphi(z)\), define
\begin{align}
 W_\ell[\varphi]&=\ell^2\varphi^2+\widehat\Gamma_{\Sph}(\varphi,\varphi),
 \label{eq:general-spherical-W}\\
 \mathfrak N_\ell[\varphi]
 &=W_\ell[\varphi]\bigl(\widehat\Delta_{\Sph}\varphi+(N-1)\ell \varphi\bigr)
       -\tfrac12\widehat\Gamma_{\Sph}(\varphi,W_\ell[\varphi]).
 \label{eq:general-spherical-N}
\end{align}
Then
\begin{equation}\label{eq:general-spherical-L}
 |\nabla G|^2=r^{2\ell-2}W_\ell[\varphi],\qquad
 \cL(G)=r^{3\ell-4}\mathfrak N_\ell[\varphi].
\end{equation}
Where \(W_\ell[\varphi]>0\), one has
\(r\diver X_G=\mathfrak N_\ell[\varphi]/W_\ell[\varphi]^{3/2}\).
These identities hold at singular quotient values whenever their
pullbacks are smooth and \(W_\ell[\varphi]>0\).
\end{proposition}

\begin{proof}
Polar differentiation gives
\[
 \nabla G=r^{\ell-1}(\ell \varphi\,\partial_r+\nabla_{\Sph}\varphi),\qquad
 \Delta G=r^{\ell-2}
       [\widehat\Delta_{\Sph}\varphi+\ell(\ell+N-2)\varphi].
\]
The radial and angular terms give
\[
 \tfrac12\Gamma(G,|\nabla G|^2)
 =r^{3\ell-4}\left[
  \ell(\ell-1)\varphi W_\ell[\varphi]
  +\tfrac12\widehat\Gamma_{\Sph}(\varphi,W_\ell[\varphi])\right].
\]
Subtracting this expression from \(|\nabla G|^2\Delta G\) proves the formula.
\end{proof}

For a nonzero homogeneous polynomial \(F\) of degree \(e\geq1\),
put \(R=r^2\) and \(S=|\nabla F|^2\).
When these quantities give a closed quotient as above, the angular
variables are
\begin{equation}\label{eq:rank-two-general-variables}
 f=\frac{F}{R^{e/2}},\qquad
 s=\frac{S}{R^{e-1}},\qquad K_{\Sph}=\{(f(x),s(x)):r=1\}.
\end{equation}
If this image has dimension two, a homogeneous invariant potential
with a positive multiplier has the form
\[
 G=r^\ell f\,a(f,s),\qquad a>0.
\]
The power family \eqref{eq:power-family} corresponds to
\(\varphi(f,s)=fs^\gamma\) and
\(\ell=e+2\alpha+2\gamma(e-1)\).  It is a restricted family of
angular multipliers.

\begin{remark}[Phase-dependent quotient potentials]
\label{cor:rank-two-paired-potentials}
Let \(N\geq3\).  Suppose \(E=\{F<0\}\) has smooth boundary off a
closed conical set \(Z\) containing the origin and satisfying the
null-set hypothesis of \cref{thm:paired-subcalibration}. Use the quotient
\eqref{eq:rank-two-general-variables}.
Choose degrees \(\ell_-\), \(\ell_+\) and put
\(h_\pm=f a_\pm(f,s)\), where the positive multipliers \(a_\pm\)
are smooth on invariant-coordinate neighborhoods of the respective
closed-phase images outside \(Z\).
Assume there that \(W_{\ell_\pm}[h_\pm]>0\), and assume
\[
 \mathfrak N_{\ell_-}[h_-]\leq0\quad(f<0),\qquad
 \mathfrak N_{\ell_+}[h_+]\geq0\quad(f>0).
\]
Then \(G_\pm=r^{\ell_\pm}h_\pm\) give
\eqref{eq:paired-weighted-comparison} with weights on the respective phases
\[
 q_\pm=\frac1r\,
  \frac{|\mathfrak N_{\ell_\pm}[h_\pm]|}
       {W_{\ell_\pm}[h_\pm]^{3/2}},
\]
extended by zero outside their phases. The boundary is area minimizing,
and strictly area minimizing if both unit-ball phase integrals are positive.
Indeed, on the regular boundary,
\(\nabla G_\pm=r^{\ell_\pm-e}a_\pm\nabla F\), so both normalized
fields equal \(\nu_E\).  The spherical equation gives their phase
signs and weights.  Apply \cref{rem:phase-local-fields},
\cref{thm:paired-subcalibration}, and
\cref{cor:paired-strict-minimality}. The two multipliers need not agree
on \(f=0\).
\end{remark}

A useful initial family is \(G=R^\alpha S^\gamma F\), with
\(S=|\nabla F|^2\).  Its full multiplier identities are collected in
\cref{app:power-identities}, especially
\cref{prop:combined-power-multiplier}.

\section[Harmonic Pullbacks and Transfer]{Harmonic Pullbacks and Subcalibration Transfer}
\label{sec:quartic-application}

Riedler proves regular-locus minimality under umbilic harmonic morphisms
and area minimization for Clifford--Simons quartics
\cite[Theorems~5.1.4 and~5.2.5]{RiedlerThesis25}, and asks about general
minimal cones \cite[Question~5.3.3(2), p.~88]{RiedlerThesis25}.
\Cref{thm:clifford-area-transfer,thm:smallest-clifford-pullback}
control the dilation term with a source-radius factor and determine the
smallest algebra, yielding \cref{cor:iterated-clifford-transfer}.
Quadratic and cubic applications appear in
\cref{sec:two-quadratic,sec:nonradial-examples}.

\subsection{The pullback identity and the dilation term}

For a degree-\(d\) homogeneous polynomial harmonic morphism
\(\phi:\R^N\to\R^m\), \(d\ge1\), the Fuglede--Ishihara
characterization, as recalled in
\cite[Theorem~2.1.2]{RiedlerThesis25},
gives, on its regular set, a dilation \(\lambda>0\) with
\[
 \Delta\phi_i=0,\qquad
 \Gamma(\phi_i,\phi_j)=\lambda^2\delta_{ij}.
\]
In particular, \(\|d\phi\|_{\mathrm{HS}}^2=m\lambda^2\).
Subscripts \(x\) and \(y\) distinguish source and target operators.

\begin{proposition}[Pullback criterion]
\label{prop:pullback-subcalibration}
Suppose that, on the regular set,
\begin{equation}\label{eq:pullback-dilation-condition}
 d\phi\bigl(\nabla_x\lambda^2\bigr)=\eta(x)\phi(x).
\end{equation}
Equivalently,
\(d\phi(\nabla_x\|d\phi\|_{\mathrm{HS}}^2)=m\eta\phi\);
this records the factor \(m\) in the differential-norm formulation
of \cite[Definition~5.1.1(1)]{RiedlerThesis25}.
Let \(H\) be a nonzero homogeneous polynomial of degree \(\ell>0\)
on \(\R^m\), put \(G=H\circ\phi\), and write
\(W_H=|\nabla_yH|^2\).  Then
\begin{align}
 |\nabla_xG|^2&=\lambda^2(W_H\circ\phi),\qquad
 \Delta_xG=\lambda^2(\Delta_yH)\circ\phi,
 \label{eq:pullback-gradient-laplacian}\\
 \cL_x(G)&=\lambda^4\cL_y(H)\circ\phi
       -\frac{\ell\eta}{2}(H W_H)\circ\phi.
 \label{eq:pullback-defect}
\end{align}
Suppose that a constant \(\kappa\in\R\) satisfies
\begin{equation}\label{eq:pullback-kappa-bound}
 \frac{|\phi(x)|^2\eta(x)}{2\lambda(x)^4}\leq\kappa
\end{equation}
at regular points with \(\phi(x)\ne0\), and that
\begin{equation}\label{eq:pullback-target-sign}
 |y|^2H(y)\cL_y(H)(y)
 \geq\kappa\ell H(y)^2W_H(y)
 \qquad(y\ne0).
\end{equation}
Then \(G\cL_x(G)\geq0\) at regular points with \(\phi(x)\ne0\).

If \(G\not\equiv0\) and a closed algebraic set \(Z\subset\R^N\)
of dimension at most \(N-2\) contains the critical set of \(G\),
the singular zeros of \(G\), and the points excluded from this
calculation, then \(\{G<0\}\) is locally perimeter minimizing and
\(\partial\llbracket\{G<0\}\rrbracket\) is locally mass minimizing.
\end{proposition}

\begin{proof}
The chain rule gives \eqref{eq:pullback-gradient-laplacian}.  Moreover,
\begin{align*}
 \tfrac12\Gamma_x(G,|\nabla_xG|^2)
 &=\frac{W_H\circ\phi}{2}
   \left\langle d\phi(\nabla_x\lambda^2),
                   (\nabla_yH)\circ\phi\right\rangle\\
 &\quad+\lambda^4\bigl(\Hess_yH(\nabla_yH,\nabla_yH)\bigr)\circ\phi.
\end{align*}
Euler's identity \(\langle y,\nabla_yH\rangle=\ell H\)
and \eqref{eq:pullback-dilation-condition} prove
\eqref{eq:pullback-defect}.  At a regular point with
\(y=\phi(x)\ne0\), the two assumed inequalities give
\[
 G\cL_x(G)
 \geq\ell H(y)^2W_H(y)
       \left(\frac{\kappa\lambda^4}{|y|^2}
                    -\frac{\eta}{2}\right)\geq0.
\]
The global conclusion follows from
\cref{rem:polynomial-finite-perimeter,cor:gradient-subcalibration},
since the algebraic dimension bound makes \(Z\) locally
\(\cH^{N-1}\)-null.
\end{proof}

Both terms in \eqref{eq:pullback-defect} vanish on a regular minimal
level cone \(H=0\); away from the cone, the dilation term can change
the sign. For example, if \(H\) is linear and
\(\eta>0\), then \(\cL_y(H)=0\) but
\(G\cL_x(G)=-\eta H^2W_H/2<0\) wherever \(H\ne0\).

\begin{remark}[Weighted interpretation]
Condition \eqref{eq:pullback-target-sign} gives the phase inequality
for the normalized gradient with density \(|y|^{-\kappa}\).
Indeed, for \(X_H=\nabla H/|\nabla H|\),
\[
 |y|^{\kappa}\diver_y\bigl(|y|^{-\kappa}X_H\bigr)
 =\frac{\cL_y(H)}{W_H^{3/2}}
       -\frac{\kappa\ell H}{|y|^2W_H^{1/2}}.
\]
It is the density of the Euler--Dierkes--Huisken functional
\(\mathcal E_\alpha(\Sigma)=\int_\Sigma |y|^\alpha\,d\mathcal H^{m-1}\)
with \(\alpha=-\kappa\).
Minimal cones are stationary on their regular parts for area functionals
with smooth positive radial densities. Subcalibration methods also apply
to proving minimality for these functionals; see, for example,
Dierkes~\cite{Dierkes26Euler}.
For the Clifford maps below, \(\alpha=-1/2\); the dilation bound
converts the weighted inequality into ordinary-area comparison.
\end{remark}

For the subclass \(\lambda^2=c|x|^{2d-2}\), with \(c>0\),
Euler's identity gives
\[
 \eta=2d(d-1)c|x|^{2d-4},\qquad
 d^2|\phi(x)|^2=|d\phi_x(x)|^2\leq\lambda^2|x|^2.
\]
Hence \eqref{eq:pullback-kappa-bound} holds with
\(\kappa=(d-1)/d\).

\subsection{Area-minimizing Clifford pullbacks}

Let \(A_1,\ldots,A_m\) be a symmetric Clifford system on \(\R^N\),
where \(m\geq2\):
\[
 A_iA_j+A_jA_i=2\delta_{ij}\Id.
\]
Set \(P_i(x)=\langle A_ix,x\rangle\), \(P=(P_1,\ldots,P_m)\),
\(R=|x|^2\), and \(Q=|P|^2\).  The Clifford relations imply
\begin{equation}\label{eq:Clifford-pullback-identities}
 \Delta P_i=0,\qquad
 \Gamma(P_i,P_j)=4R\delta_{ij},\qquad
 \Gamma(R,P_i)=4P_i.
\end{equation}
For example, a second Clifford matrix interchanges the two eigenspaces
of \(A_i\), so \(\tr A_i=0\). The classification of quadratic
harmonic morphisms by Ou and Wood, including its Clifford-system
and direct-sum description, is recalled in
\cite[Theorem~4.5.3 and Remark~4.5.5]{RiedlerThesis25}.
The maps above are normalized umbilic maps: all positive eigenvalues
of their component matrices equal one; see
\cite[Remark~5.1.12]{RiedlerThesis25}.

\begin{lemma}[Laplacian algebra under Clifford pullback]
\label{lem:Clifford-algebra-pullback}
If \(\cA\subset\R[y_1,\ldots,y_m]\) is a graded Laplacian algebra,
then \(\R[R,P^*\cA]\) is a Laplacian algebra on \(\R^N\).
Moreover, \(P\) is onto, \(Q\leq R^2\), and
\(P^{-1}(0)\setminus\{0\}\), when nonempty, is a smooth
submanifold of codimension \(m\).
For this map, \cref{prop:pullback-subcalibration} applies with
\(\lambda^2=4R\), \(\eta=16\), and \(\kappa=1/2\).
\end{lemma}

\begin{proof}
For target polynomials \(u,v\), the chain rule gives
\[
 \Delta_x(u\circ P)=4R(\Delta_yu)\circ P,\qquad
 \Gamma_x(u\circ P,v\circ P)=4R\Gamma_y(u,v)\circ P.
\]
If \(u\) has degree \(j\), then
\(\Gamma_x(R,u\circ P)=4j(u\circ P)\).
Together with \(\Delta R=2N\) and \(\Gamma(R,R)=4R\), these
identities and the product rule prove the asserted closure.

For \(x\ne0\), the vectors \(A_ix\) are orthogonal and have
length \(|x|\).  Projection of \(x\) onto their span gives
\(Q/R\leq R\).  To prove surjectivity, take \(u\in\Sph^{m-1}\)
and put \(A(u)=\sum u_iA_i\).  This symmetric involution has trace
zero, since every \(A_i\) does.  It therefore has a unit eigenvector
\(x\) with eigenvalue one.  For \(v\perp u\),
the matrices \(A(v)\) and \(A(u)\) anticommute, so
\(\langle A(v)x,x\rangle=0\).  Hence \(P(x)=u\), and scaling
proves that \(P\) is onto.  Finally, \eqref{eq:Clifford-pullback-identities}
shows that \(dP\) has rank \(m\) at every nonzero point.  The
submersion theorem proves the statement about \(P^{-1}(0)\).
Since \(dP(\nabla(4R))=16P\), the last assertion follows from
\(Q/(2R^2)\leq1/2\).
\end{proof}

In particular, if a target homogeneous polynomial \(f\) has degree
at least two and \(\nabla f(y)\ne0\) for \(y\ne0\), then
\[
 |\nabla(f\circ P)|^2=4R(|\nabla f|^2\circ P).
\]
Thus its critical set is \(Z=P^{-1}(0)\), and its zero set is
regular off \(Z\).  The set \(Z\) is algebraic and has dimension
at most \(N-m\), apart from the possible isolated origin.
Consequently it has zero \(\cH^{N-1}\)-measure. Surjectivity preserves
nonemptiness of both phases. These observations apply to the quadratic
and Cartan targets below.

For a nonflat target hypersurface cone with both phases nonempty
and the isolated critical point just assumed, this is also the
geometric singular set:
\[
 \Sing\{f\circ P=0\}=P^{-1}(0).
\]
At every nonzero point of \(P^{-1}(0)\), submersion coordinates
identify the zero set with a product of the target cone and a smooth
factor. The vertex is singular as well: a smooth cone would be a
hyperplane, which cannot be the zero set of an even homogeneous
polynomial taking both signs.

\label{sec:clifford-area-transfer}

A source-radius factor compensates for the dilation term in
\eqref{eq:pullback-defect} in sufficiently large admissible dimensions.

\begin{theorem}[Area minimization under Clifford pullback]
\label{thm:clifford-area-transfer}
Let \(f\) be a nonzero homogeneous polynomial on \(\R^m\), and let
\(Z\subset\{f=0\}\) be a closed conical real algebraic set containing
the origin and having codimension at least two.  Suppose that
\(H\in C^3(\R^m\setminus Z)\cap C^0(\R^m)\) is homogeneous of
degree \(\ell>0\), with
\[
 \operatorname{sgn}H=\operatorname{sgn}f,\qquad
 |\nabla_yH|>0\quad\text{on }\R^m\setminus Z.
\]
Put \(W_H=|\nabla_yH|^2\), and suppose that, off \(Z\),
\[
 H\cL_y(H)\geq0,\qquad H\Delta_yH\geq-MW_H
\]
for some finite \(M\geq0\).  Let \(P:\R^N\to\R^m\) be the
Clifford map above.  If \(a>0\) and
\begin{equation}\label{eq:clifford-transfer-bound}
 N\geq1+4\ell(1+M)+2a(1+M)+\frac{\ell}{a},
\end{equation}
then
\[
 G=R^aH(P),\qquad E=\{f(P)<0\},
\]
defines a two-sided subcalibration: off \(P^{-1}(Z)\),
\(\nabla_xG\ne0\) and \(G\cL_x(G)\geq0\), with strict inequality
in the two open phases.  Consequently \(E\) is locally perimeter
minimizing and \(\partial\llbracket E\rrbracket\) is locally mass
minimizing. If \(N\geq3\) and both phases are nonempty, this current
is strictly area minimizing. If both signs of \(f\) occur in every
neighborhood of each target zero, then
\[
 \operatorname{spt}\partial\llbracket E\rrbracket=\{f(P)=0\}.
\]
In particular, this last condition holds when \(f\) is harmonic.
\end{theorem}

\begin{proof}
Set \(h=H(P)\) and \(b=a(a+2\ell)\).  The chain rule and Euler's
identity give, on \(\R^N\setminus P^{-1}(Z)\),
\begin{align}
 |\nabla_xG|^2
 &=4R^{2a+1}W_H(P)+4bR^{2a-1}h^2,
 \label{eq:clifford-transfer-gradient}\\
 \cL_x(G)
 &=16R^{3a+2}(\cL_yH)(P)
   +8\bigl[a(N-1)-\ell-2b\bigr]R^{3a}hW_H(P)\notag\\
 &\quad+16bR^{3a}h^2(\Delta_yH)(P)
   +8b\bigl[a(N-1)+\ell\bigr]R^{3a-2}h^3.
 \label{eq:clifford-transfer-defect}
\end{align}
For the second formula, use the radial identity
\eqref{eq:radial-master-L} with the source function \(h\), of degree
\(2\ell\), and substitute
\[
\begin{gathered}
 |\nabla_xh|^2=4RW_H(P),\qquad
 \Delta_xh=4R(\Delta_yH)(P),\\
 \cL_x(h)=16R^2(\cL_yH)(P)-8\ell hW_H(P).
\end{gathered}
\]
These identities follow from \eqref{eq:Clifford-pullback-identities}.

Multiplication of \eqref{eq:clifford-transfer-defect} by \(h\) yields
\begin{align*}
 h\cL_x(G)\geq{}&
 8\bigl[a(N-1)-\ell-2b(1+M)\bigr]R^{3a}h^2W_H(P)\\
 &+8b\bigl[a(N-1)+\ell\bigr]R^{3a-2}h^4.
\end{align*}
The first coefficient is nonnegative by
\eqref{eq:clifford-transfer-bound}, and the last term is positive
when \(h\ne0\).  Since \(R^a>0\), this proves the claimed phase
inequality.  Formula \eqref{eq:clifford-transfer-gradient} gives
\(\nabla_xG\ne0\) off \(Z'=P^{-1}(Z)\).

The extension of \(G\) is continuous with the signs of \(f(P)\);
at the origin, use \(H(P(x))=O(|x|^{2\ell})\).
Surjectivity gives \(f(P)\not\equiv0\), hence local finite perimeter.
Since \(P\) is a submersion off the origin, stratifying \(Z\) shows
that \(Z'=P^{-1}(Z)\) has codimension at least two. Thus
\cref{cor:gradient-subcalibration} applies. For \(N\ge3\) and two
nonempty phases, the strict divergence signs give strict minimality
by \cref{thm:strict-minimality-flux}.

The additional phase assumption transfers by submersion away from
the origin and by surjectivity and homogeneity at the origin.
The support conclusion follows from \cref{rem:polynomial-finite-perimeter}.
\end{proof}

For harmonic nonconstant \(f\), \eqref{eq:Clifford-pullback-identities}
and surjectivity make \(F=f\circ P\) nonzero and harmonic.
Area minimization then implies real irreducibility by
\cref{prop:minimizing-irreducibility}. Nonharmonic targets still
require the stated phase hypothesis for full support.

The ratio bound is preserved under iteration:
\[
 \Delta_xG=4R^{a+1}(\Delta_yH)(P)
       +2a(N+2a+4\ell-2)R^{a-1}h,
\]
and \eqref{eq:clifford-transfer-gradient} gives
\begin{align}
 G\Delta_xG+M|\nabla_xG|^2
 ={}&4R^{2a+1}\bigl(H\Delta_yH+MW_H\bigr)(P)\notag\\
 &+\bigl[2a(N+2a+4\ell-2)+4Mb\bigr]R^{2a-1}h^2
 \geq0.
 \label{eq:clifford-transfer-ratio}
\end{align}
Thus the same \(M\) is available for iteration; see
\cref{cor:iterated-clifford-transfer}. In particular, \(a=1\)
preserves polynomial potentials and gives the sufficient bound
\(N\geq5\ell+3+(4\ell+2)M\).

The harmonic radial criterion supplies further target potentials.
Let \(f\) be homogeneous of degree \(e\) on \(\R^m\) and satisfy
the hypotheses of \cref{cor:harmonic-radial-existence}. Set
\(\sigma=(m-1-2e)/2>0\) and \(H=|y|^\sigma f\). Harmonicity gives
\[
 H\Delta_yH
 =\sigma(\sigma+2e+m-2)|y|^{2\sigma-2}f^2\geq0
 \qquad(y\ne0).
\]
Together with the corollary's phase sign and nonvanishing gradient,
this verifies all target hypotheses with \(Z=\{0\}\),
\(\ell=e+\sigma\), and \(M=0\). Thus
\cref{thm:clifford-area-transfer} applies whenever the source-dimension
bound \eqref{eq:clifford-transfer-bound} holds.

\subsection{The smallest algebra and iteration}
\label{subsec:clifford-smallest-algebra}

Write $q(y)=|y|^2$ on the target.

\begin{theorem}[The smallest algebra under Clifford pullback]
\label{thm:smallest-clifford-pullback}
Let $P:\R^N\to\R^m$ be the Clifford map above and let $f$ be a
nonzero essential homogeneous polynomial of degree $s\ge2$ on
$\R^m$, so $V_{\mathrm{cyl}}(f)=\{0\}$. Then
\begin{equation}\label{eq:smallest-clifford-pullback}
 \cA_R(f\circ P)=\R[R,P^*\cA_q(f)].
\end{equation}
If $Q<R^2$ somewhere, this is isomorphic to the polynomial extension
$\cA_q(f)[t]$, with target degrees doubled and $\deg t=2$.
In particular, its minimal number of homogeneous generators is one
more than that of $\cA_q(f)$.
\end{theorem}

\begin{proof}
Set $\cA=\cA_q(f)$ and $\cB=\cA_R(f\circ P)$.
By \cref{lem:Clifford-algebra-pullback}, $\R[R,P^*\cA]$ is Laplacian
and contains $f\circ P$, so $\cB\subseteq\R[R,P^*\cA]$.
For the reverse inclusion put \(\cC=\{u\in\R[y_1,\ldots,y_m]:u\circ P\in\cB\}\).
This is a unital algebra containing $f$. If $u\in\cC$, the chain
rule gives $4R(\Delta u)\circ P=\Delta(u\circ P)\in\cB$.
The quotient is a polynomial, so
\cref{prop:MR-finite-and-division} gives $(\Delta u)\circ P\in\cB$.
Thus $\cC$ is closed under $\Delta$ and, by the product rule, $\Gamma$.

It remains to recover $q$. The recursion
\eqref{eq:MR-higher-contraction} uses only multiplication and
$\Delta$, so $B_{s-1}(f,f)\in\cC$. This quadratic form is
$\langle My,y\rangle$ for a positive semidefinite $M$ whose kernel
is exactly the cylindrical directions of $f$: by polarization,
$B_{s-1}(f,f)(v)=0$ if and only if $\partial_vf\equiv0$.
Essentiality of $f$ therefore makes $M$ positive definite. Since
\[
 \Gamma(\langle My,y\rangle,\langle M^jy,y\rangle)
       =4\langle M^{j+1}y,y\rangle,
\]
all positive powers of $M$ give quadratic forms in $\cC$.
Choose a polynomial $p$ with $p(0)=0$ and $p(\lambda)=1$ at every
eigenvalue of $M$. Then $p(M)=\Id$ gives $q\in\cC$.
Consequently $\cA\subseteq\cC$, proving
\eqref{eq:smallest-clifford-pullback}.

At a point where $Q<R^2$, the Gram matrix of
$\nabla R,\nabla P_1,\ldots,\nabla P_m$ has determinant \(4^{m+1}R^{m-1}(R^2-Q)>0\).
Thus $(R,P)$ has an open set in its
image, and $R$ is algebraically independent of the pulled-back
polynomial ring. This proves the asserted isomorphism. For a finitely
generated connected graded algebra, the minimal homogeneous generator
number is $\dim(\cA_+/\cA_+^2)$, where $\cA_+$ is its positive-degree
ideal. Adjoining $t$ adds exactly one independent class.
\end{proof}

When $Q=R^2$ identically, the relation $q(P)=R^2$ must be retained;
the conclusion about one additional generator does not apply.

Choose minimal homogeneous generators
$\cA_q(f)=\R[u_1,\ldots,u_r]$, with $\deg u_i=d_i$, and write
\[
 \Delta u_i=b_i(u),\qquad \Gamma(u_i,u_j)=c_{ij}(u),\qquad q=q_*(u).
\]
The lifted variables $U_i=u_i(P)$ have the differential table
\begin{equation}\label{eq:clifford-minimal-table}
\begin{gathered}
 \Delta R=2N,\qquad \Gamma(R,R)=4R,\qquad
 \Delta U_i=4R\,b_i(U),\\
 \Gamma(R,U_i)=4d_iU_i,\qquad
 \Gamma(U_i,U_j)=4R\,c_{ij}(U).
\end{gathered}
\end{equation}
Thus the target table determines the calculation for any potential
$G=R^a\phi(U)$; only $\Delta R$ depends on the source dimension.
For a target potential factoring through $\cA_q(f)$ as $H=\phi(u)$,
this applies when $\phi$ is $C^3$ on an open neighborhood of
$u(\R^m\setminus Z)$.

\begin{proposition}[Target equivalence and dimension monotonicity]
\label{prop:clifford-transfer-monotonicity}
Fix a Clifford map $P_0:\R^{N_0}\to\R^m$ and, for $k\ge2$, set
\[
 P_k(x_1,\ldots,x_k)=\sum_{\alpha=1}^kP_0(x_\alpha),
 \qquad N_k=kN_0,\qquad R=\sum_{\alpha=1}^k|x_\alpha|^2.
\]
Assume the target hypotheses on $f,Z,H,M$ of
\cref{thm:clifford-area-transfer}, except for $H\cL(H)\ge0$.
For fixed $a>0$, put $G_k=R^aH(P_k)$. Then
\[
 H\cL(H)\ge0\text{ off }Z
 \quad\Longleftrightarrow\quad
 G_k\cL(G_k)\ge0\text{ off }P_k^{-1}(Z)
 \text{ for some }k\ge2.
\]
Once the source inequality holds for one such $k$, it holds for every
larger $k$. In particular, the target inequality is equivalent to the
source inequalities in all sufficiently large members of this family.
\end{proposition}

\begin{proof}
The image of $(R,P_k)$ is exactly \(\{(z,y)\in\R\times\R^m:z\ge|y|\}\).
Indeed, $|P_k|\le R$. Conversely, choose a unit vector $e$ with
$y=|y|e$ (arbitrary if $y=0$), and write
$y=y_1+y_2$, where
$y_1=(z+|y|)e/2$ and $y_2=-(z-|y|)e/2$.
The eigenvector construction in
\cref{lem:Clifford-algebra-pullback} realizes $P_0(x_i)=y_i$
with $|x_i|^2=|y_i|$. Two blocks therefore suffice.

For $y\notin Z$ and $z\ge|y|$, put $b=a(a+2\ell)$ and
\[
 \Psi(z,y)=z^aH(y),\qquad
 \mathcal W(z,y)=4z^{2a+1}W_H(y)+4bz^{2a-1}H(y)^2.
\]
For $N=N_k$, let $\mathcal D_N(z,y)$ be the value of
$G_k\cL(G_k)$ at any point with $(R,P_k)=(z,y)$.
Formula \eqref{eq:clifford-transfer-defect} shows that this value
is well-defined, and \eqref{eq:clifford-transfer-gradient} gives
$|\nabla G_k|^2=\mathcal W(z,y)>0$ there.
For two block-sum dimensions $N'=N_{k'}\ge N=N_k$, the same
formula gives the finite difference
\begin{equation}\label{eq:clifford-dimension-monotonicity}
 \mathcal D_{N'}(z,y)-\mathcal D_N(z,y)
 =\frac{2a(N'-N)}{z}\Psi(z,y)^2\mathcal W(z,y)\ge0.
\end{equation}
The common image therefore proves monotonicity of the source
inequalities.

The target inequality implies the source inequality for large $k$
by \cref{thm:clifford-area-transfer}. Conversely, fix $y\notin Z$.
The image calculation permits $(R,P_k)=(z,y)$ for every $z\ge|y|$.
By \eqref{eq:clifford-transfer-defect},
\[
 \lim_{z\to\infty}
 \frac{\mathcal D_{N_k}(z,y)}{16z^{4a+2}}
     =H(y)\cL(H)(y).
\]
If the source inequality holds, each expression on the left is
nonnegative. This proves necessity.
\end{proof}

Composition preserves umbilic maps
\cite[Proposition~5.1.8(1)]{RiedlerThesis25}. Induction from
\cref{thm:clifford-area-transfer,thm:smallest-clifford-pullback}
gives the following conclusions.

\begin{corollary}[Iteration]
\label{cor:iterated-clifford-transfer}
Let $f$ be a nonzero essential homogeneous polynomial of degree
$s\ge2$ on the target $\R^m$, and suppose $f,Z,H,M$ satisfy
\cref{thm:clifford-area-transfer}, with $H=\phi(u)$ as above.
Successive Clifford block sums with at least two blocks and
sufficiently large source dimensions yield minimizing phase boundaries.
After $j$ steps the defining degree is $2^js$ and the smallest algebra
is isomorphic to $\cA_q(f)[t_1,\ldots,t_j]$, with minimal generator
degrees
\[
 2^jd_1,\ldots,2^jd_r,\qquad 2,4,\ldots,2^j.
\]
Thus its minimal homogeneous generator number is exactly $r+j$.
The same ratio constant $M$ can be used at every step. If both signs
of $f$ occur near every zero, the minimizing supports are the complete
successive zero sets.
\end{corollary}

\begin{proof}
The ratio estimate \eqref{eq:clifford-transfer-ratio} preserves $M$.
The hypotheses on regularity, signs, the exceptional set, and the
nonvanishing of the gradient persist by the proof of \cref{thm:clifford-area-transfer};
the new potential remains a function of the new generators.
If $\partial_v(g\circ P)\equiv0$ for a nonzero homogeneous $g$
of degree $d>0$, then $v=0$: otherwise
$(g\circ P)(v+x)=(g\circ P)(x)$ gives vanishing order $2d$ at $v$,
whereas the submersion property of $P$ makes this equal to the
vanishing order of $g$ at $P(v)$, which is at most $d$.
The common-image calculation above gives $Q<R^2$ somewhere at each
step, so the algebra and degree assertions follow inductively from
\cref{thm:smallest-clifford-pullback}. At each stage a symmetric
Clifford system exists on $\Lambda^*\R^m$, with generators given by
exterior multiplication plus contraction; direct sums make its source
dimension arbitrarily large. Thus each required dimension bound can
be met. The support assertion follows at every stage from
\cref{thm:clifford-area-transfer}.
\end{proof}

\section[Isoparametric Hypercones]{Two-Generator Algebras and Isoparametric Hypercones}
\label{sec:low-rank}

When the radius and one defining polynomial generate a Laplacian
algebra, isoparametric geometry identifies the minimal level and
reduces the subcalibration inequality to one angular variable.

\subsection{Geometric background: parallel hypersurfaces and cones}
\label{sec:isoparametric-geometric-background}

A connected complete embedded hypersurface of the unit sphere is
\emph{isoparametric} if its principal curvatures are constant.
M\"unzner's theorem \cite{Muenzer80,Muenzer81} gives
\(g\in\{1,2,3,4,6\}\), with alternating multiplicities \(m_1,m_2\)
(equal for odd \(g\)). A degree-\(g\) Cartan--M\"unzner polynomial
describes its parallel family, with regular spherical levels in
\((-1,1)\) and two focal submanifolds at the endpoints. See
Cecil--Ryan \cite[Sections~3.2 and~3.4--3.6]{CecilRyan15} for the
geometry and Chi \cite[Section~4]{Chi20} for the classification.
The cone over a spherical hypersurface has curvatures zero radially
and \(r^{-1}\) times the spherical curvatures otherwise; it is minimal
exactly when its link is. We identify the unique minimal regular level below.

Lawlor's full curvature criterion is necessary and sufficient for area
minimization of minimal isoparametric hypercones
\cite[Section~4.3]{Lawlor91}. Their nonflat strictly minimizing members
also provide geometric models for non-affine entire solutions of the
minimal surface equation, as shown by Simon
\cite[Theorem~6]{Simon89}.

\subsection{The two-generator characterization}

For \(\cA_R(F)=\R[R,F]\), fixing the radius leaves one spherical
invariant. The Mendes--Radeschi description gives:

\begin{theorem}[One-dimensional quotients and the minimal level]
\label{thm:rank-one-classification}
Let \(N\geq3\), and let \(F\) be a nonconstant homogeneous polynomial
on \(\R^N\) with \(\cA_R(F)=\R[R,F]\). Suppose
\[
 \Sigma=\{F=0\}\cap\Sph^{N-1}\neq\varnothing,\qquad
 \nabla F\neq0\quad\hbox{on }\Sigma,
\]
and suppose \(\Sigma\) is minimal in \(\Sph^{N-1}\).
Then \(\{F=0\}\) is the cone over the minimal regular member of a
Cartan--M\"unzner isoparametric family.
In particular, \(\deg F\in\{1,2,3,4,6\}\).
\end{theorem}

\begin{proof}
Put \(e=\deg F\). The case \(e=1\) is a hyperplane, so assume
\(e\geq2\). By Mendes--Radeschi
\cite[Proposition~37(c) and its proof]{MendesRadeschi}, there are
constants \(a\neq0\) and \(b\), with \(b=0\) for odd \(e\), such
that \(\Psi=aF+br^e\) satisfies the normalized Cartan--M\"unzner
equations. Their proof gives a Cartan--M\"unzner polynomial
\(\Phi_g\), of degree \(g\in\{1,2,3,4,6\}\), such that either
\(\Psi=\Phi_g\), or
\begin{equation}\label{eq:rank-one-folded-alternative}
 \Psi=\epsilon(2\Phi_g^2-r^{2g}),\qquad \epsilon\in\{-1,1\}.
\end{equation}
On the unit sphere, \(\Sigma=\{\Psi=b\}\). Its regularity and
\(|\nabla_{\Sph^{N-1}}\Psi|^2=e^2(1-\Psi^2)\) imply \(|b|<1\).

Write \(f_g=\Phi_g|_{\Sph^{N-1}}\) and
\(\Delta\Phi_g=c_g r^{g-2}\). The spherical Cartan--M\"unzner
equations give the mean curvature of its regular levels:
\begin{equation}\label{eq:rank-one-level-mean-curvature}
 \diver_{\Sph^{N-1}}
 \frac{\nabla_{\Sph^{N-1}}f_g}{|\nabla_{\Sph^{N-1}}f_g|}
 =\frac{c_g-g(N-2)f_g}{g\sqrt{1-f_g^2}}.
\end{equation}
Since \(N\geq3\), only one regular level can be minimal.
In the folded case \eqref{eq:rank-one-folded-alternative},
\[
 \Sigma=\{f_g=t\}\cup\{f_g=-t\},\qquad
 t=\sqrt{(1+\epsilon b)/2}\in(0,1).
\]
Both levels are nonempty and regular, so they cannot both be minimal.
Thus \(\Psi=\Phi_g\), and \(\Sigma\) is its minimal regular level
\(f_g=c_g/[g(N-2)]\). The degree restriction follows.
\end{proof}

\subsection{One-variable reduction and classical classification}
\label{sec:isoparametric-reduction}

\subsubsection{Cartan--M\"unzner reduction}

Let \(\Phi\) be a Cartan--M\"unzner polynomial of degree
\(g\in\{1,2,3,4,6\}\) on \(\R^N\).
For \(g\geq2\), write the two
multiplicities as \(m_1,m_2\), reverse the sign of \(\Phi\) if
necessary so that \(1\leq m_1\leq m_2\), and put
\begin{equation}\label{eq:iso-parameters}
 s=m_1+m_2,\qquad t=\frac{m_2-m_1}{s},\qquad
 d=N-2=\frac g2s.
\end{equation}
Here \(d\) is the dimension of the minimal link.  The Euclidean
Cartan--M\"unzner equations are
\begin{equation}\label{eq:CM-euclidean}
 |\nabla\Phi|^2=g^2r^{2g-2},
 \qquad
 \Delta\Phi=\frac{g^2}{2}(m_2-m_1)r^{g-2}.
\end{equation}
Let
\begin{equation}\label{eq:iso-F-z}
 f=\frac{\Phi}{r^g},\qquad z=f-t,
 \qquad F=r^gz=\Phi-tr^g.
\end{equation}
The cone \(C=\{F=0\}\) is the cone over the unique minimal regular
member of the isoparametric family.  It is algebraic: for even \(g\),
\(r^g\) is polynomial, while for odd \(g\) the multiplicities are equal
and \(t=0\).
For \(g\geq2\), the cone is nonflat and smooth away from the origin,
so \(\Sing C=\{0\}\).

The coordinate \(z\) has a direct interpretation in the parallel
isoparametric family.  On its regular part put
\[
 \theta=\frac1g\arccos f\in(0,\pi/g),\qquad f=\cos(g\theta).
\]
Here \(\theta\) is spherical normal distance from the focal submanifold
\(M_+=\{f=1\}\) \cite[Section~4, pp.~14--16]{Chi20}.
The Cartan--M\"unzner equations give
\[
 |\nabla_{\Sph}\theta|=1,\qquad
 \Delta_{\Sph}\theta
 =d\,\frac{\cos(g\theta)-t}{\sin(g\theta)}.
\]
For the normal \(\nabla_{\Sph}\theta\), the last expression is the
divergence, or unaveraged scalar mean curvature.  Thus the minimal leaf
is characterized by
\[
 \cos(g\theta_*)=t,\qquad
 \tan^2(g\theta_*/2)=\frac{m_1}{m_2},\qquad
 z=\cos(g\theta)-\cos(g\theta_*).
\]
In particular \(z\) is a centered isoparametric coordinate, not a
signed distance.  For the local signed distance \(q=\theta-\theta_*\),
\[
 z=-g\sin(g\theta_*)q+O(q^2).
\]
We retain \(z\) because it preserves the polynomial form of the
subcalibration calculation.

On \(\Sph^{N-1}\), equations \eqref{eq:CM-euclidean} give the one-variable
diffusion closure
\begin{align}
 A(z):=|\nabla_{\Sph}z|^2
 &=g^2\bigl(1-(z+t)^2\bigr),                 \label{eq:A-z}\\
 B(z):=\Delta_{\Sph}z
 &=-g^2t-g(g+d)z.                            \label{eq:B-z}
\end{align}
The realizable interval is
\begin{equation}\label{eq:z-interval}
 I_t=[-1-t,1-t].
\end{equation}

\begin{proposition}[One-variable subcalibration equation]
\label{prop:iso-master}
Let \(\ell>0\), let \(h\) be smooth on a neighborhood of \(I_t\), and set
\[
 G=r^\ell h(z),\qquad h(0)=0,\qquad h'(0)>0.
\]
Define
\begin{equation}\label{eq:E-h}
 E(z)=\ell^2h(z)^2+A(z)h'(z)^2
\end{equation}
and
\begin{equation}\label{eq:N-h}
 \mathfrak N_\ell[h]
 =\bigl(Ah''+Bh'+(d+1)\ell h\bigr)E
                  -\frac12Ah'E'.
\end{equation}
Then \(|\nabla G|^2=r^{2\ell-2}E\) and
\(\cL(G)=r^{3\ell-4}\mathfrak N_\ell[h]\), so the sign criterion becomes
\begin{equation}\label{eq:iso-potential-sign}
 G\cL(G)=r^{4\ell-4}h(z)\mathfrak N_\ell[h](z)\geq0.
\end{equation}
If \(E>0\) throughout \(I_t\) and
\begin{equation}\label{eq:h-sign-criterion}
 \frac{h(z)}z>0\quad(z\neq0),\qquad
 z\mathfrak N_\ell[h](z)\geq0\quad(z\in I_t),
\end{equation}
then \(G\) generates a two-sided subcalibration and \(C\) is area minimizing.
Its divergence is
\begin{equation}\label{eq:iso-divergence}
 r\diver X_G=\frac{\mathfrak N_\ell[h]}{E^{3/2}}.
\end{equation}
\end{proposition}

\begin{proof}
In \cref{prop:spherical-quotient-potential}, take \(\varphi=h\).
The spherical gradient
pairing is \(\Gamma_{\Sph}(u,v)=Au'v'\), the spherical Laplacian is
\(Au''+Bu'\), and \(N-1=d+1\). Substitution gives the derivative identities and \eqref{eq:iso-potential-sign}.
For the angular profile \(h\) of \(G\), the multiplier \(a_G\) is
\[
 G=a_GF,\qquad a_G=r^{\ell-g}\frac{h(z)}z,
 \qquad \left.\frac{h(z)}z\right|_{z=0}=h'(0).
\]
The quotient is smooth and positive, while \(E>0\) excludes
critical points off the vertex. Since \(G\) is continuous at the
vertex and has the signs of the polynomial \(F\),
\cref{rem:polynomial-finite-perimeter,cor:gradient-subcalibration}
give area minimality. Formula \eqref{eq:div-normalized-gradient}
gives \eqref{eq:iso-divergence}.
\end{proof}

\subsubsection{Stability and the classical classification}

The squared norm of the second fundamental form of a minimal
isoparametric hypersurface is constant:
\begin{equation}\label{eq:iso-A2}
 |A_\Sigma|^2=d(g-1).
\end{equation}
The constant function attains the first eigenvalue \(-d(g-1)\), so
\eqref{eq:exact-Hardy-gap} gives the Hardy constant
\begin{equation}\label{eq:D-gd}
 D_{g,d}=\frac{(d-1)^2}{4}-d(g-1).
\end{equation}

Wang's stability calculation uses the discriminant
\((N-1)^2-4g(N-2)=4D_{g,d}\)
\cite[Lemma~3.4, p.~216, and proof of Theorem~B, p.~238]{Wang94}.
The following is its sharp Hardy form.

\begin{proposition}[Classical stability in Hardy form]
\label{prop:iso-stability}
For \(g>1\), stability and strict stability of the minimal isoparametric
hypercone are equivalent, and hold precisely when
\begin{equation}\label{eq:iso-stability-threshold}
 d\geq4g-2,\qquad\text{equivalently}\qquad N\geq4g.
\end{equation}
The optimal constant in \eqref{eq:strict-stability-definition} is
\(D_{g,d}\) from \eqref{eq:D-gd}.  In the stable range,
\begin{equation}\label{eq:iso-optimal-gap}
 D_{g,d}=\frac14+
       \frac{d(d-4g+2)}4\geq\frac14.
\end{equation}
If \eqref{eq:iso-stability-threshold} fails, the cone is unstable and hence
not area minimizing.
\end{proposition}

\begin{proof}
The roots of \(4D_{g,d}\), as a quadratic in \(d\), are
\(2g-1\pm2\sqrt{g(g-1)}\).  The lower root lies in \((0,1)\), and the upper
root lies strictly between \(4g-3\) and \(4g-2\).  Thus the integer link
dimensions in question give exactly \eqref{eq:iso-stability-threshold}.
Equation \eqref{eq:iso-optimal-gap} gives the gap; negative second
variation excludes minimization.
\end{proof}

For equal multiplicities, \cref{cor:harmonic-radial-existence}
with \(e=g\), \(\kappa=g(g-1)\) gives this same sharp constant.

\begin{theorem}[{Wang \cite[Theorem~A, p.~238]{Wang94}}]
\label{thm:classical-iso-classification}
For \(g\geq2\), a minimal isoparametric hypercone is area minimizing if
and only if
\begin{equation}\label{eq:classical-iso-list}
 N\geq4g,
 \qquad
 (g,m_1,m_2)\notin\{(2,1,5),(4,1,6)\},
\end{equation}
up to interchanging the multiplicities.  The \(g=1\) member is a
hyperplane and is area minimizing.
\end{theorem}

The nonflat minimizing members are strictly minimizing by Lawlor's criterion
\cite[Section~4.3]{Lawlor91}; see the argument of Tang--Zhang
\cite[Section~2]{TangZhang20}.
The two exceptions are strictly stable with optimal gap \(1/4\),
but not minimizing \cite[Theorem~B, p.~238]{Wang94}.

\subsubsection{The harmonic part of the isoparametric classification}
\label{sec:harmonic-isoparametric-scope}

For harmonic isoparametric polynomials, the dimension threshold in
\cref{ex:harmonic-isoparametric-radial} selects the following
parameters from the classical list.
The Cartan--M\"unzner equations \eqref{eq:CM-euclidean} and the
dimension relation \eqref{eq:iso-parameters} give
\[
 \Delta F=0\ \Longleftrightarrow\ m_1=m_2=m,
 \qquad N=gm+2.
\]
For \(g=4\), the possible equal multiplicities follow from the
homogeneous list of Takagi--Takahashi \cite{TakagiTakahashi72},
Stolz's multiplicity theorem \cite{Stolz99}, and the FKM construction
\cite[Theorem~4.1]{FKM81}. The sextic multiplicity restriction is due
to Abresch \cite{Abresch83}; see also
\cite[Chapter~3]{CecilRyan15}.
The first two columns give the multiplicity data; the last applies
the radial criterion:
\[
\begin{array}{c|c|c}
 g&\text{possible }N&\text{cases satisfying }N\ge4g\\ \hline
 2&2m+2,\ m\ge1&N\ge8\\
 3&5,8,14,26&14,26\\
 4&6,10&\text{none}\\
 6&8,14&\text{none}
\end{array}
\]
The FKM restriction follows directly from the Clifford relations.
A symmetric Clifford system
\(P_0,\ldots,P_a\) on \(\R^{2l}\) has FKM multiplicities
\((a,l-a-1)\), so equality requires \(l=2a+1\).
If \(a\ge2\), the operator \(P_1P_2\) preserves
\(E_+=\ker(P_0-I)\) and squares to \(-I\). Since
\(\dim E_+=l\), this forces \(l\) to be even, a contradiction.
Thus the only equal FKM pair is \((1,1)\), realized by
\(a=1,l=3\); the homogeneous list contributes the additional
quartic pair \((2,2)\). The quartic and sextic rows lie below the stability threshold in
\cref{prop:iso-stability}.

\subsection{Explicit polynomial subcalibrations}
\label{sec:isoparametric-certificates}

The one-variable calculation recovers every minimizing case, including
unequal multiplicities and the required exceptional multipliers.

\begin{theorem}[Explicit polynomial subcalibrations]
\label{thm:iso-polynomial-table}
Orient \(F\) as in \eqref{eq:iso-F-z}, put \(R=r^2\), and assume the
isoparametric multiplicities exist.  Every positive case in
\cref{thm:classical-iso-classification} admits the normalized-gradient
subcalibration generated by the polynomial potential in the following table.
\begin{center}
\renewcommand{\arraystretch}{1.18}
\begin{tabular}{@{}ccl@{}}
\toprule
\(g\)&multiplicities&potential \(G\)\\
\midrule
1&hyperplane&\(F\)\\
2&\(s\geq7\), or \((m_1,m_2)=(3,3)\)&\(RF\)\\
2&\((m_1,m_2)=(2,4)\)&\(\bigl(R+\frac3{10}F\bigr)F\)\\
3&\(m_1=m_2=4\) or \(8\)&\(RF\)\\
4&\(s\geq9\), or \(s=8\) and \(m_1\geq2\)&\(R^2F\)\\
4&\((m_1,m_2)=(3,4)\)&\(\bigl(R^2+\frac7{54}F\bigr)F\)\\
4&\((m_1,m_2)=(2,5)\)&\(\bigl(R^2+\frac7{15}F\bigr)F\)\\
4&\((m_1,m_2)=(1,7)\)&\(R\bigl(R^2+\frac12F\bigr)F\)\\
\bottomrule
\end{tabular}
\end{center}
In every row the total degree of \(G\) is at most ten, independently
of \(N\). Both sextic multiplicity pairs lie below the stability
threshold, as recorded in \cref{sec:harmonic-isoparametric-scope}.
For every nonflat row, the field generated by the displayed polynomial potential proves strict
area minimality, with the mass-gap constant in
\eqref{eq:strict-minimality-constant}.
\end{theorem}
\subsubsection{Potentials obtained by a radial factor}

We first take \(h(z)=z\).  Formula \eqref{eq:N-h} factors completely.

\begin{proposition}[Radial polynomial defect]
\label{prop:radial-Q}
For \(h=z\),
\begin{equation}\label{eq:N-z-Q}
 \mathfrak N_\ell[z]=zQ_\ell(z),
\end{equation}
where
\begin{align}
 Q_\ell(z)={}&g^2(1-t^2)\mu_\ell
 +g^2t(\ell^2-2\nu_\ell)z
 +(\ell^2-g^2)\nu_\ell z^2,              \label{eq:Q-ell}\\
 \mu_\ell={}&(d+1)\ell-\ell^2-gd,\qquad
 \nu_\ell=(d+1)\ell-gd.                  \label{eq:mu-nu}
\end{align}
Moreover,
\begin{equation}\label{eq:mu-window}
 \mu_\ell=D_{g,d}-\left(\ell-\frac{d+1}{2}\right)^2.
\end{equation}
Thus the constant term is nonnegative precisely on the admissible
degree interval obtained from the Jacobi equation.
\end{proposition}

\begin{proof}
Insert \(h=z\), \(h'=1\), \(h''=0\), and
\(E=\ell^2z^2+A\) into \eqref{eq:N-h}; collect powers of \(z\).
Identity \eqref{eq:mu-window} follows from \eqref{eq:D-gd}.
\end{proof}

\paragraph{The cases \texorpdfstring{\(g=2\) and \(g=4\)}{g=2 and g=4}.}

For \(g=2\), \(d=s\), \(\ell=4\), one has
\begin{equation}\label{eq:P2}
 Q_4=8P_2,\qquad
 P_2=(s-6)(1-t^2)-2(s-2)tz+3(s+2)z^2.
\end{equation}
The vertex of this convex quadratic is
\[
 z_*=\frac{(s-2)t}{3(s+2)},
\]
and at the right endpoint \(z_+=2m_1/s\),
\begin{equation}\label{eq:P2-right}
 P_2(z_+)=\frac{16m_1((s+2)m_1-s)}{s^2}>0.
\end{equation}
If \(z_*\leq z_+\), equivalently
\[
 m_1\geq m_{1,*}:=\frac{s(s-2)}{8(s+1)},
\]
the minimum is nonnegative precisely when
\begin{equation}\label{eq:K2}
 K_2:=3(s+2)(s-6)(1-t^2)-(s-2)^2t^2\geq0.
\end{equation}
For fixed \(s\), this quantity increases with \(m_1\).  At the entry value,
\begin{equation}\label{eq:K2-entry}
 K_2(m_{1,*})=
 \frac{3(s-2)(s+2)(s^2-8s-12)}{4(s+1)^2}\geq0
 \quad(s\geq10).
\end{equation}
For \(s=7,8,9\), direct evaluation at \(m_1=1\) gives, respectively,
\[
 \frac{23}{49},\qquad6,\qquad\frac{767}{81}.
\]
If \(z_*>z_+\), the minimum on \(I_t\) is the positive endpoint value
\eqref{eq:P2-right}.  This proves \(Q_4\geq0\) for every \(s\geq7\).
For \(s=6\), the same formula handles \((m_1,m_2)=(3,3)\), while
\((2,4)\) requires the correction in the next subsection.

For \(g=4\), \(d=2s\), \(\ell=8\),
\begin{equation}\label{eq:P4}
 Q_8=128P_4,\qquad
 P_4=(s-7)(1-t^2)-2(s-3)tz+3(s+1)z^2.
\end{equation}
Now
\[
 z_*=\frac{(s-3)t}{3(s+1)},
 \qquad
 P_4(z_+)=\frac{16m_1((s+1)m_1-s)}{s^2}>0.
\]
The vertex enters the interval at \(m_{1,*}=(s-3)/8\).  When it does, the
minimum is controlled by
\begin{equation}\label{eq:K4}
 K_4=3(s+1)(s-7)(1-t^2)-(s-3)^2t^2,
\end{equation}
and
\begin{equation}\label{eq:K4-entry}
 K_4(m_{1,*})=
 \frac{3(s-3)(s+1)(s^2-10s-3)}{4s^2}\geq0
 \quad(s\geq11).
\end{equation}
For \(s=9,10\), the values at \(m_1=1\) are \(52/27\) and \(107/25\).
For \(s=8\) and \(m_1\geq2\),
\(K_4\geq27-52(1/2)^2=14\).  This proves the \(g=4\) entries with a radial factor in
\cref{thm:iso-polynomial-table}.

\paragraph{The case \texorpdfstring{\(g=3\)}{g=3}.}

Here \(m_1=m_2=m\), \(t=0\), \(d=3m\), and we choose \(\ell=5\).  Formula
\eqref{eq:Q-ell} becomes
\begin{equation}\label{eq:g3-Q}
 Q_5(z)=9(6m-20)+16(6m+5)z^2.
\end{equation}
It is positive for \(m=4,8\), the two stable Cartan cases by
\cref{prop:iso-stability} and the classification in
\cite[Section~3.1.1]{NadirashviliTkachevVladut14}.  Since
\(r^{5}z=RF\), this gives the stated polynomial potential.

\subsubsection{Borderline polynomial corrections}

At the endpoint, cancellation of the radial defect's leading odd
term fixes the correction coefficient.

At the critical dimension \(N=4g\), one has \(d=4g-2\).  The natural
endpoint degree is \(\ell=2g\), for which \(\mu_\ell=0\).  Consider the
first transverse correction
\begin{equation}\label{eq:critical-h}
 h=z(1+\lambda z).
\end{equation}

\begin{proposition}[Forced critical correction]
\label{prop:forced-critical-lambda}
For \(t>0\), a two-sided sign near \(z=0\) forces
\begin{equation}\label{eq:lambda-star}
 \lambda=\lambda_*=
 \frac{2gt}{(2g+1)(1-t^2)}.
\end{equation}
For this correction to preserve the sign of \(z\) on all of \(I_t\), it is
necessary and sufficient that
\begin{equation}\label{eq:t-threshold}
 t<\frac{2g+1}{4g+1}.
\end{equation}
\end{proposition}

\begin{proof}
At the critical degree, \(d=4g-2\) and \(\ell=2g\).
Substitution in \eqref{eq:N-h} gives
\[
 z\mathfrak N_{2g}[z(1+\lambda z)]
 =2g^3\bigl[(2g+1)(1-t^2)\lambda-2gt\bigr]z^3+O(z^4).
\]
The lower-order terms vanish. Nonnegativity for both signs of \(z\)
therefore forces the cubic coefficient to vanish, which gives
\eqref{eq:lambda-star}. Since
\(\lambda_*>0\), the only possible sign loss for \(h/z=1+\lambda_*z\)
is at \(z=-1-t\).  The inequality
\(1-\lambda_*(1+t)>0\) is exactly \eqref{eq:t-threshold}.
\end{proof}

For \(g=2,s=6\), condition \eqref{eq:t-threshold} accepts
\((3,3)\) and \((2,4)\), but rejects \((1,5)\).  For \(g=4,s=7\), it
accepts \((3,4)\) and \((2,5)\), but rejects \((1,6)\).
The quadratic correction \(h=z+\lambda z^2\) therefore fails
precisely for the two stable non-minimizing members of the classical
list.

Substituting the three admissible unequal multiplicities in
\eqref{eq:lambda-star} gives the rational coefficients below.
Their global signs are then checked on the exact interval $I_t$.

\begin{proposition}[Three quadratic corrections]
\label{prop:quadratic-corrections}
The following identities hold on the indicated realizable intervals.
\begin{align}
g&=2,\quad (m_1,m_2)=(2,4),\quad
t=\tfrac13,\quad I_t=[-\tfrac43,\tfrac23],\notag\\
h&=z(1+\tfrac3{10}z),\quad \ell=4,\notag\\
z\mathfrak N_4[h]
&=\frac{648}{125}z^4(11z+30);                    \label{eq:24-certificate}\\[1mm]
g&=4,\quad (m_1,m_2)=(3,4),\quad
t=\tfrac17,\quad I_t=[-\tfrac87,\tfrac67],\notag\\
h&=z(1+\tfrac7{54}z),\quad \ell=8,\notag\\
z\mathfrak N_8[h]
&=-\frac{64}{19683}z^4
 (2548z^2-132237z-910926);                       \label{eq:34-certificate}\\[1mm]
g&=4,\quad (m_1,m_2)=(2,5),\quad
t=\tfrac37,\quad I_t=[-\tfrac{10}7,\tfrac47],\notag\\
h&=z(1+\tfrac7{15}z),\quad \ell=8,\notag\\
z\mathfrak N_8[h]
&=-\frac{128}{1125}z^4
 (588z^2-8015z-16550).                            \label{eq:25-certificate}
\end{align}
In every case \(h/z>0\) and \(z\mathfrak N_\ell[h]\geq0\) on \(I_t\).
\end{proposition}

\begin{proof}
The three lower bounds for \(h/z\) are \(3/5\), \(23/27\), and
\(1/3\).  In \eqref{eq:24-certificate}, \(11z+30>0\) on \(I_t\).
The quadratics in parentheses in \eqref{eq:34-certificate} and
\eqref{eq:25-certificate} are convex, and their endpoint values are both
negative; hence they are negative on the entire interval.  The identities
follow by substitution into \eqref{eq:N-h}.
\end{proof}

\begin{proposition}[The \((g,m_1,m_2)=(4,1,7)\) correction]
\label{prop:17-correction}
For \(t=3/4\), \(I_t=[-7/4,1/4]\), \(\ell=10\), and
\(h=z(1+z/2)\),
\begin{equation}\label{eq:17-certificate}
 z\mathfrak N_{10}[h]=z^2P(z),
\end{equation}
where
\begin{equation}\label{eq:P-17}
 P(z)=189z^5+1778z^4+5492z^3+4848z^2-763z+42.
\end{equation}
The polynomial \(P\) is positive on \([-7/4,1/4]\).

\end{proposition}

\begin{proof}
Substitution into \eqref{eq:N-h} gives \eqref{eq:17-certificate}.
The rational Bernstein certificate in \cref{ex:Bernstein-17} proves
\(P>0\) on the whole interval.  Also \(h/z\geq1/8\), and the
potential is \(r^{10}h=R(R^2+F/2)F\).
\end{proof}

\begin{proof}[Proof of \cref{thm:iso-polynomial-table}]
The rows with a radial factor follow from \cref{prop:radial-Q} and the calculations in
the preceding subsection.  The remaining rows follow from
\cref{prop:quadratic-corrections,prop:17-correction}.  In each case \(E>0\) on \(I_t\): if \(z\neq0\), then
\(h\neq0\) and \(E\geq\ell^2h^2>0\); at \(z=0\),
\(E(0)=A(0)h'(0)^2>0\).  Apply \cref{prop:iso-master}.
For every nonflat row the displayed numerator is nonzero, and both phases
are nonempty.  The normalized gradient is smooth away from the
vertex.  Hence \cref{cor:polynomial-strict-minimality} proves strict
minimality directly from the displayed polynomial potential.
\end{proof}

For every nonflat row except the quadratic case in ambient dimension
\(N=9\), the displayed degree satisfies \(\ell\ne(N-1)/2\).
Thus \cref{thm:ground-state-gap} already gives strict stability from
the subcalibration sign. In the exceptional case \(\ell=4=(N-1)/2\),
\cref{prop:iso-stability} gives the optimal Hardy constant \(2\).
The following boundary calculations distinguish the degree gap from
positivity of the Jacobi trace.

\subsubsection{Degenerate traces and quantitative strictness}

\begin{corollary}[Four polynomial potentials with vanishing Jacobi trace]
\label{cor:four-degenerate-certificates}
For the four polynomial potentials in \cref{thm:iso-polynomial-table}
corresponding to
\begin{equation}\label{eq:four-degenerate-list}
 (g,m_1,m_2)=(2,3,3),\ (2,2,4),\ (4,3,4),\ (4,2,5),
\end{equation}
one has \(n=d+1=4g-1\), \(\ell=2g\), and
\begin{equation}\label{eq:degenerate-trace-gap}
 \beta_G\equiv0,\qquad J_C(|\nabla G|^{-1})=0,\qquad
 \inf_{u\neq0}\frac{Q_C(u)}{\int_Cr^{-2}u^2}=\frac14.
\end{equation}
Their signed divergence satisfies, uniformly near the regular cone,
\begin{equation}\label{eq:cubic-distance-degeneracy}
 q_{X_G}(x)\asymp\frac{\dist(x,C)^3}{r^4}.
\end{equation}
In particular the fields generated by these polynomial potentials do not satisfy a positive
linear distance bound \eqref{eq:linear-distance-divergence}, although they
prove strict minimality and strict stability.
For every sufficiently small \(\eps>0\), the perturbed potential
\(G_\eps=r^{-\eps}G\) has boundary coefficient
\begin{equation}\label{eq:four-degenerate-strictification}
 \beta_{G_\eps}=\eps-\eps^2>0
\end{equation}
and satisfies the distance-weighted comparison of
\cref{prop:distance-weighted-comparison}.
\end{corollary}

\begin{proof}
Along the minimal link the gradient length \(|\nabla G|\) is constant,
since \(h'(0)=1\) in each case and \(A(0)=g^2(1-t^2)\).
Here \(\mu_\ell=0\) in \eqref{eq:mu-window}, so the boundary coefficient
vanishes.  Alternatively this follows from the absence of a linear term in
\(\mathfrak N_\ell[h]\) in the displayed factorizations.
The optimal gap \(1/4\) is \eqref{eq:iso-optimal-gap} at
\(d=4g-2\).
For \((2,3,3)\), equations \eqref{eq:P2} and \eqref{eq:N-z-Q} give
\(\mathfrak N_4[z]=192z^3\).  For the other three potentials,
\eqref{eq:24-certificate}--\eqref{eq:25-certificate} give
\(\mathfrak N_\ell[h]=z^3p(z)\) with \(p>0\) near zero and
\(E(0)>0\).  The spherical function \(z\) is a regular defining
function of the compact minimal link.  Thus \(|z|\asymp\dist(x,C)/r\),
and \eqref{eq:iso-divergence} proves
\eqref{eq:cubic-distance-degeneracy}.
The same factorizations give strict signed divergence everywhere off the
cone.  All hypotheses of \cref{prop:radial-strictification} hold, yielding
\eqref{eq:four-degenerate-strictification} and the claimed distance bound.
\end{proof}

For comparison, Niu \cite[Lemma~5.3]{Niu26} obtains quantitative
subcalibrations from the Hardt--Simon foliation under strict stability
and strict minimality. Here the explicit
potentials suffice, with a radial correction in the four cases above.

\begin{corollary}[Distance-weighted comparison for the nonflat minimizing cases]
\label{cor:iso-distance-weighted}
Every nonflat cone in \cref{thm:iso-polynomial-table} admits a homogeneous
potential, smooth and semialgebraic on \(\R^N\setminus\{0\}\), whose
normalized gradient satisfies \eqref{eq:linear-distance-divergence}.
Except for \eqref{eq:four-degenerate-list}, the polynomial potential in the
table itself has this property.  In the four remaining cases, a sufficiently
small positive rational radial perturbation has it.
\end{corollary}

\begin{proof}
The estimates \eqref{eq:K2-entry} and \eqref{eq:K4-entry} are
strict in their indicated ranges, as are their low-dimensional checks and
endpoint values.  Thus \(Q_\ell>0\) on the entire realizable interval in
all these nondegenerate rows; \eqref{eq:g3-Q} is also strictly positive.
The remaining nondegenerate correction has \(P>0\) by
\cref{prop:17-correction}.
In these cases \(\mathfrak N_\ell[h]=z p(z)\) with \(p>0\) on the
compact interval and \(E>0\).  Near the link \(|z|\) is comparable to
spherical distance, and away from it compactness supplies a positive lower
bound.  Equation \eqref{eq:iso-divergence} and homogeneity prove
\eqref{eq:linear-distance-divergence}.
Apply \cref{cor:four-degenerate-certificates} to the other four cases,
choosing a sufficiently small positive rational \(\eps\). The factor
\(r^{-\eps}=(|x|^2)^{-\eps/2}\), and hence \(G_\eps\), is
semialgebraic on \(\R^N\setminus\{0\}\); see
\cref{prop:radial-strictification}.
\end{proof}

\subsubsection{The two classical exceptions}

Both exceptional cones possess positive Jacobi fields.  Indeed the radial
Jacobi equation is
\begin{equation}\label{eq:iso-radial-Jacobi}
 \gamma(\gamma+d-1)+d(g-1)=0,
\qquad
 \gamma_\pm=-\frac{d-1}{2}\pm\sqrt{D_{g,d}}.
\end{equation}
For \((2,1,5)\) they are \(r^{-2},r^{-3}\); for \((4,1,6)\),
they are \(r^{-6},r^{-7}\). These cases are stable but not minimizing,
completing \cref{thm:classical-iso-classification}.

\begin{corollary}[{Lin's one-sided strict minimality \cite[p.~210]{Lin87}}]
\label{cor:Lin-one-sided}
Let \(x\in\R^2\), \(y\in\R^6\), and set
\[
 E=\{|y|^2<5|x|^2\},\qquad C=\partial\llbracket E\rrbracket,
 \qquad C_1=C\llcorner B_1.
\]
The cone over
\(S^1(\sqrt{1/6})\times S^5(\sqrt{5/6})\) is one-sided strictly
area minimizing in \(\overline E\).
More precisely, it minimizes against finite-perimeter changes
\(H\subset E\) of compact support.  There is \(\Theta>0\) such that
every finite-mass integral current \(T\) with
\[
 \partial T=\partial C_1,\qquad
 \operatorname{spt}T\subset\overline E,\qquad
 \operatorname{spt}T\cap B_\eps=\varnothing
\]
satisfies \(\mathbf M(T)-\mathbf M(C_1)\geq\Theta\eps^7\)
for \(0<\eps<1\).
\end{corollary}

\begin{proof}
Put \(u=|x|^2\), \(v=|y|^2\), \(R=u+v\), and \(F=v-5u\).
Take the radial potential \(G=RF\).  For functions of \(u,v\),
\[
 \Delta P=4uP_{uu}+4vP_{vv}+4P_u+12P_v,\qquad
 \Gamma(P,Q)=4uP_uQ_u+4vP_vQ_v.
\]
Hence
\begin{align*}
 \Delta G&=16(v-8u),\\
 |\nabla G|^2&=16(25u^3+24u^2v+v^3),\\
 \cL(G)&=64R(v-5u)^2(v-17u).
\end{align*}
The last identity follows by inserting the first two in
\[
 \cL(G)=|\nabla G|^2\Delta G-
 \tfrac12\Gamma(G,|\nabla G|^2),
\]
the same radial calculation as
\cref{prop:radial-Q}.
Thus \(X=\nabla G/|\nabla G|\) is smooth off the origin,
equals \(\nu_E\) on the regular cone, and satisfies
\(\diver X<0\) throughout \(E\).
Since \(\cH^{N-1}(\{0\})=0\), the one-sided part of
\cref{thm:paired-subcalibration} gives
\[
 \Per(H;U)-\Per(E;U)
 \geq\int_{E\setminus H}-\diver X\qquad(H\subset E).
\]

Set \(\Theta=\int_{E\cap B_1}-\diver X>0\).
This integral is finite, and homogeneity makes its value over
\(E\cap B_\eps\) equal to \(\Theta\eps^7\).
To pass to currents, use radial projection and an integer-valued BV
filling as in the proof of \cref{thm:strict-minimality-flux}.
The resulting function \(v\) has zero distributional gradient in
the connected open set \(\R^8\setminus\overline E\) and vanishes
there outside a compact set, so it vanishes throughout that set.
It is also constant in \(B_\eps\). Since it vanishes on the
nonempty open set \(B_\eps\setminus\overline E\), this constant
is zero. Hence \(H=\{v\geq1\}\) satisfies \(H\subset E\) and
\(H\cap B_\eps=\varnothing\), up to null sets. Thus \(E\setminus H\)
contains \(E\cap B_\eps\), and the displayed one-sided comparison
proves the claimed mass gap.
\end{proof}

\section[Gradient-Square Closure]{Gradient-Square Closure and Cone Geometry}
\label{sec:gradient-square-closure}

After the two-generator case, a natural candidate for the next closure
is
\[
 R=|x|^2,\qquad S=|\nabla F|^2,\qquad
 \mathbb R[R,F,S].
\]
The generator degrees are \((2,3,4)\) for cubics and \((2,4,6)\)
for quartics. In odd degree, closure forces degree three and identifies
the possible cubic types. In even degree, we record the closure
identities, a regular-link equation and a rectangular matrix example.
Here closure always means that this prescribed
polynomial algebra is Laplacian.

\subsection{Odd-degree rigidity and cubic closure}
\label{sec:gradient-square-rigidity}

The identity
\[
 \cL(F)=S\Delta F-\tfrac12\Gamma(F,S)
\]
relates the closure to minimality. In odd degree, the critical set
determines which of the following two rigidity arguments applies.

\begin{theorem}[Odd-degree rigidity]\label{thm:odd}
Let \(N\geq3\), and let \(F\not\equiv0\) be a homogeneous polynomial of
odd degree \(e\geq3\) on \(\R^N\). Suppose \(\R[R,F,S]\) is a
Laplacian algebra and \(\nabla F(p)=0\) at some \(p\ne0\).
Then \(e=3\).
\end{theorem}

\Cref{sec:odd-rigidity-proof} proves the theorem by degree counting
and Hessian comparisons at a nonzero critical point and a spherical
maximum, forcing $(e-2)(e-3)=0$.

With trivial critical set, the gradient square is radial:

\begin{theorem}[Odd closure with trivial critical set]
\label{thm:odd-isolated-closure}
Let \(N\geq3\), and let \(F\not\equiv0\) be homogeneous of odd degree
\(e\geq3\). If \(\R[R,F,S]\) is Laplacian and
\(\{\nabla F=0\}=\{0\}\), then
\[
 S=cR^{e-1}\qquad(c>0).
\]
Consequently \(e=3\), and, up to an orthogonal change of variables
and a nonzero constant factor, \(F\) is a Cartan cubic in one of the
dimensions \(N=5,8,14,26\).
\end{theorem}

A spherical gradient-flow first integral and incompatible Hessian
traces force $S$ to be radial. Odd closure also gives harmonicity;
the polynomial eikonal classification \cite[Theorem~1]{TkachevIsoparametric14}
and Cartan's cubic classification \cite{Cartan39} then apply.
See \cref{sec:odd-isolated-proof}.
Thus both cases force degree three; the higher-degree pullbacks in
\cref{sec:quartic-application} instead use \(Q=|P|^2\) as a generator.

The restriction \(N\geq3\) is essential. On \(\R^2\),
\(F(x,y)=\operatorname{Re}(x+iy)^e\) is harmonic and satisfies
\(S=e^2R^{e-1}\) for every positive integer \(e\).

\begin{lemma}[Three-generator cubic closure]
\label{lem:cubic-gradient-closure}
Let \(F\) be a nonzero homogeneous cubic on \(\R^N\).  Suppose
\(\R[R,F,S]\) is a Laplacian algebra, where \(S=|\nabla F|^2\).
There are constants \(\alpha,\beta,\gamma,\delta,\varepsilon\) such that
\begin{equation}\label{eq:cubic-five-constants}
\begin{aligned}
 \Delta F&=0,& \Delta S&=\alpha R,\\
 \Gamma(F,S)&=\beta RF,&
 \Gamma(S,S)&=\gamma R^3+\delta F^2+\varepsilon RS.
\end{aligned}
\end{equation}
Moreover, \(\alpha>0\) and
\(\cL(F)=-\beta RF/2\), so the regular zero set is minimal.
\end{lemma}

\begin{proof}
The generators have degrees \(2,3,4\).  Their homogeneous monomials
of degrees \(1,2,5,6\) are, respectively, none, \(R\), \(RF\), and
\(R^3,F^2,RS\).  Laplacian closure and the product rule
\eqref{eq:Gamma-polarization}
give \eqref{eq:cubic-five-constants}.  Bochner's identity
\cite[Proposition~1.47, equation~(1.236), p.~61]{ColdingMinicozzi11} gives
\(\alpha R=2|\Hess F|^2\); the Hessian of a nonzero cubic is not
identically zero, so \(\alpha>0\).  The expression for \(\cL(F)\)
follows from harmonicity.
\end{proof}

At a positive maximum \(a=F(x)\) on the unit sphere, Euler's
identity gives \(\nabla F(x)=3ax\) and
\(\Hess F(x)x=6ax\), hence \(\beta=108a^2>0\).
This gives a harmonic radial eigencubic. Its normalization and
classification follow; variational properties are treated in
\cref{sec:gradient-square-rank-two}.

\subsubsection{Normalized closure and the critical set}

\begin{proposition}[Normalized closure and the critical set]
\label{prop:cubic-singular-normalization}
Assume the hypotheses of \cref{lem:cubic-gradient-closure} and that
\(\nabla F(p)=0\) for some \(p\neq0\).  After multiplying \(F\)
by a positive constant, the closure becomes
\begin{equation}\label{eq:cubic-normalized-closure}
\begin{aligned}
 \Delta F&=0,& \Delta S&=\alpha R,&
 \alpha&=\frac{2(N+3)}3,\\
 \Gamma(F,S)&=4RF,&
 \Gamma(S,S)&=4RS+12F^2.
\end{aligned}
\end{equation}
There is an integer \(k\geq1\) such that
\begin{equation}\label{eq:cubic-dimension-quantization}
 N=6k-3,\qquad \alpha=4k.
\end{equation}
At every nonzero critical point, \(\Hess F\) has \(k\) positive
and \(k\) negative eigenvalues, all of absolute value \(|p|\), and
\(N-2k\) zero eigenvalues.
The punctured critical set \(Z^*=\{x\ne0:\nabla F(x)=0\}\) is a
real-analytic submanifold of dimension \(N-2k=4k-3\), with
\(T_pZ^*=\ker\Hess F(p)\).  In orthonormal coordinates
\(z=(u,v,w)\), where \(u,v\in\R^k\) are the positive and negative
Hessian directions and \(w\in\R^{N-2k}\),
\[
 F(p+z)=\frac{|p|}{2}\bigl(|u|^2-|v|^2\bigr)+O(|z|^3).
\]
\end{proposition}

\begin{proof}
At the given nonzero critical point, Bochner's identity makes
the trace-free Hessian nonzero, while the differentiated Euler
identity places \(p\) in its kernel. Hence \(N\geq3\).
The normalization \eqref{eq:normalizedodd}, established in the
proof of \cref{thm:odd} in \cref{sec:odd-rigidity-proof}, applies
with \(e=3\). Together with the moment identity
\eqref{eq:alphadim}, it gives \eqref{eq:cubic-normalized-closure}.

Let \(p\ne0\) be any critical point and put \(H=\Hess F(p)\).
The cubic Taylor expansion is
\[
 F(p+y)=\tfrac12\langle Hy,y\rangle+F(y),\qquad
 S(p+y)=\langle H^2y,y\rangle+O(|y|^3).
\]
Comparing quadratic terms in \(\Gamma(F,S)=4RF\) gives
\(2\langle H^3y,y\rangle=2|p|^2\langle Hy,y\rangle\),
and hence \(H^3=|p|^2H\). Bochner's identity implies \(H\ne0\),
and harmonicity gives \(\tr H=0\). Thus the nonzero eigenvalues
are \(\pm|p|\), with equal multiplicities, say \(k\).
Bochner's identity and \eqref{eq:cubic-normalized-closure} imply \(2k|p|^2=\tr(H^2)=\tfrac12\alpha|p|^2\).
Thus \(\alpha=4k\) and \(N=6k-3\), proving
\eqref{eq:cubic-dimension-quantization}.

The implicit function argument in \cref{sec:cubic-local-geometry}
proves the asserted description of \(Z^*\). The Hessian eigenvalues
give the displayed Taylor expansion; \cref{sec:tangent-instability-transfer}
then identifies the quadratic tangent and the regular graphical convergence.
\end{proof}

\subsubsection{The cubic classification within this closure}
\begin{proposition}[Cubic types under gradient-square closure]
\label{prop:cubic-gradient-classification}
\label{rem:cubic-radial-eigencubic}
Under the hypotheses of \cref{prop:cubic-singular-normalization},
the possible parameters are
\begin{equation}\label{eq:cubic-realized-dimensions}
 k\in\{1,2,3,5\},\qquad N\in\{3,9,15,27\}.
\end{equation}
These four cases are realized, respectively, by \(xyz\) on \(\R^3\)
and
\[
 F(A)=\tfrac13\operatorname{Re}\tr(A^3)
 \quad\text{on }\operatorname{Herm}_4(\mathbb F)_0,
 \qquad \mathbb F=\R,\mathbb C,\mathbb H.
\]
In particular, for an irreducible defining polynomial only the
dimensions \(9,15,27\) occur.
\end{proposition}

\begin{proof}
We apply the existing radial-eigencubic classification, with its
normalization made explicit. By
\cref{prop:cubic-singular-normalization},
\[
 \cL(F)=-2RF,\qquad
 \tr\bigl((\Hess F)^2\bigr)=2kR.
\]
The Hessian-trace identity in
\cite[Lemma~5.1(i)]{TkachevClassification10} gives the cubic trace
coefficient \(-(N+6)\lambda-6\beta_2\). Substituting
\(\lambda=-2\), \(\beta_2=2k\), and \(N=6k-3\) gives \(\tr\bigl((\Hess F)^3\bigr)=6F\).
The type parameters \((n_1,n_2)\) are the characteristic dimensions
of \cite[Definition~6.7.1, Proposition~6.11.2, and
Corollary~6.11.1]{NadirashviliTkachevVladut14}.
The dimension and cubic Hessian-trace coefficient give
\[
 n_1=1+\frac{6}{3(-2)}=0,
 \qquad n_2=\frac{N+1}{2}=3k-1.
\]
Under the additional quadratic Hessian-trace identity, a radial
eigencubic is either exceptional or one of the four Hurwitz eigencubics
of Clifford type \cite[Corollary~6.11.2 and Section~6.11.5]{NadirashviliTkachevVladut14}.
For the exceptional branch, the \(n_1=0\) classification in
\cite[Proposition~6.1]{TkachevClassification10} gives exactly the
three possible types \((n_1,n_2)=(0,5),\ (0,8),\ (0,14)\).
Among the four Clifford-type eigencubics, \(n_1=0\) occurs only for
the real product cubic in dimension three; the four algebras are listed
as \emph{mutants} in \cite[Section~6.11.5]{NadirashviliTkachevVladut14}.
This proves \eqref{eq:cubic-realized-dimensions}. The three
exceptional cases are realized by the rank-four trace models listed in
\cite[Section~6.11.6(iii)]{NadirashviliTkachevVladut14};
their closure is also calculated directly in
\cref{cor:quaternionic-trace-cubic}. The dimension-three model is
reducible, so it is excluded by irreducibility.
\end{proof}

The general classification of radial eigencubics remains incomplete
\cite[Section~2.2]{FoxTkachev24}. Within this subclass,
\cref{thm:cubic-rank-two-existence} selects only the dimension-27
quaternionic trace cone. General cubic algebras begin in
\cref{sec:cubic-algebraic-theory}.

\subsection{Even-degree closure and a regular-link identity}
\label{sec:even-gradient-structure}

Even-degree gradient-square closure need not imply harmonicity.
We record its identities and a regular-link consequence, leaving
further geometry and variational properties aside.

\begin{proposition}[The even-degree closure and its signs]
\label{prop:even-gradient-closure}
Let \(F\) be a real irreducible homogeneous polynomial of degree
\(d=2m\ge4\) on \(\R^N\), with a nonempty minimal regular zero set.
Suppose \(\R[R,F,S]\) is Laplacian, where
\(R=|x|^2\) and \(S=|\nabla F|^2\).
There are constants \(c,\alpha,\beta,\lambda,\mu\) such that
\begin{equation}\label{eq:even-rfs-closure}
\begin{aligned}
 \Delta F&=cR^{m-1},&
 \Delta S&=\alpha R^{d-2}+\beta R^{m-2}F,\\
 \cL(F)&=(\lambda R^{d-2}+\mu R^{m-2}F)F,\\
 \Gamma(F,S)&=2cR^{m-1}S
       -2\lambda R^{d-2}F-2\mu R^{m-2}F^2.
\end{aligned}
\end{equation}
Moreover, \(\alpha>0\) and \(\lambda<0\).
\end{proposition}

\begin{proof}
The degrees of \(R,F,S\) are \(2,d,2d-2\).
The homogeneous pieces of degrees \(d-2\) and \(2d-4\) are
therefore spanned by \(R^{m-1}\) and by
\(R^{d-2},R^{m-2}F\), respectively.
Minimality and \cref{prop:divisibility} make \(\cL(F)/F\)
a polynomial of degree \(2d-4\); division inside a Laplacian
algebra, \cref{prop:MR-finite-and-division}, places it in the same
homogeneous piece. The identity
\(\cL(F)=S\Delta F-\Gamma(F,S)/2\) proves the mixed entry.

Bochner's identity and \(\Gamma(R,F)=2dF\) give
\begin{equation}\label{eq:even-bochner}
 |\Hess F|^2=\frac\alpha2R^{d-2}
       +\left(\frac\beta2-cd(d-2)\right)R^{m-2}F.
\end{equation}
At a nonzero regular zero this shows \(\alpha\ge0\).
The polynomial \(F\) changes sign across that zero.
If \(\alpha=0\), nonnegativity of the left side would force
\(\beta/2-cd(d-2)=0\), and hence \(\Hess F=0\), contrary
to \(d\ge4\). Thus \(\alpha>0\).

Let \(a>0\) and \(b_0<0\) be the maximum and minimum of
\(F\) on the unit sphere. At either extremum, with value \(t\),
Euler's identity gives
\[
 \nabla F=dt x,\qquad S=d^2t^2,\qquad
 \nabla S=2(d-1)Sx.
\]
Substitution in the mixed entry, followed by division by \(2t\),
gives
\[
 d^3(d-1)t^2+(\mu-cd^2)t+\lambda=0.
\]
The two distinct values \(a,b_0\) are its roots. Consequently
\(\lambda=d^3(d-1)ab_0<0\).
\end{proof}

\begin{remark}[Harmonic normalization]
\label{rem:even-harmonic-suspension}
With \(a=c/[d(N+d-2)]\), the polynomial \(H=F-aR^m\)
is harmonic, and
\[
 |\nabla H|^2=S-2ad^2R^{m-1}F+a^2d^2R^{d-1}.
\]
Thus \(\R[R,H,|\nabla H|^2]=\R[R,F,S]\).
This change preserves the algebra, but generally changes the zero set
and the critical set; it does not reduce the geometric classification
to the harmonic case.

\end{remark}

\begin{proposition}[An isoparametric function on the regular link]
\label{prop:even-singular-link-structure}
Assume \cref{prop:even-gradient-closure} and that $F$ has a nonzero
critical point. Put $\ell=d-1$ and
$h=(\alpha-2c^2)/(-\lambda)$. On
$M=\{F=0,\ S>0\}\cap\Sph^{N-1}$, the function
$u=-2\ell^2S|_M/\lambda$ satisfies $0<u\le1$ and
\begin{equation}\label{eq:even-normalized-link}
 |\nabla_Mu|^2=4\ell^2u(1-u),\qquad
 \Delta_Mu=2\ell^2(h-1)-2\ell(N+2\ell-3)u.
\end{equation}
Thus $u$ is isoparametric in the intrinsic sense that its gradient
square and Laplacian depend only on its value.
\end{proposition}

\begin{proof}
Let $p\ne0$ be critical, $R_0=R(p)$ and $H=\Hess F(p)$.
Euler's identity gives $F(p)=0$. Comparing quadratic terms at $p$
in \eqref{eq:even-rfs-closure} and using \eqref{eq:even-bochner} gives
\begin{equation}\label{eq:even-critical-hessian}
 H(2H^2-2cR_0^{m-1}H+\lambda R_0^{d-2}\Id)=0,
 \quad \tr H=cR_0^{m-1},\quad
 \tr(H^2)=\tfrac\alpha2R_0^{d-2}.
\end{equation}
The normalized nonzero eigenvalues are among
$\eta_\pm=(c\pm\sqrt{c^2-2\lambda})/2$, with opposite signs.
Both occur: $H\ne0$ since $\alpha>0$, and if only $\eta_+$
occurred with multiplicity $a\ge1$, its trace would give
$(a-1)\eta_+=\eta_-$, a contradiction. The other case is identical.

Closure and degree counting give
\begin{align*}
 \Gamma(S,S)={}&D_0R^{2d-3}+D_1R^{3m-3}F
 +D_2R^{d-3}F^2+D_3R^{m-3}F^3\\
 &+D_4R^{d-2}S+D_5R^{m-2}FS,
\end{align*}
with the $D_3$ term omitted for $d=4$.
At $p$ its constant term gives $D_0=0$, and its quadratic terms give
$4H^4=D_1R_0^{3m-3}H/2+D_4R_0^{d-2}H^2$.
For each of the two roots, this reads $4\eta^3=D_1/2+D_4\eta$.
Using $\eta^2=c\eta-\lambda/2$ yields
\begin{equation}\label{eq:even-critical-gamma}
 D_0=0,\qquad D_1=-4c\lambda,\qquad D_4=4c^2-2\lambda.
\end{equation}

On $C_{\reg}=\{F=0,S>0\}$ take $\nu=\nabla F/\sqrt S$.
The mixed entry in \eqref{eq:even-rfs-closure} and
\eqref{eq:even-critical-gamma} imply
\[
 \nu(S)=2cR^{m-1}\sqrt S,\qquad
 |\Hess F\,\nu|^2=(c^2-\lambda/2)R^{d-2}.
\]
Subtracting $\nu(S)^2$ from $\Gamma(S,S)$ gives the tangential
gradient square. Differentiating $\Gamma(F,S)$ in direction $\nu$,
and using $\nu(R)=0$ on $C$, gives
$\Hess S(\nu,\nu)=(2c^2-\lambda)R^{d-2}$.
Since $C_{\reg}$ is minimal,
$\Delta_CS=\Delta S-\Hess S(\nu,\nu)$, and hence
\begin{equation}\label{eq:even-cone-structure}
 |\nabla_CS|^2=-2\lambda R^{d-2}S,\qquad
 \Delta_CS=-(h-1)\lambda R^{d-2}.
\end{equation}
Finally $S$ has degree $2\ell$. Writing $s=S|_M$ and separating
radial derivatives in the cone metric gives
\[
 |\nabla_Ms|^2=-2\lambda s-4\ell^2s^2,\qquad
 \Delta_Ms=-(h-1)\lambda-2\ell(N+2\ell-3)s.
\]
Normalization proves \eqref{eq:even-normalized-link}; its first
identity and $s>0$ imply $0<u\le1$.
\end{proof}

\subsection{Quartic closure and a rectangular matrix family}
\label{sec:quartic-gradient-closure}

The following example realizes generator degrees $(2,4,6)$ by
rectangular matrices, with minimality selecting a single row dimension.

\begin{proposition}[The rectangular matrix family]
\label{ex:rectangular-gradient-closure}
\label{prop:rectangular-minimal-selection}
Let \(p\ge3\), give \(M_{p,3}(\R)\) its Frobenius inner product,
and put
\begin{equation}\label{eq:rectangular-quartic-definition}
 B=X^{\mathsf T}X,\qquad P_j=\tr(B^j),\qquad
 R=P_1,\qquad F=2P_2-R^2,\qquad S=|\nabla F|^2.
\end{equation}
Then \(\cA_R(F)=\R[R,F,S]\), with three algebraically independent
generators of degrees \((2,4,6)\). Its complete closure table is
\begin{equation}\label{eq:rectangular-quartic-laplacians}
 \Delta R=6p,\qquad \Delta F=4(6-p)R,\qquad
 \Delta S=32(p+6)R^2+448F,
\end{equation}
\begin{equation}\label{eq:rectangular-quartic-gamma}
 \begin{aligned}
 \Gamma(R,R)&=4R,\quad \Gamma(R,F)=8F,\quad \Gamma(R,S)=12S,\\
 \Gamma(F,F)&=S,\\
 \Gamma(F,S)&=8RS+128(R^2+F)F,\\
 \Gamma(S,S)&=192(R^2+F)S+1024R^3F+2048RF^2.
 \end{aligned}
\end{equation}
The regular zero set is minimal if and only if \(p=5\). In that
case \(N=15\), \(F\) is nonharmonic, and
\begin{equation}\label{eq:rectangular-quartic-minimality}
 \Delta F=4R,\qquad \cL(F)=-64(R^2+F)F.
\end{equation}
\end{proposition}

\begin{proof}
Differentiation of
\(B=X^{\mathsf T}X\) gives
\begin{equation}\label{eq:rectangular-trace-differentiation}
 \nabla P_j=2jXB^{j-1},\qquad
 \Gamma(P_i,P_j)=4ijP_{i+j-1}.
\end{equation}
For the required Laplacians, the product rule gives
\[
 \Delta P_1=6p,\qquad
 \Delta P_2=4(p+4)P_1,\qquad
 \Delta P_3=6\bigl((p+5)P_2+P_1^2\bigr).
\]
These formulas, derived in \cref{app:rectangular-trace-details},
together with the product rule give
\begin{equation}\label{eq:rectangular-S-traces}
 \nabla F=4X(2B-R\Id),\qquad
 S=64P_3-64RP_2+16R^3
  =64P_3-32RF-16R^3,
\end{equation}
and \eqref{eq:rectangular-quartic-laplacians}.

The characteristic equation of the $3\times3$ matrix $B$ reduces
$P_4,P_5$ to $R,F,S$. The resulting formulas in
\cref{app:rectangular-trace-details} give the last two gradient
entries of \eqref{eq:rectangular-quartic-gamma}; Euler's identity
supplies its first row. Thus \cref{prop:MR-generator-test} proves closure;
conversely every Laplacian algebra containing \(R,F\) contains
\(S=\Gamma(F,F)\), proving the equality with \(\cA_R(F)\).
On diagonal rectangular matrices with three distinct positive squared
entries \(t_1,t_2,t_3\), the Jacobian of
\((P_1,P_2,P_3)\) is a nonzero Vandermonde determinant times six.
The triangular change in \eqref{eq:rectangular-S-traces} is invertible,
so \(R,F,S\) are algebraically independent.

Finally the displayed table gives
\[
 \cL(F)=4(5-p)RS-64(R^2+F)F.
\]
At a matrix with \(B=\operatorname{diag}(4,1,1)\), one has
\(F=0\), \(R=6\), and \(S=768>0\). Thus the regular zero
set can be minimal only for \(p=5\); for this value the curvature
identity \eqref{eq:div-normalized-gradient} proves minimality.
The Laplacian in \eqref{eq:rectangular-quartic-laplacians} proves
the last assertion. For comparison, the member with \(p=6\) is
harmonic and has the same gradient pairings, but its regular zero
set is not minimal.
\end{proof}

\subsection{The boundary of gradient-square closure}

The quadratic compositions of the next section have a different
three-generator structure:

\begin{proposition}[An independent quartic generator]
\label{prop:quartic-composition-not-rfs}
Let \(R,T,F\) be algebraically independent homogeneous polynomials
of degrees \(2,4,4\), with \(R=|x|^2\), and put
\(S=|\nabla F|^2\). Suppose, for constants \(a,b,h\) with \(b>0\),
\[
 S=4R(aF+bT),\qquad
 \Delta S=8(N+8)(aF+bT)+8(a^2+bh)R^2.
\]
Then \(\R[R,F,S]\) is not Laplacian, although
\[
 \R[R,F,S][R^{-1}]=\R[R,T,F][R^{-1}].
\]
\end{proposition}

\begin{proof}
The degree-four part of \(\R[R,F,S]\), whose displayed generators
have degrees \(2,4,6\), is contained in
\(\operatorname{span}\{R^2,F\}\). Algebraic independence places
\(T\) outside this span. Its coefficient in \(\Delta S\) is
\(8(N+8)b\ne0\), so \(\Delta S\notin\R[R,F,S]\).
After adjoining \(R^{-1}\), however,
\(T=b^{-1}(S/(4R)-aF)\), which proves both inclusions in the
displayed equality.
\end{proof}

The orthogonal quartic blocks in \cref{sec:two-quadratic} satisfy
these identities by \eqref{bh:reduced-table}, with
\(\cA_R(F)=\R[R,T,F]\). Thus the \((2,4,6)\) and \((2,4,4)\)
closures differ as polynomial algebras despite agreeing after localization.

\section[Quadratic Compositions and Quartics]{Quadratic Compositions and Quartic Hypercones}
\label{sec:two-quadratic}

Starting with OT--FKM splitting quartics, we develop the two-block
normal form, quadratic Jordan compositions, and a common theorem
for orthogonal quartic data. Its realization table guides the quartic
applications. A Gram sextic extends the construction, while
higher-bidegree octonionic pullbacks exhibit transverse instability.

\subsection{OT--FKM splitting and the two-block normal form}

Allowing the two radii in the balanced OT--FKM sections of
\cite{CuiFKM25} to vary independently gives singular hypercones
with a $(2,2,4)$ Laplacian algebra.

Let $x,y\in\R^\ell$. Choose skew-symmetric matrices
$A_1,\ldots,A_{m-1}$ satisfying
\[
 A_iA_j+A_jA_i=-2\delta_{ij}\Id.
\]
Assume
\begin{equation}\label{bh:range}
 2\leq m\leq\ell-2.
\end{equation}
Only parameters for which this Clifford representation exists are
included. In particular, $A_1$ is a complex structure, so $\ell$ is even.
Set
\begin{gather}\label{bh:quartic-data}
 u=\norm{x}^2,\quad v=\norm{y}^2,\quad R=u+v,\quad T=uv,\qquad
 Q=\ip{x}{y}^2+\sum_{i=1}^{m-1}\ip{A_i x}{y}^2,\\
 c=\frac{m-1}{\ell-2},\quad a=1-2c,\quad b=c(1-c),\quad
 F=Q-cT.\label{bh:centered}
\end{gather}
Here $0<c<1$ and $b>0$; write $C=\{F=0\}$.

\begin{remark}[Relation to the original OT--FKM cone]\label{bh:fkm-relation}
In these coordinates, the Cartan--M\"unzner OT--FKM polynomial \cite{FKM81} is
\[
 \Phi=R^2-2(u-v)^2-8Q
      =-R^2+8(1-c)T-8F.
\]
On $u=v$, the equation $F=0$ is equivalent to
$\Phi=(1-2c)R^2$. Without the balance condition it is instead
$\Phi=-R^2+8(1-c)T$, whose normalized right-hand side varies with
$T/R^2$. The codimension-two cone studied in \cite{CuiFKM25}
is exactly the balanced section $C\cap\{u=v\}$, whose only singular
point is the vertex; the full hypercone has the two singular axes
described below.
\end{remark}

\subsubsection{Three generators of degrees \texorpdfstring{$(2,2,g)$}{(2,2,g)}}\label{sec:two-block-extension}

Quadratic Jordan closure \cite[Proposition~37(b)]{MendesRadeschi}
and Euler-operator bigrading \cite[Proposition~19(a)]{MendesRadeschiQuadratic}
give a two-block, separately harmonic normal form for Laplacian
algebras generated by two quadratics and one higher-degree polynomial.
This includes the OT--FKM splitting quartics; higher-degree examples
appear at the section's end.

\begin{proposition}\label{prop:two-block-normal}
Let $\cA=\R[R,D,F]$ be a graded Laplacian algebra with independent
quadratic generators $R,D$ and a homogeneous generator $F$ of degree
$e>2$, with $F\notin\R[R,D]$. After an orthogonal
change of coordinates and a linear change of $R,D$, its quadratic
part is generated by $u=|x|^2$, $v=|y|^2$, where
$x\in\R^p$, $y\in\R^q$. There is a radial polynomial $H(u,v)$
such that $P=F-H$ is bihomogeneous of some bidegree $(a,b)$,
$a+b=e$, and $\Delta_xP=\Delta_yP=0$.

Suppose that $p,q\geq2$ and both bidegrees are positive. Then closure
is equivalent to identities
\begin{align}
 |\nabla_xP|^2&=a^2\bigl(Ku^{a-1}v^b+
               Lu^{a/2-1}v^{b/2}P\bigr),\label{eq:two-block-x}\\
 |\nabla_yP|^2&=b^2\bigl(Ku^av^{b-1}+
               Lu^{a/2}v^{b/2-1}P\bigr),\label{eq:two-block-y}
\end{align}
for constants $K>0$ and $L\in\R$. The terms involving $L$ are
omitted unless $a,b$ are both even. Necessarily
\begin{equation}\label{eq:two-block-dimensions}
 \frac{p-2}{a}=\frac{q-2}{b}.
\end{equation}
In particular, $P$ restricts to an isoparametric function on
$\Sph^{p-1}\times\Sph^{q-1}$.
\end{proposition}

\begin{proof}
Write $D(z)=\langle Bz,z\rangle$, with $B$ symmetric. The product
rule gives $\Gamma(D,D)\in\cA_2=\operatorname{span}\{R,D\}$,
so $B^2$ is a linear combination of $B$ and the identity. Since
$D$ is independent of $R$, $B$ has exactly two eigenvalues.
Its orthogonal eigenspaces give $u,v$.

Put $\mathcal B=\R[u,v]$. The operator
$x\cdot\nabla_x=\Gamma(u,\cdot)/2$ preserves $\cA$, and
$\cA_e=\mathcal B_e+\R F$. Decomposing $F$ by bidegree therefore
shows that all its nonradial terms have the same $x$-degree $a$.
Subtract the other radial terms to obtain a bihomogeneous generator
$P_0$ of bidegree $(a,b)$. This algebra is now bigraded. The distinct
bidegree components of $\Delta P_0$ belong to $\mathcal B$, since
they have degree less than $e$.

They vanish unless $a,b$ are even. In the even mixed case write
$\Delta_xP_0=\lambda_xu^{a/2-1}v^{b/2}$ and
$\Delta_yP_0=\lambda_yu^{a/2}v^{b/2-1}$. Commuting these partial
Laplacians gives
$\lambda_x b(b+q-2)=\lambda_y a(a+p-2)$.
Subtracting the appropriate multiple of $u^{a/2}v^{b/2}$ kills
both. If one bidegree is zero, the subtraction only uses the active
block. This proves the separately harmonic normal form.

The two bidegree components of $\Gamma(P,P)$ also lie in the
algebra. Their degree is $2e-2$, so their expressions contain at
most the first power of $P$. Matching bidegrees gives
\eqref{eq:two-block-x}--\eqref{eq:two-block-y} initially with
possibly different constant and linear coefficients. On the unit
product spheres, let $f=P|_{\Sph^{p-1}\times\Sph^{q-1}}$.
The separate tangential gradient squares are quadratics in $f$,
with leading coefficients $-a^2$ and $-b^2$. Its mean is zero,
by separate harmonicity. Its maximum $M>0$ and minimum $m<0$
are roots of both quadratics, since both partial gradients vanish
at these points. Thus the two quadratics are respectively
$a^2(K+Lf-f^2)$ and $b^2(K+Lf-f^2)$, with
$K=-mM>0$ and $L=m+M$. This proves the common coefficients.
Conversely, the two identities and the separate Euler and Laplace
identities give every generator entry of the closure table, so
\cref{prop:MR-generator-test} proves sufficiency.

If $\mu_2$ is the average of $f^2$, integration by parts on each
sphere gives
\[
 a^2K=a(2a+p-2)\mu_2,\qquad
 b^2K=b(2b+q-2)\mu_2.
\]
This implies \eqref{eq:two-block-dimensions}. Finally,
\[
 |\nabla f|^2=(a^2+b^2)(K+Lf-f^2),\qquad
 \Delta f=-[a(a+p-2)+b(b+q-2)]f,
\]
which are the isoparametric equations on the product.
\end{proof}

The block dimensions $p,q$ need
not agree; if $p=q>2$, \eqref{eq:two-block-dimensions} forces $a=b$.
The change $F\mapsto F-H(u,v)$ preserves the algebra but generally
changes its zero set, so the minimal level must still be determined
within each explicit family. If one bidegree is zero, the active block
reduces to \cite[Proposition~37(c)]{MendesRadeschi}, with an independent
Euclidean factor adjoined.

\subsection{Quadratic Jordan algebras and the composition equation}
\label{sec:quadratic-jordan-framework}

Quadratic-generated Laplacian algebras organize compositions of
arbitrary even degree. For a real symmetric matrix $A$, write
$q_A(x)=\langle Ax,x\rangle$.
Then
\begin{equation}\label{eq:quadratic-jordan-correspondence}
 \Delta q_A=2\tr A,\qquad
 \Gamma(q_A,q_B)=4q_{A\circ B},\qquad
 A\circ B=\tfrac12(AB+BA).
\end{equation}
Thus a quadratic-generated algebra $\cA$, containing $R=|x|^2$,
is Laplacian exactly when its quadratic part $J=\cA_2$ is a unital
Jordan algebra under $f\circ g=\Gamma(f,g)/4$
\cite[Proposition~37(b)]{MendesRadeschi}. The matrix realization matters: its multiplicities enter $\Delta q_A$.
The Jordan classification gives Clifford and real, complex, or
quaternionic Gram models; see
\cite[Theorems~B--C, Section~3 and Appendix~A]{MendesRadeschiQuadratic}.

The next proposition is the low-dimensional consequence of this
classification; it concerns algebras generated entirely in degree two.

\begin{proposition}[Three and four quadratic generators]
\label{prop:three-quadratic-normal-forms}
\label{prop:low-dimensional-quadratic-jordan}
Let $\cA$ be generated by homogeneous quadratics, contain $R$, and
be Laplacian. If $k=\dim\cA_2$ is three or four, the following
list is complete, up to orthogonal source coordinates and a linear
change of quadratic generators.
\begin{center}
\small
\begin{tabular}{cl>{\raggedright\arraybackslash}p{0.54\textwidth}}
\toprule
$k$ & $J$ & Quadratic coordinates \\
\midrule
3 & $\R^3$ & $|x_1|^2,|x_2|^2,|x_3|^2$ on three nonzero orthogonal blocks. \\
3 & $\operatorname{Sym}_2(\R)$ & $|x|^2,|y|^2,\langle x,y\rangle$, with $x,y\in\R^\ell$. \\
4 & $\R^4$ & Four nonzero orthogonal block radii. \\
4 & $\R\oplus\operatorname{Sym}_2(\R)$ & An independent block radius and the preceding real Gram coordinates. \\
4 & $\operatorname{Herm}_2(\mathbb C)$ & $|z|^2,|w|^2,\Re\langle z,w\rangle,\Im\langle z,w\rangle$, with $z,w\in\mathbb C^\ell$. \\
\bottomrule
\end{tabular}
\end{center}
Here $k$ is the number of minimal quadratic algebra generators;
algebraic relations may occur among them.
\end{proposition}

\begin{proof}
The correspondence \eqref{eq:quadratic-jordan-correspondence} embeds
$J$ as a unital Jordan subalgebra of the symmetric matrices. In the
formally real Jordan classification recalled in
\cite[Appendix~A]{MendesRadeschiQuadratic}, the simple factors of
real dimension at most four are $\R$, $\operatorname{Sym}_2(\R)$,
and $\operatorname{Herm}_2(\mathbb C)$, of dimensions $1,3,4$.
The spin factors of dimensions three and four are the latter two
matrix algebras. Taking direct sums gives precisely the displayed
list. A primitive one-dimensional factor acts by an orthogonal
projection; these projections split the source. The representation
argument in \cite[Section~3 and Appendix~A]{MendesRadeschiQuadratic}
gives repeated standard real or complex modules for the remaining
factors, hence their Gram coordinates. Finally, a degree-two element
of a quadratic-generated algebra is a linear combination of its
quadratic generators, proving the assertion about their minimal number.
\end{proof}

For the real two-column model use
$R=|x|^2+|y|^2$, $p=|x|^2-|y|^2$, and $q=2\langle x,y\rangle$.
Its image is $\{R\geq0:p^2+q^2\leq R^2\}$ for $\ell\geq2$;
for $\ell=1$ only the boundary $R^2=p^2+q^2$ is attained.
The complex two-column model similarly has coordinates
$R,p,q_1,q_2$, with $q_1+iq_2=2\langle z,w\rangle$ and
$p=|z|^2-|w|^2$. Its image is
$\{R\geq0:p^2+q_1^2+q_2^2\leq R^2\}$ for $\ell\geq2$;
for $\ell=1$ only its boundary is attained. These statements follow by realizing positive
semidefinite $2\times2$ Gram matrices of the allowed rank.

An arbitrary quadratic span need not be closed: on $\R^3$,
$\operatorname{span}\{R,x_1^2,x_1x_2\}$ omits
$\Gamma(x_1x_2,x_1x_2)=x_1^2+x_2^2$. Adjoining $x_2^2$ gives
$\R\oplus\operatorname{Sym}_2(\R)$. Every finite list has a unital
matrix Jordan closure of dimension at most $N(N+1)/2$; restricting
its size is essential, since all even-degree homogeneous polynomials
are polynomials in the full collection $x_ix_j$.

Choose a basis $z=(z_1,\ldots,z_k)$ of $J$ and put
\begin{equation}\label{eq:quadratic-jordan-profile-data}
 \begin{gathered}
 B_{ij}(z)=\tfrac14\Gamma(z_i,z_j),\qquad
 b_i=\tfrac14\Delta z_i,\\
 A_P=\sum_{i,j}B_{ij}P_iP_j,\qquad
 \mathscr D P=\sum_{i,j}B_{ij}P_{ij}+\sum_i b_iP_i.
 \end{gathered}
\end{equation}
The entries of $B$ are linear and the $b_i$ are constant.
The quotient calculus in \cref{thm:finite-reduction} gives
\begin{equation}\label{eq:quadratic-jordan-profile-pullback}
 |\nabla P(z)|^2=4A_P(z),\qquad
 \Delta P(z)=4(\mathscr D P)(z),\qquad
 \cL(P(z))=8\mathscr T(P)(z),
\end{equation}
where
\begin{equation}\label{eq:three-quadratic-profile-numerator}
 \mathscr T(P)=2A_P\mathscr D P
       -\sum_{i,j}B_{ij}P_i\partial_jA_P.
\end{equation}
These identities are evaluated on the actual generator image, also
when the generators satisfy relations.

\begin{proposition}[Composition criterion]
\label{prop:three-quadratic-profile-criterion}
\label{prop:quadratic-jordan-profile-criterion}
Let $P$ be a real homogeneous polynomial of degree $d$ and set
$F=P(z)$. Its regular zero set is minimal if and only if
$\mathscr T(P)=0$ on the part of the generator image where
$P=0$ and $A_P>0$. In particular, $P\mid\mathscr T(P)$ is
sufficient. It is also necessary if $P$ is irreducible and has a
regular zero in an open subset of the image where $B$ is positive
definite. In that case the quotient has degree $2d-2$, while $F$
has degree $2d$.
\end{proposition}

\begin{proof}
The first assertion follows from
\eqref{eq:quadratic-jordan-profile-pullback} and the level-set
mean-curvature identity. Positive definiteness makes the generator
map a submersion and gives $A_P>0$ at regular profile zeros.
Minimality therefore makes $\mathscr T(P)$ vanish on a relatively
open smooth real patch of $P=0$. The patch-divisibility argument in
\cref{prop:minimizing-irreducibility} gives $P\mid\mathscr T(P)$.
Its degree is $3d-2$, which proves the last assertion.
\end{proof}

For orthogonal radii $u_i=|x_i|^2$, with $n_i=\dim V_i$, the data are
\begin{equation}\label{eq:three-radii-profile-operator}
 B=\operatorname{diag}(u_i),\qquad
 \mathscr D P=\sum_i u_iP_{ii}+\tfrac12\sum_i n_iP_i.
\end{equation}
Thus the criterion specializes to
\begin{equation}\label{eq:three-radii-minimality}
 P\ \bigm|\
 \left(\sum_i u_iP_i^2\right)
 \left(2\sum_i u_iP_{ii}+\sum_i n_iP_i\right)
 -\sum_i u_iP_i\partial_i\left(\sum_j u_jP_j^2\right),
\end{equation}
under the irreducibility and interior-zero hypotheses above.

\subsubsection{Clifford profiles with a radial variable}

For a symmetric Clifford system on $\R^N$ with $s\geq2$, let
$p_i(x)=\langle P_ix,x\rangle$ for $1\leq i\leq s$.
The algebra $\R[R,p_1,\ldots,p_s]$ has
\begin{equation}\label{eq:three-spin-profile-operator}
 B=\begin{pmatrix}R&p^T\\p&R\Id_s\end{pmatrix},\qquad
 \mathscr D P=R(P_{RR}+\Delta_pP)
      +2\sum_i p_iP_{Rp_i}+\tfrac N2P_R.
\end{equation}
This follows from \eqref{eq:Clifford-pullback-identities}; the real
and complex two-column models have $(s,N)=(2,2\ell)$ and $(3,4\ell)$.

If a homogeneous profile $f(p)$ is independent of $R$, the identity
\eqref{eq:pullback-defect} reads
\begin{equation}\label{eq:quadratic-Clifford-profile-transfer}
 \cL_x(f(p))=16R^2(\cL_pf)(p)
             -8d f(p)|\nabla_pf|^2(p),\qquad d=\deg f.
\end{equation}
Thus a degree-$d$ minimal target cone yields a degree-$2d$ minimal
source equation, and a harmonic target gives a harmonic source.
This is the Clifford specialization of Riedler's cone-pullback
theorem \cite[Theorem~5.1.4]{RiedlerThesis25}; its subcalibration
transfer is \cref{thm:clifford-area-transfer}. In particular, minimal
cubic targets supply sextic equations. Effective factors and the
variational conclusions are determined by the corresponding results
in \cref{sec:quartic-application}.

Allowing the radial coordinate also includes the OT--FKM splitting
in \eqref{bh:centered}. For a homogeneous profile $P(t,y)$ of degree
$d$, define
\[
 K_P=|\nabla_yP|^2-P_t^2,\qquad
 \square P=\Delta_yP-P_{tt},\qquad
 \Gamma_L(P,K_P)=\nabla_yP\cdot\nabla_yK_P-P_t(K_P)_t.
\]
On $F=P(R,p)=0$, the same profile table gives
\begin{align}
 |\nabla F|^2&=4R K_P,\notag\\
 \cL(F)&=16R^2\left(K_P\square P
                  -\tfrac12\Gamma_L(P,K_P)\right)
             +8(N-2)RK_PP_t.
 \label{eq:radial-Clifford-profile-equation}
\end{align}
For clarity, Euler's identity first gives, on the whole source,
\[
 |\nabla F|^2=4RK_P+8dPP_t,\qquad
 \Delta F=4R\square P+(2N+8d-8)P_t.
\]
Differentiate the first expression before setting $P=0$;
$\Gamma(F,R)=4dF$ and the homogeneity of $K_P$, of degree $2d-2$,
then give \eqref{eq:radial-Clifford-profile-equation}.
For $P=|y|^2-ct^2$, $0<c<1$, this equation selects
$c=2(s-1)/(N-2)$, the familiar minimal isoparametric level.
The splitting quartic instead has profile
$\tfrac14(\sum_{i=1}^m p_i^2+cp_0^2-cR^2)$, while Riedler's
quartic is independent of $R$. Both therefore belong to the same
composition calculus.

\begin{remark}[An elementary radial profile]
\label{sec:three-quadratic-profiles}
In the real two-column model, the radial variable gives quartic minimal levels: for
$\ell>2$ set
\[
 W=p^2+q^2,\qquad c=\frac1{\ell-1},\qquad
 F=W-cR^2,\qquad \Phi=R^2-2W.
\]
Equation \eqref{eq:radial-Clifford-profile-equation} selects this $c$;
$|\nabla\Phi|^2=16R^3$ and $\Delta\Phi=8(\ell-3)R$ identify
$F=0$ as its minimal isoparametric level. Its smallest algebra is
$\R[R,F]$, since $W=F+cR^2$. This minimal level is covered by
\cref{sec:isoparametric-reduction}.
\end{remark}

\subsubsection{Real and complex Gram profiles}

Let $X=[x_1\ \cdots\ x_r]\in\R^{p\times r}$ and set
$h_{ij}=\langle x_i,x_j\rangle$. The independent entries of
$H=X^TX$ span $\operatorname{Sym}_r(\R)$, with $R=\tr H$, and
\begin{equation}\label{eq:real-Gram-differential-table}
 \begin{split}
 \Delta h_{ij}&=2p\delta_{ij},\\
 \Gamma(h_{ij},h_{kl})
 &=\delta_{ik}h_{jl}+\delta_{il}h_{jk}
   +\delta_{jk}h_{il}+\delta_{jl}h_{ik}.
 \end{split}
\end{equation}
Indeed, $\nabla_{x_a}h_{ij}=\delta_{ai}x_j+\delta_{aj}x_i$;
taking inner products and traces proves the table, including the
diagonal entries. For $Z\in\mathbb C^{p\times r}$, use the real
Euclidean metric and write
$h_{ij}=\sum_a\overline{z_{ai}}z_{aj}$, so $H=Z^*Z$.
The complex-bilinear extension of $\Gamma$ satisfies
\begin{equation}\label{eq:complex-Gram-differential-table}
 \Delta h_{ij}=4p\delta_{ij},\qquad
 \Gamma(h_{ij},h_{kl})
       =2(\delta_{jk}h_{il}+\delta_{il}h_{kj}).
\end{equation}
To verify it, use $\Delta=4\sum_{a,b}\partial_{z_{ab}}
\partial_{\bar z_{ab}}$ and
$\Gamma(f,g)=2\sum_{a,b}(f_{z_{ab}}g_{\bar z_{ab}}
+f_{\bar z_{ab}}g_{z_{ab}})$. Taking real and imaginary parts
gives the real table on $\operatorname{Herm}_r(\mathbb C)$.
These are the standard Gram realizations in
\cite[Section~3 and Section~7]{MendesRadeschiQuadratic}.

Their generator images consist of positive semidefinite Gram matrices
of rank at most $p$. For $p\geq r$ the positive definite part is
open and the Gram map is a submersion there: for a prescribed
symmetric or Hermitian increment $A$, the variation
$\dot X=\tfrac12 XH^{-1}A$, or
$\dot Z=\tfrac12 ZH^{-1}A$, has $\dot H=A$.
Consequently \cref{prop:quadratic-jordan-profile-criterion} applies
directly to profiles with a regular zero in this positive definite
part. For $p<r$ the same differential identities apply on the
rank-constrained image. The constructions below select minimal
profiles from these differential tables.

\subsubsection{Orthogonal quartic blocks and their weights}
\label{sec:Clifford-composition-background}

The OT--FKM and Riedler quartics are quadratic profiles in Clifford
coordinates and the radius. The spin-factor realization of
\cref{sec:quadratic-jordan-framework} gives their common quartic table.
The underlying Hurwitz matrix construction and its reduction to
Clifford representations are described in
\cite[Chapter~0, pp.~2--3; Chapter~1, Lemma~1.5 and Proposition~1.7]{Shapiro2000}.

Let $\mathcal C=(P_0,\ldots,P_m)$ be a symmetric Clifford system on
$\R^{2\ell}$, with $m\geq1$, and put
\[
 R=|z|^2,\qquad p_i(z)=\langle P_i z,z\rangle,
 \qquad p=(p_0,\ldots,p_m).
\]
Their differential table is \eqref{eq:Clifford-pullback-identities},
and \cref{lem:Clifford-algebra-pullback} gives $|p|^2\leq R^2$.

In the product $f\circ g=\Gamma(f,g)/4$, the quadratic space
$J=\operatorname{span}\{R,p_0,\ldots,p_m\}$ has multiplication
\begin{equation}\label{eq:Clifford-composition-quadratic-table}
 R\circ R=R,\qquad R\circ p_i=p_i,\qquad
 p_i\circ p_j=\delta_{ij}R.
\end{equation}
Thus $J$ is a spin-factor Jordan algebra. The Clifford relations
give the following table for its quartic block norms.

\begin{lemma}[Orthogonal quartic blocks]
\label{lem:Clifford-block-quartics}
Partition $\{0,\ldots,m\}$ into nonempty sets $I_1,\ldots,I_s$,
and set
\[
 E_j=\sum_{i\in I_j}p_i^2\quad(1\leq j\leq s),\qquad
 E_*=R^2-\sum_{i=0}^m p_i^2.
\]
Then all these polynomials are nonnegative and
\begin{align}
 &E_*+\sum_{j=1}^s E_j=R^2,\qquad
 \Gamma(E_i,E_j)=16RE_i\delta_{ij},
 \label{eq:Clifford-composition-quartic-table}\\
 &\Delta E_j=8|I_j|R,\qquad
 \Delta E_*=8(\ell-m)R.\nonumber
\end{align}
Here the pairing identity includes the index $*$.
\end{lemma}

\begin{proof}
For disjoint blocks the pairing in
\eqref{eq:Clifford-pullback-identities} gives
$\Gamma(E_i,E_j)=0$; for one block it gives
$\Gamma(E_j,E_j)=16RE_j$ and $\Delta E_j=8|I_j|R$.
Also $\Gamma(R^2,E_j)=16RE_j$, so $E_*$ is orthogonal to every
$E_j$.  Expanding its gradient square gives
$\Gamma(E_*,E_*)=16RE_*$.  Finally,
$\Delta R^2=8(\ell+1)R$ proves the asserted Laplacian of $E_*$.
\end{proof}

\begin{remark}[Geometric interpretation]
\label{rem:weighted-Clifford-foliation}
For $\ell>m+1$, the Clifford map on the unit sphere has connected
fibers and hemisphere quotient metric. If all $|I_j|\geq2$, pulling
back the target $\prod_j\mathrm{SO}(|I_j|)$-orbits gives the common
quartic-block levels: this is Radeschi's composed foliation
\cite[Theorems~A--B and Section~3]{Radeschi14}. The one-leaf target
gives the OT--FKM foliation. In the splitting model the target
$\mathrm{SO}(m)$-action fixes $p_0$; its leaves retain $u-v$, while
the quartic subalgebra retains only $(u-v)^2$.
The balanced $(3,7,7)$ model in
\cref{sec:weighted-Clifford-quartic} uses $\mathcal C_{9,1}$ and
target $\mathrm{SO}(3)\times\mathrm{SO}(7)$-orbits, with profile
$E_+=E_-$. A profile equation selects a union of leaves; the
differential table requires no connectedness assumption.
\end{remark}

For three quartic blocks, the weights in the table are the target
block dimensions and the remaining weight $\ell-m$. The following
criterion identifies every minimal linear combination with a regular
zero meeting the positive part of this table.

\begin{proposition}[Minimal levels of orthogonal quartic generators]
\label{prop:orthogonal-quartics}
Suppose that nonzero homogeneous quartic polynomials $E_0,E_1,E_2$ satisfy
\begin{equation}\label{eq:orthogonal-quartic-data}
 E_0+E_1+E_2=R^2,\qquad
 \Gamma(E_i,E_j)=16RE_i\delta_{ij},\qquad
 \Delta E_i=8r_iR
\end{equation}
for constants $r_i>0$. They generate a Laplacian algebra together
with $R$, of generator degrees $(2,4,4)$ after eliminating one $E_i$.
Let $F=\sum_i\alpha_iE_i$ have a regular zero point with all $E_i>0$,
and put $\sigma=\sum_i r_i\alpha_i$.
Then $F=0$ is minimal if and only if its nonzero coefficients have
exactly two values $\alpha>0>\beta$, satisfying
\begin{equation}\label{eq:quartic-level-condition}
 \alpha(r_+-1)+\beta(r_--1)=0,\qquad
 r_+=\sum_{\alpha_i=\alpha}r_i,\quad r_-=\sum_{\alpha_i=\beta}r_i.
\end{equation}
If $F=0$ is minimal and none of the coefficients vanishes,
$\R[R,F]$ is already Laplacian.
\end{proposition}

\begin{proof}
The norm identities give $E_i\geq0$. Where they are positive,
their gradient matrix is positive definite, so their image has
interior. Set $A_j=\sum_i\alpha_i^jE_i$. The closure table gives
\[
 \Gamma(F,F)=16RA_2,\qquad
 \cL(F)=128R^2(\sigma A_2-A_3)-64FA_2.
\]
Thus minimality on the regular zero plane is equivalent to
$\sigma A_2-A_3=\kappa F$ for a constant $\kappa$. The image contains
an open piece of that plane, so comparison of its linear coefficients
is valid. For each nonzero $\alpha_i$ this reads
$\alpha_i(\sigma-\alpha_i)=\kappa$. There are at most two such
values, and a zero point with positive $E_i$ requires both signs.
Their sum is $\sigma$, which is precisely
\eqref{eq:quartic-level-condition}. These conditions also make
$\cL(F)$ divisible by $F$ globally and prove sufficiency.
In the minimal case, if all coefficients are nonzero, then
$A_2=(\alpha+\beta)F-\alpha\beta R^2$; the displayed norm and
Laplacian identities prove closure of $\R[R,F]$.
\end{proof}

\begin{remark}[More quartic blocks]\label{rem:quartic-many-blocks}
For any number of nonzero orthogonal quartic blocks summing to $R^2$,
the same coefficient comparison permits only two nonzero coefficient
values, provided the regular zero meets the positive-block region.
Group equal positive, negative, and zero coefficients and add their
weights. This reduces the minimal-level equation to
\eqref{eq:quartic-level-condition}, with two generators when no block
is unused. Different representations can still yield different cones.
\end{remark}

\begin{remark}[Compatible Cartan--M\"unzner equations]
\label{rem:compatible-quartic-CM}
Setting $\Phi_i=R^2-2E_i$ in \eqref{eq:orthogonal-quartic-data} gives
\[
 \sum_i\Phi_i=R^2,\qquad |\nabla\Phi_i|^2=16R^3,\qquad
 \Delta\Phi_i=4(N+2-4r_i)R.
\]
For distinct $i,j,k$ it also gives
$\Gamma(\Phi_i,\Phi_j)=-16R\Phi_k$.
Conversely, the sum and the individual gradient identities imply
these cross pairings; together with the displayed Laplacians they
recover \eqref{eq:orthogonal-quartic-data}.
These are compatible Cartan--M\"unzner-type equations, allowing
degenerate quartic representations: for example,
$E_i=(|x|^2-|y|^2)^2$ with equal-dimensional blocks has $r_i=1$.
\end{remark}

For the quartic algebra $\cA=\R[R,E_0,E_1,E_2]$ itself,
$\cA_2=\R R$. In the Clifford construction above, it lies in a
larger quadratic algebra. The criterion below determines its
realizable weights.

\begin{lemma}[Weights and image of the quartic blocks]
\label{lem:quartic-block-weights-image}
For the nonzero blocks in \eqref{eq:orthogonal-quartic-data}, the
weights are positive integers and $N=2(r_0+r_1+r_2-1)$. Their
spherical image is exactly
\begin{equation}\label{eq:quartic-simplex-image}
 (E_0,E_1,E_2)(\Sph^{N-1})
 =\{(u_0,u_1,u_2)\in[0,1]^3:u_0+u_1+u_2=1\}.
\end{equation}
\end{lemma}

\begin{proof}
Summing the Laplacian identities in \eqref{eq:orthogonal-quartic-data}
gives $N=2(r_0+r_1+r_2-1)$. The norm identities imply $E_i\geq0$.
On the unit sphere,
$u_i=E_i|_{\Sph^{N-1}}$ satisfies
\[
 |\nabla_{\Sph}u_i|^2=16u_i(1-u_i),\qquad
 \Delta_{\Sph}u_i=8r_i-4(N+2)u_i.
\]
A constant $u_i$ would be zero or one, contradicting the assumption
that all three blocks are nonzero.  Its minimum and maximum are
therefore zero and one.  Integration of the Laplacian identity
gives $r_i>0$.  At a nonzero zero of $E_i$, write
$H_i=\Hess E_i$.  Comparing quadratic terms in
$|\nabla E_i|^2=16RE_i$ gives
\[
 H_i^2=8RH_i,\qquad \operatorname{tr}H_i=8r_iR.
\]
Thus $H_i=8RP_i$ for an orthogonal projection $P_i$ of rank $r_i$;
in particular, each weight is an integer.

Nonnegativity and the sum of the blocks place the spherical image
in the displayed simplex. Its intersection with the simplex interior
is nonempty, since $E_0E_1E_2$ is a nonzero polynomial. There the
spherical gradient matrix
$16(\operatorname{diag}(u)-uu^{\mathsf T})$ has rank two, so the
intersection is relatively open by the submersion theorem. It is
also relatively closed, since the spherical image is compact.
Connectedness of the simplex interior, followed by taking its closure,
proves \eqref{eq:quartic-simplex-image}.
\end{proof}

The Clifford-module restriction gives an exact numerical range for
\eqref{eq:orthogonal-quartic-data}. We use the conventional function
\[
 (\delta(1),\ldots,\delta(8))=(1,2,4,4,8,8,8,8),
 \qquad \delta(k+8)=16\delta(k).
\]
Here $\delta(k)$ is the real irreducible module dimension for the
negative Clifford algebra on $k-1$ generators. A symmetric Clifford
system with $p\ge2$ matrices on $\R^{2\ell}$ exists exactly when
$\delta(p-1)\mid\ell$ \cite[Section~3.5]{FKM81}.
See also \cite[Chapter~0, pp.~3--4]{Shapiro2000} for the equivalent
Hurwitz--Radon restriction.

\begin{proposition}[Realizable quartic weights]
\label{prop:quartic-weight-realizability}
Let $r_0,r_1,r_2$ be positive integers, write their rearrangement as
$a\le b\le c$, and set $\ell=a+b+c-1$.
Nonzero quartics satisfying \eqref{eq:orthogonal-quartic-data} on
$\R^{2\ell}$ exist if and only if \(\delta(a+b-1)\mid\ell\).
Whenever this condition holds, a realization is obtained from one
Clifford system by two blocks of squared quadratic coordinates and
their complement in $R^2$.
\end{proposition}

\begin{proof}[Construction and parameter restrictions]
\leavevmode\par\nobreak
Assume the divisibility condition. Take $a+b$ symmetric Clifford
matrices, with quadratic coordinates $p_1,\ldots,p_{a+b}$.
The three polynomials
\[
 \sum_{\alpha=1}^{a}p_\alpha^2,\qquad
 \sum_{\alpha=a+1}^{a+b}p_\alpha^2,\qquad
 R^2-\sum_{\alpha=1}^{a+b}p_\alpha^2
\]
have weights $(a,b,c)$ by \cref{lem:Clifford-block-quartics}.
Their positive Laplacian coefficients ensure that none is zero.

For any realization, summing the Laplacians forces
$N=2(r_0+r_1+r_2-1)$. Positivity and integrality of the weights follow from
\cref{lem:quartic-block-weights-image}.
The converse uses the completed classification with four principal
curvatures and compatibility of the three blocks. See
\cref{app:quartic-weight-realizability} and its weight-one,
Clifford-joining, and exceptional-triple lemmas
\cref{lem:quartic-weight-one,lem:quartic-Clifford-joining,lem:quartic-exclude-155}.
\end{proof}

This is a classification of realizable weights, not of simultaneous
orthogonal equivalence classes of triples. For example, $(3,3,5)$
is excluded since $\delta(5)=8\nmid10$, although the individual
Cartan--M\"unzner multiplicity pairs $(2,7),(2,7),(4,5)$ all occur.
The criterion is symmetric in the weights; the subcalibration
range in \cref{thm:quartic-table-existence} instead depends on
which two blocks are chosen for the defining polynomial.

\begin{corollary}[The weights with active sum nine]
\label{cor:quartic-nine-weights}
Let $r_0,r_1,r_2$ be positive integers with $r_1,r_2\geq2$ and
$r_1+r_2=9$. A quartic realization exists if and only if $r_0$ is
a positive multiple of eight. Its ambient dimension is $N=2r_0+16$.
\end{corollary}

\begin{proof}
Up to exchanging the active weights, their pairs are $(2,7)$,
$(3,6)$, and $(4,5)$. For $1\leq r_0\leq7$, none of the corresponding module
dimensions divides $\ell=r_0+8$; the seven checks are listed in
\cref{tab:quartic-nine-low-weights}.
If $r_0\geq8$, the two
smallest weights are the active pair. Hence $a+b-1=8$ and the same
criterion is $8=\delta(8)\mid r_0+8$, equivalently $8\mid r_0$.
\end{proof}

\subsection{Subcalibrations from orthogonal quartic data}
\label{sec:quartic-table-existence}

The orthogonal-quartic table determines the minimal level, singular
set, and subcalibration of every realization in the following sufficient
minimizing range. Exact thresholds for particular families also use
the instability arguments below.

Throughout this subsection, let $E_0,E_1,E_2$ be nonzero homogeneous
quartic polynomials on $\R^N$ satisfying
\eqref{eq:orthogonal-quartic-data}, with $R=|x|^2$.  Suppose
$r_1,r_2\geq2$, and put
\begin{gather}
 \begin{gathered}
 h=r_1+r_2,\qquad \rho=r_0+r_1+r_2,\\
 c=\frac{r_1-1}{h-2},\qquad a=1-2c,\qquad b=c(1-c)>0,
 \end{gathered}
 \label{eq:quartic-general-weights}\\
 T=\frac{E_1+E_2}{4},\qquad
 F=\frac{(1-c)E_1-cE_2}{4}.
 \label{eq:quartic-general-generators}
\end{gather}
By \cref{lem:quartic-block-weights-image}, $N=2\rho-2$.
The choice of $c$ is the minimal-level condition
\eqref{eq:quartic-level-condition}.  Direct substitution gives
\begin{equation}\label{bh:reduced-table}
 \begin{gathered}
 (\Gamma(I_i,I_j))_{I=(R,T,F)}=
 \begin{pmatrix}
 4R&8T&8F\\
 8T&4RT&4RF\\
 8F&4RF&4R(aF+bT)
 \end{pmatrix},\\
 (\Delta R,\Delta T,\Delta F)=(2N,2hR,2aR).
 \end{gathered}
\end{equation}
Thus the gradient table is independent of $r_0$, while the ambient
dimension records all three weights.  By
\cref{prop:MR-generator-test}, $\R[R,T,F]$ is a Laplacian algebra.
It is exactly $\cA_R(F)$: writing $S=|\nabla F|^2$, one obtains
\[
 S=4R(aF+bT),\qquad
 \Delta S=8(N+8)(aF+bT)+8(a^2+bh)R^2.
\]
Since $b>0$, the second identity recovers $T$ in $\cA_R(F)$.
The generators are algebraically independent.  Indeed, the nonzero
polynomial $E_0E_1E_2$ is positive somewhere; there the gradients of
the three blocks are independent, and the change to $(R,T,F)$ is
locally invertible.  Fixing $R=1$ therefore leaves spherical quotient
dimension two.

\begin{lemma}[Geometry determined by the quartic data]
\label{lem:quartic-table-geometry}
At each fixed $R\geq0$, the full image of $(T,F)$ is
described by
\begin{equation}\label{eq:quartic-table-image}
 0\leq T\leq R^2/4,\qquad -cT\leq F\leq(1-c)T.
\end{equation}
The critical set of $F$ and the geometric singular set of
$C=\{F=0\}$ are both $Z=\{T=0\}$, and
$\dim Z\leq N-h$.  The hypersurface $C$ is minimal outside $Z$,
both phases approach every point of $C$, and
$\partial\llbracket\{F<0\}\rrbracket$ has support $C$.
\end{lemma}

\begin{proof}
The simplex image in \cref{lem:quartic-block-weights-image} and
homogeneity prove \eqref{eq:quartic-table-image}. Moreover,
\[
 |\nabla F|^2=R\bigl((1-c)^2E_1+c^2E_2\bigr),
\]
so the critical set of $F$ is exactly $Z$.  The reduced table gives
\begin{equation}\label{bh:minimality}
 \cL(F)=-8F\{b(R^2+2T)+2aF\},
\end{equation}
which proves regular minimality.

Nonzero points of $Z$ exist because $E_0$ attains one on the unit
sphere. Fix such a point $p$ and write $R_0=R(p)$ and
$\Hess E_i(p)=8R_0P_i$ for $i=1,2$, using the projections obtained
in the proof of \cref{lem:quartic-block-weights-image}. The quadratic terms of
$\Gamma(E_1,E_2)=0$ give $P_1P_2+P_2P_1=0$.  Taking the trace
shows $P_1P_2=0$.  Hence
\[
 \Hess T(p)=2R_0(P_1+P_2),\qquad
 \Hess F(p)=2R_0\bigl((1-c)P_1-cP_2\bigr).
\]
The first Hessian has rank $h$.  Applying the implicit function
theorem to $h$ independent components of $\nabla T$ places $Z$
locally in a smooth submanifold of codimension $h$.  Consequently
$\dim Z\leq N-h$ and $\cH^{N-1}(Z)=0$.
The second Hessian has positive and negative indices $r_1,r_2$,
respectively, so both phases occur arbitrarily near every such $p$.
If the zero set were a smooth hypersurface at $p$, its leading
quadratic term would vanish on its tangent hyperplane and thus have
rank at most two.  Its actual rank is $h\geq4$, a contradiction.
Thus these points are geometrically singular.

Homogeneity gives the same two-phase conclusion at the origin,
and regularity gives it on $C\setminus Z$.  Polynomial finite
perimeter and the local BV Poincar\'e inequality then identify the
support of the boundary current with $C$.  The vertex is singular
as well: a cone smooth at its vertex would be a hyperplane, whereas
$C$ has nonzero singular points.
\end{proof}

\begin{theorem}[Polynomial subcalibrations from orthogonal quartics]
\label{thm:quartic-table-existence}
For the data \eqref{eq:orthogonal-quartic-data} and
\eqref{eq:quartic-general-weights}--\eqref{eq:quartic-general-generators},
suppose that either
\[
 h\geq9,
 \qquad\text{or}\qquad
 r_1=r_2=4\ \text{ and }\ r_0\geq5.
\]
There exists a homogeneous polynomial $G=AF$ of degree at most ten,
with $A>0$ on $T>0$, whose critical set is $Z=\{T=0\}$ and whose
normalized gradient satisfies
\[
 \operatorname{sgn}(G)\,\diver X_G>0\qquad\text{off }C=\{F=0\}.
\]
The current $\partial\llbracket\{F<0\}\rrbracket$ has full support
$C$ and is strictly area minimizing.  Its regular part is strictly
stable, with the explicit lower bound
\begin{equation}\label{eq:bihomogeneous-Hardy}
 Q_C(\varphi)\geq
 \left[\frac{(N-17)^2}{4}+8h-72\right]
 \int_Cr^{-2}\varphi^2,\qquad
 \varphi\in C_c^\infty(C_{\reg}\setminus\{0\}),\quad r=\sqrt R.
\end{equation}
\end{theorem}

\subsubsection{Selecting the multiplier}

At $p\in Z\setminus\{0\}$, \eqref{bh:minimality} gives
$\lambda_F(p)=-8bR(p)^2<0$. Hence
\cref{lem:critical-multiplier-defect} forces every admissible smooth
multiplier to vanish on $Z$.
For $A\in\R[R,T]$, scaling a nonzero point of $Z$ gives
$A(R,0)=0$ for all $R>0$, so $T$ divides $A$ as a polynomial in
$R,T$.  The lowest two nonzero homogeneous choices are therefore
$T$ and $RT$, up to a positive constant.

These two candidates suffice in the stated ranges:
\begin{equation}\label{bh:potentials}
 G=\begin{cases}
 TF,&h\geq11,\ \text{or }h=10\ \text{and }r_1,r_2\geq3,\\
 RTF,&h=9,\ \text{or }\{r_1,r_2\}=\{2,8\},\\
 RTF,&r_1=r_2=4,\ r_0\geq5.
 \end{cases}
\end{equation}
With $s=T/R^2$ and $z=F/T$, these are $R^4s^2z$ and $R^5s^2z$:
the radial factor changes the defect but preserves vanishing on $Z$.
We use \cref{thm:finite-reduction,thm:exact-multiplier} and, for
$h=9$, the realizability restriction \cref{cor:quartic-nine-weights}.

\subsubsection{The degree-eight potential and its range}

For $G_8=TF$, the reduced table gives
\begin{gather}
 \norm{\nabla G_8}^2=4RT(bT^2+aTF+3F^2),\qquad
 \Delta G_8=R\{2aT+2(h+4)F\},\label{bh:G8-basic}\\
 \begin{split}\label{bh:G8-defect}
 \cL(G_8)=8TF\big\{&bT^2[(h-8)R^2-4T]\\
                  &+aTF[(h-4)R^2-4T]
                    +3F^2[hR^2-4T]\big\}.
 \end{split}
\end{gather}
Introduce the dimensionless variables
\[
 s=T/R^2\in(0,1/4],\qquad z=F/T\in[-c,1-c].
\]
The braces in \eqref{bh:G8-defect}, divided by $R^2T^2$, are
\begin{equation}\label{bh:p8}
 p_8(z,s)=b(h-8-4s)+a(h-4-4s)z+3(h-4s)z^2.
\end{equation}
Since
\[
 b+az=(1-c)^2(z+c)+c^2(1-c-z)>0,
\]
we have $\partial_s p_8=-4(b+az+3z^2)<0$.  It is enough to prove
\begin{equation}\label{bh:q8}
 q(z):=p_8(z,1/4)=b(h-9)+a(h-5)z+3(h-1)z^2>0.
\end{equation}

First take one of the degree-eight cases in \eqref{bh:potentials},
so $h\geq10$. Interchanging $E_1,E_2$ preserves the problem, so we
may assume
$1/(h-2)\leq c\leq1/2$ and $a\geq0$.  For $z\geq0$,
\eqref{bh:q8} is positive.  On $[-c,0]$ its vertex is
$z_*=-a(h-5)/(6(h-1))$, which lies in this interval exactly when
\[
 c\geq c_*:=\frac{h-5}{8(h-2)}.
\]
If $c<c_*$, the minimum is attained at $-c$, where \(q(-c)=4c((h-1)c-1)>0\).
If $c\geq c_*$, the minimum is positive exactly when
\begin{equation}\label{bh:D8}
 D_8(c):=12(h-1)(h-9)c(1-c)-(h-5)^2(1-2c)^2>0.
\end{equation}
This expression increases on $0<c\leq1/2$.  For $h\geq13$,
\[
 D_8(c_*)=
 \frac{3(h-1)(h-5)(h^2-14h+21)}{4(h-2)^2}>0.
\]
For $h=12,11$, the bound $c\geq1/(h-2)>c_*$ gives, respectively,
\[
 D_8(1/10)=107/25>0,\qquad D_8(1/9)=52/27>0.
\]
For $h=10$ with $r_1,r_2\geq3$, one has $c\geq1/4>c_*$ and
$D_8(1/4)=14>0$. This proves positivity in every degree-eight
case of \eqref{bh:potentials}.

The same candidate does not give a uniform result for $h=10$.
For $c=1/8$, hence $a=3/4$ and $b=7/64$, one has
\begin{equation}\label{eq:quartic-degree-eight-obstruction}
 p_8(-1/16,1/4)=-5/256<0.
\end{equation}
By \cref{lem:quartic-table-geometry}, this point is attained in every
realization with $(r_1,r_2)=(2,8)$. Thus $TF$ has the wrong divergence
sign at points of the complement. For $h=9$, one has
$q(z)=4az+24z^2$ with $a\ne0$, so $q$ is negative for small $z$ of
the opposite sign. For $r_1=r_2=4$, one instead has
$q(z)=-b+21z^2$, negative near zero. The exact image assertion again
realizes these negative values with $F\ne0$. Hence, within the
theorem's range, the degree-eight choices in \eqref{bh:potentials}
are precisely the cases in which $TF$ has the required strict sign.

\subsubsection{The degree-ten potential in the remaining ranges}

For $G_{10}=RTF$, the same table gives
\begin{equation}\label{bh:G10}
 \cL(G_{10})=8R^5T^3F\,p_{10}(z,s),
\end{equation}
where
\begin{align}
 p_{10}(z,s)={}&b[h-8+(2h-23)s]
       +a[h-4+(2h-14)s]z\nonumber\\
       &+[3h+(15h-33)s+(18h+27)s^2]z^2\nonumber\\
       &+2(r_0-1)s\bigl[b+az+3(1+3s)z^2\bigr].
       \label{bh:p10}
\end{align}
Because $r_0\geq1$ and $b+az>0$ on the image interval, the last
line is nonnegative.  Thus increasing the unused block weight
$r_0$ can only improve this certificate.

For $h=9$, \cref{cor:quartic-nine-weights} gives $r_0\geq8$.
At $r_0=8$, the full polynomial is
\begin{equation}\label{eq:quartic-nine-certificate}
 p_{10}(z,s)=b(1+9s)+a(5+18s)z+
 (27+144s+315s^2)z^2.
\end{equation}
Since $c\in[1/7,6/7]$, one has $a^2\leq25/49$ and
$4b\geq24/49$. The negative discriminant is therefore bounded below by
\begin{align}
 &4b(1+9s)(27+144s+315s^2)-a^2(5+18s)^2\nonumber\\
 &\qquad\geq
 \frac{23+4788s+30564s^2+68040s^3}{49}>0.
 \label{eq:quartic-nine-discriminant}
\end{align}
The leading coefficient is positive, so \eqref{eq:quartic-nine-certificate}
is positive for all real $z$ and $0<s\leq1/4$. Increasing $r_0$
preserves positivity on the image interval.

When $h=10$, the first three terms are
\[
 b(2-3s)+a(6+6s)z+(30+117s+207s^2)z^2.
\]
Here $c\in[1/8,7/8]$, so $a^2\leq9/16$ and
$b=(1-a^2)/4$.  The negative of the discriminant is bounded below by
\begin{align*}
 &4b(2-3s)(30+117s+207s^2)-a^2(6+6s)^2\\
 &\hspace{10mm}\geq
 6+\frac{45}{2}s+\frac{117}{16}s^2-\frac{4347}{16}s^3
 \geq\frac{1797}{1024}>0.
\end{align*}
The final inequality follows from $0<s\leq1/4$ by discarding the
positive linear and quadratic terms. The leading coefficient is
positive, so this polynomial is positive for all real $z$ and all
allowed $s$.

Finally, for $r_1=r_2=4$ one has $h=8$, $a=0$, and $b=1/4$.
The full polynomial becomes
\begin{equation}\label{eq:quartic-balanced-eight-certificate}
 p_{10}(z,s)=\frac{2r_0-9}{4}s+
 \bigl[24+(81+6r_0)s+(153+18r_0)s^2\bigr]z^2.
\end{equation}
This is positive for every $s>0$ if $r_0\geq5$.
Together these calculations prove the required sign in every
case of \eqref{bh:potentials}.

\subsubsection{Global minimality and the Hardy bound}

\begin{proof}[Completion of the proof of \cref{thm:quartic-table-existence}]
\leavevmode\par\nobreak
Both multipliers are positive on $T>0$. Euler's identity excludes
critical points where $F\ne0$, and $|\nabla F|^2=4bRT$ excludes them
on $F=0$, $T>0$. On $Z$, both potentials and their gradients vanish.
Thus they have critical set $Z$ and the same phases as $F$.
The full-support and dimension conclusions of
\cref{lem:quartic-table-geometry} verify the hypotheses of
\cref{cor:gradient-subcalibration,thm:removable-subcalibration}.
Their comparison, followed by \cref{cor:polynomial-strict-minimality},
gives strict integral-current mass minimization and the gap
$\Theta\varepsilon^{N-1}$.

For stability, \eqref{eq:boundary-potential-algebraic} gives
\[
 \beta_{TF}=2(h-8)\frac{R^2}{T}-8,\qquad
 \beta_{RTF}=2(h-8)\frac{R^2}{T}+4\rho-50.
\]
Since $R^2/T\geq4$ and $h\geq8$, the homogeneous Hardy estimate
\cref{thm:ground-state-gap}, with cone dimension $N-1=2\rho-3$,
gives the same lower bound for both potential degrees:
\begin{align*}
 \left(8-\frac{N-1}{2}\right)^2+8(h-8)-8
 &=\frac{(N-17)^2}{4}+8h-72,\\
 \left(10-\frac{N-1}{2}\right)^2+8(h-8)+4\rho-50
 &=\frac{(N-17)^2}{4}+8h-72.
\end{align*}
For $h\geq10$ this coefficient is positive. For $h=9$,
\cref{cor:quartic-nine-weights} gives $N=2r_0+16\geq32$, so the
coefficient $(N-17)^2/4$ is positive as well. In the remaining
case, $h=8$ and $r_0\geq5$ imply $N\geq24$, so it is at least
$17/4$.  This proves \eqref{eq:bihomogeneous-Hardy}.
\end{proof}

For the splitting specialization $r_0=1$, $h=\ell$, one has
$N=2\ell$, and the Hardy coefficient becomes
$\ell^2-9\ell+1/4$. For $r_0>1$ and unequal active weights,
\cref{prop:quartic-weight-realizability} determines realizability,
and the inequalities above give a sufficient minimizing range.

\subsubsection{The balanced specialization}

\begin{corollary}[The degree-eight potential for balanced quartics]
\label{prop:balanced-quartic-multiplier}
Suppose that the three nonzero quartics in
\eqref{eq:orthogonal-quartic-data} have weights $(r_0,\nu,\nu)$,
with $\nu\geq5$. Set
\[
 F=\frac{E_+-E_-}{4},\qquad W=\frac{E_++E_-}{4},\qquad G=WF.
\]
The conclusions of \cref{thm:quartic-table-existence} hold for $G$,
with $\Sigma=\{W=0\}$. In this normalization, $\Delta F=0$,
$|F|\leq W\leq R^2/4$, and
\begin{equation}\label{eq:balanced-quartic-smallest-algebra}
 \cA_R(F)=\R[R,W,F]
\end{equation}
is a polynomial algebra with generator degrees $(2,4,4)$.
Moreover,
\begin{align}
 |\nabla G|^2&=4RW(W^2+3F^2),\qquad
 \Delta G=4(\nu+2)RF,\label{eq:balanced-quartic-gradient}\\
 \cL(G)&=16FW\left\{
 W^2\big[(\nu-4)R^2-2W\big]
 +3F^2\big[\nu R^2-2W\big]\right\}.
 \label{eq:balanced-quartic-certificate}
\end{align}
\end{corollary}

\begin{proof}
Take $T=W$, $F_{\mathrm{gen}}=F/2$, $h=2\nu\geq10$, and $c=1/2$
in \cref{thm:quartic-table-existence}. Its table becomes
\begin{equation}\label{eq:balanced-quartic-table}
 \Delta(R,W,F)=(2N,4\nu R,0),\qquad
 (\Gamma(I_i,I_j))_{I=(R,W,F)}=
 \begin{pmatrix}
 4R&8W&8F\\
 8W&4RW&4RF\\
 8F&4RF&4RW
 \end{pmatrix}.
\end{equation}
Substituting this normalization in
\eqref{bh:G8-basic}--\eqref{bh:G8-defect} gives
\eqref{eq:balanced-quartic-gradient}--\eqref{eq:balanced-quartic-certificate}.
\end{proof}

The following table organizes the applications of the common
quartic theorem. The splitting, Riedler, and weight-two families
have exact thresholds, obtained by combining this theorem with
instability arguments. Each construction also identifies its
singular set, including the real rank-one strata in the weight-two
families.

\begin{table}[htbp]
\centering\small
\setlength{\tabcolsep}{5pt}
\renewcommand{\arraystretch}{1.15}
\caption{Euclidean quartic realizations of
\cref{thm:quartic-table-existence}. Only parameters admitting the
specified Clifford realization are included. The intrinsic
Simons-cone result in \cref{thm:Hopf-Simons-section} uses a
separate differential table.}
\label{tab:quartic-weight-realizations}
\begin{tabular}{@{}lll@{}}
\toprule
Family & $(r_0,r_1,r_2)$ & Minimizing parameters used here\\
\midrule
OT--FKM splitting & $(1,m,\ell-m)$ & $\ell\ge10$, $2\le m\le\ell-2$\\
Riedler pullback & $(N/2-2k+1,k,k)$ & $N\ge32$, $k\ge4$\\
Hopf pairing & $(1,8,8)$ & Ambient dimension $32$\\
Balanced Clifford example & $(3,7,7)$ & Ambient dimension $32$\\
Additional block example & $(2,5,10)$ & Ambient dimension $32$\\
Weight-two families & $(2,m,4k-m-1)$ & $m=2,3$, $k\ge3$\\
Unequal Lawson pullback & $(22,5,6)$ & Ambient dimension $64$\\
\bottomrule
\end{tabular}
\end{table}

\subsection{The splitting, Hopf and Riedler families}
\label{sec:quartic-family-completions}

We apply the common theorem to the splitting construction, the Hopf
pairing and Riedler's pullbacks. The splitting and Riedler families
admit complete variational classifications. The balanced Hopf section
also illustrates intrinsic comparison inside a Simons cone, and the
degree-nine Riedler potential supplies an additional Hardy improvement.

\subsubsection{The OT--FKM splitting criterion}

For the data \eqref{bh:quartic-data}--\eqref{bh:centered}, the exact
minimizing threshold is as follows.

\begin{theorem}\label{bh:main}
The hypersurface $C=\{F=0\}\subset\R^{2\ell}$ is minimal on its
regular part, and
\[
 \Sing C=(\R^\ell\times\{0\})\cup(\{0\}\times\R^\ell).
\]
The oriented boundary $\partial\llbracket\{F<0\}\rrbracket$ is
area minimizing if and only if $\ell\geq10$. In all these minimizing
cases it is strictly area minimizing in the sense of
\cref{def:four-variational-notions}. The polynomial potentials and
strict-stability bound of \cref{thm:quartic-table-existence} apply,
with degrees ten for $\ell=10$ and eight for $\ell\geq12$.
For the remaining
admissible dimensions $\ell=4,6,8$, the regular part of $C$ is unstable.
\end{theorem}

\paragraph{Closure and singularities.}

For $x\ne0$, the vectors \(x,A_1x,\ldots,A_{m-1}x\)
are orthogonal and have length $|x|$. Consequently $0\leq Q\leq T$.
Direct differentiation gives, in the variables $I=(u,v,Q)$,
\begin{equation}\label{bh:original-table}
 (\Gamma(I_i,I_j))=
 \begin{pmatrix}4u&0&4Q\\0&4v&4Q\\4Q&4Q&4RQ\end{pmatrix},
 \qquad
 \Delta I=(2\ell,2\ell,2mR).
\end{equation}
For example,
$\nabla_yQ=2\sum_{i=0}^{m-1}\ip{A_i x}{y}A_i x$, with $A_0=\Id$,
so $\norm{\nabla_yQ}^2=4uQ$; differentiating in $x$ gives $4vQ$.
By \cref{prop:MR-generator-test},
\begin{equation}\label{bh:algebra}
 \cA=\R[u,v,Q]=\R[R,u-v,F]
\end{equation}
is a graded Laplacian algebra, with generator degrees $(2,2,4)$.
The determinant of its gradient matrix is $64RQ(T-Q)$, which is
positive at points with $u,v>0$ and $0<Q<T$. These points exist by
choosing the components of $y$ both in and perpendicular to the
span above. Thus the three generators are algebraically independent;
the spherical quotient dimension is two. The separately harmonic
normal form of \cref{prop:two-block-normal} is
$P=Q-(m/\ell)uv$; the minimal polynomial $F=Q-cuv$ differs from
it by a radial term in the two blocks.

With
$p_0=u-v$ and $p_{i+1}=2\langle A_i x,y\rangle$ for
$A_0=\Id$, the matrices in \cref{rem:split-composed-foliation}
form a symmetric Clifford system, and
\[
 F=\frac14\left(\sum_{i=1}^{m}p_i^2
                    +cp_0^2-cR^2\right).
\]
Its blocks $E_0=p_0^2$, $E_1=4Q$, and $E_2=4(T-Q)$
have weights $(1,m,\ell-m)$. The substitutions in
\eqref{bh:centered} therefore give \eqref{bh:reduced-table} with
$h=\ell$ and $N=2\ell$.
The calculation after \eqref{bh:reduced-table} recovers
$\cA_R(F)=\R[R,T,F]$, of degrees $(2,4,4)$. Independence follows
since $(u,v,Q)\mapsto(u+v,uv,Q-cuv)$ has Jacobian $u-v$.
Both algebras have spherical quotient dimension two, but
\eqref{bh:algebra} also retains the sign of $u-v$.

The geometry follows from \cref{lem:quartic-table-geometry}. For
$\ell\ge10$, evenness of $\ell$ and
\cref{thm:quartic-table-existence} give the variational conclusions.
The remaining dimensions are excluded by the following local
model and angular calculation.

Fix $x_0\ne0$, and write
$W_{x_0}=\operatorname{span}\{x_0,A_1x_0,\ldots,A_{m-1}x_0\}$.
The first nonzero term of $F$ at $(x_0,0)$ is
\begin{equation}\label{bh:transverse}
 \norm{x_0}^2\big((1-c)\norm{y_\parallel}^2
                         -c\norm{y_\perp}^2\big),
 \quad y_\parallel\in W_{x_0},\quad y_\perp\in W_{x_0}^\perp.
\end{equation}
The tangent cone is the product of $\R^\ell$ with this nonlinear
Lawson cone. Its transverse link is
\[
 \Sph^{m-1}(\sqrt c)\times
 \Sph^{\ell-m-1}(\sqrt{1-c})\subset\Sph^{\ell-1}.
\]
Both multiplicities are positive by \eqref{bh:range}. Smooth local
orthonormal frames for $W_x$ and $W_x^\perp$ identify a neighborhood
of the stratum with a smoothly varying family of these quadratic cones.
The same description holds on the other coordinate subspace. 

\begin{remark}[Clifford coordinates of the OT--FKM splitting]
\label{rem:split-composed-foliation}
Put $A_0=\Id$. The symmetric Clifford system
\[
 P_0=\begin{pmatrix}\Id&0\\0&-\Id\end{pmatrix},\qquad
 P_{i+1}=\begin{pmatrix}0&A_i^T\\A_i&0\end{pmatrix}
 \quad(0\leq i\leq m-1)
\]
has coordinates $p_0=u-v$ and $p_{i+1}=2\langle A_i x,y\rangle$.
The common geometric interpretation is given in
\cref{sec:Clifford-composition-background}.
\end{remark}

\subsubsection{The unstable OT--FKM splitting dimensions}

For $\ell=4,6$, the transverse Lawson cone in \eqref{bh:transverse}
is unstable: its link has dimension $\ell-2$ and constant squared
curvature $\ell-2$, giving the Hardy margin
\begin{equation}\label{bh:Lawson-margin}
 \frac{(\ell-3)^2}{4}-(\ell-2)<0.
\end{equation}
\Cref{lem:tangent-instability-transfer} transfers a negative
variation from this transverse cone to $C_{\rm reg}$, using the
local model \eqref{bh:transverse}.

For $\ell=8$, the transverse Lawson cone can be area minimizing;
instability follows instead from the angular calculation below.

\begin{lemma}\label{bh:curvature}
On the regular part of the quartic cone,
\begin{equation}\label{bh:AC}
 \norm{A_C}^2=(\ell-2)\frac{R}{T}+\frac{2(\ell-1)}R.
\end{equation}
On $\Lambda=C\cap\Sph^{2\ell-1}$, put $t=T|_\Lambda$. Then
\begin{gather}\label{bh:link-table}
 \norm{A_\Lambda}^2=\frac{\ell-2}{t}+2(\ell-1),\qquad
 \norm{\nabla_\Lambda t}^2=4t(1-4t),\\
 \Delta_\Lambda t=2(\ell-1)-8\ell t.\nonumber
\end{gather}
\end{lemma}
\begin{proof}
Put $W=\norm{\nabla F}^2$ and $\nu=\nabla F/\sqrt W$.
The Euclidean Bochner identity gives
$\norm{\nabla^2F}^2=\Delta W/2-\Gamma(F,\Delta F)$.
On $F=0$, \eqref{bh:reduced-table} yields
\begin{align*}
 W&=4bRT,\\
 \norm{\nabla^2F}^2&=4(a^2+b\ell)R^2+8b(\ell+4)T,\\
 \norm{\nabla\sqrt W}^2&=4[(a^2+b)R^2+5bT],\\
 \nabla^2F(\nu,\nu)&=\Delta F=2aR.
\end{align*}
Now use
\[
 \norm{A_C}^2=
 \frac{\norm{\nabla^2F}^2-2\norm{\nabla\sqrt W}^2
                     +(\nabla^2F(\nu,\nu))^2}{W}
\]
to obtain \eqref{bh:AC}. The normal Hessian of any polynomial $V$ can also be computed
from the closure table, using
\[
 \Hess V(\nu,\nu)=\frac{
 \Gamma(F,\Gamma(F,V))-\tfrac12\Gamma(V,W)}{W}.
\]
On $F=0$, this gives
\[
 \Gamma(T,F)=0,\qquad
 \Hess T(\nu,\nu)=2R-\frac{4T}{R},\qquad
 \Delta_C T=2(\ell-1)R+\frac{4T}{R}.
\]
Subtracting the radial contribution of the degree-four function $T$
on the $(2\ell-1)$-dimensional cone gives the last identity in
\eqref{bh:link-table}. Subtracting the squared radial derivative from
$\Gamma(T,T)=4RT$ gives the gradient identity.
\end{proof}

For $\ell=8$, the positive function $\phi=t^{-1}$ satisfies
\begin{equation}\label{bh:eigenfunction}
 (\Delta_\Lambda+\norm{A_\Lambda}^2)\phi=46\phi.
\end{equation}
To apply this identity, we cut off near the singular set of the link.
Its two singular strata have dimension seven and
transverse dimension seven in $\Lambda$. The local quadratic-cone
description after \eqref{bh:transverse} implies
$t\asymp\rho^2$ and transverse volume density comparable to
$\rho^6\,d\rho$, where $\rho$ is distance to the singular stratum.
Thus $\phi\asymp\rho^{-2}$, its gradient is $O(\rho^{-3})$, and
$\phi$, its gradient and the potential term have the required
finite integrals.

Choose cutoffs $\eta_\varepsilon$ vanishing for $\rho<\varepsilon$
and equal to one for $\rho>2\varepsilon$, with gradient bounded by
$O(\varepsilon^{-1})$. Then
\[
 \int_\Lambda\phi^2\norm{\nabla\eta_\varepsilon}^2
 =O(\varepsilon^{-4}\varepsilon^{-2}\varepsilon^7)
 =O(\varepsilon)\longrightarrow0.
\]
Multiplying \eqref{bh:eigenfunction} by
$\eta_\varepsilon^2\phi$ and integrating by parts on the regular
part proves that the angular Rayleigh quotients of
$\eta_\varepsilon\phi$ tend to $-46$. The cone has dimension fifteen,
so \eqref{eq:exact-Hardy-gap} gives compactly supported regular
variations whose scale-invariant quotients tend to
\begin{equation}\label{bh:negative-margin}
 \frac{(15-2)^2}{4}-46=-\frac{15}{4}<0.
\end{equation}
This proves instability. Together with \eqref{bh:Lawson-margin} and
the evenness of $\ell$, it completes Theorem~\ref{bh:main}.

\subsubsection{The Hopf pairing and its balanced section}
\label{sec:Hopf-pairing}

The Hopf pairing realizes weights $(1,8,8)$ and has two linear
singular strata. Its balanced section also gives an intrinsic
comparison inside a Simons cone.

Let \(h:\R^{16}\to\R^9\) be the quadratic extension of the standard
octonionic Hopf map \(\Sph^{15}\to\Sph^8\), normalized by
\begin{equation}\label{eq:Hopf-normalization}
 |h(x)|^2=|x|^4,\qquad
 Dh(x)Dh(x)^{\mathsf T}=4|x|^2\Id_9,\qquad \Delta h=0.
\end{equation}
For polarization of the spherical norm identity and the Hopf construction
from norm-preserving bilinear multiplication, see
\cite[Chapter~15]{Shapiro2000}.
The coordinates are \(h_i(x)=\langle A_ix,x\rangle\) for a nine-matrix Clifford
system on \(\R^{16}\); see \cite[Theorem~A and Section~1.2]{Radeschi14}.
On \(\R^{16}\oplus\R^{16}\), set
\begin{equation}\label{eq:Hopf-pairing-definitions}
 u=|x|^2,\quad v=|y|^2,\quad R=u+v,\quad V=u-v,\quad T=uv,
 \qquad F(x,y)=\langle h(x),h(y)\rangle.
\end{equation}
In particular, \(|F|\leq T\), and \(F\) has bidegree \((2,2)\).

\begin{proposition}[The pairing closure]\label{prop:Hopf-closure}
The algebra \(\cA=\R[u,v,F]=\R[R,V,F]\) is a graded Laplacian
algebra of spherical quotient dimension two. Its closure table, in
the ordered generators \((u,v,F)\), is
\begin{equation}\label{eq:Hopf-closure}
 (\Delta u,\Delta v,\Delta F)=(32,32,0),\qquad
 \bigl(\Gamma(I_i,I_j)\bigr)=
 \begin{pmatrix}
  4u&0&4F\\
  0&4v&4F\\
  4F&4F&4uv(u+v)
 \end{pmatrix}.
\end{equation}
The quartic zero set \(C=\{F=0\}\) is minimal on its regular part,
and its geometric singular set is
\begin{equation}\label{eq:Hopf-singular-set}
 Z=(\R^{16}\times\{0\})\cup(\{0\}\times\R^{16}).
\end{equation}
\end{proposition}

\begin{proof}
Equation \eqref{eq:Hopf-normalization} gives
\[
 |\nabla_xF|^2=4uv^2,\qquad
 |\nabla_yF|^2=4u^2v,\qquad \Delta F=0.
\]
Euler's identity in each block gives the remaining entries of
\eqref{eq:Hopf-closure}, and \cref{prop:MR-generator-test}
proves Laplacian closure. Surjectivity of the Hopf map gives
the full image
\[
 \{(u,v,f):u,v\geq0,\ |f|\leq uv\}.
\]
Its interior proves algebraic independence and spherical quotient
dimension two. Since $uv$ varies at fixed $R$ and $F=0$,
$\Gamma(F,F)=4Ruv\notin\R[R,F]$.

Since \(\Gamma(F,R)=8F\) and \(\Gamma(F,T)=4RF\), we obtain
\[
 \cL(F)=-\tfrac12\Gamma(F,4RT)=-8(R^2+2T)F.
\]
This proves regular minimality and identifies the critical set as
\(Z\). At \((0,y_0)\), \(y_0\ne0\), the transverse quadratic term
is \(\langle A(h(y_0))x,x\rangle\), where
\(A(a)=\sum_i a_iA_i\). Its eigenvalues are
\(\pm|y_0|^2\), each with multiplicity eight. Hence the tangent
cone is the product of a quadratic cone of signature \((8,8)\)
with \(\R^{16}\), and the point is geometrically singular.
The other axis is identical; the vertex is singular as well.
\end{proof}

\begin{remark}[A second realization of the pairing table]
The Clifford construction in \cref{sec:Clifford-composition-background}
also realizes this table. Put
\(Q=u^2+v^2-2F\). For a nine-matrix Clifford system
\(P_0,\ldots,P_8\) on \(\R^{32}\), the generators
\[
 \widehat R=|z|^2,\qquad
 \widehat V=\langle P_0z,z\rangle,\qquad
 \widehat Q=\sum_{i=0}^8\langle P_iz,z\rangle^2
\]
have the same gradient and Laplacian identities as \((R,V,Q)\).
Both images are
\(\{(R,V,Q):R\geq0,\ V^2\leq Q\leq R^2\}\).
The inverse relations
\(u=(R+V)/2\), \(v=(R-V)/2\), and \(F=(R^2+V^2-2Q)/4\)
therefore identify their graded differential algebras.
\end{remark}

\begin{theorem}[A polynomial subcalibration of the pairing cone]
\label{thm:Hopf-pairing}
The potential \(G=TF=|x|^2|y|^2F\) has degree eight and
subcalibrates \(C=\{F=0\}\). Its oriented boundary current
\(\partial\llbracket\{F<0\}\rrbracket\) is strictly area minimizing
in \(\R^{32}\). Moreover,
\begin{equation}\label{eq:Hopf-Hardy-bound}
 Q_C(\varphi)\geq\frac{449}{4}\int_C r^{-2}\varphi^2,
 \qquad \varphi\in C_c^\infty(C_{\reg}),\quad r=\sqrt R.
\end{equation}
\end{theorem}

\begin{proof}
The three blocks
\[
 E_0=(u-v)^2,\qquad E_1=2(T+F),\qquad E_2=2(T-F)
\]
have weights $(1,8,8)$. The centered generator $F/2$ gives
$h=16$, $r_0=1$, and $c=1/2$ in
\cref{thm:quartic-table-existence}. Its balanced specialization
\cref{prop:balanced-quartic-multiplier} gives $G=TF$ and
\begin{align}
 \Delta G&=40RF,\qquad
 |\nabla G|^2=4RT(T^2+3F^2),\label{eq:Hopf-gradient}\\
 \cL(G)&=32TF\bigl[T^2(2R^2-T)
                    +3F^2(4R^2-T)\bigr].\label{eq:Hopf-defect}
\end{align}
The Hardy constant in \eqref{eq:bihomogeneous-Hardy} is $449/4$
at $(N,h)=(32,16)$.
\end{proof}

\begin{remark}[Relation with the OT--FKM and Riedler constructions]
\label{rem:Hopf-FKM-Riedler}
The matrices \(B_i=\operatorname{diag}(A_i,-A_i)\) form a Clifford
system on \(\R^{32}\). Its OT--FKM polynomial
\cite[Section~1.3]{Radeschi14} is
\[
 \Psi=R^2-2\sum_{i=1}^9\langle B_i(x,y),(x,y)\rangle^2
     =4F-V^2.
\]
It has multiplicities \((8,7)\). Applying the centering formula
\eqref{eq:iso-F-z} to \(-\Psi\), with ordered multiplicities
\((7,8)\), shows that its minimal regular hypercone is
\(\{\Psi+R^2/15=0\}\). Nonvertex singularities distinguish $C$ from every cone over a
regular isoparametric hypersurface. They also distinguish it from
Riedler's family: its singular dimension is $N-2k$, so congruence
would require $N=32$, $2k=16$, but sixteen Clifford matrices
require ambient dimension at least $256$
\cite[Remark~5.2.6(1) and Table~4.1]{RiedlerThesis25}. Although
\(F=\langle\cdot,\cdot\rangle\circ(h,h)\), the product map
\(\Phi(x,y)=(h(x),h(y))\) satisfies
\[
 D\Phi D\Phi^{\mathsf T}=4\operatorname{diag}(u\Id_9,v\Id_9).
\]
It is not horizontally conformal on \(u\ne v\), so the pullback
criterion in \cref{prop:pullback-subcalibration} does not apply
to this map on \(\R^{32}\).
\end{remark}

\paragraph{The balanced section inside a Simons cone.}
\label{sec:Hopf-intrinsic-extension}

For the pairing polynomial above, the ambient space for comparison
is now itself a minimal cone. Let
\[
 \mathcal B=\{V=0\}=\{|x|=|y|\},\qquad
 D=\{V=F=0\},\qquad E_{\mathcal B}=\mathcal B\cap\{F<0\}.
\]
The Simons cone \(\mathcal B\) is a singular space of dimension
\(31\), smooth away from the origin. We use its induced metric
and the orientation determined by \(\nabla V\).
Intrinsic perimeter is variation in the smooth part, extended across
the vertex by the cutoff below. Write it as \(\Per_{\mathcal B}\), and set
\(D_a=\partial\llbracket E_{\mathcal B}\rrbracket
\mathbin{\llcorner}B_a(0)\).

\begin{theorem}[The pairing section of a Simons cone]
\label{thm:Hopf-Simons-section}
The cone \(D\) has dimension \(30\), is smooth away from the
origin, and is minimal in \(\R^{32}\). The intrinsic potential
\(G_{\mathcal B}=RF\) subcalibrates it in \(\mathcal B\).
For every \(a>0\), the integral current \(D_a\) minimizes mass
among integral \(30\)-currents supported in
\(\mathcal B\cap\overline B_a(0)\) with boundary \(\partial D_a\).
Thus \(D_a\) solves the Plateau problem in the singular ambient
space \(\mathcal B\), with its Euclidean induced area.
\end{theorem}

\begin{proof}
On \(\mathcal B\setminus\{0\}\), the gradients of \(V\) and
\(F\) are nonzero and orthogonal by \eqref{eq:Hopf-closure}.
Hence \(D\setminus\{0\}\) is smooth of codimension two.
The unit normal to \(\mathcal B\) is
\(n_{\mathcal B}=(x,-y)/\sqrt R\). The bidegree of \(F\) gives
\(\Hess F((x,-y),(x,-y))=-4F\). Since \(\mathcal B\) is minimal,
the induced operators satisfy
\begin{equation}\label{eq:Hopf-intrinsic-closure}
 \begin{gathered}
 \Delta_{\mathcal B}R=62,\qquad
 \Delta_{\mathcal B}F=4F/R,\\
 \Gamma_{\mathcal B}(R,R)=4R,\qquad
 \Gamma_{\mathcal B}(R,F)=8F,\qquad
 \Gamma_{\mathcal B}(F,F)=R^3.
 \end{gathered}
\end{equation}
Here \(\nabla F\) is tangent to \(\mathcal B\), and \(T=R^2/4\).
The same chain rule as in \cref{thm:finite-reduction} now yields
\begin{equation}\label{eq:Hopf-intrinsic-defect}
 \begin{gathered}
 \Delta_{\mathcal B}(RF)=82F,\qquad
 |\nabla_{\mathcal B}(RF)|^2=R^5+20RF^2,\\
 \cL_{\mathcal B}(RF)=16RF(2R^4+85F^2),
 \end{gathered}
\end{equation}
where \(\cL_{\mathcal B}\) is \eqref{eq:L-definition} for the
induced metric. The normalized gradient
\[
 X=\frac{\nabla_{\mathcal B}(RF)}{|\nabla_{\mathcal B}(RF)|}
\]
is smooth on \(\mathcal B\setminus\{0\}\), equals the exterior
normal along \(D\setminus\{0\}\), and its divergence has the
strict sign of \(F\) off \(D\).

We apply the proof of \cref{thm:removable-subcalibration} in the
smooth part of \(\mathcal B\). Only the vertex requires a cutoff.
For a radial cutoff \(\eta_\eps\), zero on \(B_\eps(0)\) and
one outside \(B_{2\eps}(0)\),
\[
 \int_{\mathcal B\cap B_{2\eps}(0)}
       |\nabla_{\mathcal B}\eta_\eps|\,d\cH^{31}
 =O(\eps^{30}).
\]
Hence extending a bounded characteristic function across the vertex
adds no perimeter mass. The estimate $|\diver_{\mathcal B}X|\le c/r$
from homogeneity makes the divergence error vanish as well. For every relatively compact open set \(U\subset\mathcal B\)
and every locally finite-perimeter set \(H\subset\mathcal B\) with
\(H\triangle E_{\mathcal B}\Subset U\),
\begin{equation}\label{eq:Hopf-intrinsic-comparison}
 \Per_{\mathcal B}(H;U)-\Per_{\mathcal B}(E_{\mathcal B};U)
 \geq\int_{H\triangle E_{\mathcal B}}
       |\diver_{\mathcal B}X|\,d\cH^{31}\geq0.
\end{equation}
The set \(E_{\mathcal B}\) has locally finite perimeter because
its boundary is a cone with a smooth compact link.

For the asserted integral-current comparison, cone a competitor
minus \(D_a\) to the origin inside \(\mathcal B\cap\overline B_a(0)\).
This gives an integral \(31\)-dimensional filling. On the oriented smooth part of \(\mathcal B\), write this filling as
\(\llbracket k\rrbracket_{\mathcal B}\), with integer-valued
locally BV multiplicity, and set
\(v=\one_{E_{\mathcal B}}+k\) and \(H=\{v\geq1\}\).
The integer-level decomposition used in
\cref{thm:strict-minimality-flux} gives
\(\lvert D\one_H\rvert\leq|Dv|\) on the smooth part.
Writing \(T\) for the competitor, for \(b>a\) we obtain
\[
 0\leq \Per_{\mathcal B}(H;B_b)-
          \Per_{\mathcal B}(E_{\mathcal B};B_b)
 \leq \mathbf M(T)-\mathbf M(D_a).
\]
The inequalities use \eqref{eq:Hopf-intrinsic-comparison} and
cancellation of the common exterior mass; $D$ has no mass on a sphere.
The representation and integer-level decomposition hold in smooth
charts. The vertex cutoff extends the characteristic function,
and no positive-dimensional integral current is supported at the
vertex. Thus the comparison includes competitors meeting it.

Finally, the divergence in \eqref{eq:Hopf-intrinsic-defect} vanishes
on \(D\), so \(D\) is minimal in \(\mathcal B\).
For its unit normal \(\nu\) within \(\mathcal B\), the two block
lengths satisfy \(|\nu_x|=|\nu_y|\), by the partial gradient
identities above and \(u=v\). Hence
\(\mathrm{II}_{\mathcal B}(\nu,\nu)=0\). Since
\(\operatorname{tr}_{T\mathcal B}\mathrm{II}_{\mathcal B}=0\),
its trace on \(TD\) is also zero. The second fundamental form
composition formula proves Euclidean minimality of \(D\setminus\{0\}\).
The smooth compact link gives
\(\int_{D\cap B_{2\eps}}|\nabla_D\eta_\eps|\,d\cH^{30}
=O(\eps^{29})\to0\), so the Euclidean first variation also vanishes
across the vertex.
\end{proof}

The Euclidean area-minimizing cones of codimension two in
\cite[Theorem~B]{CuiFKM25} are also area-minimizing hypersurfaces within
the Simons cones containing them. Indeed, let \(K\subset\mathcal S\)
be such a cone and put \(K_a=K\mathbin{\llcorner}B_a(0)\).
Every integral current \(T\) supported in
\(\mathcal S\cap\overline B_a(0)\) with
\(\partial T=\partial K_a\) is also an admissible Euclidean
competitor, so \(\mathbf M(K_a)\leq\mathbf M(T)\).
The induced area on \(\mathcal S\) is the Euclidean area.

\subsubsection{Riedler's Clifford--Simons quartics}

The construction and area-minimizing range
\(N\geq32\), \(k\geq4\) are due to Riedler
\cite[Section~5.2 and Theorem~5.2.5]{RiedlerThesis25}.
The common quartic theorem gives polynomial potentials; the
degree-nine potential below improves their Hardy bound. The
sixteen-dimensional endpoint is a splitting cone.

Let \(A_1,\ldots,A_{2k}\) be a symmetric Clifford system on
\(\R^N\), and write
\(P_i(x)=\langle A_i x,x\rangle\) and
\(Q=\sum_{i=1}^{2k}P_i^2\). Define
\[
 F=\sum_{i=1}^{k}P_i^2-\sum_{i=k+1}^{2k}P_i^2,
 \qquad C^4_{N,k}=\{x\in\R^N:F(x)=0\}.
\]
This particular Clifford realization requires
$\delta(2k-1)\mid N/2$ \cite[Section~3.5]{FKM81}.
Here $F=f\circ P$, where
$f(y)=\sum_{i\leq k}y_i^2-\sum_{i>k}y_i^2$, so the composition
is the pullback of a Simons cone. The unequal-target example
\cref{ex:lawson-pullback} extends this connection with the transfer
theorem of \cref{sec:quartic-application}.
At the comparison endpoint \(k=1\),
\(F=(P_1-P_2)(P_1+P_2)\) is reducible and harmonic;
\cref{prop:minimizing-irreducibility} excludes area minimization.
For the minimizing parameters below, the same proposition gives
irreducibility.
The chain rule gives the identities recorded in
\cite[Lemma~5.2.11]{RiedlerThesis25}:
\begin{equation}\label{eq:quartic-ambient-closure}
 \bigl(\Gamma(I_i,I_j)\bigr)_{I=(R,Q,F)}=
 \begin{pmatrix}
 4R&8Q&8F\\
 8Q&16RQ&16RF\\
 8F&16RF&16RQ
 \end{pmatrix},\qquad
 \Delta I=(2N,16kR,0).
\end{equation}
The determinant is \(1024R(R^2-Q)(Q^2-F^2)\), so the spherical
quotient has dimension two wherever \(0<Q<R^2\) and \(|F|<Q\).
For a concrete example, take the quadratic octonionic Hopf map
\(\widetilde P=(P_1,\ldots,P_9):\R^{16}\to\R^9\), normalized by
\(|\widetilde P|^2=R^2\), and retain its first eight components.
Then \(k=4\) and \(Q=R^2-P_9^2\). Choose an image with
\(P_9\neq0\) and nonzero components in both four-dimensional blocks.
This gives \(0<Q<R^2\) and \(|F|<Q\). Taking two copies of this
eight-matrix Clifford system gives \(N=32\), \(k=4\), in the
minimizing range below; the same point, placed in one copy, retains
the strict inequalities.
When \(Q\not\equiv R^2\), these generators correspond to
\eqref{eq:orthogonal-quartic-data} via
\[
 (E_0,E_1,E_2)=\left(R^2-Q,\frac{Q+F}{2},\frac{Q-F}{2}\right),
 \qquad (r_0,r_1,r_2)=\left(\frac N2-2k+1,k,k\right).
\]
For $N\ge32$, $k\ge4$, normalize $T=Q/4$, $F_{\rm gen}=F/8$.
Lemma~\ref{lem:quartic-table-geometry} gives regular minimality,
the singular set $\{Q=0\}$, and full support. If $k=4$, then
$r_0\ge9$ and \cref{thm:quartic-table-existence} gives $G=RQF$;
if $k\ge5$, \cref{prop:balanced-quartic-multiplier} gives $G=QF$.
Thus the boundary current is strictly area minimizing in this range,
with the Hardy bound \eqref{eq:bihomogeneous-Hardy}. Positive
constant rescalings do not affect normalized gradients.

The defining polynomial itself satisfies
\[
 |\nabla F|^2=16RQ,\qquad \Delta F=0,\qquad
 \cL(F)=-64F(Q+2R^2).
\]
The last identity follows from
$\cL(F)=-\Gamma(F,\Gamma(F,F))/2$ and
\eqref{eq:quartic-ambient-closure}.

\begin{example}[Riedler's subcalibration potential]
\label{ex:Riedler-correct-power}
For \(N\geq32\) and \(k\geq4\), Riedler's thesis uses
\cite[Section~5.2.2, Lemmas~5.2.12--5.2.13]{RiedlerThesis25}
\begin{equation}\label{eq:Riedler-correct-potential}
 G=RQ^{3/4}F=16^{-3/4}R^{1/4}S^{3/4}F,
 \qquad S=|\nabla F|^2=16RQ.
\end{equation}
This potential has degree nine. The parameters in
\cref{prop:harmonic-power-Jacobi} are
\((e,\sigma,\gamma)=(4,1/2,3/4)\).
Using \eqref{eq:quartic-ambient-closure}, we obtain the boundary Jacobi coefficient
\begin{equation}\label{eq:Riedler-positive-boundary}
 \beta_G=(2N-41)+(12k-47)\frac{R^2}{Q}>0.
\end{equation}
Since \(Q\leq R^2\), \cref{thm:ground-state-gap} gives
\[
 Q_C(u)\geq
 \left[\left(9-\frac{N-1}{2}\right)^2+2N+12k-88\right]
 \int_Cr^{-2}u^2.
\]
This lower bound exceeds the general quartic bound
\eqref{eq:bihomogeneous-Hardy}, with $h=2k$, by
\[
 \left[\frac{(N-19)^2}{4}+2N+12k-88\right]
 -\left[\frac{(N-17)^2}{4}+16k-72\right]
 =N+2-4k=2r_0>0.
\]
Thus Riedler's degree-nine potential gives the stronger Hardy bound.
\end{example}

\begin{corollary}[Instability of Riedler's sixteen-dimensional quartic]
\label{thm:Riedler-unstable}
The cone $C^4_{16,4}$ is unstable on its regular part and is not area
minimizing. More precisely, there are compactly supported variations on the regular
part with scale-invariant quotients of the second variation tending to
$-15/4$.
\end{corollary}

\begin{proof}
Write $F_R$ for the defining quartic of $C^4_{16,4}$ and let
$A_1,\ldots,A_8$ be its symmetric Clifford matrices. The volume
matrix $A_9=A_1\cdots A_8$ satisfies
\[
 A_9^T=A_9,\qquad A_9^2=I,\qquad A_9A_i=-A_iA_9.
\]
Its two eigenspaces have dimension eight, since $A_1$ exchanges
them isometrically. Choosing compatible orthonormal bases gives
coordinates $(x,y)\in\R^8\oplus\R^8$ in which
\[
 A_{i+1}=\begin{pmatrix}0&B_i\\B_i^T&0\end{pmatrix}
 \quad(0\leq i\leq7),\qquad B_0=I.
\]
The matrices $B_1,\ldots,B_7$ are skew-symmetric and satisfy
$B_iB_j+B_jB_i=-2\delta_{ij}I$. Thus the eight vectors $B_i y$
$(0\leq i\leq7)$ are pairwise orthogonal, all of length $|y|$, and
\begin{gather*}
 \sum_{i=0}^7\langle x,B_i y\rangle^2=|x|^2|y|^2,
 \qquad Q=4|x|^2|y|^2,\\
 \frac{F_R}{8}
 =\sum_{i=0}^3\langle x,B_i y\rangle^2
       -\frac12|x|^2|y|^2=F_{\rm split}.
\end{gather*}
Consequently $C^4_{16,4}$ is orthogonally congruent to a splitting
cone \eqref{bh:centered} with $(\ell,m)=(8,4)$.
Its Clifford representation is the balanced one: the symmetric
involution $B_1B_2B_3$ anticommutes with $B_4$, so its two
eigenspaces have dimension four.

In the original variables, $\Lambda=C^4_{16,4}\cap\Sph^{15}$ and
$q=Q|_\Lambda=4t$, so \eqref{bh:link-table} reads
\begin{equation}\label{eq:Riedler-16-link}
 |A_\Lambda|^2=\frac{24}{q}+14,\qquad
 |\nabla_\Lambda q|^2=16q(1-q),\qquad
 \Delta_\Lambda q=56-64q.
\end{equation}
The $\ell=8$ argument above transfers under this orthogonal
identification, and \eqref{bh:negative-margin} gives the asserted
quotients and instability.
\end{proof}

This is the case left undecided in
\cite[Remark~5.2.6(3)]{RiedlerThesis25}. Together with
Theorem~5.2.5 and the exclusions in Remark~5.2.6(1)--(2)
of that thesis, it gives
\[
 C^4_{N,k}\text{ is area minimizing}
 \quad\Longleftrightarrow\quad N\geq32\text{ and }k\geq4
\]
among the admissible Clifford--Simons quartics.

\subsection{Further quartic realizations}

The remaining realizations in \cref{tab:quartic-weight-realizations}
have different singular geometry. Their subcalibrations and Hardy
bounds follow from the same quartic theorem.

\subsubsection{Two infinite families with real rank-one singular sets}
\label{sec:orthq-families}

An unused block of weight two gives two explicit nonharmonic
families in $\R^{8k}$. Their singular set is the real rank-one
matrix cone, and their area-minimizing range has the sharp threshold
$k\ge3$.

Identify $\R^{4k}$ with $\mathbb H^k$ and let $J_1,J_2,J_3$ be
componentwise right multiplication by the quaternionic units
$\mathbf i,\mathbf j,\mathbf k$. For $x,y\in\mathbb H^k$, $k\ge2$, put
$R=|x|^2+|y|^2$ and $W=|x|^2|y|^2-\langle x,y\rangle^2$, and define
\begin{align}
 F_{k,2}&=(4k-3)
  (\langle x\mathbf i,y\rangle^2+\langle x\mathbf j,y\rangle^2)-W,
  \label{eq:orthq-two-families}\\
 F_{k,3}&=(4k-3)
  (\langle x\mathbf i,y\rangle^2+\langle x\mathbf j,y\rangle^2
                         +\langle x\mathbf k,y\rangle^2)-2W.
  \nonumber
\end{align}
Here the displayed pairings are real Euclidean inner products.
For $h=\sum_\alpha\bar x_\alpha y_\alpha$, one has
$\langle x\mathbf i,y\rangle=\Im_{\mathbf i}h$ and similarly for
the other units. Thus the second family is also the Gram composition
\begin{equation}\label{eq:orthq-quaternionic-Gram}
 F_{k,3}=(4k-3)|h|^2-(4k-5)(\Re h)^2-2|x|^2|y|^2.
\end{equation}

\begin{theorem}[Two families and their sharp threshold]
\label{thm:orthq-infinite-families}
For $k\ge2$ and $m=2,3$, the polynomial $F_{k,m}$ defines an
irreducible minimal hypercone $C_{k,m}\subset\R^{8k}$ with the
algebra and singular geometry in \cref{prop:orthq-construction}.
The cone is area minimizing if and only if $k\ge3$.
For $k\ge3$, the degree-eight potential
\begin{equation}\label{eq:orthq-potential}
 G_{k,m}=WF_{k,m}
\end{equation}
subcalibrates the full boundary current, which is strictly area
minimizing. The regular cone satisfies
\begin{equation}\label{eq:orthq-Hardy}
 Q_{C_{k,m}}(\varphi)\ge
 \left(16k^2-36k-\frac{31}{4}\right)
 \int_{C_{k,m}}r^{-2}\varphi^2.
\end{equation}
At $k=2$ both cones are unstable on their regular parts.
\end{theorem}

Both sequences arise from the following construction, which gives
their algebra, irreducibility, and transverse singular geometry.

Let $J_1,\ldots,J_m$ be skew-symmetric orthogonal endomorphisms of
$\R^\ell$ satisfying
\[
 J_iJ_j+J_jJ_i=-2\delta_{ij}\Id,
 \qquad m\ge2,\qquad\ell-m-1\ge2.
\]
For $x,y\in\R^\ell$, put
\begin{gather}
 R=|x|^2+|y|^2,\qquad
 W=|x|^2|y|^2-\langle x,y\rangle^2,\qquad
 Q=\sum_{i=1}^m\langle J_ix,y\rangle^2,
 \label{eq:orthq-data}\\
 \begin{gathered}
 a=\ell-m-2,\qquad b=m-1,\qquad d=a-b=\ell-2m-1,\\
 F=(\ell-3)Q-bW.
 \end{gathered}
 \label{eq:orthq-polynomial}
\end{gather}
The vectors $x,J_1x,\ldots,J_mx$ are mutually orthogonal and have
length $|x|$, so $0\le Q\le W\le R^2/4$.

\begin{proposition}[Construction and exact algebra]
\label{prop:orthq-construction}
The polynomial $F$ is real irreducible and nonharmonic, has nonempty
regular zero set, and has essential space $\R^{2\ell}$.
Its regular zero set is minimal, and
\begin{equation}\label{eq:orthq-minimality}
 |\nabla F|^2=4R(dF+abW),\qquad
 \cL(F)=-8F\bigl(2dF+2abW+abR^2\bigr).
\end{equation}
Moreover,
\[
 \cA_R(F)=\R[R,W,F]
\]
is a polynomial algebra on three generators of degrees $(2,4,4)$.
Its critical set is the geometric singular set of $C=\{F=0\}$:
\begin{equation}\label{eq:orthq-singular}
 \Sing C=\{W=0\}
 =\{(x,y):\operatorname{rank}_{\R}[x\ y]\le1\}.
\end{equation}
This set has dimension $\ell+1$ and is smooth away from the origin.
At each nonzero singular point the transverse tangent cone is the
Lawson quadratic cone of signature $(m,\ell-m-1)$.
\end{proposition}

\begin{proof}
The symmetric Clifford system associated with
$\Id,J_1,\ldots,J_m$ has quadratic coordinates
$|x|^2-|y|^2$, $2\langle x,y\rangle$, and
$2\langle J_ix,y\rangle$; see \cite[Section~3]{FKM81}.
The blocks
\begin{equation}\label{eq:orthq-blocks}
 E_0=R^2-4W,\qquad E_1=4Q,\qquad E_2=4(W-Q)
\end{equation}
satisfy \eqref{eq:orthogonal-quartic-data} with weights
$(2,m,\ell-m-1)$ by \cref{lem:Clifford-block-quartics}. In the normalization of
\eqref{eq:quartic-general-generators}, $T=W$ and
$F_{\rm gen}=F/(\ell-3)$. Thus \eqref{bh:reduced-table} gives
\begin{equation}\label{eq:orthq-table}
 \begin{gathered}
 (\Delta R,\Delta W,\Delta F)
 =(4\ell,2(\ell-1)R,2dR),\\
 (\Gamma(I_i,I_j))_{I=(R,W,F)}=
 \begin{pmatrix}
 4R&8W&8F\\8W&4RW&4RF\\8F&4RF&4R(dF+abW)
 \end{pmatrix}.
 \end{gathered}
\end{equation}
Equation \eqref{eq:orthq-minimality} follows from
\eqref{bh:minimality}. The smallest-algebra and independence assertions
follow from the argument after \eqref{bh:reduced-table}.
In this case $R$ is already obtained from $F$ by one Laplacian:
$\ell$ is even because $J_1^2=-\Id$, whereas $2m+1$ is odd, so
$d\ne0$ and $R=\Delta F/(2d)$. In particular $F$ is nonharmonic.

For the remaining assertions, fix $x\ne0$ and decompose
\[
 \R^\ell=\R x\oplus U_x^+\oplus U_x^-,\qquad
 U_x^+=\operatorname{span}\{J_1x,\ldots,J_mx\}.
\]
For $y=y_0+y_++y_-$, the exact fibre formulas are
\begin{equation}\label{eq:orthq-fibre}
 W=|x|^2(|y_+|^2+|y_-|^2),\qquad
 F=|x|^2(a|y_+|^2-b|y_-|^2).
\end{equation}
They exhibit regular zeros and both phases. Writing
$F=y^{\mathsf T}M(x)y$, the matrix $M(x)$ has rank $\ell-1\ge4$
for every real $x\ne0$, with kernel $\R x$. It has the same rank
over $\R(x)$, so this quadratic in $y$ cannot factor into linear
forms over that field. Its coefficients, which are homogeneous
quadratics in $x$, have no nonconstant common divisor. A linear
common divisor would give $M(x)=0$ at a nonzero point of a real
hyperplane. A quadratic common divisor would give $M(x)=p(x)B$
for a constant matrix $B$, and $M(x)x=0$ would force $B=0$.
Both contradict the rank assertion. Gauss's lemma therefore proves
real irreducibility.

A constant annihilating direction $(v_1,v_2)$ would give two
vanishing derivatives of distinct bidegrees $(1,2)$ and $(2,1)$.
The second implies $M(x)v_2=0$ for every $x\ne0$, hence $v_2=0$;
interchanging $x,y$ gives $v_1=0$. Thus the ambient space is essential.
Finally $dF+abW=a^2Q+b^2(W-Q)$, so \eqref{eq:orthq-minimality}
identifies the critical set with $W=0$. Equality in real
Cauchy--Schwarz gives \eqref{eq:orthq-singular}. The charts $y=tx$
with $x\ne0$, and their counterparts with $y\ne0$, give its smoothness
and dimension. Smooth local frames in \eqref{eq:orthq-fibre} give
the transverse quadratic cone. At $(x,tx)$ an isometric normal
coordinate is $w=\Pi_{x^\perp}(\delta y-t\delta x)/\sqrt{1+t^2}$; the
nonzero Hessian eigenvalues are $2aR$ and $-2bR$, with multiplicities
$m$ and $\ell-m-1$. Their sum is at least four, so the zero set
cannot be smooth there. This also proves that these critical points
are geometric singularities, as in \cref{lem:quartic-table-geometry}.
\end{proof}

\begin{proof}[Proof of \cref{thm:orthq-infinite-families}]
The weights are $(2,m,4k-m-1)$, and their active sum is $h=4k-1$.
For $k\ge3$, this is at least eleven. The normalization
$T=W$, $F_{\rm gen}=F_{k,m}/(4k-3)$ therefore gives
\eqref{eq:orthq-potential} from \eqref{bh:potentials}, with $W>0$
off the rank-one singular set. The strict-minimality and full-support
conclusions follow from \cref{thm:quartic-table-existence}. In this normalization,
\begin{equation}\label{eq:orthq-potential-defect}
 \begin{split}
 \cL(WF)=8WF\bigl\{&[3(\ell-1)R^2-12W]F^2\\
 &+d[(\ell-5)R^2-4W]WF
 +ab[(\ell-9)R^2-4W]W^2\bigr\},
 \end{split}
\end{equation}
which is \eqref{bh:G8-defect} after the same constant rescaling.
The Hardy coefficient in \eqref{eq:bihomogeneous-Hardy}, with
$N=8k$ and $h=4k-1$, is exactly
\[
 \frac{(8k-17)^2}{4}+8(4k-1)-72
 =16k^2-36k-\frac{31}{4}>0.
\]

At $k=2$, the transverse tangent cones from
\cref{prop:orthq-construction} have signatures $(2,5)$ and $(3,4)$
in $\R^7$. Each is a six-dimensional Lawson cone whose link has
$|A|^2=5$, so its radial Hardy margin is
$(6-2)^2/4-5=-1$. \Cref{lem:tangent-instability-transfer},
applied to the flat $\R^9$
factor and the smooth local model \eqref{eq:orthq-fibre}, transfers
a compactly supported negative variation to the regular cone.
Thus both $k=2$ cones are unstable and not area minimizing.
\end{proof}

\begin{proposition}[Distinction from the preceding families]
\label{prop:orthq-distinct}
For $k\ge3$, the cones in \eqref{eq:orthq-two-families} are not
orthogonally congruent to the standard OT--FKM splitting cones or
to pure full Clifford pullbacks of nondegenerate Lawson quadratic
cones. For each fixed $k$, the two cones are not congruent to one another.
\end{proposition}

\begin{proof}
The splitting cones in $\R^{2\ell}$ have singular set the two
coordinate $\ell$-planes, whereas \eqref{eq:orthq-singular} has
dimension $\ell+1$. For a full Clifford quadratic map
$P:\R^{2\ell}\to\R^p$, the identity
$DP\,DP^{\mathsf T}=4R\Id_p$ shows that the nonvertex singular
set of a nondegenerate Lawson pullback is $P^{-1}(0)$, of codimension
$p$ if nonempty. For the balanced target this is also recorded in
\cite[Remark~5.2.6(1)]{RiedlerThesis25}. Congruence would force $p=\ell-1$.
The Clifford module restriction
\cite[Section~3.5]{FKM81}, equivalently
\cite[Table~4.1]{RiedlerThesis25}, would then require \(\delta(\ell-2)=2^{2k-1}\mid4k=\ell\),
which is impossible for $k\ge3$ because $2^{2k-1}>4k$.
Finally the transverse signatures $(2,4k-3)$ and $(3,4k-4)$
are distinct even after interchanging their two entries, and hence
distinguish the two sequences.
\end{proof}

\begin{remark}[The weights $(2,5,10)$]
\label{ex:quartic-2510}
Taking $\ell=16$ and $m=5$ in \cref{prop:orthq-construction}
gives the weights $(2,5,10)$ and $F=13Q-4W$ on $\R^{32}$.
The required Clifford system exists since $\delta(6)=8\mid16$.
Here $h=15$, so \cref{thm:quartic-table-existence} gives the
potential $WF$, strict area minimization, and strict stability.
\end{remark}

\subsubsection{A complex rank-one singular set: weights \texorpdfstring{$(3,7,7)$}{(3,7,7)}}
\label{sec:weighted-Clifford-quartic}

The balanced criterion \cref{prop:balanced-quartic-multiplier} gives
a harmonic example with a complex rank-one singular set.
Let $E_i=L_{e_i}$, $1\leq i\leq7$, be left multiplication by an
orthonormal imaginary basis of the octonions.  Alternativity gives
$E_i^T=-E_i$ and $E_iE_j+E_jE_i=-2\delta_{ij}\Id_8$.
On $\R^{16}=\R^8\oplus\R^8$ put
\begin{equation}\label{eq:weighted-Clifford-matrices}
 J=\begin{pmatrix}0&\Id_8\\-\Id_8&0\end{pmatrix},\qquad
 B_i=\begin{pmatrix}E_i&0\\0&-E_i\end{pmatrix}.
\end{equation}
These eight skew-symmetric matrices satisfy the negative Clifford
relations.  For $x,y\in\R^{16}$ define
\begin{align}
 R&=|x|^2+|y|^2,\nonumber\\
 W&=|x|^2|y|^2-\langle x,y\rangle^2-\langle Jx,y\rangle^2,
 \label{eq:weighted-Clifford-definition}\\
 F&=2\sum_{i=1}^7\langle B_ix,y\rangle^2-W.\nonumber
\end{align}

\begin{proposition}[The complex rank-one example]
\label{thm:weighted-Clifford-quartic}
The polynomial $F$ in \eqref{eq:weighted-Clifford-definition} is
real irreducible and harmonic. Its cone $C=\{F=0\}\subset\R^{32}$ has
geometric singular set
\begin{equation}\label{eq:weighted-Clifford-singular-set}
 \Sing C=\Sigma=\{W=0\}
 \cong\{(z,w)\in\mathbb C^8\times\mathbb C^8:
             \operatorname{rank}_{\mathbb C}[z\ w]\leq1\}.
\end{equation}
This set has real dimension $18$ and is smooth away from the origin.
At each nonzero singular point the tangent cone is orthogonally
congruent to $\R^{18}\times\{|\xi|^2=|\eta|^2\}$, with
$\xi,\eta\in\R^7$. The potential $G=WF$ subcalibrates the full
boundary current, which is strictly area minimizing. Its smallest
Laplacian algebra is $\R[R,W,F]$, and \eqref{eq:bihomogeneous-Hardy} gives
\begin{equation}\label{cor:weighted-Clifford-Hardy}
 Q_C(\varphi)\geq\frac{385}{4}\int_C r^{-2}\varphi^2,
 \qquad \varphi\in C_c^\infty(C_{\reg}\setminus\{0\}).
\end{equation}
\end{proposition}

\begin{proof}
The Clifford system associated with $\Id_{16},J,B_1,\ldots,B_7$
has quadratic coordinates $p_0=|x|^2-|y|^2$ and
$p_A=2\langle Ax,y\rangle$. Apply \cref{lem:Clifford-block-quartics} with
\begin{align}
 E_0&=p_0^2+p_{\Id}^2+p_J^2=R^2-4W,\nonumber\\
 E_+&=\sum_{i=1}^7p_{B_i}^2=2(W+F),\label{eq:weighted-Clifford-blocks}\\
 E_-&=R^2-E_0-E_+=2(W-F).\nonumber
\end{align}
Here $\ell=16$, $m=9$, and the weights are $(3,7,7)$, so
\cref{prop:balanced-quartic-multiplier} gives harmonicity, the smallest
algebra, and the subcalibration. Irreducibility follows from
\cref{prop:minimizing-irreducibility}; the Hardy coefficient uses
$(N,h)=(32,14)$.

Identify $(\R^{16},J)$ with $\mathbb C^8$.  The expression for
$W$ is the complex Gram determinant of the two columns $x,y$.
Thus \cref{lem:quartic-table-geometry} and equality in Cauchy--Schwarz
give \eqref{eq:weighted-Clifford-singular-set}.
The rank-one matrices form a smooth complex manifold of dimension
$8+2-1=9$, so $\dim_{\R}\Sigma=18$ and
$\Sigma\setminus\{0\}$ is smooth.

For $x\ne0$, the vectors $x,Jx,B_1x,\ldots,B_7x$ are orthogonal
with length $|x|$. Decompose $y=y_0+y_++y_-$ along
$K_x=\operatorname{span}\{x,Jx\}$,
$U_x^+=\operatorname{span}\{B_1x,\ldots,B_7x\}$, and
$U_x^-=(K_x\oplus U_x^+)^\perp$. The last two spaces have dimension seven, and
\begin{equation}\label{eq:weighted-Clifford-transverse-form}
 W=|x|^2(|y_+|^2+|y_-|^2),\qquad
 F=|x|^2(|y_+|^2-|y_-|^2).
\end{equation}
Local frames give transverse signature $(7,7)$; interchanging $x,y$
covers $\Sigma\setminus\{0\}$. The Hessian formulas in
\cref{lem:quartic-table-geometry}, with $T=W$ and $F_{\rm gen}=F/2$,
give eigenvalues $\pm2R(p)$, each sevenfold, and kernel dimension
eighteen; \cref{sec:tangent-instability-transfer} gives the asserted
orthogonal tangent-cone model.
\end{proof}

The singular dimension $18$ distinguishes this cone from the
splitting and Hopf pairing cones in $\R^{32}$, whose singular sets
have dimension $16$. For a pure full Clifford pullback of a
nondegenerate Lawson cone, \eqref{eq:Clifford-pullback-identities}
makes the nonvertex singular set the zero fiber of the Clifford map.
Congruence would therefore require fourteen quadratic coordinates.
Such a system on $\R^{32}$ would require $\delta(13)=128\mid16$,
contrary to the module restriction \cite[Section~3.5]{FKM81}.
In particular, the example is distinct from Riedler's family.

\subsubsection{An unequal Lawson pullback}

\begin{example}[An unequal quadratic target]
\label{ex:lawson-pullback}
A Lawson cone in $\R^5\oplus\R^6$ pulls back to a quartic in
$\R^{64}$, illustrating \cref{prop:pullback-subcalibration}.
On the target, put
\[
 T=|u|^2+|v|^2,\qquad f=5|u|^2-4|v|^2,\qquad H=Tf.
\]
The target potential satisfies \eqref{eq:pullback-target-sign} with
$\ell=4$ and $\kappa=1/2$, strictly off $f=0$; the calculation is
recorded in \cref{app:application-details}.

Choose eleven Clifford matrices on \(\R^{64}\), and set
\(F=f\circ P\), \(Q=T\circ P=|P|^2\), and
\(R=|x|^2\) on the source. Thus \(G=H\circ P=QF\).
The source blocks
\[
 E_0=R^2-Q,\qquad E_1=\sum_{i=1}^{5}P_i^2,\qquad
 E_2=\sum_{i=6}^{11}P_i^2
\]
have weights $(22,5,6)$. Such matrices exist by
\cite[Table~4.1]{RiedlerThesis25}. With $T_{\rm gen}=Q/4$,
$F_{\rm gen}=F/36$, \cref{thm:quartic-table-existence} gives the
same potential $QF$, critical set $P^{-1}(0)$, full support, and
strict area minimality. Here $Q>0$ off $P^{-1}(0)$, so the
multiplier has the required positive domain. The Hardy coefficient
in \eqref{eq:bihomogeneous-Hardy} is $2273/4$.
\end{example}

\subsection{An example from complex Gram coordinates}
\label{sec:gram-composition-constructions}

The following sextic is obtained from the nine quadratic coordinates
of a complex Hermitian Gram matrix. Its profile uses the imaginary
part of a cyclic product of three off-diagonal entries. It shows
that the composition criterion also produces higher-degree minimal
regular zero sets.

\begin{example}[A complex Gram sextic]
\label{thm:complex-gram-cycle-sextic}
For every $k\ge2$, let $x_1,x_2,x_3\in\mathbb C^k$ and
$h_{ij}=x_i^*x_j$.  Then
\begin{equation}\label{eq:complex-gram-cycle-sextic}
 F=\operatorname{Im}(h_{12}h_{23}h_{31})
\end{equation}
is a real harmonic homogeneous polynomial of degree six on $\R^{6k}$
with nonempty minimal regular zero set.

More explicitly, set $r_i=h_{ii}$, $R=\sum_i r_i$,
\[
 a_1=|h_{23}|^2,\qquad a_2=|h_{31}|^2,\qquad
 a_3=|h_{12}|^2,\qquad
 t=\operatorname{Re}(h_{12}h_{23}h_{31}).
\]
With $\{i,j,l\}=\{1,2,3\}$ in each summand, define
\begin{align}
 S&=\sum_{i=1}^3 r_i a_i(a_j+a_l)-2t\sum_{i=1}^3a_i,
 \label{eq:complex-gram-cycle-gradient}\\
 B&=\sum_{i=1}^3a_i(r_i+r_j)(r_i+r_l)
       -\sum_{i=1}^3a_i^2-2Rt.
 \label{eq:complex-gram-cycle-weight}
\end{align}
Then
\begin{equation}\label{eq:complex-gram-cycle-minimality}
 \Delta F=0,\qquad |\nabla F|^2=S,\qquad
 \Gamma(F,S)=4BF,\qquad \cL(F)=-2BF.
\end{equation}
\end{example}

\begin{proof}
Use the complex bilinear extension of $\Gamma$ in
\eqref{eq:complex-Gram-differential-table}.  For
$Z=h_{12}h_{23}h_{31}$ and $A=a_1+a_2+a_3$, the product rule gives
\[
 \Gamma(Z,Z)=4AZ,\qquad
 \Gamma(Z,\overline Z)=2\sum_i r_i a_i(a_j+a_l).
\]
Taking real and imaginary parts proves
$|\nabla F|^2=S$.  The same table gives
\begin{equation}\label{eq:complex-gram-cycle-pairings}
 \Gamma(F,r_i)=4F,\quad
 \Gamma(F,a_i)=2(R-r_i)F,\quad
 \Gamma(F,t)=2AF.
\end{equation}
Applying these three identities to the expression for $S$ gives
$\Gamma(F,S)=4BF$. Indeed, writing
$D=\sum_i r_i a_i$ and $E=\sum_i r_i a_i^2$, one has
$S=AD-E-2tA$. After the $A^2$ terms cancel, the product rule gives
\[
 \Gamma(F,S)=4F\left[
 RD+\frac A2\sum_i r_i(R-r_i)-\sum_i a_i^2
 -\sum_i r_i a_i(R-r_i)-2Rt\right].
\]
The coefficient of $a_i$ in the brackets is
$r_i^2+r_1r_2+r_2r_3+r_3r_1=(r_i+r_j)(r_i+r_l)$,
so the brackets equal $B$. This is a polynomial identity, valid
also when the Gram entries satisfy rank relations, as for $k=2$.
For harmonicity, with $x_2,x_3$ fixed, $F$ is the Hermitian
quadratic form in $x_1$ with matrix
\[
 \frac{h_{23}x_2x_3^*-\overline{h_{23}}x_3x_2^*}{2i}.
\]
Its trace is zero, so its block Laplacian vanishes.  The same
argument applies to the other two blocks.  The point
$x_1=x_3=e_1$, $x_2=e_1+e_2$ has $F=0$ and $S=2$, proving
nonempty regularity; \eqref{eq:complex-gram-cycle-minimality}
then proves minimality.

\end{proof}

\subsection{Higher bidegrees and transverse instability}

The two-block equations also admit higher bidegrees. In the following
octonionic examples, an unstable transverse tangent cone excludes
area minimization.

Identify $\R^8$ with $\mathbb O$ and set
\[
 \mu(x,y)=x\overline y,\qquad x,y\in\mathbb O.
\]
Then $\norm{\mu}^2=T$ and
\begin{equation}\label{bh:mu}
 D\mu\,D\mu^{\mathsf T}=R\Id_8,\qquad
 \Delta\mu=0,\qquad D\mu(\nabla R)=4\mu.
\end{equation}
These follow because multiplication by a fixed nonzero octonion is
a linear similarity, with its squared norm as dilation squared;
see \cite[Chapter~1, Lemma~1.3(1) and Appendix, pp.~15--16 and 21--22]{Shapiro2000}
for the norm-multiplicative algebra and similarity identities.

We use T.~Wang's octonionic pullback construction
\cite[Section~3.4]{WangTeng26}. For $g=3,4,6$, choose the normalized
Cartan--M\"unzner polynomial $\Phi_g$ on $\R^8$ with multiplicities
respectively $(2,2)$, $(1,2)$, and $(1,1)$. Set
\[
 P_g=\Phi_g\circ\mu,\qquad F_g=P_g-t_gT^{g/2},\qquad
 (\kappa_g,t_g)=(0,0),(8,1/3),(0,0),
\]
in the same order, where $\Delta\Phi_g=\kappa_g|z|^{g-2}$ and
terms with zero coefficient are omitted. The value $t_g$ is the
minimal level from \cref{sec:low-rank}. These polynomials have
bidegree $(g,g)$.

\begin{proposition}\label{bh:octonionic}
Each $\R[u,v,P_g]$ is a graded Laplacian algebra with three
algebraically independent generators. Each $C_g=\{F_g=0\}\subset\R^{16}$
is a minimal hypercone, with singular set equal to the two coordinate
$\R^8$ subspaces. For $g=3,4,6$, these hypercones are unstable.
\end{proposition}
\begin{proof}
The closure identities are
\begin{gather}\label{bh:higher-table}
 \Delta u=\Delta v=16,\quad\Gamma(u,v)=0,\quad
 \Gamma(u,u)=4u,\quad\Gamma(v,v)=4v,\\
 \Gamma(u,P_g)=\Gamma(v,P_g)=2gP_g,\quad
 \Gamma(P_g,P_g)=g^2RT^{g-1},\nonumber\\
 \Delta P_g=\kappa_gRT^{g/2-1}.\nonumber
\end{gather}
The last expression is polynomial (zero for $g=3$).
The gradient determinant $16g^2R(T^g-P_g^2)$ is positive at
intermediate levels. Thus \cref{prop:MR-generator-test} gives
closure and \eqref{eq:background-quotient-Gram} gives independence
and spherical quotient dimension two.

Put $H_g(z)=\Phi_g(z)-t_g\norm{z}^g$, with the second term zero when
$t_g=0$. This is a homogeneous polynomial in every case under
consideration, and its zero set is the minimal isoparametric
hypercone in $\R^8$. A useful consequence of \eqref{bh:mu} is, for any
homogeneous $H$ of degree $h$,
\begin{equation}\label{bh:mu-pullback}
 \cL(H\circ\mu)=R^2(\cL(H))\circ\mu
                    -2h(H\norm{\nabla H}^2)\circ\mu.
\end{equation}
Indeed $\norm{\nabla(H\circ\mu)}^2=R(\norm{\nabla H}^2)\circ\mu$,
$\Delta(H\circ\mu)=R(\Delta H)\circ\mu$, and
$\Gamma(H\circ\mu,R)=4h(H\circ\mu)$. Substitution proves
\eqref{bh:mu-pullback}. Both terms vanish on $H_g\circ\mu=0$, proving
minimality. Off the coordinate subspaces $\mu\ne0$, the target
minimal level is regular, and \eqref{bh:mu} gives a nonzero source
gradient. Along either coordinate subspace that gradient vanishes.

At $(x_0,0)$ with $x_0\ne0$, the first nonzero term is
$H_g(x_0\overline y)$, so the tangent cone is a product of $\R^8$
with an orthogonal image of the target minimal cone. The target
link has dimension six and constant squared curvature $6(g-1)$.
Its classical Hardy margin is
\[
 D_g=\frac{25}{4}-6(g-1)=
 \begin{cases}-23/4,&g=3,\\-47/4,&g=4,\\-95/4,&g=6.
 \end{cases}
\]
Each is negative, so \cref{lem:tangent-instability-transfer}
proves instability of $C_g$. The target cone
is nonlinear, so these coordinate strata are genuine singularities.
\end{proof}

\begin{remark}[The quadratic target and the sign obstruction]
For a quadratic target with both multiplicities positive, let
$Q=\norm{\operatorname{pr}_W\mu}^2$, $m=\dim W$, and
$c=(m-1)/6$. This $Q$ obeys \eqref{bh:original-table} with $\ell=8$.
The curvature and cutoff proof above applies without a choice of
Clifford splitting, so the full $(2,2)$ hypercone $Q=cT$ is also
unstable. Thus all four bidegrees in this $\R^8\times\R^8$ construction
give unstable hypercones. The minimizing question for a codimension-two balance section is separate.
\end{remark}

\section[Metrized Cubic Algebras]{Metrized Cubic Algebras and Differential Closure}
\label{sec:cubic-algebraic-theory}

For an essential irreducible harmonic cubic with minimal regular zero
set, the smallest Laplacian algebra has at most four generators in the
radial case and exactly five in the nonradial case
(\cref{thm:harmonic-cubic-five-bound}). Radial transport identities
and the Clifford table prepare \cref{sec:cubic-potentials}.
The nonradial weight operator gives an intrinsic graded splitting,
five differential invariants, and a reduction to Cartan cubics in
dimensions $8,14,26$ (\cref{cor:nr-cartan-dimensions}); compatible
splittings still require classification. For the five-generator proof,
pass from \cref{subsec:cubic-multiplication} to
\cref{sec:nonradial-cubics,subsec:nr-recovery,subsec:nr-five-generators};
coordinate models appear in \cref{sec:nonradial-examples}.

\subsection[Hsiang's problem and radial cubic classification]
{Hsiang's radial cubic problem and the classification work of
Peng--Xiao and Tkachev}
\label{sec:cubic-history}

Within his broader study of algebraic minimal cones, Hsiang
\cite{Hsiang67} singled out the classification of cubic forms satisfying
\[
 \cL(F)=-\gamma RF,\qquad R=|x|^2,
\]
with a constant $\gamma$. These are now called \emph{radial
eigencubics}; radiality refers to the quadratic weight in this equation,
not to radial dependence of $F$. This is the radial cubic classification
problem considered here; see \cite[Section~1.1]{TkachevClassification10}
and \cite[Section~1]{FoxTkachev24}.
Hsiang's invariant-theoretic constructions recover Cartan cubics in
$\R^5$ and $\R^8$ and give the $SO(4)$- and $SU(4)$-invariant
models in $\R^9$ and $\R^{15}$; see
\cite[Section~6.1]{TkachevClassification10}.
Tongyan Liu \cite{Liu89} later gave an explicit octonionic construction
of a further cubic in $\R^{27}$, the $Sp(4)$-invariant model;
see \cite[Example~4 and Theorem~C]{PengXiao93} and
\cite[Section~6, p.~1562]{HoppeTkachev19}.

Peng--Xiao \cite{PengXiao93} and, independently, Tkachev
\cite{TkachevClifford10,TkachevClassification10} developed closely
related approaches to radial minimal cubics.\footnote{Tkachev records
that he learned of Peng--Xiao's 1993 paper from Zizhou Tang in November
2010; see \cite[Remark~1, p.~40]{TkachevMiklyukov19}.}
Peng--Xiao assumed harmonicity; Tkachev proved it for irreducible radial
cubics \cite[Theorem~1]{TkachevClassification10}; the algebraic
form of this result, with cubes of linear forms as the exception, is
\cite[Theorem~6.6.1]{NadirashviliTkachevVladut14}.

\paragraph{Clifford and exceptional cubics.}
Following Tkachev, an irreducible radial cubic is of \emph{Clifford
type} if it is orthogonally equivalent, up to a nonzero constant factor,
to
\[
 F_{\mathcal A}(y,z)=\sum_{i=0}^{q}z_i\langle A_i y,y\rangle,
 \qquad A_iA_j+A_jA_i=2\delta_{ij}\Id,
\]
where $y\in\R^{2m}$, $z\in\R^{q+1}$, and the $A_i$ are symmetric
matrices on $\R^{2m}$. The remaining irreducible
radial cubics are called \emph{exceptional}; see
\cite[Section~1.2]{TkachevClassification10} and
\cite[Definitions~2.2 and~2.8]{FoxTkachev24}.
In particular, Cartan cubics belong to the exceptional branch.
The Clifford family will be written explicitly in
\cref{sec:invariants} to compute its differential closure.

\paragraph{Radial normal forms and the classical model families.}
\label{rem:cubic-radial-normal-form-history}
Choosing a positive maximum direction on the unit sphere, with
coordinate $z$, and normalizing $F$ gives
\[
 \begin{gathered}
 F=z^3+zB+C,\qquad
 B=-3|\xi|^2-\tfrac32|\eta|^2+\tfrac32|\zeta|^2,\\
 C=C_{111}+C_{102}+C_{012}+\sigma,\qquad \sigma=C_{030}.
 \end{gathered}
\]
The subscripts denote multidegrees in $(\xi,\eta,\zeta)$, whose
dimensions are $n_1,n_2,n_3$, respectively. They satisfy
$n_3=2n_1+n_2-2$ and
$N=3n_1+2n_2-1$; see \cite[Section~1, Theorem~1]{PengXiao93}
and \cite[Proposition~3.1]{TkachevClassification10}.

The component $\sigma$ satisfies
$|\nabla\sigma|^2=\tfrac92|\eta|^4$.
The zero case means that the $\eta$-space is absent ($n_2=0$).
Otherwise Peng--Xiao \cite[Theorem~2]{PengXiao93} classify $\sigma$
as reducible or a scaled Cartan cubic; a different proof of the
eikonal classification is given in
\cite[Theorem~1.1]{TkachevCartan10}.
The case $\sigma=0$ yields the Cartan cubics, whereas the case
$\sigma\ne0$ with $\sigma$ reducible yields the Clifford cubics
\cite[Theorems~A and~B]{PengXiao93}; see also
\cite[Proposition~3.4 and Theorem~3]{TkachevClassification10}.
In the remaining branch, $n_2\in\{5,8,14,26\}$.
The cases $n_1=0$ and $n_1=1$ were treated in
\cite[Theorems~C and~D]{PengXiao93} and
\cite[Sections~6.2--6.3]{TkachevClassification10}:
the former yields the three models invariant under $SO(4)$, $SU(4)$,
and $Sp(4)$, and the latter yields four models interpreted through
complexified rank-three Jordan norms
\cite[Section~6.11.6(iii)--(iv)]{NadirashviliTkachevVladut14}.

Tkachev also identifies a hidden Clifford system encoded by
$C_{111}+3^{-1/2}C_{102}$ when $n_1>0$, imposing Hurwitz--Radon restrictions on the
multiplicities \cite[Proposition~3.3]{TkachevClassification10}.
His construction of a $G_2$-invariant cubic of type $(4,5)$ and his
correspondence with isoparametric quartics provide further examples and
restrictions \cite[Section~6.4 and Theorem~5]{TkachevClassification10},
using the isoparametric classification results of
Cecil--Chi--Jensen \cite{CecilChiJensen07} and Immervoll
\cite{Immervoll08}. The exceptional branch has only finitely many
possible multiplicity types, with $N\le72$
\cite[Theorem~6.11.2]{NadirashviliTkachevVladut14}.
This bounds the possible types rather than completing the
classification of their realizations.

Subsequent work of Tkachev and his collaborators recast radial minimal
cubics in the language of commutative nonassociative metrized algebras,
giving intrinsic interpretations of their Clifford and Jordan
structures \cite{TkachevJordan14,NadirashviliTkachevVladut14,FoxTkachev24}.
We follow this terminology and recall the algebraic preliminaries
needed to study their Laplacian algebras.

\subsection{Metrized multiplication and low-degree structure}
\subsubsection{The multiplication associated with a cubic}
\label{subsec:cubic-multiplication}

The radial equation above also defines the associated
\emph{radial eigencubic (REC) algebras}
\cite[Definition~6.4.1]{NadirashviliTkachevVladut14}.
We allow an arbitrary quadratic weight in \(\cL(F)=-qF\).

We use the \emph{Freudenthal--Springer algebra} \(V(F)\) of a cubic
\cite[Definition~6.3.1 and equations~(6.4.1)--(6.4.4)]{NadirashviliTkachevVladut14}.
Let \(V\) be a Euclidean space and let \(F\in\R[V]\) be homogeneous
of degree three. Define a bilinear multiplication by
\begin{equation}\label{eq:cubic-polarized-product}
 \langle xy,z\rangle=D^3F[x,y,z],\qquad x,y,z\in V.
\end{equation}
The inner product determines \(xy\); symmetry of \(D^3F\) gives
\[
 xy=yx,\qquad \langle xy,z\rangle=\langle x,yz\rangle.
\]
Thus \(V(F)\) is commutative with invariant Euclidean inner product.
Its \emph{Freudenthal multiplication} gives self-adjoint operators
\(L_xy=xy\). No unit, associativity, or Jordan identity is assumed;
in particular, \(L_{x^2}y=(xx)y\) and \(L_x^2y=x(xy)\) need not agree.

\begin{proposition}[Differential quantities and multiplication]
\label{prop:cubic-multiplication-dictionary}
With the normalization \eqref{eq:cubic-polarized-product},
\begin{equation}\label{eq:cubic-multiplication-dictionary}
 F(x)=\frac16\langle x^2,x\rangle,\qquad
 \nabla F(x)=\frac12x^2,\qquad
 \Hess F(x)=L_x,\qquad \Delta F(x)=\tr L_x.
\end{equation}
Here the Hessian is identified with a self-adjoint endomorphism using
the Euclidean inner product. Conversely, a commutative bilinear
multiplication with an invariant Euclidean inner product determines a
cubic by the first formula, and its polarization recovers the multiplication.
\end{proposition}

\begin{proof}
Homogeneity gives \(D^3F[x,x,x]=6F(x)\). Differentiating the
first formula and using invariance of the inner product gives
\[
 DF(x)[v]=\frac16\bigl(2\langle xv,x\rangle+
                         \langle x^2,v\rangle\bigr)
        =\frac12\langle x^2,v\rangle.
\]
Another derivative gives \(D(\nabla F)(x)[v]=xv\); taking its
trace proves the last assertion. For the converse, commutativity and
invariance make \(\langle xy,z\rangle\) fully symmetric, so the
third derivative of \(\langle x^2,x\rangle/6\) is precisely this
trilinear form.
\end{proof}

Writing \(x^3=x(x^2)\), substitution in
\(\cL(F)=|\nabla F|^2\Delta F-
\Hess F(\nabla F,\nabla F)\) shows that, for a quadratic form \(q\),
the equation \(\cL(F)=-qF\) is equivalent to
\[
 |x^2|^2\tr L_x-\langle x^2,x^3\rangle
 =-\frac23q(x)\langle x^2,x\rangle.
\]
If \(F\) is harmonic, then \(\tr L_x=0\) and the identity reduces to
\(\langle x^2,x^3\rangle=\frac23q(x)\langle x^2,x\rangle\).
The further specialization \(q=\gamma R\) gives the radial REC
identity in \cite[Proposition~6.4.1]{NadirashviliTkachevVladut14},
which uses the same cubic normalization and the parameter
\(\lambda=-\gamma\).

The metrized algebra on \(V\) encodes the derivatives of \(F\), whereas
\(\cA_R(F)\) contains scalar differential invariants such as
\(S=|\nabla F|^2=|x^2|^2/4\). In particular, the critical points
of \(F\) are exactly the square-zero elements of its metrized algebra.

The cylindrical directions are encoded by
\[
 D_vF(x)=\frac12\langle x^2,v\rangle
        =\frac12\langle L_vx,x\rangle,
\]
which gives \(V_{\mathrm{cyl}}(F)=\{v:L_v=0\}\), the annihilator
of its metrized algebra.

\subsubsection{The quadratic weight and linear generators}

An irreducible cubic with nonempty minimal regular zero set satisfies,
by \cref{prop:divisibility},
\begin{equation}\label{eq:cubic-general-weight}
 \cL(F)=-qF,
\end{equation}
where \(q=-\lambda_F\) is a homogeneous quadratic polynomial,
possibly zero. Irreducibility is used here to obtain the polynomial
equation; its radial specialization was introduced in
\cref{sec:cubic-history}.

For brevity put \(\cB=\cA_R(F)\) and
\(\cB_d=\cB\cap\R[V]_d\), the degree-\(d\) homogeneous part.

\begin{proposition}[Quadratic information in the closure]
\label{prop:cubic-weight-in-minimal-algebra}
Suppose \(\dim V\geq2\) and \(F\not\equiv0\) is a homogeneous
cubic satisfying \eqref{eq:cubic-general-weight} with a polynomial
\(q\). Then \(q\in\cB_2\). If \(\cB_2=\R R\), then
\[
 \Delta F=0,\qquad q=\gamma R\quad\text{with }\gamma>0.
\]
In particular, a nonradial weight forces a quadratic element of
\(\cA_R(F)\) independent of \(R\).
\end{proposition}

\begin{proof}
The identity \(\cL(F)=S\Delta F-\tfrac12\Gamma(F,S)\)
shows that \(\cL(F)\in\cB\). The polynomial-quotient property
in \cref{prop:MR-finite-and-division} therefore gives
\(q=-\cL(F)/F\in\cB_2\). If \(\cB_2=\R R\), the linear
polynomial \(\ell=\Delta F\) satisfies \(\ell^2\in\R R\).
A nonzero square of a linear form has rank one, whereas \(R\)
has rank \(\dim V\geq2\); hence \(\ell=0\).
We already know \(q=\gamma R\). At a positive maximum \(a\) of
\(F\) on the unit sphere, attained at \(x\), homogeneity gives
\(\nabla F(x)=3ax\) and \(\Hess F(x)x=6ax\).
Evaluating \eqref{eq:cubic-general-weight} there gives
\(-\gamma a=-54a^3\), so \(\gamma=54a^2>0\).
\end{proof}

\begin{remark}[Additional quadratic generators]
\label{rem:cubic-nonradial-closure}
Since \(\R[R,F,S]_2=\R R\),
\cref{prop:cubic-weight-in-minimal-algebra} shows that its Laplacian
closure forces radiality. Adjoining a nonradial \(q\) need not close
the algebra. Conversely, the radial unbalanced Clifford cubics have
several quadratic generators; see \cref{thm:clifford-minimal-algebra}.
\end{remark}

\begin{remark}[Linear generators in the nonharmonic case]
\label{rem:cubic-nonharmonic-linear-splitting}
If \(\Delta F\ne0\), then \(\cB_1\ne0\). Choose orthonormal
linear coordinates \(t_1,\ldots,t_r\) spanning \(\cB_1\), and put
\(W=\operatorname{span}\{\nabla t_i\}\). Then
\[
 \cB=\R[t_1,\ldots,t_r]\otimes\cB',\qquad
 \cB'=\cB\cap\R[W^\perp],\qquad \cB'_1=0.
\]
Indeed, \(\partial_{t_i}h=\Gamma(t_i,h)\in\cB\), so the finite
Taylor identity
\[
 h(0,y)=\sum_\alpha\frac{(-t)^\alpha}{\alpha!}\partial_t^\alpha h(t,y),
\]
applied also to the derivatives of \(h\), puts every coefficient
in \(\cB'\). This algebra is Laplacian and contains
\(|y|^2=R-\sum_i t_i^2\). The cubic \(F\) can still have mixed
terms, and the splitting gives no uniform generator bound for \(\cB'\).
\end{remark}

\subsubsection{One cubic generator and orthogonal blocks}

The quadratic part of a Laplacian algebra is a Jordan algebra of
self-adjoint endomorphisms; see Mendes--Radeschi
\cite{MendesRadeschiQuadratic}. The following restriction on its
cubic part forces its endomorphisms to commute.

\begin{theorem}[Orthogonal blocks forced by the cubic part]
\label{thm:cubic-one-dimensional-degree-three}
Let \(F\not\equiv0\) be an essential homogeneous cubic on a Euclidean
space \(V\) of dimension at least two, and suppose
\(\cB_3=\R F\). Then \(\cB_1=0\) and \(\Delta F=0\).
There is an orthogonal decomposition into nonzero subspaces,
unique up to permutation,
\begin{equation}\label{eq:cubic-canonical-blocks}
 V=V_1\oplus\cdots\oplus V_r,\qquad
 \cB_2=\operatorname{span}\{R_1,\ldots,R_r\},\qquad
 R_i(x)=|\operatorname{pr}_{V_i}x|^2.
\end{equation}
Moreover, \(F\) is multihomogeneous of positive degrees
\((d_1,\ldots,d_r)\), with \(d_1+\cdots+d_r=3\).
Consequently \(r\leq3\), and, up to permutation, the possible degree
patterns are
\[
 (3),\qquad (2,1),\qquad (1,1,1).
\]
If \(F\) also satisfies \eqref{eq:cubic-general-weight}, then
\(q=\sum_i\gamma_iR_i\) for constants \(\gamma_i\).
\end{theorem}

\begin{proof}
If \(0\ne\ell\in\cB_1\), then both \(\ell R\) and \(\ell^3\)
are nonzero elements of the one-dimensional space \(\cB_3\).
Cancellation would give \(\ell^2=cR\), contradicting the rank
argument above. Thus \(\cB_1=0\), and in particular \(\Delta F=0\).

Represent \(\cB_2\) by a subspace \(J\) of self-adjoint
endomorphisms: \(B\in J\) when \(q_B(x)=\langle Bx,x\rangle\)
belongs to \(\cB_2\). Since
\[
 \Gamma(q_B,q_C)=4q_{B\circ C},\qquad
 B\circ C=\tfrac12(BC+CB),
\]
\(J\) contains \(\Id\) and is closed under the Jordan product.
For \(B\in J\), define \(E_Bh(x)=\langle Bx,\nabla h(x)\rangle\).
Then \(2E_BF=\Gamma(q_B,F)\in\cB_3\), so
\(E_BF=\mu_BF\) for some real number \(\mu_B\).

We claim that all matrices in \(J\) commute. Fix \(B,C\in J\)
and put \(K=[B,C]\). The identity
\([E_B,E_C]=-E_K\) gives \(E_KF=0\). Since \(K\) is
skew-adjoint, the orthogonal transformations \(\exp(tK)\) preserve
\(F\). Every orthogonal symmetry of \(F\) fixes every element of
\(\cB\): its invariant polynomial algebra is Laplacian and contains
\(R,F\), and hence contains \(\cA_R(F)\). In particular,
\(q_B(\exp(tK)x)=q_B(x)\), which upon differentiation gives
\([B,K]=0\). In an orthonormal basis diagonalizing \(B\),
the \((i,j)\) entry of \([B,[B,C]]\) is
\((b_i-b_j)^2C_{ij}\). Thus \([B,C]=0\), as claimed.

Simultaneously diagonalize \(J\), and group together the coordinates
on which every member of \(J\) has the same diagonal value. These
groups define the joint eigenspaces \(V_1,\ldots,V_r\).
A generic element of \(J\) has distinct eigenvalues on distinct
groups. Its spectral projections \(P_i\) are polynomials in this
element and therefore belong to \(J\), because on commuting
matrices the Jordan product is ordinary multiplication. Hence
\(J=\operatorname{span}\{P_1,\ldots,P_r\}\).
The projections are precisely the primitive nonzero idempotents of
this commutative matrix algebra, proving uniqueness in
\eqref{eq:cubic-canonical-blocks}.

Finally, \(E_{P_i}\) is the Euler operator in the variables of
\(V_i\), and \(E_{P_i}F=d_iF\) for a real scalar \(d_i\).
Expansion into monomials shows that each \(d_i\) is a nonnegative
integer and that \(F\) has degree \(d_i\) in that block.
Total homogeneity gives \(\sum_i d_i=3\). If \(d_i=0\), \(F\)
would be independent of \(V_i\), so
\(V_i\subset V_{\mathrm{cyl}}(F)=\{0\}\), contradicting
\(V_i\ne\{0\}\). This proves the degree list. The final assertion follows
from \cref{prop:cubic-weight-in-minimal-algebra}.
\end{proof}

Clifford cubics can also require a quartic generator
(\cref{thm:clifford-minimal-algebra}). The assumption
\(\cB_3=\R F\) is absent from the subcalibration criteria below
and fails for the nonradial models, whose two cubic components are
separated by a block Euler operator.

\subsection{Critical Hessians and multiplier degeneracy}
\label{sec:cubic-square-zero-obstruction}

By \cref{prop:cubic-multiplication-dictionary}, critical points are
square-zero elements \(c^2=0\), and \(L_c=\Hess F(c)\) governs
the first transverse term of \(F\). The Hessian identity below follows
from \eqref{eq:cubic-general-weight} for an arbitrary quadratic weight
\(q\).

\begin{proposition}[Critical Hessians and multiplier degeneracy]
\label{prop:cubic-general-singular-spectrum}
\label{prop:cubic-positive-multiplier-obstruction}
\label{cor:cubic-essential-multiplier-obstruction}
Let $F$ be a nonzero homogeneous cubic on $\R^N$ satisfying
\eqref{eq:cubic-general-weight}, where $q$ is a quadratic form. Let $c$ be a critical point, and put
$A=\Hess F(c)$ and $\tau=\tr A$. Then
\begin{enumerate}[label=\textup{(\roman*)}]
\item The Hessian satisfies
\begin{equation}\label{eq:cubic-general-hessian-polynomial}
 A^3-\tau A^2=\frac{q(c)}2 A.
\end{equation}
In particular, every nonzero eigenvalue satisfies
$\lambda^2-\tau\lambda-q(c)/2=0$. If $q(c)\ne0$, then $A\ne0$.
If $F$ is harmonic and $A\ne0$, then $q(c)>0$ and the nonzero eigenvalues are
$\pm\sqrt{q(c)/2}$, with equal multiplicities.
\item If $q(c)>0$ and $a$ is $C^2$ near $c$, with $a(c)>0$,
then $G=aF$ cannot satisfy $F\cL(G)\ge0$ throughout a neighborhood
of $c$.
\item If $F$ is essential and harmonic, every nonzero critical point
has $A\ne0$ and $q(c)>0$. Hence a nonnegative smooth multiplier
giving the two-sided subcalibration sign must vanish at every such
point across which it extends smoothly.
\end{enumerate}
\end{proposition}

\begin{proof}
For (i),
Euler's identity gives $F(c)=0$. Write
\[
 F(c+y)=h(y)+T(y),\qquad
 h(y)=\tfrac12\langle Ay,y\rangle,
\]
where $T$ is homogeneous of degree three and is not identically zero.
If $q(c)\ne0$ and $A=0$, the left side of \eqref{eq:cubic-general-weight}, expanded
at $c$, is homogeneous of degree five in $y$, whereas the right side
has the nonzero term $-q(c)T(y)$ of degree three. This is impossible.

The terms of degree two in the same equation are
\[
 \tau\langle A^2y,y\rangle-\langle A^3y,y\rangle
 =-\frac{q(c)}2\langle Ay,y\rangle.
\]
All matrices involved are symmetric, so this proves
\eqref{eq:cubic-general-hessian-polynomial}. If $F$ is harmonic,
$\tau=0$. If in addition $A\ne0$, a nonzero real eigenvalue gives
$q(c)/2=\lambda^2>0$. The zero trace then forces the multiplicities
of the two nonzero eigenvalues to agree.

For (ii), \cref{lem:critical-multiplier-defect} gives
\(\cK_F(a)(c)=-a(c)^3q(c)<0\). Since (i) gives \(A\ne0\),
choose \(y\) with \(h(y)\ne0\). Continuity and
\(F(c+ry)=r^2h(y)+O(r^3)\) yield
\begin{equation}\label{eq:cubic-positive-multiplier-leading-sign}
 F(c+ry)\cL(aF)(c+ry)
 =-a(c)^3q(c)r^4h(y)^2+o(r^4)<0
\end{equation}
for small \(r>0\). The negative product implies
\(\nabla(aF)\ne0\), so the sign fails where the field is defined;
harmonicity is not used.

For (iii), \(L_c=0\) would put \(c\) in the annihilator
\(V_{\mathrm{cyl}}(F)=\{0\}\). Thus (i) gives \(q(c)>0\).
The multiplier conclusion is also the harmonic obstruction of
\cref{prop:harmonic-positive-multiplier-obstruction}.
\end{proof}

For essential harmonic $F$ as above, \cref{sec:tangent-instability-transfer}
identifies the tangent at every nonzero critical point as
$\{(u,v)\in\R^s\oplus\R^s:|u|=|v|\}\times\R^{N-2s}$,
where $2s$ is the Hessian rank, with smooth convergence on compact
regular subsets. Radial powers $R^\beta$ fail as smooth positive
multipliers there, whereas $S$ has the required vanishing.
The following trace bound supplies removability for
\cref{thm:cubic-unified-sextic}.

\begin{lemma}[A trace bound for the critical set]
\label{lem:cubic-trace-removability}
Let \(F\not\equiv0\) be a harmonic cubic on \(\R^N\), \(N\ge2\),
satisfying \(\cL(F)=-2RF\), and put \(Z=\{\nabla F=0\}\).
If \(\Delta S\ge16R\), then every \(c\in Z\setminus\{0\}\)
satisfies \(\operatorname{rank}\Hess F(c)\ge8\), and \(Z\) is locally
contained near \(c\) in a smooth submanifold of codimension eight.
In particular, \(Z\) has locally zero \(\cH^{N-1}\)-measure.
\end{lemma}

\begin{proof}
At \(c\in Z\setminus\{0\}\),
\cref{prop:cubic-general-singular-spectrum} gives
\(A^3=|c|^2A\), where \(A=\Hess F(c)\). Its nonzero eigenvalues
are \(\pm|c|\), and Bochner's identity yields
\[
 |c|^2\operatorname{rank}A=\tr(A^2)
   =\tfrac12\Delta S(c)\ge8|c|^2.
\]
Choose eight components of \(\nabla F\) with independent differentials.
The implicit function theorem places \(Z\) locally in their smooth
common zero set of codimension eight. This proves the local codimension
claim and, by a countable cover, nullity off the origin. The origin is
also null since \(N\ge2\).
If \(N<8\), the rank inequality simply excludes nonzero critical points.
\end{proof}

\subsection{Radial transport and Clifford closure}
\subsubsection{Intrinsic transport and the scalar trace form}

For harmonic \(F\), the identity \(\Gamma(F,S)=4RF\) is the
normalization \(\lambda=-2\) for REC algebras. We use this normalization
and its polarized identities from
\cite[Section~6.7.1, equations~(6.7.6)--(6.7.7)]{NadirashviliTkachevVladut14}.

\begin{lemma}[Intrinsic radial transport identities]
\label{lem:cubic-radial-transports}
Let \(F\) be a homogeneous cubic on a Euclidean space \(V\), with
\(\Gamma(F,S)=4RF\), where \(R=|x|^2\) and \(S=|\nabla F|^2\).
Define
\begin{equation}\label{eq:cubic-unified-D}
 D=\frac{\Gamma(S,S)-4RS}{12}
   =\frac{|x^3|^2-R|x^2|^2}{12}.
\end{equation}
Then
\begin{equation}\label{eq:cubic-unified-table}
 \Gamma(F,D)=2SF,\qquad
 \Gamma(S,D)=8RD,\qquad
 \Gamma(D,D)=4SD.
\end{equation}
Neither harmonicity nor a scalar quadratic trace form is needed.
\end{lemma}

\begin{proof}
Use the multiplication \eqref{eq:cubic-polarized-product}, and put
\[
 \begin{gathered}
 y=x^2,\quad z=xy,\quad u=y^2,\quad w=yz,\\
 r=|x|^2,\quad c=\langle x,y\rangle,\quad
 s=|y|^2,\quad t=|z|^2,\quad \delta=t-rs.
 \end{gathered}
\]
Thus \(F=c/6\), \(S=s/4\), \(D=\delta/12\), and the hypothesis is
\(P(x):=\langle y,z\rangle-4rc/3=0\). Its gradient gives the
polarized identity
\begin{equation}\label{eq:cubic-radial-polarization}
 Q(x):=4xz+u-\tfrac83cx-4ry=0.
\end{equation}
Pairing with \(y\) gives
\begin{equation}\label{eq:cubic-radial-polarization-eta}
 \eta:=\langle y,u\rangle=\tfrac83c^2-4\delta.
\end{equation}
The coefficient of \(\tau^2\) in \(P(x+\tau y)=0\), with \(y\)
held fixed, and the pairing \(\langle Q,z\rangle=0\) give, respectively,
\[
 \begin{aligned}
 4\langle z,xz\rangle+6\langle u,z\rangle
   &=\tfrac{16}{3}r^2c+\tfrac{28}{3}cs,\\
 4\langle z,xz\rangle+\langle u,z\rangle
   &=\tfrac{16}{3}r^2c+\tfrac83cs.
 \end{aligned}
\]
Subtracting, and then using \(DQ(x)[y]/2-xQ(x)=0\), yields
\begin{equation}\label{eq:cubic-radial-polarization-uz}
 \langle u,z\rangle=\tfrac43cs,\qquad
 xu+4w=4sx+\tfrac83cy.
\end{equation}
Equation \eqref{eq:cubic-radial-polarization-eta} gives
\(D=2F^2-F(\nabla F)\). Differentiating and applying the second
identity in \eqref{eq:cubic-radial-polarization-uz}, we obtain
\begin{equation}\label{eq:cubic-intrinsic-D-gradient}
 \nabla D=\tfrac c3y-\tfrac18xu=\tfrac12(w-sx).
\end{equation}

Write \(b=\langle z,w\rangle\) and \(U=|u|^2\). Pairing
\eqref{eq:cubic-radial-polarization} with \(u\), and using
\(\langle z,xu\rangle=4s^2+32rc^2/9-4b\), gives
\[
 b=s^2+r\delta+\frac U{16}.
\]
The coefficient of \(\tau^3\) in \(P(x+\tau y)=0\) is
\[
 2U+4\langle z,xu\rangle+4b
 =\tfrac43\bigl(r\eta+6c\langle y,z\rangle+3s^2\bigr).
\]
Substitution of \eqref{eq:cubic-radial-polarization-eta} and
\(\langle Q,u\rangle=0\) reduces this to
\(b=s^2-U/4+8r\delta/3\). The two expressions for \(b\) imply
\begin{equation}\label{eq:cubic-radial-polarization-norms}
 U=\tfrac{16}{3}r\delta,\qquad
 b=s^2+\tfrac43r\delta.
\end{equation}

Differentiate \(\langle u,z\rangle=4cs/3\) in the direction \(u\):
\[
 4\langle w,xu\rangle+\langle u,yu\rangle+2|xu|^2
 =4s\eta+\tfrac{16}{3}c\langle z,u\rangle.
\]
Here \(P(y)=0\) gives \(\langle u,yu\rangle=4s\eta/3\), while
\eqref{eq:cubic-radial-polarization-uz} gives
\[
 \langle w,xu\rangle=4st+\tfrac{32}{9}c^2s-4|w|^2,
 \qquad |xu|^2=16rs^2-32st+16|w|^2.
\]
Substitution yields
\begin{equation}\label{eq:cubic-radial-polarization-w}
 |w|^2=\tfrac73st-\tfrac43rs^2.
\end{equation}
Finally, \(\nabla F=y/2\), \(\nabla S=z\), and
\eqref{eq:cubic-intrinsic-D-gradient} give
\[
 \begin{aligned}
 \Gamma(F,D)&=\tfrac14(\langle u,z\rangle-sc)=cs/12=2SF,\\
 \Gamma(S,D)&=\tfrac12(b-s^2)=2r\delta/3=8RD,\\
 \Gamma(D,D)&=\tfrac14(|w|^2-2st+rs^2)=s\delta/12=4SD.
 \end{aligned}
\]
\end{proof}

\begin{proposition}[Closure under a scalar quadratic trace form]
\label{prop:cubic-scalar-killing-closure}
Suppose, in addition to the hypotheses of
\cref{lem:cubic-radial-transports}, that \(F\) is harmonic and
\(\tr(L_x^2)=kR\) for a constant \(k\). Writing \(N=\dim V\), one has
\begin{equation}\label{eq:cubic-scalar-killing-laplacians}
 \Delta S=2kR,\qquad
 \Delta D=\left(2k-\frac{2N}{3}\right)S,
 \qquad \cA_R(F)=\R[R,F,S,D].
\end{equation}
In particular, this algebra has at most four generators.
\end{proposition}

\begin{proof}
Polarization gives \(\tr(L_xL_v)=k\langle x,v\rangle\).
Differentiating \eqref{eq:cubic-radial-polarization} and dividing
by four yields the operator identity
\[
 L_z+L_xL_y+L_yL_x+2L_x^3-2rL_x
 -\tfrac23c\Id-2x\otimes y-2y\otimes x=0,
\]
where \((a\otimes b)v=\langle b,v\rangle a\).
Taking traces, using harmonicity, gives
\[
 \tr L_x^3=\alpha c,\qquad \alpha=N/3+2-k.
\]
Its derivative in the direction \(y\) is
\(\tr(L_yL_x^2)=\alpha s\). Therefore
\[
 \begin{aligned}
 \operatorname{div}(yz)
 &=2\tr(L_zL_x)+\tr L_y^2+2\tr(L_yL_x^2)
   =(3k+2\alpha)s,\\
 \operatorname{div}(sx)&=(N+4)s.
 \end{aligned}
\]
Taking the divergence in \eqref{eq:cubic-intrinsic-D-gradient}
proves the formula for \(\Delta D\); Bochner's identity gives that
for \(\Delta S\). Together with \(\Delta R=2N\), \(\Delta F=0\),
homogeneity, and \cref{lem:cubic-radial-transports}, these formulas
show that \(\R[R,F,S,D]\) is Laplacian. Conversely,
\(S=\Gamma(F,F)\) and \eqref{eq:cubic-unified-D} put \(S,D\) in
every Laplacian algebra containing \(R,F\), proving equality.
\end{proof}

\begin{remark}[The radial bound and its trace input]
\label{rem:cubic-harmonic-generator-bound}
An irreducible harmonic radial eigencubic of exceptional type,
in the terminology of \cref{sec:cubic-history}, has \(\tr(L_x^2)=k|x|^2\)
for a constant \(k\), in the multiplication convention
\eqref{eq:cubic-polarized-product}. This is
\cite[Proposition~6.11.1]{NadirashviliTkachevVladut14}; see also
\cite[Proposition~5.2, equation~(5.11)]{TkachevClassification10}.
The scalar trace condition is not a consequence of radiality alone:
in the Clifford family, \eqref{eq:DeltaS} is scalar precisely when
\(m=q+1\).

Consequently every irreducible radial eigencubic has at most four
ordinary algebra generators in \(\cA_R(F)\). After normalizing
\(F\), the exceptional branch follows from this trace theorem and
\cref{prop:cubic-scalar-killing-closure}. In the Clifford branch,
the same proposition applies when \(m=q+1\); when
\(m\ne q+1\), \cref{thm:clifford-minimal-algebra} recovers the
block radii and identifies the smallest algebra directly.
\end{remark}

\subsubsection{The Clifford family and its differential table}
\label{sec:invariants}
We write the Clifford family of
\cite[Theorem~B]{PengXiao93} in the symmetric-system notation of
\cite{TkachevClifford10}.

Let \(q\geq1\), and let
\(\mathcal A=(A_0,\ldots,A_q)\) be a symmetric Clifford system on
\(\R^{2m}\), so that
\begin{equation}\label{eq:clifford-system}
 A_iA_j+A_jA_i=2\delta_{ij}\Id_{\R^{2m}}.
\end{equation}
For \((y,z)\in\R^{2m}\oplus\R^{q+1}\), put
\(A_z=\sum_{i=0}^qz_iA_i\).  The corresponding cubic and its zero cone are
\begin{equation}\label{eq:clifford-F}
 F=F_{\mathcal A}(y,z)=\sum_{i=0}^qz_i\langle A_i y,y\rangle
 =y^TA_zy,\qquad \mathcal C_{\mathcal A}=\{F=0\}.
\end{equation}
The Freudenthal--Springer algebra of \(F/2\) is a polar algebra in the
normalization of
\cite[Definition~6.5.1 and Propositions~6.5.1--6.5.2]{NadirashviliTkachevVladut14}.
We call the pair \((m,q)\) admissible when such a Clifford system exists.

Set
\[
 R=|y|^2+|z|^2,\qquad S=|\nabla F|^2,\qquad N=2m+q+1.
\]
The matrices are trace free, so \(F\) is harmonic.  The Clifford
relations give \(\cL(F)=-8RF\), and hence the regular zero set is
minimal; the calculations are given below.
We orient the cone as \(\partial\llbracket\{F<0\}\rrbracket\)
and use the local mass-minimizing convention for integral-current
competitors with compactly supported difference that bounds.

\paragraph{Differential identities.}

Introduce
\[
 \tau_i=y^TA_i y,\qquad
 \tau=(\tau_0,\ldots,\tau_q),\qquad
 A_\tau=\sum_{i=0}^{q}\tau_iA_i,
\]
and the scalar functions
\[
 u=|y|^2,\qquad v=|z|^2,\qquad R=u+v,\qquad T=|\tau|^2.
\]
The Clifford relations imply
\[
 A_z^2=v\Id,\qquad A_\tau^2=T\Id,\qquad
 A_zA_\tau+A_\tau A_z=2F\Id.
\]
Consequently,
\begin{equation}\label{eq:clifford-contractions}
 |A_zy|^2=uv,\quad |A_\tau y|^2=uT,\quad
 \langle A_zy,A_\tau y\rangle=uF.
\end{equation}

The vectors $A_0y,\ldots,A_qy$ are orthogonal, each of length $|y|$.
Bessel's inequality, followed by Cauchy--Schwarz, gives
\begin{equation}\label{eq:box-constraints}
 0\leq T\leq u^2,\qquad 0\leq F^2\leq vT.
\end{equation}
We shall also use
\[
 \langle A_i y,A_z y\rangle=uz_i,\qquad
 \langle A_i y,A_\tau y\rangle=u\tau_i.
\]

Set
\[
 U=\nabla F,\qquad S=|U|^2,\qquad Z=\nabla S,\qquad
 J=|Z|^2,\qquad K=\Delta S.
\]

The four functions \(I=(u,v,T,F)\) have the following gradient
inner products and Laplacians:
\begin{equation}\label{eq:clifford-closure-table}
 \begin{aligned}
 \bigl(\langle\nabla I_i,\nabla I_j\rangle\bigr)&=
 \begin{pmatrix}
 4u&0&8T&4F\\
 0&4v&0&2F\\
 8T&0&16uT&8uF\\
 4F&2F&8uF&4uv+T
 \end{pmatrix},\\
 \Delta I&=(4m,\,2(q+1),\,8(q+1)u,\,0).
 \end{aligned}
\end{equation}
For example, \(\nabla_yT=4A_\tau y\), so
\(|\nabla T|^2=16uT\) and \(\langle\nabla T,\nabla F\rangle=8uF\).
Euler's identity and the Clifford relations give the remaining entries.
Since \(R=u+v\), \cref{prop:MR-generator-test} shows that
\(\R[u,v,T,F]\) is a Laplacian algebra.
Hence \cref{thm:finite-reduction} applies to potentials \(G(u,v,T,F)\).

\begin{proposition}[Gradients and Laplacians]
\label{prop:first-order}
One has
\begin{subequations}
\label{eq:U}\label{eq:S}\label{eq:Z}\label{eq:Clifford-gradS-square}
\label{eq:DeltaF}\label{eq:DeltaS}\label{eq:UZ}
\begin{equation}
\begin{aligned}
 U&=(2A_z y,\tau),& S&=4uv+T,\\
 Z&=(8vy+4A_\tau y,8uz),& \Delta F&=0,\\
 K&=16mv+16(q+1)u,& \langle U,Z\rangle&=16RF,\\
 J&=16\bigl(4uvR+(u+4v)T\bigr).
\end{aligned}
\tag{\theparentequation}
\end{equation}
\end{subequations}
\end{proposition}

\begin{proof}
Differentiate to obtain $U,Z$; their squared norms follow from
\eqref{eq:clifford-contractions}. Conjugation by $A_j$, $j\ne i$,
sends $A_i$ to $-A_i$, proving trace-freeness and $\Delta F=0$.
The remaining formulas are
$K=\Delta(4uv+T)$ and $\langle U,Z\rangle=\Gamma(F,4uv+T)$
in \eqref{eq:clifford-closure-table}.
\end{proof}

\begin{theorem}[The smallest algebra of an unbalanced Clifford cubic]
\label{thm:clifford-minimal-algebra}
If \(m\ne q+1\), then
\[
 \cA_R(F)=\R[u,v,T,F],\qquad
 (\cA_R(F))_2=\operatorname{span}\{u,v\},\qquad
 (\cA_R(F))_3=\R F.
\]
The two coordinate blocks are therefore the intrinsic quadratic
blocks of \cref{thm:cubic-one-dimensional-degree-three}, with degree
pattern \((2,1)\).
\end{theorem}

\begin{proof}
Every Laplacian algebra containing \(R,F\) contains
\(S=\Gamma(F,F)\) and \(\Delta S\). Equation~\eqref{eq:DeltaS}
recovers the two block radii by
\[
 u=\frac{\Delta S/16-mR}{q+1-m},\qquad v=R-u.
\]
It then recovers \(T=S-4uv\). Thus \(\R[u,v,T,F]\) is
contained in \(\cA_R(F)\), and the reverse inclusion follows from
its closure table \eqref{eq:clifford-closure-table}. The generators
have degrees \(2,2,4,3\), giving the two homogeneous components
stated above; \(u,v\) are linearly independent.
Finally, if a constant vector \((\xi,\zeta)\) is a translation
direction of \(F\), its derivative has terms
\(2\langle A_z y,\xi\rangle\) and \(y^TA_\zeta y\) of
different bidegrees, so both vanish. Invertibility of \(A_0\)
gives \(\xi=0\), and \(A_\zeta^2=|\zeta|^2\Id\) gives
\(\zeta=0\). The block theorem therefore applies.
\end{proof}

When \(m=q+1\), one has $\Delta S=16mR$. Applying
\cref{prop:cubic-scalar-killing-closure} to the normalized cubic
$F/2$ therefore gives
\[
 \cA_R(F)=\R[R,F,S,vT]\subsetneq\R[u,v,T,F].
\]
Indeed, $|\nabla(F/2)|^2=S/4$, and its generator $D$ from
\eqref{eq:cubic-unified-D} is $vT/4$. In the present normalization,
\[
 vT=\frac{\Gamma(S,S)-16RS}{48}\in\cA_R(F).
\]
The generators $R,F,S,vT$ have degrees $2,3,4,6$, so the quadratic
component of this smallest algebra is $\R R$. Since $u,v$ are
linearly independent, the inclusion is strict. The larger algebra
remains useful for the coordinate calculations below.

Since \(Z=2\Hess F(U)\), \eqref{eq:UZ} and harmonicity give
\[
 \cL(F)=-\tfrac12\Gamma(F,S)=-8RF,
\]
recovering the radial minimality identity for this family
directly from the closure table.

\subsection{The weight operator and nonradial spectral rigidity}
\label{sec:nonradial-cubics}
\label{sec:cubic-generator-bound}

\begin{theorem}[The five-generator bound]
\label{thm:harmonic-cubic-five-bound}
Let $F$ be a nonzero irreducible homogeneous harmonic cubic on a
Euclidean space $V$, with $V=V_{\mathrm{ess}}(F)$, and suppose its
regular zero set is minimal.
Then the smallest Laplacian algebra $\cA_R(F)$ is
generated by at most five polynomials. More precisely:
\begin{enumerate}[label=\textup{(\roman*)}]
\item If $F$ is radial, $\cA_R(F)$ has at most four generators.
\item If $F$ is nonradial, $\cA_R(F)$ is a polynomial algebra on five
algebraically independent homogeneous generators of degrees
$(2,2,3,3,4)$. Its spherical quotient dimension is four.
\end{enumerate}
The upper bound five is attained by the examples of
\cref{sec:nonradial-models}.
\end{theorem}

Extending the Peirce framework of
\cite[Sections~6.6--6.8]{NadirashviliTkachevVladut14} to the weight
operator gives spectral rigidity and a graded splitting. Irreducibility
forces an exact eikonal core. The five resulting invariants close under
the differential operations, are recovered from $R,F$, and have rank
five. Spectral exclusion and core exactness are proved in
\cref{app:nr-spectral-exclusions,app:nr-core-exactness-proof}.

\subsubsection{Maximal idempotents and the weight operator}
\label{subsec:nr-maximal-idempotents}

By \cref{lem:square-free-critical-set}, an irreducible harmonic $F$
has a nonsingular real zero. Thus \cref{prop:divisibility}
gives $\cL(F)=-qF$, where homogeneity makes $q$ quadratic.

Use the commutative Euclidean metrized algebra and normalization of
\cref{prop:cubic-multiplication-dictionary}.
Harmonicity is the condition $\tr L_x=0$ for every $x$;
we call such an algebra \emph{exact}, following
\cite[Section~3.2 and Definition~4.1]{FoxTkachev24}.
For a self-adjoint operator \(T\), the equation
\(\cL(F)=-\frac32\ip{Tx}{x}F\) is then equivalent to
\begin{equation}\label{eq:nr-Hsiang-identity}
 \ip{x^2}{x^3}=\ip{Tx}{x}\ip{x^2}{x}.
\end{equation}
This defines an \emph{exact Hsiang algebra}, with symmetric bilinear weight
$b(x,y)=\ip{Tx}{y}$. The radial case is the special case in which
$T$ is scalar.
Here \(x^3=x(x^2)\) and \(x^4=x(x^3)\); no associativity is assumed.
Differentiating \eqref{eq:nr-Hsiang-identity} gives
\begin{equation}\label{eq:nr-vector-identity}
 4x^4+(x^2)^2-2\ip{x^2}{x}Tx-3\ip{Tx}{x}x^2=0.
\end{equation}

As shown in \cref{subsec:cubic-multiplication},
$\operatorname{Ann}(A)=\{v:vA=0\}$ is the cylindrical subspace of $F$.
Metric invariance makes $A_{\mathrm{ess}}=\operatorname{Ann}(A)^\perp$
a subalgebra; its orthogonal splitting from the annihilator has zero
mixed products.

We use the spherical maximum construction of an idempotent
\cite[Lemma~6.6.1]{NadirashviliTkachevVladut14} and derive the spectral
bounds below.

\begin{lemma}[Maximal normalization]\label{lem:nr-max-idempotent}
Let \(A\) be a commutative metrized algebra with a Euclidean metric
and nonzero multiplication. After multiplying its product
by a positive constant, there is a unit idempotent \(c\) such that
\begin{equation}\label{eq:nr-spherical-maximum}
 \max_{|x|=1}\ip{x^2}{x}=\ip{c^2}{c}=1.
\end{equation}
Every eigenvalue of \(L_c\) on \(c^\perp\) lies in
\([-1,\frac12]\).
\end{lemma}

\begin{proof}
Put \(h(x)=\ip{x^2}{x}\). Polarization shows that \(h\) is a
nonzero odd cubic, so its maximum \(M\) on the unit sphere is positive.
At a maximizer \(c\), the Lagrange multiplier equation is \(c^2=Mc\).
Replacing the product by \(xy/M\) proves
\eqref{eq:nr-spherical-maximum}. It preserves exactness and replaces
\(T\) in \eqref{eq:nr-Hsiang-identity} by \(T/M^2\).
The constrained Hessian of \(h\) at \(c\) is
\(6L_c-3\Id\) on \(c^\perp\), giving the upper bound.
For a unit eigenvector \(u\perp c\), with \(cu=ru\), maximality gives
\[
 h\left(-\frac c2\pm\frac{\sqrt3}2u\right)
 =-\frac18-\frac98r\pm\frac{3\sqrt3}8h(u)\le1.
\]
Averaging gives \(r\ge-1\).
\end{proof}

Fix this normalization, which rescales the cubic by a positive constant.

\begin{lemma}[The weight operator at a unit idempotent]
\label{lem:nr-idempotent-quotient}
Every unit idempotent \(e\) in an algebra satisfying
\eqref{eq:nr-Hsiang-identity} satisfies
\begin{equation}\label{eq:nr-quotient-polynomial}
 Te=e,\qquad T=p(L_e)-6P_e,\qquad p(t)=4t^3+4t^2-t,
\end{equation}
where \(P_ex=\ip{e}{x}e\). In particular, \([T,L_e]=0\).
\end{lemma}

\begin{proof}
At \(e\), the scalar identity gives \(\ip{Te}{e}=1\), and the
vector identity gives \(Te=e\). Its derivative at \(e\) is
\[
 8L_e^3y+8L_e^2y-2L_ey-2Ty-12\ip{e}{y}e=0.
\]
This proves \eqref{eq:nr-quotient-polynomial}. Since \(L_e\) is
self-adjoint and \(L_ee=e\), it commutes with \(P_e\).
\end{proof}

Write \(E_r=\ker(L_c-r\Id)\). By maximality \(E_1=\R c\), and
\begin{equation}\label{eq:nr-Peirce-quotient}
 Tu=p(r)u\qquad (u\in E_r\subset c^\perp).
\end{equation}

\subsubsection{Fusion identities and spectral rigidity}
\label{subsec:nr-spectral-rigidity}

We adapt the radial polarization argument in
\cite[Lemma~6.6.2]{NadirashviliTkachevVladut14}, specifically equation~(6.6.8).
For nonradial $T$, we derive the fusion rule directly from
\eqref{eq:nr-vector-identity}, using
\begin{equation}\label{eq:nr-fusion-polynomial}
 \mathscr K(r,s,t)=4(r^2+s^2+t^2+rs+rt+st)+2(r+s+t)-3.
\end{equation}

\begin{lemma}[Fusion and the square identity]\label{lem:nr-fusion-square}
For \(u\in E_r\), \(v\in E_s\), \(w\in E_t\), all
perpendicular to \(c\), one has
\begin{equation}\label{eq:nr-fusion}
 \mathscr K(r,s,t)\ip{uv}{w}=0.
\end{equation}
For \(u\in E_r\subset c^\perp\), set
\(Q_r(t)=\mathscr K(r,r,t)\). Then
\begin{equation}\label{eq:nr-square-operator}
 Q_r(L_c)u^2=3\bigl(2r+p(r)\bigr)|u|^2c.
\end{equation}
Writing \(u^2=r|u|^2c+\sum_tq_t\), with
\(q_t\in E_t\cap c^\perp\), gives
\begin{equation}\label{eq:nr-square-norm}
 \sum_t(t+4r)|q_t|^2=4r^2(3r-1)(r+1)|u|^4.
\end{equation}
\end{lemma}

\begin{proof}
The mixed coefficient of \(ab\) in
\eqref{eq:nr-vector-identity} at \(c+au+bv\), projected to
\(c^\perp\) and divided by two, is the projection of
\[
 \bigl[4L_c^2+(4r+4s+2)L_c
       +4(r^2+rs+s^2)+2(r+s)-3\bigr]uv.
\]
Pairing with \(w\) proves \eqref{eq:nr-fusion}. The coefficient
of \(a^2\) at \(c+au\) gives
\eqref{eq:nr-square-operator}; thus \(Q_r(t)q_t=0\).
For the norm identity, pair the vector identity at \(u\) with \(c\).
Metric invariance gives
\[
 \ip{u^4}{c}=r|u^2|^2,\qquad
 \ip{(u^2)^2}{c}=r^2|u|^4+\sum_t t|q_t|^2.
\]
The term containing \(Tu\) pairs to zero, and the last term is
\(-3p(r)r|u|^4\). Rearranging yields
\eqref{eq:nr-square-norm}, since
\(3p(r)r-4r^3-r^2=4r^2(3r-1)(r+1)\).
\end{proof}

\begin{proposition}[A uniform spectral interval]
\label{prop:nr-spectral-interval}
Put \(r_0=-(2+\sqrt7)/6\). Then
\begin{equation}\label{eq:nr-Peirce-interval}
 \operatorname{Spec}(L_c|_{c^\perp})
 \subset\{-1,0\}\cup[r_0,-1/6]\cup[1/3,1/2].
\end{equation}
Moreover,
\begin{equation}\label{eq:nr-quotient-interval}
 \operatorname{Spec}T\subset\{0\}\cup
 \left[\frac7{27},\frac{17+7\sqrt7}{27}\right],
 \qquad \ker T=\operatorname{Ann}(A).
\end{equation}
\end{proposition}

\begin{proof}
All possible output eigenvalues \(t\) in
\eqref{eq:nr-square-norm} lie in \([-1,1/2]\). If
\(-1<r<r_0\), then \(Q_r(t)=4t^2+(8r+2)t+12r^2+4r-3\)
is decreasing on that interval and
\(Q_r(1/2)=12r^2+8r-1>0\). All \(q_t\) therefore vanish,
contradicting the negative right side of
\eqref{eq:nr-square-norm}. If \(-1/6<r<0\), both endpoint
values
\[
 Q_r(-1)=(6r+1)(2r-1),\qquad Q_r(1/2)=12r^2+8r-1
\]
are negative. Convexity makes \(Q_r\) negative throughout the
interval, giving the same contradiction. If \(0<r<1/3\),
then \(Q_r(-1)<0\) and \(Q_r(0)<0\); hence every admissible
root is positive. The left side of the norm identity is nonnegative
and the right side is negative. This proves
\eqref{eq:nr-Peirce-interval}. At \(r=1/3\), the norm identity
also shows that every \(q_t\) vanishes, because \(t+4/3>0\).

Applying \(p\) to these intervals gives the spectral inclusion in
\eqref{eq:nr-quotient-interval}. The only zero of \(p\) allowed
by \eqref{eq:nr-Peirce-interval} is \(0\). Thus \(u\in\ker T\)
implies \(cu=0\), and \eqref{eq:nr-square-operator} gives
\[
 (4L_c^2+2L_c-3\Id)u^2=0.
\]
Its two roots \((-1\pm\sqrt{13})/4\) lie outside
\([-1,1/2]\), and it is nonzero at \(1\). Consequently \(u^2=0\).
The coefficient of \(t^3\) in the vector identity at \(x+tu\)
is then \(8L_u^3x\): all other terms have degree at most two,
since \(Tu=u^2=0\). Hence \(L_u^3=0\), and self-adjointness
gives \(L_u=0\). Conversely, an annihilator element is orthogonal
to \(c=c^2\) and is killed by \(L_c\), so
\eqref{eq:nr-quotient-polynomial} kills it as well.
\end{proof}

By \cref{prop:nr-spectral-interval}, restriction to $A_{\mathrm{ess}}$
preserves exactness and the defining identity, and
$T|_{A_{\mathrm{ess}}}$ is positive definite.

Let
\begin{equation}\label{eq:nr-eigenvalue-one-space}
 B=E_T(1)=\R c\oplus Z\oplus H\oplus W,
 \qquad Z=E_{-1},\quad H=E_{-1/2},\quad W=E_{1/2},
\end{equation}
since \(p(t)-1=(t+1)(2t-1)(2t+1)\).
The fusion polynomial gives the following possible outputs on
\(c^\perp\):
\[
\begin{array}{c|cccccc}
(r,s)&(-1,-1)&(-1,-\frac12)&(-1,\frac12)&
(-\frac12,-\frac12)&(-\frac12,\frac12)&(\frac12,\frac12)\\ \hline
 t&\text{none}&\frac12&-\frac12,\frac12&-\frac12&
 -1,\frac12&-1,-\frac12
\end{array}
\]
(The additional root \(1\) in the fourth column is absent on
\(c^\perp\).) In particular,
\begin{equation}\label{eq:nr-B-subalgebra}
 B^2\subset B,\qquad BB^\perp\subset B^\perp.
\end{equation}

A Euclidean commutative metrized algebra is \emph{eikonal} if
\(x^3=|x|^2x\), equivalently \(|x^2|^2=|x|^4\).
Equivalently, the associated cubic $6F$ is a cubic eiconal in the sense
of \cite[Definition~3.1.1]{NadirashviliTkachevVladut14}; exactness is a
separate condition.

\begin{lemma}[The negative Peirce spaces]\label{lem:nr-negative-spaces}
\begin{enumerate}[label=\textup{(\roman*)}]
\item If \(H\ne0\), then \(\operatorname{Spec}T\subset\{1,0\}\).
\item If \(Z\ne0\), then
\(\operatorname{Spec}T\subset\{1,7/27,0\}\).
For \(U=E_T(7/27)\), one has
\begin{equation}\label{eq:nr-U-peirce}
 U=P\oplus Q,\quad P=E_{1/3},\quad Q=E_{-1/6},\qquad
 L_z|_P=0,\quad (L_z|_Q)^2=\frac1{12}|z|^2\Id_Q.
\end{equation}
\end{enumerate}
\end{lemma}

The proof is given in \cref{app:nr-spectral-exclusions}.

Exactness excludes $Z=H=W=0$.

\begin{lemma}\label{lem:nr-not-one-dimensional}
In an exact Hsiang algebra whose multiplication is not identically zero,
\(E_T(1)\ne\R c\).
\end{lemma}

The proof is given in \cref{app:nr-spectral-exclusions}.

\begin{theorem}[Spectral rigidity]\label{thm:nonradial-spectral}
Every exact Euclidean Hsiang algebra whose multiplication is not
identically zero admits a positive rescaling for which
\begin{equation}\label{eq:nr-rigid-spectrum}
 \operatorname{Spec}T\subset\{0,7/27,1\},\qquad
 \ker T=\operatorname{Ann}(A).
\end{equation}
On $A_{\mathrm{ess}}$, either $T|_{A_{\mathrm{ess}}}=\Id$ or
\begin{equation}\label{eq:nr-quotient-splitting}
 A_{\mathrm{ess}}=B\oplus U,\qquad T|_B=\Id,\qquad T|_U=\frac7{27}\Id.
\end{equation}
\end{theorem}

\begin{proof}
The kernel assertion is \cref{prop:nr-spectral-interval}.
If \(H\ne0\) or \(Z\ne0\), use
\cref{lem:nr-negative-spaces}.
Otherwise the fusion table gives \(w^2=|w|^2c/2\) for \(w\in W\).
The norm identity with \(r=1/2\) then reads
\(0=3|w|^4/4\), so \(W=0\). This would give
\(B=\R c\), contrary to \cref{lem:nr-not-one-dimensional}.
\end{proof}

\subsection{The eikonal core and its representation}
\label{subsec:nr-recovery}

Work on $A_{\mathrm{ess}}$, denoted by $A$, with nonradial weight operator.
By \cref{thm:nonradial-spectral} and \eqref{eq:nr-B-subalgebra},
\begin{equation}\label{eq:nr-recovery-splitting}
 A=B\oplus U,\qquad T|_B=\Id,\qquad T|_U=\tfrac7{27}\Id,
 \qquad B^2\subset B,\quad BU\subset U.
\end{equation}

\begin{proposition}[The eikonal core and its grading]\label{prop:nr-eikonal-core}
Assume $T|_{A_{\mathrm{ess}}}$ is not scalar.
Then \(H=0\), \(Z\ne0\), and the subalgebra \(B=E_T(1)\)
is eikonal. Moreover, $U^2\subset B$, so $A=B\oplus U$ is a
$\mathbb Z_2$-grading.
\end{proposition}

\begin{proof}
\Cref{lem:nr-negative-spaces}(i) and
\cref{thm:nonradial-spectral} give \(H=0\).
If \(Z=0\), the proof of \cref{thm:nonradial-spectral} would give
\(B=\R c\), so \(Z\ne0\). Fusion and invariance give
\[
 zz'=-\ip{z}{z'}c,\qquad ZW\subset W,\qquad
 w^2=\frac12|w|^2c+z_w,\quad z_w\in Z.
\]
For a unit \(z\in Z\), the self-adjoint map \(L_z\) preserves
\(W\). If \(w\in W\) is an eigenvector with eigenvalue \(y\),
then \(L_{e_\pm}w=(-1/4\pm\sqrt3y/2)w\), where
\(e_\pm=(-c\pm\sqrt3z)/2\). Since \(w\perp e_\pm\) and
\(Tw=w\), the quotient formula gives
\[
 p(-1/4+\sqrt3y/2)=p(-1/4-\sqrt3y/2)=1.
\]
The roots of \(p(t)-1\) are \(-1,-1/2,1/2\). Two such roots
with sum \(-1/2\) must be \(-1,1/2\), so \(y^2=3/4\).
Consequently \((L_z|_W)^2=3\Id/4\).
Scaling and the norm identity at \(w\in W\) yield
\[
 |zw|^2=\frac34|z|^2|w|^2,\qquad |z_w|^2=\frac34|w|^4.
\]
For \(b=tc+z+w\), its square has the orthogonal components
\[
 b^2=\left(t^2-|z|^2+\frac12|w|^2\right)c
          +(-2tz+z_w)+(tw+2zw).
\]
On taking norms, the terms linear in \(t\) cancel because
\(\ip{z}{z_w}=\ip{zw}{w}\). The displayed norm identities then
give \(|b^2|^2=|b|^4\). Differentiating within the metrized
subalgebra \(B\) proves \(b^3=|b|^2b\).

For the grading, use $U=P\oplus Q$ from \eqref{eq:nr-U-peirce}.
For the input eigenvalue pairs $(1/3,1/3)$, $(-1/6,-1/6)$,
and $(1/3,-1/6)$, the fusion rule \eqref{eq:nr-fusion}
permits, among the eigenvalues on $c^\perp$, respectively no output,
only $-1$, and only $1/2$.
Metric invariance therefore gives
\[
 p^2=\tfrac13|p|^2c,\qquad q^2\in\R c\oplus Z,\qquad pq\in W
 \qquad(p\in P,\ q\in Q).
\]
Polarization yields $U^2\subset B$; the other two grading relations
are already contained in \eqref{eq:nr-B-subalgebra}.
\end{proof}

At the maximal unit idempotent $c$, the decomposition is therefore
\[
 B=\R c\oplus Z\oplus W,\qquad U=P\oplus Q,
\]
with $L_c$-eigenvalues $1,-1,1/2,1/3,-1/6$, respectively, and $Z\ne0$.
To determine the action of $B$ on $U$, set $M_b=L_b|_U$ and
$\{A_1,A_2\}=A_1A_2+A_2A_1$.

\begin{proposition}\label{prop:nr-representation-identities}
For $b,d\in B$,
\begin{equation}\label{eq:nr-representation-quadratic}
\begin{aligned}
 M_{b^2}&=6M_b^2-\tfrac13|b|^2\Id_U,\\
 M_{bd}&=3\{M_b,M_d\}-\tfrac13\ip bd\Id_U.
\end{aligned}
\end{equation}
\begin{equation}\label{eq:nr-representation-cubic}
 108M_b^3-9|b|^2M_b-\ip{b^2}b\Id_U=0.
\end{equation}
\end{proposition}
\begin{proof}
Every unit idempotent $e\in B$ satisfies $\ip{e^2}e=1$, so it
attains the global spherical maximum in
\eqref{eq:nr-spherical-maximum}.
The maximality bounds and \cref{lem:nr-idempotent-quotient}, applied
on $U\subset e^\perp$, give
\[
 \operatorname{Spec}M_e\subset\{1/3,-1/6\}.
\]
Indeed, the third root $-7/6$ of $p(t)=7/27$ is below the lower
bound $-1$. Hence
$M_e=6M_e^2-\Id_U/3$.

Consider the homogeneous quadratic operator polynomial \(\mathscr Q(b)=M_{b^2}-6M_b^2+\tfrac13|b|^2\Id_U\).
Every nonzero $b\in B$ lies in an eikonal subalgebra of dimension
at most two. Indeed, differentiating $b^3=|b|^2b$ in direction $b^2$
gives
\[
 (b^2)^2=2\ip{b^2}b\,b-|b|^2b^2,
\]
so $\operatorname{span}\{b,b^2\}$ is such a subalgebra.
In dimension one it is generated by a unit idempotent, and hence
$\mathscr Q$ vanishes there.
In dimension two choose a positive spherical maximum $e$ for its cubic.
Since $|e^2|=1$, the multiplier equation gives $e^2=e$ without
rescaling. For a unit vector $z\perp e$, write
$ez=\alpha z$, $z^2=\alpha e+\beta z$.
Linearization at $e$ gives $2\alpha^2+\alpha=1$, while $z^3=z$
gives $\alpha\beta=0$ and $\alpha^2+\beta^2=1$.
Thus $\alpha=-1$, $\beta=0$, and
\[
 e,\qquad \frac{-e+\sqrt3z}{2},\qquad
 \frac{-e-\sqrt3z}{2}
\]
are three pairwise nonproportional unit idempotents.
A homogeneous quadratic polynomial on a plane vanishing on three
distinct lines is zero. Applying this to each matrix entry of
$\mathscr Q$ proves $\mathscr Q(b)=0$.
Polarization gives the second quadratic identity.
Finally, apply it to $b,d=b^2$. Since $b^3=|b|^2b$,
\[
 |b|^2M_b
 =3\{M_b,6M_b^2-\tfrac13|b|^2\Id_U\}
       -\tfrac13\ip{b^2}b\Id_U
 =36M_b^3-2|b|^2M_b-\tfrac13\ip{b^2}b\Id_U,
\]
which gives the cubic identity.
\end{proof}

The remaining identities on $U$ are as follows.

\begin{lemma}\label{lem:nr-U-cubic}
For $u\in U$, set $s=|u|^2$, $w=u^2\in B$, and $v=wu\in U$.
Then
\begin{equation}\label{eq:nr-U-identities}
 v=\tfrac s9u,\qquad |w|^2=\tfrac{s^2}9,\qquad
 w^2=\tfrac s3w,\qquad uv=\tfrac s9w.
\end{equation}
Moreover, for every $b\in B$,
\begin{equation}\label{eq:nr-mixed-polarization}
 -3wb+18u(bu)=sb.
\end{equation}
\end{lemma}
\begin{proof}
Pair the first identity in \eqref{eq:nr-representation-quadratic}
with $u$ and differentiate in $b$ to obtain
\eqref{eq:nr-mixed-polarization}.
At $u$, \eqref{eq:nr-vector-identity} reads
$4uv+w^2=7sw/9$.
Together with \eqref{eq:nr-mixed-polarization} at $b=w$, this gives
$uv=sw/9$ and $w^2=sw/3$.
Since $B$ is eikonal, $|w^2|^2=|w|^4$, and therefore
\[
 |u^2|^2\left(|u^2|^2-\tfrac19|u|^4\right)=0
 \qquad(u\in U).
\]
The first factor is not the zero polynomial: $U\ne0$ and the
only possible eigenvalues $1/3,-1/6$ of $M_c$ are nonzero, so
$\ip{u^2}c=\ip{M_cu}u$ is not identically zero.
The polynomial ring on $U$ is an integral domain, whence
$|u^2|^2=|u|^4/9$ identically.
Differentiating this identity within $U$ gives $4wu=4su/9$.
\end{proof}

\begin{proposition}[Exactness of the core]\label{prop:nr-core-exactness}
If the original cubic is irreducible over $\R$, then $B$ is exact,
\[
 \tr_B L_b=\tr_U M_b=0\qquad(b\in B),
\]
and $\dim B=\dim U+2$.
\end{proposition}
The proof is given in \cref{app:nr-core-exactness-proof}.

\subsection{The intrinsic five-generator closure}
\label{subsec:nr-five-generators}

Use the normalization in \cref{lem:nr-max-idempotent}.

\begin{theorem}[Five generators in the nonradial branch]
\label{thm:nonradial-five-generators}
Let $F$ be a nonzero irreducible homogeneous harmonic cubic on a
Euclidean space $V=V_{\mathrm{ess}}(F)$, with minimal regular zero set
and nonradial quadratic weight. After a constant rescaling of $F$, let
$V=B\oplus U$ be the orthogonal splitting of
\cref{thm:nonradial-spectral}. For $x=b+u$, define
\begin{equation}\label{eq:nr-five-invariants}
 r=|b|^2,\qquad s=|u|^2,\qquad
 a=\ip{b^2}b,\qquad p=\ip{bu}u,\qquad h=|bu|^2.
\end{equation}
Then
\begin{equation}\label{eq:nr-five-algebra}
 \cA_R(F)=\R[r,s,a,p,h],\qquad
 R=r+s,\qquad F=\frac{a+3p}{6}.
\end{equation}
These five polynomials are algebraically independent. In particular,
the smallest Laplacian algebra of $F$ requires exactly five ordinary
algebra generators, of degrees $2,2,3,3,4$.
\end{theorem}

\begin{proof}
By \cref{thm:nonradial-spectral,prop:nr-eikonal-core},
both $B$ and $U$ are nonzero, $B$ is eikonal, and
$B^2\subset B$, $BU\subset U$, $U^2\subset B$.
By \cref{prop:nr-core-exactness}, irreducibility makes $B$ exact,
and the symmetric operators $M_b=L_b|_U$ satisfy
\begin{equation}\label{eq:nr-five-traces}
 \tr_B L_b=0,\qquad \tr_U M_b=0\qquad(b\in B).
\end{equation}
Write $m=\dim U$, $n=\dim B=m+2$, and $w=u^2\in B$.
The identities of
\cref{prop:nr-representation-identities,lem:nr-U-cubic}, together with
\eqref{eq:nr-mixed-polarization}, give
\begin{align}
 M_{b^2}&=6M_b^2-\frac r3\Id_U,
 &108M_b^3&=9rM_b+a\Id_U,\label{eq:nr-five-representation}\\
 |w|^2&=\frac{s^2}{9},
 &w^2&=\frac s3w,\notag\\
 M_wu&=\frac s9u,
 &18u(bu)&=3bw+sb.\label{eq:nr-five-mixed}
\end{align}
The eikonal identity on $B$ gives $b(b^2)=rb$ and $|b^2|^2=r^2$.

\emph{Differential closure.}
The five bidegrees with respect to $B\oplus U$ are, respectively,
\[
 (2,0),\quad(0,2),\quad(3,0),\quad(1,2),\quad(2,2).
\]
Thus, for any of the
five generators $f$ of bidegree $(i,j)$,
\[
 \Gamma(r,f)=2if,\qquad \Gamma(s,f)=2jf.
\]
All remaining gradient products are
\begin{equation}\label{eq:nr-five-gamma}
\begin{aligned}
 \Gamma(a,a)&=9r^2,\\
 \Gamma(a,p)&=18h-rs,\\
 \Gamma(p,p)&=\frac{s^2}{9}+4h,\\
 \Gamma(a,h)&=rp+\frac{sa}{3},\\
 \Gamma(p,h)&=\frac{rp}{3}+\frac{2sp}{9}+\frac{sa}{27},\\
 \Gamma(h,h)&=\frac{r+s}{3}h+\frac{p^2}{9}+\frac{ap}{27}.
\end{aligned}
\end{equation}
To verify the table intrinsically, metric invariance and
\eqref{eq:nr-five-mixed} give the gradients, written in their
$B\oplus U$ components,
\begin{equation}\label{eq:nr-five-gradients}
 \nabla a=(3b^2,0),\qquad
 \nabla p=(w,2M_bu),\qquad
 \nabla h=\left(\frac{bw}{3}+\frac{sb}{9},\,2M_b^2u\right).
\end{equation}
Their pairings, using \eqref{eq:nr-five-representation}, yield the
first five lines of \eqref{eq:nr-five-gamma}.
For the last line, polarize the eikonal identity on $B$ at $w$ to obtain
\[
 2w(wb)+w^2b=|w|^2b+2\ip wb w.
\]
Pairing with $b$ and using
$\ip{b^2}w=6h-rs/3$ gives
\[
 |bw|^2=\frac{s^2r}{9}+p^2-sh.
\]
Consequently
\[
 |\nabla_Bh|^2=\frac{sh}{3}+\frac{p^2}{9},\qquad
 |\nabla_Uh|^2=4\ip{M_b^4u}u
             =\frac{rh}{3}+\frac{ap}{27},
\]
which proves the final pairing.

The Laplacians are
\begin{equation}\label{eq:nr-five-laplacians}
 \Delta r=2(m+2),\qquad \Delta s=2m,\qquad
 \Delta a=\Delta p=0,\qquad
 \Delta h=\frac{mr+(m+2)s}{9}.
\end{equation}
To verify these formulas, exactness of $B$ implies
$\sum_{\alpha=1}^n e_\alpha^2=0$ for an orthonormal basis of $B$.
Summing the quadratic identity in \eqref{eq:nr-five-representation},
and separately taking its trace, gives
\[
 \sum_{\alpha=1}^n M_{e_\alpha}^2=\frac n{18}\Id_U,
 \qquad \tr M_b^2=\frac m{18}r.
\]
Consequently,
\[
 \Delta_Bh=2\sum_{\alpha=1}^n|M_{e_\alpha}u|^2=\frac n9s,
 \qquad \Delta_Uh=2\tr M_b^2=\frac m9r.
\]
The other formulas in \eqref{eq:nr-five-laplacians} follow from
\eqref{eq:nr-five-traces}. The differential chain rule now proves
that $\R[r,s,a,p,h]$ is a Laplacian algebra.

\emph{Recovery from $R$ and $F$.}
The grading gives $F=(a+3p)/6$. Applying
\eqref{eq:nr-five-gamma} and \eqref{eq:nr-five-laplacians} gives
\begin{equation}\label{eq:nr-five-gradient-square}
 S=\frac14(r-s/3)^2+4h,\qquad
 \Delta S=\frac{(10m+36)r+(2m+4)s}{9}.
\end{equation}
Since $R=r+s$ and $m>0$, we recover the two radii directly:
\begin{equation}\label{eq:nr-five-radius-recovery}
 r=\frac{9\Delta S-(2m+4)R}{8(m+4)},\qquad s=R-r.
\end{equation}
Both belong to $\cA_R(F)$, because $S=\Gamma(F,F)$ and
$\Delta S$ do. Setting $D=\Gamma(r,F)=a+p$ then gives
\[
 p=3F-\frac D2,\qquad
 a=\frac{3D}{2}-3F,\qquad
 h=\frac{\Gamma(a,p)+rs}{18}.
\]
Thus all five generators belong to $\cA_R(F)$. Since their algebra
is Laplacian and contains $R,F$, equality follows.

\emph{Algebraic independence.}
At the maximal unit idempotent $c$,
\cref{prop:nr-eikonal-core} supplies a nonzero $(-1)$-Peirce space
$Z\subset B$. Choose a unit vector $z\in Z$. Then $z^2=-c$, so at
$b=z$ one has $r=1$, $a=0$, and
\[
 \nabla_B r=2z,\qquad \nabla_B a=-3c.
\]
In particular, the differential of $(r,a)$ has rank two.
By \eqref{eq:nr-five-representation}, every eigenvalue of $M_z$ lies in
\[
 \{\lambda_1,\lambda_2,\lambda_3\}
   =\left\{-\frac1{2\sqrt3},0,\frac1{2\sqrt3}\right\}.
\]
Let their multiplicities be $k_1,k_2,k_3$, allowing zero provisionally.
The identities $\tr M_z=0$, $\tr M_z^2=m/18$, and
$k_1+k_2+k_3=m$ imply $k_1=k_2=k_3=m/3>0$.
Choose $u=u_1+u_2+u_3$ with a nonzero component in each eigenspace.
For fixed $b=z$, the map from $t_i=|u_i|^2$ to the remaining invariants
is
\[
 \begin{pmatrix}s\\p\\h\end{pmatrix}
 =
 \begin{pmatrix}
  1&1&1\\
  \lambda_1&\lambda_2&\lambda_3\\
  \lambda_1^2&\lambda_2^2&\lambda_3^2
 \end{pmatrix}
 \begin{pmatrix}t_1\\t_2\\t_3\end{pmatrix}.
\]
The Vandermonde matrix is invertible, and the differentials $dt_i$
are independent because $u_i\ne0$. Thus $(s,p,h)$ has rank three
in the $U$ directions, whereas $(r,a)$ depends only on $b$ and has
rank two. The full five-invariant map has rank five at this point.
Its components are therefore algebraically independent. Their
polynomial algebra has transcendence degree five and cannot be
generated by fewer than five elements.
\end{proof}

\begin{corollary}[Cartan reduction and ambient dimensions]
\label{cor:nr-cartan-dimensions}
Under the hypotheses and normalization of
\cref{thm:nonradial-five-generators}, the cubic
\begin{equation}\label{eq:nr-cartan-reduction}
 H=a-9p=6\bigl(\Gamma(r,F)-5F\bigr)
\end{equation}
satisfies
\[
 \Delta H=0,\qquad |\nabla H|^2=9R^2.
\]
It is orthogonally equivalent to a normalized Cartan cubic, and
\[
 \dim V\in\{8,14,26\}.
\]
Writing $C=H/6$ and $C_B=C\circ\operatorname{pr}_B$, one has
\begin{equation}\label{eq:nr-cartan-reconstruction}
 F=\frac{4C_B-C}{3}.
\end{equation}
\end{corollary}

\begin{proof}
The formula for $H$ follows from $6F=a+3p$ and
$\Gamma(r,F)=a+p$. By \eqref{eq:nr-five-gradients},
\[
 \nabla H=(3b^2-9u^2,-18bu).
\]
The identities $|b^2|^2=r^2$, $|u^2|^2=s^2/9$, and
$\ip{b^2}{u^2}=6h-rs/3$ therefore give
\[
 |\nabla H|^2
 =9r^2+9s^2-54(6h-rs/3)+324h=9(r+s)^2.
\]
Also $\Delta H=0$ by \eqref{eq:nr-five-laplacians}.

Put $m=\dim U$. The multiplicity calculation in the proof of
\cref{thm:nonradial-five-generators} gives $m/3\in\mathbb N$
and $m>0$. Since $\dim B=m+2$, it follows that
$\dim V=2m+2\ge8$.
Cartan's classification of harmonic cubic eikonals
\cite{Cartan39}, or \cite[Theorem~1.1]{TkachevCartan10},
now identifies $H$ with a Cartan cubic in dimension $5,8,14$, or
$26$. In the latter theorem, the additional reducible possibility
$x_N^3-3x_N\sum_{i<N}x_i^2$ has Laplacian $6(2-N)x_N$,
so harmonicity and $N\ge8$ exclude it. The same dimension bound
excludes $N=5$. Finally, $C_B=a/6$ and $C=(a-9p)/6$ give
\eqref{eq:nr-cartan-reconstruction}.
\end{proof}

The reduction recovers $F$ from a Cartan cubic and a compatible
splitting $B\oplus U$. Classifying $F$ up to equivalence therefore
also requires classifying these splittings.

The models in \cref{sec:nonradial-models} realize the sharp bound
and express \(r,s\) and the other invariants in coordinates.
The programs check their closure, recovery, and rank-five minor;
the universal spectral and trace arguments are proved above and in
\cref{app:cubic-structure}.

\begin{proof}[Proof of \cref{thm:harmonic-cubic-five-bound}]
The divisibility argument in \cref{subsec:nr-maximal-idempotents}
and the cubic correspondence associate an exact Hsiang algebra to $F$.
In the radial case, write $\cL(F)=-\gamma RF$, where $\gamma>0$
by the maximal normalization above. Harmonicity gives
$\Gamma(F,S)=2\gamma RF$. Replacing $F$ by
$\widetilde F=\sqrt{2/\gamma}\,F$ gives
$\Gamma(\widetilde F,|\nabla\widetilde F|^2)=4R\widetilde F$,
as required by \cref{prop:cubic-scalar-killing-closure}; this rescaling
does not change $\cA_R(F)$.
The branch argument in \cref{rem:cubic-harmonic-generator-bound}
then gives the radial four-generator bound.

In the nonradial case, \cref{thm:nonradial-five-generators} gives the
polynomial algebra with the stated generator degrees. Restriction to
$R=1$ reduces its quotient dimension from five to four. The examples
in \cref{prop:nonradial-algebra} establish sharpness.
\end{proof}

\section[Cubic Subcalibration Criteria]{Cubic Subcalibration Criteria and Radial Models}
\label{sec:cubic-potentials}

The radial table of \cref{sec:cubic-algebraic-theory} gives a degree-six
subcalibration from two invariant inequalities or one sufficient trace
bound. An additional identity yields a degree-seven polynomial alternative.

\subsection{Existence from the radial cubic table}
\label{sec:cubic-unified-sextic}

The common defect formula also applies to the nonradial models of
\cref{sec:nonradial-examples}, using their five-generator table.

\begin{lemma}[The common harmonic potential]
\label{lem:cubic-harmonic-common-potential}
Let \(F\not\equiv0\) be a homogeneous harmonic cubic on \(\R^N\),
\(N\ge2\), satisfying \(\cL(F)=-qF\) for a homogeneous quadratic polynomial
\(q\). Put \(K=\Delta S\) and \(T=\Gamma(S,S)\). For
\(G=F\sqrt{RS}\), there is a homogeneous polynomial
\(\mathcal P_F\) of degree eighteen such that
\begin{equation}\label{eq:cubic-harmonic-common-defect}
 \cL(G)=\frac{F\mathcal P_F}{16\sqrt R\,S^{3/2}}
 \qquad\text{on }\{S>0\}.
\end{equation}
If \(\mathcal P_F\ge0\) there and
\(Z=\{S=0\}\) has locally zero \(\cH^{N-1}\)-measure, then
\(G\) generates a two-sided subcalibration. The phase \(\{F<0\}\)
is locally perimeter minimizing and its boundary current is locally
mass minimizing, with support \(\{F=0\}\).
\end{lemma}

\begin{proof}
Using \(\Gamma(F,S)=2qF\) and homogeneity gives, on \(S>0\),
\begin{align}
 Q_G:=|\nabla G|^2
 &=RS^2+F^2\left[R\left(2q+\frac{T}{4S}\right)+11S\right],
 \label{eq:cubic-harmonic-common-gradient}\\
 \Delta G
 &=G\left[\frac{N+9}{R}+\frac{2q+K/2}{S}
                         -\frac{T}{4S^2}\right].
 \label{eq:cubic-harmonic-common-laplacian}
\end{align}
Substitution into \(\cL(G)=Q_G\Delta G-\Gamma(G,Q_G)/2\)
gives \eqref{eq:cubic-harmonic-common-defect} with a polynomial
\(\mathcal P_F\). Its explicit coefficient formula is
\(\mathcal P_F=C_0+FC_1+F^2C_2\), with
\eqref{eq:cubic-nonharmonic-coefficients} specialized to
\(t=0\) and \(H_2=|\Hess F|^2=K/2\).
\Cref{app:nonharmonic-cubic-certificate} gives the full expansion
and derivation.
Since \(Z\subset\{F=0\}\) by Euler's identity, \(G\) extends
continuously with the signs of \(F\). Its positive multiplier and degree
six give \(\nabla G\ne0\) off \(Z\) by the homogeneous-multiplier
verification following \cref{cor:gradient-subcalibration}.
That corollary and \cref{thm:removable-subcalibration} give comparison;
harmonicity identifies the full support by
\cref{rem:polynomial-finite-perimeter}.
\end{proof}

By \cref{prop:cubic-weight-in-minimal-algebra} and closure,
\(\mathcal P_F\in\cA_R(F)\); its sign is tested on the actual generator
image. No separate condition \eqref{eq:cubic-nonharmonic-PSD} is needed.

For the radial existence criterion, use \(D\) from
\eqref{eq:cubic-unified-D} and put \(K=\Delta S=2\tr(L_x^2)\).
By \cref{lem:cubic-radial-transports}, the criterion depends only
on the radial table and \(K,R,S,D\).

\begin{theorem}[Subcalibrations from the radial cubic table]
\label{thm:cubic-unified-sextic}
Let \(F\not\equiv0\) be a homogeneous harmonic cubic on a Euclidean
space of dimension \(N\ge8\), with \(\Gamma(F,S)=4RF\). Define \(D\)
by \eqref{eq:cubic-unified-D} and put \(K=\Delta S\).
Assume, at every point,
\begin{equation}\label{eq:cubic-unified-trace-bound}
 K\ge16R
\end{equation}
and
\begin{equation}\label{eq:cubic-unified-H}
 \mathscr H:=(N-12)S^2+\frac12R(K-16R)S-12RD\ge0.
\end{equation}
Then the phase \(\{F<0\}\) is locally perimeter minimizing, and its
oriented boundary current is strictly area minimizing with support
\(\{F=0\}\). These conclusions are realized by the degree-six potential
\[
 G=F\sqrt{RS}
   =\frac1{12}|x|\,|x^2|\langle x^2,x\rangle
\]
which is continuous on \(V\), smooth on \(V\setminus Z\), and has
\(\{G<0\}=\{F<0\}\), where \(Z=\{S=0\}\).
Its gradient is nonzero off \(Z\), and its normalized gradient is a
two-sided subcalibration whose divergence has the strict sign of \(F\)
when \(F\ne0\).
\end{theorem}

The polynomial \(\mathscr H\) has a direct boundary interpretation.
Apply \cref{prop:harmonic-power-Jacobi} with \(e=3\),
\(\sigma=1\), \(\gamma=1/2\), \(q=2R\), and
\(T=4RS+12D\). Since \(w=1/(rS)\), this gives
\begin{equation}\label{eq:cubic-H-Jacobi-meaning}
 \left.\mathscr H\right|_{C_{\reg}}
 =RS^2\left(-\frac{J_Cw}{w}\right).
\end{equation}
Since \(G\) has degree six, \cref{thm:ground-state-gap} also gives
\[
 Q_C(u)\geq\left(6-\frac{N-1}{2}\right)^2\int_Cr^{-2}u^2,
 \qquad u\in C_c^\infty(C_{\reg}\setminus\{0\}).
\]
This proves strict stability whenever \(N\ne13\), including when
the link is singular.

\begin{lemma}[Bounds for radial cubics]
\label{lem:cubic-unified-bounds}
Assume the structural hypotheses of \cref{thm:cubic-unified-sextic},
without its dimension and inequality assumptions. Then
\begin{equation}\label{eq:cubic-unified-bounds}
 D\ge0,\qquad S\le\frac{R^2}{3},\qquad
 3RD\le S^2,\qquad 9D\le RS,\qquad F^2\le D.
\end{equation}
\end{lemma}

\begin{proof}
By \cref{lem:cubic-radial-transports}, the last identity in
\eqref{eq:cubic-unified-table} gives \(D\ge0\) on \(S>0\).
At \(S=0\), the nonnegative function \(S\) has zero gradient, and \eqref{eq:cubic-unified-D} gives \(D=0\).
Take a positive maximum of \(S\) on \(R=1\). Homogeneity gives
\(\nabla S=4Sx\) there. Pairing with \(\nabla D\) and using
Euler's identity gives \(24SD=8D\). Thus \(S=1/3\) if \(D>0\).
If \(D=0\), then \(16S^2=\Gamma(S,S)=4S\), so \(S=1/4\).
Homogeneity proves the bound for \(S\).

Cauchy--Schwarz gives \((8RD)^2\le(4RS+12D)(4SD)\).
When \(D>0\), this is
\(D(4R^2-3S)\le RS^2\). Since \(3S\le R^2\), it implies
\(3RD\le S^2\). The assertion is immediate when \(D=0\).
Similarly, Euler's identity and \eqref{eq:cubic-unified-table} give
\[
 36D^2=\langle x,\nabla D\rangle^2\le4RSD,
 \qquad
 4S^2F^2=\Gamma(F,D)^2\le4S^2D.
\]
These prove the final two bounds wherever cancellation is possible.
At \(S=0\), Euler's identity for \(F\) gives \(F=0\), so the
remaining zero cases also follow.
\end{proof}

\begin{proof}[Proof of \cref{thm:cubic-unified-sextic}]
Specializing \cref{lem:cubic-harmonic-common-potential} to
\(q=2R\) and \(T=4RS+12D\) gives, on \(S>0\),
\begin{align}
 Q:=|\nabla G|^2
 &=RS^2+F^2\left(\frac{3DR}{S}+5R^2+11S\right),
 \label{eq:cubic-unified-gradient}\\
 \Delta G
 &=G\left(\frac{N+9}{R}+\frac{K}{2S}
                     +\frac{3R}{S}-\frac{3D}{S^2}\right).
 \label{eq:cubic-unified-laplacian}
\end{align}
In particular, \(Q\ge RS^2>0\). Substituting these formulas into
\(\cL(G)=Q\Delta G-\Gamma(G,Q)/2\) yields
\begin{equation}\label{eq:cubic-unified-defect}
 \cL(G)=\frac{G}{RS^2}
             \left(RS^2\mathscr H+F^2\mathscr J\right),
\end{equation}
where
\begin{align}
 \mathscr J={}&RD\left[\tfrac32R(K-16R)
                  +15(R^2-3S)+3(N-8)S\right]\notag\\
 &+\tfrac52(K-16R)R^3S+45R^4S
             +\tfrac{11}{2}(K-16R)RS^2\notag\\
 &+(5N+33)R^2S^2+(11N+44)S^3.
 \label{eq:cubic-unified-J}
\end{align}
This calculation uses only \cref{lem:cubic-radial-transports} and
homogeneity; no derivative of \(K\) is involved. Every term in \(\mathscr J\)
is nonnegative by \cref{lem:cubic-unified-bounds},
\eqref{eq:cubic-unified-trace-bound}, and \(N\ge8\). The term
\(45R^4S\) is positive. Hence \(F\cL(G)>0\) wherever \(F\ne0\).

The critical set is removable by \cref{lem:cubic-trace-removability}.
Since \(\mathcal P_F=16(RS^2\mathscr H+F^2\mathscr J)\),
\cref{lem:cubic-harmonic-common-potential} gives the minimizing
and support conclusions.
Both open phases are nonempty because \(F\) is odd, and the strict
divergence sign makes both integrals in \eqref{eq:two-phase-flux}
positive. The conical sets \(E=\{F<0\},Z\) and zero-homogeneous field \(X_G\)
therefore satisfy \cref{thm:strict-minimality-flux}, proving strict
area minimality.
\end{proof}

\begin{corollary}[Existence from a trace bound]
\label{cor:cubic-unified-trace}
Under the structural hypotheses of \cref{thm:cubic-unified-sextic},
with \(N\ge8\), the single bound
\begin{equation}\label{eq:cubic-unified-simple-trace}
 \tr(L_x^2)\ge
       \max\left\{8,\frac{40-N}{3}\right\}|x|^2
\end{equation}
implies its two inequality assumptions and all its conclusions.
\end{corollary}

\begin{proof}
The bound implies \(K\ge16R\). By
\cref{lem:cubic-unified-bounds},
\[
 \mathscr H\ge(N-16)S^2+\tfrac12R(K-16R)S.
\]
For \(N\ge16\) this is nonnegative. For \(8\le N<16\), the bound
\(S\le R^2/3\), with the negative coefficient \(N-16\), gives
\[
 \mathscr H\ge RS\left[\frac K2+\frac{N-40}{3}R\right]\ge0.
\]
\end{proof}

The trace test allows nonscalar \(K\), without a closure condition
on \(\Delta D\). The low-dimensional Clifford endpoint instead uses
\eqref{eq:cubic-unified-H} directly.

\subsection{Polynomial alternatives of degree seven}
\label{sec:cubic-algebraic-existence}

To obtain polynomial potentials in place of \(F\sqrt{RS}\), we use
the degree-seven potential \(SF\). Its defect identity holds on the full
radial cubic table. We first impose the additional relation \(D=F^2\);
\cref{prop:exc-SF-certificate} instead assumes a scalar trace coefficient
and allows an independent \(D\).

\begin{lemma}[The gradient-square defect]
\label{lem:cubic-SF-defect}
Let \(F\) be a nonzero homogeneous harmonic cubic satisfying
\(\Gamma(F,S)=4RF\), put \(K=\Delta S\), and define \(D\) by
\eqref{eq:cubic-unified-D}.
For \(G=SF\), one has
\begin{equation}\label{eq:cubic-SF-general-gradient}
 Q:=|\nabla G|^2=S^3+12F^2(RS+D),\qquad
 \Delta G=(K+8R)F,
\end{equation}
and
\begin{equation}\label{eq:cubic-SF-general-defect}
 \begin{aligned}
 \frac{\cL(G)}F={}&(K-16R)S^3-30DS^2\\
   &+F^2\bigl[(12K-72R)D+12RKS-96S^2\bigr].
 \end{aligned}
\end{equation}
The quotient denotes its polynomial extension across \(F=0\).
No scalar-trace condition or identity for \(\Delta D\) is required.
\end{lemma}

\begin{proof}
Apply the product rule and \cref{lem:cubic-radial-transports} to
\(Q=|\nabla(SF)|^2\), then substitute in
\(\cL(SF)=Q\Delta(SF)-\Gamma(SF,Q)/2\).
The individual pairings are recorded in \cref{app:application-details}.
\end{proof}

\begin{corollary}[A polynomial gradient-square certificate]
\label{thm:cubic-algebraic-trace-certificate}
Let \(V\) be a Euclidean commutative metrized algebra of dimension
\(N\ge3\). Suppose that, for every $x\in V$,
\begin{align}
 \tr L_x&=0,\label{eq:cubic-algebraic-trace-zero}\\
 \langle x^2,x^3\rangle
 &=\frac43|x|^2\langle x^2,x\rangle,
 \label{eq:cubic-algebraic-radial-identity}\\
 |x^3|^2&=|x|^2|x^2|^2+
                  \frac13\langle x^2,x\rangle^2,
 \label{eq:cubic-algebraic-norm-identity}\\
 \tr(L_x^2)&\ge8|x|^2.
 \label{eq:cubic-algebraic-trace-bound}
\end{align}
Set
\[
 F(x)=\frac16\langle x^2,x\rangle,\qquad
 R=|x|^2,\qquad S=|\nabla F|^2,\qquad G=SF.
\]
Then $G$ is a polynomial potential of degree seven with
\[
 \{G<0\}=\{F<0\},\qquad
 \{\nabla G=0\}=\{\nabla F=0\}.
\]
Its normalized gradient is a two-sided subcalibration away from
this common critical set, which has locally zero
$\cH^{N-1}$-measure. Hence
$\partial\llbracket\{F<0\}\rrbracket$ is locally mass minimizing.
The divergence has a strict sign at every point with $F\ne0$.
\end{corollary}

\begin{proof}
The multiplication dictionary turns the three identities into
\(\Delta F=0\), \(\Gamma(F,S)=4RF\), and
\(\Gamma(S,S)=4RS+12F^2\), so \(D=F^2\).
Bochner's identity gives \(K=\Delta S=2\tr(L_x^2)\ge16R\),
and \cref{lem:cubic-unified-bounds} gives \(S\le R^2/3\).
Substituting these facts in \cref{lem:cubic-SF-defect} yields the
nonnegative decomposition
\eqref{eq:cubic-variable-trace-positive-decomposition}, which is
strict wherever \(F\ne0\). The same gradient identity identifies
the exceptional set with \(\{\nabla F=0\}\), and
\cref{lem:cubic-trace-removability,cor:gradient-subcalibration}
give the minimizing conclusion. The coefficient decomposition is proved in
\cref{app:application-details}.
\end{proof}

\begin{remark}[Relation to gradient-square closure]
\label{rem:cubic-algebraic-certificate-scope}
The norm identity \eqref{eq:cubic-algebraic-norm-identity} is exactly
\(D=F^2\). In the three-generator closure of
\cref{sec:gradient-square-rank-two},
\cref{prop:cubic-singular-normalization} gives \(K=4kR\), so the
trace bound becomes \(k\ge4\). The quaternionic rank-four trace
model has \(k=5\) and satisfies these hypotheses
(\cref{cor:quaternionic-trace-cubic}). Unlike this certificate,
\cref{thm:cubic-unified-sextic} allows an independent \(D\);
\(SF\) provides a polynomial alternative when the displayed defect
is nonnegative.
\end{remark}

The cubic--algebra dictionary and the multiplier obstruction do not
require harmonicity. For a general quadratic weight in
\(\cL(F)=-qF\), \cref{app:nonharmonic-cubic-certificate} retains
the trace terms in the defect of \(F\sqrt{RS}\).
\Cref{thm:cubic-nonharmonic-certificate} gives a sufficient
polynomial positivity test.

\paragraph{Applications to the classical cubic models.}
The normal forms in \cref{sec:cubic-history} give three-generator
rank-four trace cubics and four-generator complexified rank-three Jordan
norms and \(G_2\) trilinear cubics. The table collects their
\(F\sqrt{RS}\) data and the nonradial model of
\cref{sec:nonradial-examples}.

For the Cartan--M\"unzner cubics \(\Phi\) in dimensions \(14,26\),
set \(F=\Phi/(3\sqrt3)\). Then \(S=R^2/3\), \(D=R^3/27\),
and \(K=4(N+2)R/3\). The common potential is proportional to
\(R^{3/2}F\); the polynomial alternative \(RF\) is given in
\cref{thm:iso-polynomial-table,ex:harmonic-isoparametric-radial}.
\[
\begin{array}{l|r|c}
\text{model}&N&\text{verification data}\\ \hline
\text{Cartan, quaternionic}&14&K/R=64/3\\
\text{Cartan, octonionic}&26&K/R=112/3\\
\text{rank-four quaternionic trace}&27&K/R=20\\
\text{complexified rank-three Jordan norm},\ d=2&18&K/R=16\\
\text{complexified rank-three Jordan norm},\ d=4&30&K/R=24\\
\text{complexified rank-three Jordan norm},\ d=8&54&K/R=40\\
G_2\text{ trilinear}&21&K/R=24\\
\text{nonradial quaternionic model}&26&\text{five-generator closure}
\end{array}
\]
The first seven rows satisfy the radial trace bound
\(K/R\geq\max\{16,2(40-N)/3\}\).
The final row is verified in \cref{thm:nonradial-minimizing}.
In every row the boundary current is strictly area minimizing.

\subsection{Three-generator closure and rank-four trace cubics}
\label{sec:gradient-square-rank-two}

For the closure and realizations in
\cref{prop:cubic-singular-normalization,prop:cubic-gradient-classification},
we determine area minimization, trace models, and the sharp stability
constant. The degree-seven polynomial $G=SF$ complements the common
degree-six potential $F\sqrt{RS}$.

\subsubsection{The tangent obstruction and the existence theorem}

The following closure criterion retains its displayed generality:
for example, the reducible polynomial \(F=xyz\) on \(\R^3\)
satisfies \eqref{eq:cubic-normalized-closure} with \(k=1\), by
direct differentiation. This comparison endpoint is not subject to
the introductory irreducibility convention.

\begin{theorem}[Area minimization from gradient-square closure]
\label{thm:cubic-rank-two-existence}
Let \(F\) be a nonzero homogeneous cubic on \(\R^N\), \(N\geq3\),
and let \(S=|\nabla F|^2\).  Suppose that
\(\R[R,F,S]\) is a Laplacian algebra and that \(F\) has a
nonzero critical point.  Write \(N=6k-3\) as in
\cref{prop:cubic-singular-normalization}.  Then
\begin{equation}\label{eq:cubic-minimality-threshold}
 \partial\llbracket\{F<0\}\rrbracket\text{ is locally mass minimizing}
 \quad\Longleftrightarrow\quad k\geq4.
\end{equation}
Every minimizing member is strictly area minimizing in the sense of
\cref{def:four-variational-notions}, with polynomial potential
\(G=SF\) of degree seven. It is also strictly stable, with the
optimal Hardy constant determined in \cref{thm:sharp}.
\end{theorem}

\begin{proof}
Normalize as in \eqref{eq:cubic-normalized-closure} and suppose
\(k\geq4\). The dictionary in
\cref{prop:cubic-multiplication-dictionary} translates the normalized
gradient identities into the product identities of
\cref{thm:cubic-algebraic-trace-certificate}. Moreover, Bochner's
identity gives
\[
 K=\Delta S=4kR,\qquad
 \tr(L_x^2)=\tfrac12K=2kR\geq8R.
\]
That theorem therefore gives the polynomial subcalibration potential
\(G=SF\), including the required sign and the removability of its
critical set. Its divergence is strictly signed wherever \(F\ne0\).
Since \(G\) is homogeneous and polynomial and both phases are nonempty,
\cref{cor:polynomial-strict-minimality} gives the strict mass gap.

The sharp stability constant is computed in \cref{thm:sharp}.
Harmonicity and \cref{prop:minimizing-irreducibility} also show that
every minimizing member is irreducible.

For the converse, \cref{prop:cubic-gradient-classification} identifies
\(k=1\) with the reducible harmonic cubic \(xyz\), whose phases cannot
minimize perimeter by \cref{prop:minimizing-irreducibility}.
For \(k=2,3\), take a nonzero critical point \(p\).
The Taylor expansion in \cref{prop:cubic-singular-normalization} gives
\[
 r^{-2}F(p+rz)\longrightarrow
 \frac{|p|}{2}\bigl(|u|^2-|v|^2\bigr),
 \qquad z=(u,v,w),\quad u,v\in\R^k.
\]
The transverse quadratic cone is unstable, since its link has
\(|A|^2=2k-2\) and
\[
 \frac{(2k-3)^2}{4}-(2k-2)<0.
\]
By \cref{sec:tangent-instability-transfer}, the Taylor expansion gives
smooth regular graphical convergence. \Cref{lem:tangent-instability-transfer}
therefore makes these cases unstable and excludes area minimization.
\end{proof}

\subsubsection{The rank-four trace realizations}

The trace cubics on $\operatorname{Herm}_4(\mathbb F_d)_0$,
$d=1,2,4$, realize the three exceptional cases in dimensions
$9,15,27$ \cite[Section~6.11.6(iii)]{NadirashviliTkachevVladut14}.
They admit the following uniform calculation.

\begin{corollary}[The rank-four trace cubics]
\label{cor:quaternionic-trace-cubic}
Let \(\mathbb F_d=\R,\mathbb C,\mathbb H\), with real dimension
\(d=1,2,4\), and equip
\[
 V=\operatorname{Herm}_4(\mathbb F_d)_0,\qquad
 \dim_{\R}V=3+6d,
\]
with \(\langle A,B\rangle=\operatorname{Re}\tr(AB)\).
For \(F(A)=\tfrac13\operatorname{Re}\tr(A^3)\),
\(R=\tr(A^2)\), and \(S=|\nabla F|^2\), one has
\[
 \cA_R(F)=\R[R,F,S],\qquad
 \Delta F=0,\quad \Delta S=4(d+1)R,
\]
\[
 \Gamma(F,S)=4RF,\qquad \Gamma(S,S)=4RS+12F^2.
\]
The generators are algebraically independent. The oriented zero cone
is area minimizing if and only if \(d=4\); in that case \(G=SF\)
proves strict area minimality. The common potential \(F\sqrt{RS}\)
also gives a subcalibration in this case.
\end{corollary}

\begin{proof}
Projection onto the trace-free subspace gives
\[
\begin{gathered}
 \nabla F=A^2-\tfrac R4\Id,\qquad
 S=\tr(A^4)-\tfrac{R^2}{4},\\
 (\Hess F)_A(B)=AB+BA-\tfrac12\operatorname{Re}\tr(AB)\Id.
\end{gathered}
\]
Write \(p_j=\sum_i\lambda_i^j\) after diagonalizing \(A\), where
\(p_1=0\). If \(P\) is the orthogonal projection onto the
sum-zero diagonal vectors, the diagonal Hessian is
\(2P\operatorname{diag}(\lambda_i)P\), with trace zero and
squared Hilbert--Schmidt norm \(2R\). Each off-diagonal entry
contributes \(d\) eigenvalues \(\lambda_i+\lambda_j\). Thus
\[
 \Delta F=0,\qquad
 |\Hess F|^2=2R+d\sum_{i<j}(\lambda_i+\lambda_j)^2
             =2(d+1)R,
\]
and \(\Delta S=4(d+1)R\). Also,
\(\nabla S=4A^3-RA-p_3\Id\). Newton's identities give
\[
 p_5=\tfrac56Rp_3,\qquad
 p_6=\tfrac34Rp_4-\tfrac18R^3+\tfrac13p_3^2.
\]
Substitution proves the two gradient identities. Together with
Euler's identities, \cref{prop:MR-generator-test} gives Laplacian
closure. Since \(S=\Gamma(F,F)\) belongs to \(\cA_R(F)\),
this is the smallest algebra containing \(R,F\).
At distinct eigenvalues, the Vandermonde determinant shows that
\(p_2,p_3,p_4\) have independent differentials on \(p_1=0\);
hence \(R,F,S\) are algebraically independent.

The gradient formula and the transverse model in
\cref{prop:cubic-singular-normalization} identify the singular set as
\[
 \Sing\{F=0\}
 =\{A\in V:A^2=\tfrac14\tr(A^2)\Id\}.
\]
Its nonzero matrices have eigenvalues \((\lambda,\lambda,-\lambda,-\lambda)\),
with \(\lambda>0\). In particular,
\(\operatorname{diag}(1,1,-1,-1)\) is a nonzero critical point.
Since \(\Delta S=4kR\) with \(k=d+1\),
\cref{thm:cubic-rank-two-existence} proves all the minimizing and
nonminimizing assertions. For the common potential, put \(D=F^2\).
The displayed identities imply
\[
 \Gamma(F,D)=2SF,\qquad\Gamma(S,D)=8RD,\qquad
 \Gamma(D,D)=4SD.
\]
For \(d=4\), \(N=27\) and \(\tr L_x^2=10R\), so
\cref{cor:cubic-unified-trace} applies.
\end{proof}

\Needspace{12\baselineskip}
\subsubsection{Exact link geometry and the sharp Hardy constant}
On the sphere write \(f=F|_{\Sph^{N-1}}\) and
\(s=S|_{\Sph^{N-1}}\).
Let \(M=Z^*\cap \Sph^{N-1}\). Since \(Z^*\) is conical and radial
vectors are transverse to the sphere, \(M\) is a compact smooth
submanifold of dimension \(4k-4\). The zero link
\(\Sigma=\{F=0\}\cap \Sph^{N-1}\) has regular part of dimension \(6k-5\).
Its transverse dimension along \(M\) is thus \(2k-1\).
If \(\rho\) denotes spherical distance to \(M\), then, near \(M\),
\begin{equation}\label{eq:tube}
 s\asymp\rho^2,\qquad
 \cH^{N-2}\bigl(\Sigma_{\reg}\cap\{\rho<\eps\}\bigr)
 \leq C\eps^{2k-1}.
\end{equation}
The proof of these estimates is given in \cref{sec:cubic-local-geometry}.

\paragraph{Two geometric identities.}
With \(\nu=\nabla F/\sqrt S\), the curvature identity proved in
\cref{prop:harmonic-power-Jacobi} is
\(|A_C|^2=(\Delta S)/(2S)-|\nabla S|^2/(2S^2)\).
Substituting \(\Delta S=4kR\) and \(|\nabla S|^2=4RS\)
on the regular cone gives
\begin{equation}\label{eq:Aexact}
 |A_C|^2=\frac{2(k-1)R}{S},\qquad
 |A_\Sigma|^2=\frac{2(k-1)}s.
\end{equation}

The spherical cross term \(\Gamma_{\Sph}(f,s)=4f(1-3s)\) vanishes on
the zero link. Therefore
\begin{equation}\label{eq:linksgrad}
 |\nabla_\Sigma s|^2=4s-16s^2.
\end{equation}
In particular \(0<s\leq1/4\) on \(\Sigma_{\mathrm{reg}}\), improving
the global bound \(s\leq1/3\) in
\eqref{eq:cubic-gradient-only-S-bound}.

To compute the link Laplacian, differentiate the closure table:
\begin{align*}
 \Gamma(F,\Gamma(F,S))&=4RS+24F^2,\\
 \Hess F(\nabla F,\nabla S)&=\tfrac12|\nabla S|^2=2RS+6F^2.
\end{align*}
Their difference is \(\Hess S(\nabla F,\nabla F)\); hence
\(\Hess S(\nu,\nu)=2R\) on the cone. Since \(S\) has degree four,
its spherical normal Hessian at \(r=1\) equals \(2-4s\).
Subtract this from \(\Delta_{\Sph}s=4k-4(N+2)s\). The link is minimal,
so there is no mean-curvature term, and
\begin{equation}\label{eq:linksLap}
 \Delta_\Sigma s=4k-2-4(N+1)s
 =4k-2-(24k-8)s.
\end{equation}

\paragraph{A positive eigenfunction and its domain.}
Let \(J_\Sigma=\Delta_\Sigma+|A_\Sigma|^2\).
For any real \(p\), equations \eqref{eq:Aexact}--\eqref{eq:linksLap} give
\begin{equation}\label{eq:powerJ}
 J_\Sigma(s^{-p})
 =\left\{\frac{4p^2+(6-4k)p+2k-2}{s}
           +4p(6k-6-4p)\right\}s^{-p}.
\end{equation}
For \(k\geq4\), define
\begin{equation}\label{eq:p}
 \delta_k=\sqrt{4k^2-20k+17}>0,\qquad
 p_k=\frac{2k-3-\delta_k}{4}>0,\qquad
 B_k=4p_k(6k-6-4p_k).
\end{equation}
The coefficient of \(s^{-1}\) in \eqref{eq:powerJ} vanishes, so
\begin{equation}\label{eq:eig}
 \phi=s^{-p_k}>0,\qquad J_\Sigma\phi=B_k\phi.
\end{equation}

\begin{theorem}[Optimal Hardy stability constant]\label{thm:sharp}
Under \eqref{eq:cubic-normalized-closure}, with a nonvertex critical point and \(k\geq4\),
the Friedrichs bottom of \(-J_\Sigma\), initially defined on
\(C_c^\infty(\Sigma_{\mathrm{reg}})\), is \(-B_k\).
The optimal constant in
\[
 Q_C(u)\geq H_k\int_Cr^{-2}u^2,
 \qquad u\in C_c^\infty(C_{\mathrm{reg}}\setminus\{0\}),
\]
is
\begin{equation}\label{eq:sharpH}
 \boxed{H_k=\bigl(k+\sqrt{4k^2-20k+17}\bigr)^2.}
\end{equation}
In particular, for \(k=5\), \(N=27\), \(H_5=42+10\sqrt{17}\).
\end{theorem}

\begin{proof}
For every compactly supported smooth \(v\) on the regular link,
the positive-solution identity applied to \eqref{eq:eig} gives
\begin{equation}\label{eq:linkground}
 \int_\Sigma\bigl(|\nabla v|^2-|A_\Sigma|^2v^2\bigr)
 =\int_\Sigma\phi^2\left|\nabla(v/\phi)\right|^2-B_k\int_\Sigma v^2.
\end{equation}
The bottom is therefore at least \(-B_k\).

By \eqref{eq:tube}, \(\phi\asymp\rho^{-2p_k}\), and \(2k-3-4p_k=\delta_k>0\).
Thus \(\phi\in L^2(\Sigma)\), with finite gradient and potential
energy near \(M\). Choose smooth cutoffs \(\eta_\epsilon\), zero
when \(\rho<\epsilon\), one when \(\rho>2\epsilon\), with
\(|\nabla\eta_\epsilon|\leq C/\epsilon\). The local volume bound gives
\[
 \int_\Sigma\phi^2|\nabla\eta_\epsilon|^2
 =O(\epsilon^{2k-3-4p_k})=O(\epsilon^{\delta_k})\longrightarrow0.
\]
Use \(v=\eta_\epsilon\phi\) in \eqref{eq:linkground}. Its Rayleigh
quotient tends to \(-B_k\), proving the claimed bottom.
The same estimates give convergence to \(\phi\) in the Friedrichs
form norm. Testing \eqref{eq:eig} first against compactly supported
smooth functions and then using form density shows that \(\phi\)
belongs to the Friedrichs operator domain, with eigenvalue \(-B_k\)
for \(-J_\Sigma\).

The cone dimension is \(n=N-1=6k-4\).  By
\eqref{eq:exact-Hardy-gap}, the optimal cone constant is
\[
 \frac{(n-2)^2}{4}+\inf\operatorname{spec}(-J_\Sigma)
 =9(k-1)^2-B_k.
\]
Finally,
\[
 9(k-1)^2-B_k
 =\bigl(3(k-1)-4p_k\bigr)^2=(k+\delta_k)^2,
\]
which proves \eqref{eq:sharpH}.
\end{proof}

\begin{remark}
For the larger root in \eqref{eq:p},
\(2k-3-4p=-\delta_k\), so the cutoff energy does not tend to zero.
The radicand \(\delta_k^2\) equals four times the Hardy stability
constant of the transverse balanced quadratic cone:
\[
 \delta_k^2=4\left[\frac{(2k-3)^2}{4}-(2k-2)\right].
\]
\end{remark}

\subsection{Four-generator radial models}
\label{sec:exceptional-four-generator}

\subsubsection{Complexified Jordan norms}
\label{exc-jordan-models}

The sextic \(D\) is now independent: \(D=(\Gamma(S,S)-4RS)/12\).
In the preceding trace family, \(D=F^2\).

Let \(\mathbb K_d=\R,\mathbb C,\mathbb H,\mathbb O\) for
\(d=1,2,4,8\), respectively, and set
\(J_d=\operatorname{Herm}_3(\mathbb K_d)\), writing \(J=J_d\) locally.
Its Jordan product is \(x\bullet y=(xy+yx)/2\), its unit is \(e\),
and its Euclidean inner product is \(\tr(x\bullet y)\).
Thus \((J_d)_0=e^\perp=\{x\in J_d:\tr x=0\}\) denotes its trace-free part.
On the real vector space
underlying \(V=J\otimes_{\R}\mathbb C\), use the real part of
the Hermitian extension of this inner product. Thus
\(|x+iy|^2=|x|^2+|y|^2\) and \(\dim_{\R}V=6(d+1)\).
Let \(N\) be the complexification
of the Jordan determinant. A Jordan frame consists of three
idempotents \(c_i\) with \(c_i\bullet c_j=\delta_{ij}c_i\),
\(\tr c_i=1\), and \(\sum_i c_i=e\). The normalization is
\[
 N(z_1c_1+z_2c_2+z_3c_3)=z_1z_2z_3,
 \qquad |z_1c_1+z_2c_2+z_3c_3|^2=\sum_i|z_i|^2.
\]
These are the complexified rank-three models in
\cite[Section~6.11.6(iv)]{NadirashviliTkachevVladut14}; the polynomial
there is \(2\operatorname{Re}N\), with the same zero set.
Their REC characteristic dimensions are \((n_1,n_2)=(1,3d+2)\).

The determinant has the coordinate formula
\begin{equation}\label{exc-norm-coordinates}
 z=\begin{pmatrix}
 a&\bar u&\bar v\\ u&b&w\\ v&\bar w&h
 \end{pmatrix},\qquad
 N(z)=abh-a\,n(w)-b\,n(v)-h\,n(u)
       +2\operatorname{Sc}((\bar u w)v).
\end{equation}
Here \(a,b,h\in\mathbb C\) and
\(u,v,w\in\mathbb K_d\otimes_{\R}\mathbb C\);
the bar is algebra conjugation fixing complex scalars,
\(n(u)=u\bar u\), and \(\operatorname{Sc}\) is the complex-linear
extension of the scalar part. Thus \(n\) is a complex quadratic
form, distinct from the Hermitian norm \(|u|^2\).

The quadratic adjoint \(z^\#\) satisfies
\((dN)_z(v)=\tr(z^\#\bullet v)\), with the complex bilinear
trace form. The rank-three identities are
\begin{equation}\label{exc-adjoint}
 z\bullet z^\#=N(z)e,\qquad
 (z^\#)^\#=N(z)z,\qquad
 \tr(z\bullet z^\#)=3N(z).
\end{equation}
They extend from \(J\) to \(V\) as polynomial identities; see
\cite[Section~3.2.7]{NadirashviliTkachevVladut14} for the Jordan algebra
background. Put
\begin{equation}\label{exc-invariants}
 F=\operatorname{Re}N,\qquad R=|z|^2,\qquad
 S=|z^\#|^2,\qquad D=|N(z)|^2.
\end{equation}
Holomorphicity gives \(\Delta F=0\) and \(|\nabla F|^2=S\).
The Jordan product on \(J\) is distinct from the metrized product
obtained by polarizing \(F\) on \(V\).

\begin{theorem}[The smallest closure of the complexified models]
\label{thm:exc-jordan-closure}
For each \(d\in\{1,2,4,8\}\), the four polynomials in
\eqref{exc-invariants} are algebraically independent and
\[
 \cA_R(F)=\R[R,F,S,D].
\]
Their degrees are \(2,3,4,6\); in particular,
\((\cA_R(F))_2=\R R\) and \((\cA_R(F))_3=\R F\).
The Laplacian identities are
\begin{equation}\label{exc-jordan-laplacians}
 \Delta R=12(d+1),\qquad \Delta F=0,\qquad
 \Delta S=4(d+2)R,\qquad \Delta D=4S.
\end{equation}
The gradient pairings, apart from those determined by
\(\Gamma(R,P)=2\deg(P)P\), are
\begin{equation}\label{exc-gradient-table}
\begin{aligned}
 \Gamma(F,F)&=S,& \Gamma(F,S)&=4RF,& \Gamma(F,D)&=2SF,\\
 \Gamma(S,S)&=4RS+12D,& \Gamma(S,D)&=8RD,&
 \Gamma(D,D)&=4SD.
\end{aligned}
\end{equation}
In particular, \(\cL(F)=-2RF\), so the regular zero set is minimal.
Throughout \(V\),
\begin{equation}\label{exc-global-inequalities}
 R^2\geq3S,\qquad RS\geq9D,\qquad S^2\geq3RD,
 \qquad 0\leq F^2\leq D.
\end{equation}
\end{theorem}

\begin{proof}
\emph{Differential identities.}
For real \(a\in J\), \(T_a:z\mapsto a\bullet z\)
is self-adjoint for the Hermitian trace metric, and
\eqref{exc-adjoint} gives \((dN)_z(T_az)=(\tr a)N(z)\).
Consequently \(\exp(itT_a)\) is unitary and multiplies \(N\)
by \(\exp(it\tr a)\). Real Jordan automorphisms are also
unitary and preserve \(N\). These transformations preserve
\(R,S,D\) and transport identities involving \(N\) with its
phase factor.

The linearized Jordan identity implies that the commutators
\([T_a,T_b]\) are derivations, whose exponentials are automorphisms.
Take the standard diagonal frame \((c_i)\). At
\(x=\sum_i r_ic_i\), with distinct positive \(r_i\), the
Peirce decomposition \cite[Section~3.2.6]{NadirashviliTkachevVladut14} is
\[
 J=\bigoplus_i\R c_i\ \oplus\!\bigoplus_{i<j}J_{ij},
 \qquad \dim_{\R}J_{ij}=d.
\]
Here \(J_{ij}\) consists of matrices supported in the \((i,j)\)
and \((j,i)\) entries, and
\(c_i\bullet a=c_j\bullet a=a/2\) for \(a\in J_{ij}\).
Varying the \(r_i\) supplies the real diagonal directions, while
the derivations and unitary infinitesimal transformations give
\[
 [T_{c_i},T_a]x=\tfrac14(r_j-r_i)a,\qquad
 iT_ax=\tfrac12(r_i+r_j)ia,\qquad iT_{c_i}x=ir_ic_i.
\]
These directions span the underlying real space of \(V\).
The map obtained by varying the diagonal entries and the transformation
parameters is therefore a submersion at \(x\), so its image contains
a real open subset of \(V\).

On the diagonal, with \(t_i=r_i^2\),
\(R=e_1(t), S=e_2(t), D=e_3(t)\), where
\(e_1(t)=\sum_i t_i\), \(e_2(t)=\sum_{i<j}t_it_j\), and
\(e_3(t)=t_1t_2t_3\). The first derivatives of
these three invariants in the off-diagonal directions vanish.
Differentiation gives the last row of \eqref{exc-gradient-table}.
Extending \(\Gamma\) complex bilinearly, diagonal differentiation
gives \(\Gamma(N,S)=4RN\), whose real part is the desired identity
for \(F\). Holomorphicity gives the remaining pairings with
\(F\). The complex Hessian of \(N\) has diagonal-block
squared norm \(2R\); on \(J_{ij}\) its block is
\(-r_k\Id\), where \(\{i,j,k\}=\{1,2,3\}\), contributing
\(dR\) in total. Therefore
\[
\begin{gathered}
 |\Hess_{\mathbb C}N|^2=(d+2)R,\qquad
 |\Hess_{\R}F|^2=2(d+2)R,\\
 \Delta S=2|\Hess_{\R}F|^2=4(d+2)R.
\end{gathered}
\]
Here \(|\Hess_{\R}\operatorname{Re}(e^{i\theta}N)|^2
=2|\Hess_{\mathbb C}N|^2\), so the real Hessian norm is
unchanged by the determinant phase.
Also \(\Delta|N|^2=4|z^\#|^2\). Phase covariance transports
these identities to the open subset just obtained, and polynomiality
then proves them on all of \(V\).

\emph{The smallest closure.}
The differential identities satisfy
\cref{prop:cubic-scalar-killing-closure}, which gives the claimed
smallest algebra. At distinct positive \(t_i\), the Jacobian
of \((e_1,e_2,e_3)\) is nonzero. Changing the overall complex
phase preserves \(R,S,D\) and changes \(F\) whenever \(N\)
is nonreal. Hence the four differentials are independent at such
a point.

\emph{Global inequalities.}
The gradient table and \cref{lem:cubic-unified-bounds} give
\eqref{exc-global-inequalities}.
\end{proof}

\subsubsection{Polynomial potentials and the Jordan minimizing range}
\label{exc-multipliers}

The following polynomial alternative also applies to the model in
\cref{exc-g2-model}.

\begin{proposition}[An algebraic sign criterion for the multiplier \(S\)]
\label{prop:exc-SF-certificate}
Let \(F\) be a harmonic cubic on a Euclidean space \(V\), set
\(R=|x|^2\) and \(S=|\nabla F|^2\), and let \(D\) be a
sextic satisfying \eqref{exc-gradient-table} and
\eqref{exc-global-inequalities}. Suppose \(\Delta S=cR\).
For \(G=SF\), one has \(\cL(G)=F\mathcal P_c\), where
\begin{equation}\label{exc-SF-certificate}
\begin{aligned}
 \mathcal P_c={}&(c-\tfrac{58}{3})RS^3
       +\tfrac{10}{3}S^2(RS-9D)\\
 &+F^2\bigl[(12c-72)RD+12cS(R^2-3S)+(36c-96)S^2\bigr].
\end{aligned}
\end{equation}
If \(c\geq58/3\), then \(G\cL(G)\geq0\), and
\(\nabla G\neq0\) wherever \(S>0\).
\end{proposition}

\begin{proof}
Set \(K=cR\) in \eqref{eq:cubic-SF-general-defect} and regroup
to obtain \eqref{exc-SF-certificate}. Every summand is nonnegative
by \eqref{exc-global-inequalities} when \(c\geq58/3\).
The gradient formula \eqref{eq:cubic-SF-general-gradient} gives
\(|\nabla G|^2\geq S^3\).
\end{proof}

\begin{theorem}[Area minimization for the complexified Jordan models]
\label{thm:exc-jordan-minimizing}
For the four models in \cref{exc-jordan-models},
\(\{F<0\}\) is locally perimeter minimizing if and only if
\(d\in\{2,4,8\}\). In all three minimizing dimensions,
\(G=F\sqrt{RS}\) generates a subcalibration. For \(d=4,8\),
the polynomial potential \(SF\) does so as well.
\end{theorem}

\begin{proof}
For \(d=2,4,8\), the dimension and trace coefficient are
\[
 (N,K/R)=(18,16),\ (30,24),\ (54,40),\qquad K=\Delta S.
\]
Since \(\tr L_x^2=K/2\), each pair satisfies
\cref{cor:cubic-unified-trace}; this proves the minimizing assertions
for the common potential. For \(d=4,8\),
\cref{prop:exc-SF-certificate} additionally gives the polynomial
potential \(SF\).

The singular set is the Jordan rank-at-most-one cone
\[
 \Sing\{F=0\}=\Sigma=\{z\in V:z^\#=0\}.
\]
To describe its nonzero stratum, take
\(a\neq0\) in \eqref{exc-norm-coordinates} and set
\[
 B=b-\frac{n(u)}a,\quad C=h-\frac{n(v)}a,\quad
 W=w-\frac{u\bar v}a.
\]
The determinant formula gives
\begin{equation}\label{exc-schur}
 N(z)=a\bigl(BC-n(W)\bigr).
\end{equation}
This identity also holds for \(d=8\). Indeed, multiplicativity
of \(n\) and associativity of the scalar part give
\[
 n\!\left(w-\frac{u\bar v}{a}\right)
 =n(w)-\frac{2}{a}\operatorname{Sc}((\bar u w)v)
       +\frac{n(u)n(v)}{a^2},
\]
which proves \eqref{exc-schur} after expanding \(aBC\).
The coordinate change is holomorphically invertible, so
\(dN=0\) is exactly \(B=C=W=0\) in this chart. Hence
these points are geometrically singular: the transverse real
quadratic form \(\operatorname{Re}(a(BC-n(W)))\) is nondegenerate
and indefinite. Moreover,
\(\Sigma\setminus\{0\}\) is smooth of complex dimension
\(1+2d\), and real codimension \(2(d+2)\). To see that these
charts cover the nonzero set, recall that the spectral theorem makes
primitive idempotents span \(J\), and that real Jordan automorphisms
act transitively on them
\cite[Theorem~III.1.2 and Corollary~IV.2.7]{FarautKoranyi94}.
Some primitive idempotent therefore has nonzero trace pairing with
any nonzero \(z\), and an automorphism takes it to \(c_1\).

For \(d=4,8\), the codimension bound and
\eqref{eq:cubic-SF-general-gradient} allow
\cref{cor:gradient-subcalibration} to be applied to \(SF\).
Harmonicity identifies the boundary-current support with \(\{F=0\}\)
by \cref{rem:polynomial-finite-perimeter}.

For \(d=1\), take \(p\in\Sigma\setminus\{0\}\) and put
\(H=\Hess F(p)\). The chart \eqref{exc-schur} shows that
\(H\) has real rank \(2(d+2)\).
\Cref{prop:cubic-general-singular-spectrum}, with \(q=2R\), gives
\(H^3=R(p)H\). Since \(\tr H=0\), its nonzero eigenvalues
are \(\pm\sqrt{R(p)}\), each of multiplicity \(d+2\).
Thus the tangent cone, with the induced Euclidean metric, is
\begin{equation}\label{exc-jordan-tangent}
 C_{d+1,d+1}\times\R^{4d+2},\qquad
 C_{q,q}=\{(x,y)\in\R^{q+1}\oplus\R^{q+1}:|x|=|y|\}.
\end{equation}
For \(d=1\), the normal factor \(C_{2,2}\) is unstable by the
Hardy criterion in \cref{sec:tangent-instability-transfer}.
\Cref{lem:tangent-instability-transfer} therefore rules out local
perimeter minimization.
\end{proof}

For \(d=2\), complexification identifies \(V\) isometrically
with \(M_3(\mathbb C)\) with its Frobenius metric, and \(N\)
with the determinant. In the notation of \cref{sec:complex-determinants},
\(R=E_1\), \(S=E_2\), and \(D=E_3\). Thus both the closure
table and the potential above are the three-by-three cases of the
construction there.

\subsubsection{The \texorpdfstring{$G_2$}{G2} trilinear model}
\label{exc-g2-model}

Let \(\times\) be the standard cross product on
\(\operatorname{Im}\mathbb O=\R^7\), and set
\[
 F(a,b,c)=\langle a\times b,c\rangle,
 \qquad (a,b,c)\in(\R^7)^3.
\]
Equivalently, \(F(a,b,c)=\varphi(a,b,c)\), where \(\varphi\)
is the associative calibration on \(\R^7\) of Harvey and Lawson
\cite[Section~IV.1.A]{HarveyLawson82a}.
This is the classical model in
\cite[Section~6.4, equation~(6.26), Proposition~6.3]{TkachevClassification10},
whose defining polynomial \(\operatorname{Re}(a(bc))\) is
\(-F\). Up to similarity, its subtracted Jordan model is the cubic
norm restricted to the trace-orthogonal complement of
\(\operatorname{Herm}_3(\R)\) in \(\operatorname{Herm}_3(\mathbb O)\);
see \cite[Proposition~6.4.4 and Table~6.4.1]{NadirashviliTkachevVladut14}.
Let \(A\) be the real \(7\times3\) matrix with
columns \(a,b,c\), and put
\[
 B=A^tA,\qquad R=\tr B,\qquad S=e_2(B),\qquad D=\det B,
\]
where \(e_2(B)\) is the second elementary symmetric polynomial
of the eigenvalues of \(B\).

The compact real group \(G_2=\operatorname{Aut}(\mathbb O)\)
acts diagonally on \(V=(\R^7)^3\). Schwarz's invariant-generation
theorem \cite[Theorem~4 and equations~(5.1)--(5.3)]{Schwarz83} gives
\[
 \R[V]^{G_2}=\R[B_{ij},F].
\]
The commuting \(SO(3)\) action on the columns fixes \(F\)
and conjugates \(B\). Averaging polynomial expressions over
\(SO(3)\), whose conjugation invariants on symmetric matrices
are generated by the elementary symmetric polynomials, therefore gives
\[
 \R[V]^{G_2\times SO(3)}=\R[R,F,S,D].
\]
This invariant ring equals \(\cA_R(F)\):

\begin{theorem}[The trilinear model]
\label{thm:exc-g2}
The four polynomials \(R,F,S,D\) are algebraically independent,
and \(\cA_R(F)=\R[R,F,S,D]\). In particular,
\((\cA_R(F))_2=\R R\) and \((\cA_R(F))_3=\R F\).
They satisfy
\eqref{exc-gradient-table} and \eqref{exc-global-inequalities},
with
\begin{equation}\label{exc-g2-laplacians}
 \Delta R=42,\qquad \Delta F=0,\qquad
 \Delta S=24R,\qquad \Delta D=10S.
\end{equation}
Both \(G=F\sqrt{RS}\) and the polynomial potential \(SF\)
give subcalibrations of \(\{F<0\}\subset\R^{21}\), which is
locally perimeter minimizing.
\end{theorem}

\begin{proof}
The cross-product identity
\(|a\times b|^2=|a|^2|b|^2-\langle a,b\rangle^2\)
gives \(|\nabla F|^2=S\), while trilinearity gives
\(\Delta F=0\). Alternation gives
\(\langle\nabla_{a_i}F,a_j\rangle=\delta_{ij}F\), with
\(a_i\) denoting the columns. In matrix notation,
\[
 \nabla_A S=2A(RI-B),\qquad
 \nabla_A D=2A\operatorname{adj}B.
\]
These formulas and the three-by-three Newton identities give
\eqref{exc-gradient-table}. Direct differentiation of the Gram
invariants gives \(\Delta S=24R\) and \(\Delta D=10S\).
The global inequalities \eqref{exc-global-inequalities} follow
from \cref{lem:cubic-unified-bounds}.

The smallest-algebra assertion now follows from
\cref{prop:cubic-scalar-killing-closure}.
For independence,
choose orthonormal \(e_1,e_2\), put \(e_3=e_1\times e_2\),
and choose a unit \(e_4\) orthogonal to their span. At
\[
 a=r_1e_1,\qquad b=r_2e_2,\qquad
 c=r_3(\cos\theta\,e_3+\sin\theta\,e_4),
\]
the three invariants \(R,S,D\) are the elementary symmetric
functions of \(r_1^2,r_2^2,r_3^2\), while
\(F=r_1r_2r_3\cos\theta\). Distinct positive \(r_i^2\)
and \(0<\theta<\pi\) give four independent differentials.

Here \(N=21\) and \(K=24R\), so
\cref{cor:cubic-unified-trace} gives the common potential.
For the polynomial potential, \cref{prop:exc-SF-certificate} applies
with \(c=24\). The singular set of the zero cone is
\[
 \Sing\{F=0\}=\Sigma=\{S=0\}
 =\{A\in M_{7,3}(\R):\operatorname{rank}A\leq1\}.
\]
At a nonzero rank-one matrix, a column rotation reduces to
\((a,0,0)\), whose transverse leading form
\(\langle a\times\delta b,\delta c\rangle\) is indefinite of rank twelve.
The nonzero stratum has dimension
\(7+3-1=9\), hence codimension \(12\). The gradient bound
\eqref{eq:cubic-SF-general-gradient} and \cref{cor:gradient-subcalibration}
therefore give the same minimizing phase.
\end{proof}

\begin{remark}
For both families, Bochner's identity gives
\(2|L_x|_{\mathrm{HS}}^2=c|x|^2\), where \(c>0\). Hence
\(L_v=0\) implies \(v=0\), and their smallest closures have the intrinsic
degree pattern \((3)\) of \cref{thm:cubic-one-dimensional-degree-three}.
For the \(G_2\) model this differs from its trilinear coordinate
presentation: the larger invariant algebra
\(\R[V]^{G_2}\) contains their separate radii, whereas the
smallest algebra is also invariant under column rotations and
does not contain \(|a|^2,|b|^2,|c|^2\).
\end{remark}

\subsection{Area-minimizing Clifford cubic cones}
\label{sec:Clifford-classification}
\begin{theorem}[Area-minimizing classification]
\label{thm:clifford-classification}
For every admissible symmetric Clifford system with \(q\geq1\), the
oriented multiplicity-one cone \(\mathcal C_{\mathcal A}\) is area
minimizing if and only if
\begin{equation}\label{eq:clifford-classification-list}
 \boxed{m\geq4,\qquad q\geq3,\qquad(m,q)\neq(4,3).}
\end{equation}
Every area-minimizing member in this family is strictly area minimizing.
\end{theorem}

The positive range uses \(G=F\sqrt{RS}\) and
\cref{thm:cubic-unified-sextic}; singular-stratum obstructions exclude
the low-dimensional cases. The remaining pair \((4,3)\) requires a
link variation:

\begin{theorem}[Instability of the pair \((4,3)\)]
\label{thm:Clifford-43-instability}
For every Clifford system with \((m,q)=(4,3)\), the cone
\(\cC_\cA\subset\R^{12}\) is unstable.  More precisely, if
\(\Lambda=\cC_\cA\cap\Sph^{11}\), then
\(\phi=S^{3/4}(1-S/2)\) belongs to the energy space
\(\mathcal D(\Lambda)\) defined in
\cref{sec:instability} and satisfies
\[
 \int_{\Lambda_{\reg}}
 \left(|\nabla_\Lambda\phi|^2-|A_\Lambda|^2\phi^2
       +\frac{81}{4}\phi^2\right)\dd\cH^{10}<0,
\]
where \(A_\Lambda\) is the second fundamental form of
\(\Lambda\subset\Sph^{11}\).
Every tangent cone of \(\cC_\cA\) at a nonzero point is area minimizing.
At such a singular point it is the product of the seven-dimensional Simons
cone and \(\R^4\).
\end{theorem}

\subsection{The Clifford positivity range}
\label{sec:Clifford-common-potential}

\begin{proposition}[The common potential for Clifford cubics]
\label{prop:Clifford-common-potential}
Suppose either \(m\geq8\), \(q\geq3\), or \(m\geq4\), \(q\geq4\).
Then
\[
 G=F\sqrt{RS}
\]
generates a two-sided subcalibration for \(\{F<0\}\), and
\(\cC_{\cA}\) is strictly area minimizing with the prescribed orientation.
The gradient of \(G\) is nonzero on \(\{S>0\}\), and
\(F\cL(G)>0\) wherever \(F\ne0\).
\end{proposition}

\begin{proof}
For \cref{thm:cubic-unified-sextic}, use the normalization
\begin{equation}\label{eq:Clifford-normalized-data}
 \begin{aligned}
 f&=\frac F2,\qquad s=|\nabla f|^2=\frac S4=uv+\frac T4,
 \qquad D=\frac{vT}{4},\\
 \kappa&=\Delta s=\frac K4=4mv+4(q+1)u.
 \end{aligned}
\end{equation}
The closure table \eqref{eq:clifford-closure-table} gives
\begin{equation}\label{eq:Clifford-unified-table}
 \Gamma(f,s)=4Rf,\qquad \Gamma(s,s)=4Rs+12D.
\end{equation}
Thus \(\cL(f)=-2Rf\), and the remaining transport identities are
\eqref{eq:cubic-unified-table}, with \((F,S)\) replaced by \((f,s)\).
If \(m\geq8\) and \(q\geq3\), then
\(N=2m+q+1\geq20\) and \(\kappa\geq16R\).
The trace criterion \cref{cor:cubic-unified-trace} applies directly.

For the second range, write the remaining sign expression as
\[
 \mathscr H_{m,q}
 =(2m+q-11)s^2+\frac12R\bigl(4mv+4(q+1)u-16R\bigr)s-12RD.
\]
At \((m,q)=(4,4)\), direct substitution gives
\begin{equation}\label{eq:Clifford-H44}
 \mathscr H_{4,4}
 =\frac{T(8u^2+T)}{16}+(u^2-T)v(2u+3v)\geq0,
\end{equation}
by \eqref{eq:box-constraints} and \(u,v\geq0\).
For every \(m\geq4\), \(q\geq4\),
\begin{align*}
 \mathscr H_{m,q}
 &=\mathscr H_{4,4}
   +\bigl(2(m-4)+q-4\bigr)s^2\\
 &\quad+2R\bigl((m-4)v+(q-4)u\bigr)s\geq0,
\end{align*}
and \(\kappa\geq16R\). Hence \cref{thm:cubic-unified-sextic}
applies in this range as well.
Finally,
\[
 f\sqrt{Rs}=\frac14F\sqrt{RS}.
\]
This positive factor preserves the normalized gradient and the
strict divergence sign, giving the result for the original \(F\).
\end{proof}

In the normalized variables of \eqref{eq:Clifford-normalized-data},
the sextic is generally independent of the cubic:
\[
 D-f^2=\frac{vT-F^2}{4}\geq0.
\]
This Cauchy--Schwarz defect need not vanish, so the polynomial
gradient-square criterion requiring \(D=f^2\) does not apply in general.

\paragraph*{Dimensions in which Clifford systems exist}

For a symmetric Clifford system $(A_0,\ldots,A_q)$ on $\R^{2m}$,
the Hurwitz--Radon admissibility condition
\cite[Chapter~0, pp.~3--4]{Shapiro2000} is \(q\leq\rho(m)\),
with the Clifford-system convention of \cite{FKM81,TkachevClifford10}.
Here, if $m=2^{4a+b}c$ with $c$ odd and $0\leq b\leq3$, then \(\rho(m)=8a+2^b\).
In particular, \(\rho(4)=4,\; \rho(5)=1,\; \rho(6)=2,\; \rho(7)=1\).
Thus an admissible pair with $q\geq3$ has either $m=4$ or $m\geq8$.
If $m=4$, then $q\in\{3,4\}$.

\begin{proof}[Sufficiency in \cref{thm:clifford-classification}]
By admissibility, every pair in
\eqref{eq:clifford-classification-list} has either \(m\geq8\),
\(q\geq3\), or \((m,q)=(4,4)\). Both cases are covered by
\cref{prop:Clifford-common-potential}.
\end{proof}

\subsection{Singular strata and tangent cones}
\label{sec:tangents}

From \eqref{eq:U}, the critical set is
\[
 \Sigma:=\{\nabla F=0\}
 =\Sigma_0\cup\Sigma_+,
\]
where
\[
 \Sigma_0=\{(0,z):z\in\R^{q+1}\},\qquad
 \Sigma_+=\{(y,0):\tau(y)=0\}.
\]
Indeed, $A_z^2=v\Id$ shows that $A_zy=0$ forces $z=0$ or $y=0$.

At a nonzero point of $\Sigma_+$, the gradients
$\nabla\tau_i=2A_i y$ are mutually orthogonal, so this stratum is
smooth there and has ambient codimension $2q+2$.  The stratum
$\Sigma_0$ has ambient codimension $2m$.  In the sufficient range
$m\geq4$, $q\geq3$, and with $N=2m+q+1$,
\begin{equation}\label{eq:Sigma-thin}
 \cH^{N-1}(\Sigma\cap \mathcal K)=0
 \quad\text{for every compact }\mathcal K.
\end{equation}

Define the generalized Simons cone
\[
 \mathcal S_k
 =\{(a,b)\in\R^k\oplus\R^k:|a|=|b|\}.
\]
Let $E=\{F<0\}$.

\Needspace{5\baselineskip}
\begin{proposition}[Unique tangents with multiplicity one]
\label{prop:tangent-cones}
\leavevmode
\begin{enumerate}[label=\textup{(\roman*)}]
\item At a nonzero point $p=(0,z_0)$ of $\Sigma_0$, choose orthogonal
coordinates
\[
 \R^{2m+q+1}=\R^m\oplus\R^m\oplus\R^{q+1}.
\]
The unique set blow-up is
\begin{equation}\label{eq:tangent-Sigma0}
 E_p^{(0)}=\{(a,b,\eta):|a|<|b|\},\qquad
 \partial E_p^{(0)}=\mathcal S_m\times\R^{q+1}.
\end{equation}
\item At $p=(y_0,0)\in\Sigma_+\setminus\{0\}$, after an orthogonal
identification
$\R^{2m+q+1}=\R^{q+1}\oplus\R^{q+1}
 \oplus\R^{2m-q-1}$, the unique set blow-up is
\begin{equation}\label{eq:tangent-Sigmaplus}
 E_p^{(+)}=\{(x,x',\xi^\perp):|x|<|x'|\},\qquad
 \partial E_p^{(+)}
 =\mathcal S_{q+1}\times\R^{2m-q-1}.
\end{equation}
\item At the vertex, every dilation fixes $E$ and $\cC_{\cA}$, so the
tangent cone is $\cC_{\cA}$ itself.
\end{enumerate}
In each case, the oriented current tangent is the boundary of the
stated set, and the varifold tangent has multiplicity one on its regular
part.  The perimeter measures converge under dilation.  These statements
hold for every admissible Clifford system, without an area-minimizing
assumption.
\end{proposition}

\begin{proof}
The Taylor argument in \cref{sec:tangent-instability-transfer} gives
the set tangents. We use exact local coordinates also to determine
the current and varifold limits. Near
$p=(0,z_0)$, $z_0\ne0$, the identity $A_z^2=|z|^2\Id$ gives
smooth positive and negative eigenspaces. Their dimensions are both
$m$, since $A_w$ interchanges them whenever $0\ne w\perp z$.
Choose smooth orthonormal frames and write $y=P(z)(a,b)$. Then
\[
 F(y,z)=|z|\bigl(|a|^2-|b|^2\bigr).
\]
The coordinate map has orthogonal derivative at $(0,z_0)$, since
its derivatives in $z$ are multiplied by $y=0$. This gives
\eqref{eq:tangent-Sigma0}, including the stated set side.

Near $p=(y_0,0)\in\Sigma_+$, $y_0\ne0$, the vectors
$e_i=A_i y_0/|y_0|$, $0\le i\le q$, are orthonormal.
Let $P_\perp$ be the orthogonal projection onto their common
orthogonal complement. The map
\[
 y\longmapsto\left(
 a_i(y)=\frac{\tau_i(y)}{2|y_0|}\ (0\le i\le q),\quad
 \xi^\perp=P_\perp(y-y_0)\right)
\]
has orthogonal derivative at $y_0$, because
$da_i|_{y_0}(\xi)=\langle e_i,\xi\rangle$, and is a smooth
local coordinate map. With $z=\eta$ and
\[
 x=\frac{a+\eta}{\sqrt2},\qquad
 x'=\frac{a-\eta}{\sqrt2},
\]
we have exactly
\[
 F(y,\eta)=2|y_0|a\cdot\eta
          =|y_0|\bigl(|x|^2-|x'|^2\bigr).
\]
This proves \eqref{eq:tangent-Sigmaplus} with its stated side.

In either case, after the stated orthogonal identification, a smooth
local diffeomorphism \(\Phi\) takes the quadratic sublevel model onto
\(E\), with \(\Phi(0)=p\) and \(D\Phi(0)=\Id\).
The rescaled maps \(\Phi_r(w)=(\Phi(rw)-p)/r\) converge to the identity
in \(C^1\) on compact sets.  Each quadratic model has locally finite
perimeter and multiplicity one on its regular boundary.  The area
formula for \(\Phi_r\), whose tangential Jacobians converge uniformly
to one, gives convergence of the perimeter measures and the associated
varifolds.  The $C^1$ convergence also gives $L^1_{\rm loc}$ convergence of the
sublevel sets, whose model boundaries are Lebesgue null, and identifies
the oriented current limit with the boundary of the stated model.  At the vertex, all rescalings
agree exactly by homogeneity.
\end{proof}

\begin{remark}
All nonzero algebraic singular points are genuine geometric
singularities: neither
\eqref{eq:tangent-Sigma0} nor \eqref{eq:tangent-Sigmaplus} is a
hyperplane. Thus \(\Sing\cC_\cA=\Sigma_0\cup\Sigma_+\).
The transverse Simons cones are area minimizing exactly when
$m\geq4$ and $q+1\geq4$, respectively
\cite{BDGG69,DePhilippisPaolini}. For a transverse parameter at least
two, this is also exactly the stable range \cite[Section~6.1]{Simons68}; the parameter-one planar
crossing is stable on its regular part but is not minimizing. In particular,
for $(m,q)=(4,3)$ both strata have the same tangent model
$\mathcal S_4\times\R^4$.  The corresponding multiplicity-one current tangents are area minimizing;
at regular points the tangent is a multiplicity-one hyperplane.

For the Hurwitz cubics \(F(x,y,z)=\operatorname{Re}((xy)z)\), with
\(x,y,z\in\mathbb K\) and \(\mathbb K=\mathbb C,\mathbb H,\mathbb O\),
the singular set is the union of the three mutually orthogonal
coordinate subspaces:
\[
 \Sing\{F=0\}
 =(\mathbb K\oplus0\oplus0)\cup(0\oplus\mathbb K\oplus0)
   \cup(0\oplus0\oplus\mathbb K).
\]
Indeed, \(|\nabla F|^2=|x|^2|y|^2+|y|^2|z|^2+|z|^2|x|^2\).
The complex case gives the explicit cubic tangent in
\cref{sec:tangent-instability-transfer}.
\end{remark}

\begin{lemma}
\label{lem:Sigmaplus-nonempty}
If $m>q$, then $\Sigma_+\setminus\{0\}$ is nonempty.
\end{lemma}

\begin{proof}
After an orthogonal change of coordinates, the first two Clifford
matrices have the standard form
\begin{equation}\label{eq:standard-first-two}
 A_0=
 \begin{pmatrix}\Id&0\\0&-\Id\end{pmatrix},
 \qquad
 A_1=
 \begin{pmatrix}0&\Id\\\Id&0\end{pmatrix},
\end{equation}
and, for $i\geq2$,
\begin{equation}\label{eq:standard-rest}
 A_i=
 \begin{pmatrix}0&E_i\\-E_i&0\end{pmatrix},
\end{equation}
where the $E_i$ are skew-symmetric orthogonal matrices.  Fix
$b\neq0$ in $\R^m$.  The $q$ vectors
$b,E_2b,\ldots,E_qb$ are mutually orthogonal.  Since $m>q$, choose
$a\neq0$ of the same length as $b$ and orthogonal to their span.
For $y=(a,b)$, all components of $\tau(y)$ vanish.
\end{proof}

\begin{proof}[Necessity in \cref{thm:clifford-classification}]
Suppose first that $m\leq3$.  The stratum
$\Sigma_0\setminus\{0\}$ has dimension $q+1$ and consists of genuine
singular points by \cref{prop:tangent-cones}.  Its ambient codimension
is $2m<8$.  Federer's codimension-seven bound for the singular set of an
area-minimizing hypersurface \cite{Federer70} therefore excludes area
minimization.

Now suppose $q\leq2$ and $m\geq4$.  Then $m>q$, so
\cref{lem:Sigmaplus-nonempty} applies.  The smooth stratum
$\Sigma_+\setminus\{0\}$ has ambient codimension $2q+2<8$ and again
consists of genuine singular points.  The same dimension bound excludes area minimization in this case as well.
\end{proof}

\subsection{The exceptional pair \texorpdfstring{$(4,3)$}{(4,3)}}

\label{sec:instability}

\paragraph*{Second variation and the radial Hardy inequality}

Fix $(m,q)=(4,3)$ and write
\[
 \Lambda=\cC_{\cA}\cap\Sph^{11},\qquad
 \mathscr S=\Sing(\cC_{\cA})\cap\Sph^{11},\qquad
 \mathscr S_0=\Sigma_0\cap\Sph^{11},\qquad
 \mathscr S_+=\Sigma_+\cap\Sph^{11}.
\]
Thus $\dim\Lambda=10$, while the two singular strata $\mathscr S_0$ and
$\mathscr S_+$ each have dimension $3$.
Let \(\mathcal D(\Lambda)\) be the closure of
\(C_c^\infty(\Lambda_{\reg})\) in the norm with square
\(\int_{\Lambda_{\reg}}(|\nabla\psi|^2+(1+|A_\Lambda|^2)\psi^2)\).
The radial Hardy inequality in cone dimension \(n=11\) has optimal
constant \((n-2)^2/4=81/4\).  Separating the radial variable in the
second variation of area therefore gives the quadratic form
\[
 \mathcal B(\varphi,\psi)=\int_{\Lambda_{\reg}}
 \left(\langle\nabla_\Lambda\varphi,\nabla_\Lambda\psi\rangle
       -|A_\Lambda|^2\varphi\psi+\frac{81}{4}\varphi\psi\right).
\]
It remains to find \(\phi\in\mathcal D(\Lambda)\) with
\(\mathcal B(\phi,\phi)<0\).

For the present link, the smooth compact singular strata and their
nondegenerate transverse quadratic models give
\begin{equation}\label{eq:link-tube}
 \cH^{10}(N_r(\mathscr S)\cap\Lambda)\leq Cr^7,
 \qquad |A_\Lambda|\leq C\rho^{-1},\qquad
 \rho=\dist(\cdot,\mathscr S).
\end{equation}
Indeed, the local coordinates in \cref{prop:tangent-cones} identify
the transverse zero sets with seven-dimensional quadratic cones in
eight normal variables.  Their volume element has radial order
\(\rho^6\dd\rho\), and their second fundamental forms have size
\(O(\rho^{-1})\).  Restricting to the unit sphere removes the radial
direction along the singular stratum.  The coordinate maps and their
inverses have bounded first and second derivatives on a finite cover
of the compact singular strata.  The same estimates therefore hold on
\(\Lambda\), which proves \eqref{eq:link-tube}.

\paragraph*{Curvature identities on the link}

Use $u,v,S,J$ from \cref{sec:invariants}. On
$\Lambda$ one has $R=u+v=1$ and $F=0$.

\begin{lemma}[Link curvature]
\label{lem:link-curvature}
On $\Lambda_{\reg}$,
\begin{align}
 |A_\Lambda|^2&=\frac{32}{S}-\frac{J}{2S^2},
 \label{eq:link-curvature}\\
 |\nabla_\Lambda S|^2&=J-16S^2.
 \label{eq:link-gradS}
\end{align}
\end{lemma}

\begin{proof}
The curvature identity in the proof of
\cref{prop:harmonic-power-Jacobi} is
\(|A_C|^2=K/(2S)-J/(2S^2)\). Here \eqref{eq:DeltaS} gives
\(K=64\) at \((m,q)=(4,3)\), \(R=1\), proving
\eqref{eq:link-curvature}: the radial principal curvature is zero,
so \(|A_C|=|A_\Lambda|\) at radius one.
Moreover, \(Z\perp\nabla F\) on \(F=0\) by \eqref{eq:UZ}, and homogeneity gives
\(\langle Z,(y,z)\rangle=4S\).  Hence
\(\nabla_\Lambda S=Z-4S(y,z)\), proving
\eqref{eq:link-gradS}.
\end{proof}

\paragraph*{Integration along the Hopf fibers}

For $(4,3)$, the standard block form allows us to extend the Clifford
system by one generator.  In the notation of
\eqref{eq:standard-first-two}--\eqref{eq:standard-rest}, put
$E_4=E_2E_3$ and
\[
 A_4=\begin{pmatrix}0&E_4\\-E_4&0\end{pmatrix}.
\]
Then $(A_0,\ldots,A_4)$ is a Clifford system on $\R^8$.  For
$\xi\in\Sph^7$, the map
\begin{equation}\label{eq:hopf-map}
 \Pi(\xi)
 =\bigl(\langle A_0\xi,\xi\rangle,\ldots,
        \langle A_4\xi,\xi\rangle\bigr)
\end{equation}
is the quaternionic Hopf map $\Sph^7\to\Sph^4$; see, for example,
\cite{FKM81}.  To verify its measure normalization, write
\(p_i(\xi)=\langle A_i\xi,\xi\rangle\).  If
$\xi=(a,b)\in\R^4\oplus\R^4$, then
$b,E_2b,E_3b,E_4b$ are mutually orthogonal with the same length;
this proves $|\Pi(\xi)|=1$.  The Clifford relations also give
\begin{equation}\label{eq:hopf-horizontal}
 \nabla_{\Sph^7}p_i=2(A_i\xi-p_i\xi),\qquad
 \langle\nabla_{\Sph^7}p_i,\nabla_{\Sph^7}p_j\rangle
 =4(\delta_{ij}-p_ip_j).
\end{equation}
Thus \(D\Pi\) multiplies lengths by \(2\) on the horizontal space
\((\ker D\Pi)^\perp\), and its four-dimensional Jacobian is the
constant \(2^4\).
The quaternionic Hopf fibers are congruent round \(3\)-spheres.
Integration over the fibers gives, for every integrable function \(h\),
\[
 \int_{\Sph^7}h(\Pi(\xi))\dd\xi
 =c_H\int_{\Sph^4}h(p)\dd p,
 \qquad c_H>0.
\]

The Hopf map also describes the second singular stratum: $\mathscr S_+$ is the inverse image under $\Pi$ of the two poles
$(0,0,0,0,\pm1)\in\Sph^4$.  Hence $\mathscr S_+$ has two connected
components, each a Hopf fiber diffeomorphic to $S^3$.

Set
\[
 \theta=\frac{T}{u^2}\in[0,1].
\]
If $s=\langle A_4\xi,\xi\rangle$, then $\theta=1-s^2$.
The following lemma expresses integration on the link in terms of
$u$ and $\theta$.

\begin{lemma}[Integration on the link]
\label{lem:hopf-coarea}
There is a constant $c_0>0$ such that every integrable function
$h=h(u,\theta)$ satisfies
\begin{align}
 \int_{\Lambda_{\reg}}h\,\dd\cH^{10}
 &=c_0\int_0^1\int_0^1
     h(u,\theta)\,
     u^2(1-u)^{1/2}
     \left(\frac{\theta}{1-\theta}\right)^{1/2}
     S^{1/2}\,\dd\theta\,\dd u,
 \label{eq:hopf-coarea}\\
 S&=4u(1-u)+u^2\theta,                                  \notag\\
 J&=16\left[
       4u(1-u)+\bigl(u+4(1-u)\bigr)u^2\theta
     \right].
 \label{eq:Clifford-link-SJ}
\end{align}
\end{lemma}

\begin{proof}
Write $y=\sqrt u\,\xi$ and $z=\sqrt{1-u}\,\zeta$, with
$\xi\in\Sph^7$ and $\zeta\in\Sph^3$.  Spherical measure on $\Sph^{11}$
decomposes, up to a positive constant, as
\[
 u^3(1-u)\,\dd u\,\dd\xi\,\dd\zeta.
\]
If
$\bar\tau=(\langle A_0\xi,\xi\rangle,\ldots,
\langle A_3\xi,\xi\rangle)$, then
$|\bar\tau|=\sqrt\theta$ and
\[
 F=u\sqrt{1-u}\,\langle\zeta,\bar\tau\rangle.
\]
For fixed $0<u<1$ and $\xi$ with $\theta>0$, the equation $F=0$
defines an equatorial $\Sph^2$ in the $\zeta$-sphere.  Along this equator,
\begin{equation}\label{eq:z-coarea}
 |\nabla_{\Sph^3,\zeta}F|=u\sqrt{1-u}\sqrt\theta.
\end{equation}
The coordinate exceptional sets have zero $\cH^{10}$-measure in
$\Lambda$. Indeed, $u=0$ and $u=1$ have dimensions three and seven.
For $0<u<1$, the set $\theta=0$ has dimension $1+3+3=7$, since
$\xi$ lies in two Hopf fibers and $\zeta$ is unrestricted; the set
$\theta=1$ has dimension $1+6+2=9$, since $\xi$ lies above an
equatorial $\Sph^3\subset\Sph^4$ and $\zeta$ lies in an equatorial
$\Sph^2$. We may therefore work on $0<u<1$, $0<\theta<1$ without
changing the link integral.

On $\{F=0\}\cap\Sph^{11}$, Euler's identity gives
$|\nabla_{S^{11}}F|=|\nabla F|=\sqrt S$.
Applying the coarea formula first on $\Sph^{11}$ and then on the
$\zeta$-sphere gives, for a constant $c_1>0$,
\[
 \int_{\Lambda_{\reg}}h\dd\cH^{10}
 =c_1\int_0^1u^2(1-u)^{1/2}
       \int_{\Sph^7}h(u,\theta)\frac{S^{1/2}}{\sqrt\theta}
       \dd\xi\,\dd u.
\]

It remains to integrate in $\xi$.  The Hopf map
\eqref{eq:hopf-map} has congruent fibers and constant Jacobian.
For the last coordinate $s$ on $\Sph^4$, spherical measure has density
proportional to $(1-s^2)\dd s$.  Since $\theta=1-s^2$, the two
branches of $s$ contribute
\[
 \frac{\theta}{\sqrt{1-\theta}}\,\dd\theta.
\]
Substitution into the preceding integral proves
\eqref{eq:hopf-coarea}; \eqref{eq:Clifford-link-SJ} follows from
the formulas for $S$ and $J$ in \eqref{eq:S} with
$R=1$ and $T=u^2\theta$.
\end{proof}

\paragraph*{A test function proving instability}

For $a\in\R$, put $\phi_a=S^a$ on $\Lambda_{\reg}$.  Combining
\cref{lem:link-curvature,lem:hopf-coarea} gives, whenever the
integral is finite,
\begin{align}
 \frac1{c_0}\mathcal B(\phi_a,\phi_b)
 &=\int_0^1\int_0^1
    u^2(1-u)^{1/2}
    \left(\frac{\theta}{1-\theta}\right)^{1/2}
    \Bigg[
       \left(ab+\frac12\right)JS^{a+b-3/2}              \notag\\
 &\hspace{27mm}
       -32S^{a+b-1/2}
       +\left(\frac{81}{4}-16ab\right)S^{a+b+1/2}
    \Bigg]\dd\theta\,\dd u.
 \label{eq:monomial-link-form}
\end{align}
For $a,b\in\{3/4,7/4\}$, every power of $S$ in
\eqref{eq:monomial-link-form} is a nonnegative integer.  Expanding
\eqref{eq:Clifford-link-SJ} and using only
\begin{align*}
 \int_0^1u^{i+2}(1-u)^{1/2}\dd u
   &=\mathrm B\left(i+3,\frac32\right),\\
 \int_0^1\theta^{j+1/2}(1-\theta)^{-1/2}\dd\theta
   &=\mathrm B\left(j+\frac32,\frac12\right),
\end{align*}
we obtain
\[
 \frac1{c_0\pi}
 \begin{pmatrix}
  \mathcal B(\phi_{3/4},\phi_{3/4})
   &\mathcal B(\phi_{3/4},\phi_{7/4})\\
  \mathcal B(\phi_{7/4},\phi_{3/4})
   &\mathcal B(\phi_{7/4},\phi_{7/4})
 \end{pmatrix}
 =
 \begin{pmatrix}
  \dfrac{96}{715}
   &\dfrac{471936}{1616615}\\[2mm]
  \dfrac{471936}{1616615}
   &\dfrac{571904}{969969}
 \end{pmatrix}.
\]
In particular, for
\[
 \phi=\phi_{3/4}-\frac12\phi_{7/4}
      =S^{3/4}\left(1-\frac12S\right),
\]
\begin{equation}\label{eq:negative-link-form}
 \mathcal B(\phi,\phi)
 =-c_0\pi\,\frac{9952}{969969}<0.
\end{equation}

Near either singular stratum of $\mathscr S$, the nondegenerate transverse
quadratic models give
\[
 S\asymp\rho^2,
 \qquad
 J=O(\rho^2),
 \qquad
 |\nabla_\Lambda S|=O(\rho),
\]
where $\rho=\dist(\cdot,\mathscr S)$; the last estimate also follows from
\eqref{eq:link-gradS}.  Therefore
\[
 |\phi|=O(\rho^{3/2}),
 \qquad
 |\nabla_\Lambda\phi|=O(\rho^{1/2}).
\]
In particular, both $|\nabla_\Lambda\phi|^2$ and
$|A_\Lambda|^2\phi^2$ are $O(\rho)$.  Since the transverse volume
element has order $\rho^6\dd\rho$, all terms in the form energy are
locally integrable.  Choose smooth cutoffs \(\chi_\eps\) that vanish
for \(\rho\leq\eps\), equal one for \(\rho\geq2\eps\), and satisfy
\(|\nabla_\Lambda\chi_\eps|\leq C/\eps\).
The estimate \eqref{eq:link-tube} and the preceding pointwise bounds give
\[
 \|(1-\chi_\eps)\phi\|_{\mathcal D}^2
 \leq C\eps^8\longrightarrow0.
\]
Indeed, the gradient and curvature terms are bounded by
\(C\int_0^{2\eps}\rho\,\rho^6\dd\rho=O(\eps^8)\);
the cutoff derivative contributes at most
\(C\eps^{-2}\eps^3\eps^7=O(\eps^8)\).
Thus \(\chi_\eps\phi\in C_c^\infty(\Lambda_{\reg})\) approximates
\(\phi\) in \(\mathcal D(\Lambda)\).
The cone stability formula \eqref{eq:exact-Hardy-gap} and
\eqref{eq:negative-link-form} prove
\cref{thm:Clifford-43-instability}.  Together with the sufficient and
dimension-obstruction arguments above, this also completes the proof
of \cref{thm:clifford-classification}.

\begin{remark}[Power multipliers at \((4,3)\)]
\label{sec:obstruction}
\label{prop:Clifford-power-obstruction}
Instability already excludes every positive smooth multiplier with the
two-sided sign near the regular cone: \cref{thm:Jacobi-trace} and
\eqref{eq:ground-state-identity} would otherwise imply stability. In particular,
no choice of exponents in
\[
 G_{\alpha,\gamma}=R^\alpha S^\gamma F,
 \qquad(\alpha,\gamma)\in\R^2,
\]
can give a subcalibration for \((m,q)=(4,3)\).
\Cref{app:Clifford-power-obstruction} additionally locates explicit
points in the actual invariant image with \(R,S>0\), \(F\ne0\),
and \(\nabla G_{\alpha,\gamma}\ne0\), where
\(F\,\diver(\nabla G_{\alpha,\gamma}/|\nabla G_{\alpha,\gamma}|)<0\).
\end{remark}

\subsection{Unequal dimensions: Hopf incidence cubics}

Let $d\in\{2,4,8\}$ and let $h_d:\R^{2d}\to\R^{d+1}$ be the
standard quadratic Hopf map, with $|h_d(y)|^2=|y|^4$. Set
\[
 F_d(x,y)=\langle x,h_d(y)\rangle,\qquad
 x\in\R^{d+1},\quad y\in\R^{2d},\quad u=|y|^2,\quad v=|x|^2.
\]
The bidegree is $(1,2)$, and the Hopf identities give
\begin{equation}\label{eq:incidence-closure}
 (\Gamma(I_i,I_j))_{I=(u,v,F_d)}=
 \begin{pmatrix}4u&0&4F_d\\0&4v&2F_d\\
 4F_d&2F_d&u^2+4uv\end{pmatrix},\qquad
 \Delta I=(4d,2(d+1),0).
\end{equation}
Its determinant is $16(u+4v)(u^2v-F_d^2)$, positive on an open set.
Thus the algebra has three independent generators of degrees $(2,2,3)$.
Furthermore $\cL(F_d)=-8(u+v)F_d$, and the singular set of the
minimal zero hypercone is $\{y=0\}$. At a nonzero point of this
stratum the transverse quadratic form has signature $(d,d)$.

These are the Clifford cubics with parameters $(m,q)=(d,d)$.
By \cref{thm:clifford-classification}, the cases $d=4,8$ are area
minimizing, in ambient dimensions $13,25$. For $d=2$, the transverse
Lawson cone in $\R^4$ is unstable, and the localization argument
following \eqref{bh:Lawson-margin} gives instability in $\R^7$.
For \(d=4\), the degree-six potential is critical:
\(6=(13-1)/2\). Substituting \(T=u^2\) into
\eqref{eq:Clifford-H44} and \eqref{eq:cubic-H-Jacobi-meaning} gives
\[
 \beta_G=\frac{9u^2}{(u+4v)^2}>0
 \quad\text{on }C_{\reg},\qquad
 \inf_{C_{\reg}\cap\Sph^{12}}\beta_G=0.
\]
As \(u\to0\) near the singular stratum, the bound degenerates.
It proves stability but gives no uniform positive Hardy gap.
The three-generator description uses $|h_d(y)|^2=u^2$.
For a general Clifford system, $|\tau(y)|^2$ is a separate quartic
invariant. The computations in \cref{sec:invariants} use the
four-generator algebra $\R[u,v,|\tau|^2,F]$.

The corresponding regular levels on
\(\Sph^d\times\Sph^{2d-1}\) are isoparametric hypersurfaces.
Indeed, on the product of unit round spheres
\(P_d=\Sph^d\times\Sph^{2d-1}\), put \(f_d=F_d|_{P_d}\).
The partial degrees \((1,2)\) and \eqref{eq:incidence-closure} give
\[
 |\nabla_{P_d}f_d|^2=5(1-f_d^2),\qquad
 \Delta_{P_d}f_d=-5d f_d.
\]
Thus the levels \(f_d=t\), \(-1<t<1\), are isoparametric,
and \(f_d=0\) is minimal in \(P_d\).

\section[Nonradial Cubics and Pullbacks]{Nonradial Cubic Examples and Cubic Pullbacks}
\label{sec:nonradial-examples}

Cartan deformations attain the five-generator bound of
\cref{thm:harmonic-cubic-five-bound} in the dimensions allowed by
\cref{cor:nr-cartan-dimensions}. Their two weight eigenvalues in
\eqref{eq:nonradial-minimal-equation} distinguish them from the radial
irreducible cubics known to Fox and Tkachev
\cite[Introduction and Section~2.1]{FoxTkachev24}.
The models in \(\R^8,\R^{14}\) are unstable; five-generator inequalities
prove strict area minimality in \(\R^{26}\). We then pull back the
minimizing cubic targets.

\subsection{Nonradial harmonic models}
\label{sec:nonradial-models}

\subsubsection{One-parameter deformations of Cartan cubics}

Use the notation \(J_d=\operatorname{Herm}_3(\mathbb K_d)\) from
\cref{sec:exceptional-four-generator}. For the three Cayley--Dickson
inclusions \(\mathbb K_d\subset\mathbb K_{2d}\), \(d=1,2,4\), the
involution \(a+b\ell\mapsto a-b\ell\) preserves multiplication
and conjugation by the Cayley--Dickson formula. Its entrywise extension
therefore preserves the Jordan product and defines an automorphism
\(\sigma\) of \(J_{2d}\) with fixed algebra \(J_d\).
Relative to the common diagonal Jordan frame, it refines each
non-diagonal Peirce space as
\[
 J_{ij}\cong\mathbb K_{2d}
       =\mathbb K_d\oplus\mathbb K_d\ell,
 \qquad J_{ij}^{+}\cong\mathbb K_d,\quad
 J_{ij}^{-}\cong\mathbb K_d\ell.
\]
Thus the fixed algebra remains simple of rank three, and its
non-diagonal Peirce spaces have half the dimensions of those of
\(J_{2d}\).
On \(V=(J_{2d})_0\), with inner product
\(\ip{x}{y}=\frac12\tr(x\bullet y)\), this gives the orthogonal splitting
\[
 V=V_+\mathbin{\perp}V_-,\qquad
 V_+=(J_d)_0,\qquad V_-=U\cong\mathbb K_d^3.
\]
The Cartan cubic \(N=\det_{J_{2d}}|_V\) is invariant under \(\sigma\).
Hence \(N(b+u)=N(b-u)\) for \(b\in V_+\), \(u\in V_-\).
Since $N$ is cubic and even in $u$, it has the unique bidegree decomposition
\[
 N(b+u)=A(b)+\mathsf B(b,u),\qquad
 \deg A=(3,0),\quad \deg\mathsf B=(1,2).
\]
We deform the mixed component by setting
\begin{equation}\label{eq:nonradial-pencil}
 F_t=A+t\mathsf B,\qquad t\in\R.
\end{equation}
The choice of these pairs makes \(A=\det_{J_d}|_{(J_d)_0}\) itself a
Cartan cubic; see \cite[Theorem~3.1.1 and
Sections~3.2.6, 6.4.4--6.4.5]{NadirashviliTkachevVladut14} and
\cite{TkachevJordan14}. Both \(N\) and \(A\) are therefore harmonic.
Since \(\Delta=\Delta_b+\Delta_u\) and \(\mathsf B\) is linear in \(b\),
\[
 0=\Delta N=\Delta_bA+\Delta_u\mathsf B,
 \qquad \Delta F_t=(1-t)\Delta_bA=0.
\]
Thus the entire family is harmonic.
With \(\mathbb K=\mathbb K_d\), this identifies \(V\) with
\[
 V_d=\operatorname{Herm}_3(\mathbb K)_0\oplus\mathbb K^3,\qquad
 \ip{(Z,u)}{(Z',u')}=\tfrac12\operatorname{Re}\tr(ZZ')
                         +\operatorname{Re}(u^*u').
\]
Thus \(\dim V_d=6d+2\). The half-trace metric here differs from the
trace metric on the complexification in
\cref{sec:exceptional-four-generator}. Set
\begin{equation}\label{eq:nonradial-invariants}
 \rho=\tfrac12\tr Z^2,\quad \nu=|u|^2,\quad
 A=\det Z,\quad \mathsf{B}=-u^*Zu,\quad W=4u^*Z^2u.
\end{equation}
Then \(R=\rho+\nu\). For the nonradial member \(F_{-1/3}\), the
dictionary with the intrinsic invariants
\eqref{eq:nr-five-invariants} is
\[
 (r,s,a,p,h)=\left(\rho,\nu,\frac{3\sqrt3}{2}A,
             -\frac{\mathsf B}{2\sqrt3},\frac W{48}\right).
\]
The intrinsic normalization is \(\sqrt3 F_{-1/3}/4\), with
\(3/16\) times the quadratic weight below. We verify the model
by its coordinate table before invoking the intrinsic theorem.

To verify the mixed term and the isometry, write
\[
 Z=\begin{pmatrix}r_1&c&\bar b\\ \bar c&r_2&a\\ b&\bar a&r_3\end{pmatrix},
 \qquad r_1+r_2+r_3=0.
\]
In the determinant formula
\(r_1r_2r_3-r_1|a|^2-r_2|b|^2-r_3|c|^2+2\operatorname{Re}(abc)\),
replace \((a,b,c)\) by
\((a+v_1\ell,b+v_2\ell,c+v_3\ell)\). The quadratic part in
\(v\), with \(u=(\bar v_1,\bar v_2,\bar v_3)\), is
\(-u^*Zu=\mathsf{B}\). This verifies the isometry and the signs, including
the inclusion \(\mathbb H\subset\mathbb O\). In particular,
\[
 F_{-1/3}=\frac{4A-N}{3}
   =\frac13\tr\bigl(\sigma(x)\bullet x^{\bullet2}\bigr),
\]
where \(x\in(J_{2d})_0\) corresponds to \((Z,u)\).
After the common scaling above, this is the Cartan reconstruction in
\cref{cor:nr-cartan-dimensions}. These models realize all three ambient
dimensions allowed there; a classification of compatible splittings
would still be needed to classify all nonradial cubics.

\begin{proposition}[The nonradial differential algebra]
\label{prop:nonradial-algebra}
The algebra \(\R[\rho,\nu,A,\mathsf B,W]\) is Laplacian. For
\(F=F_{-1/3}\), it is the smallest Laplacian algebra containing
\(R,F\):
\[
 \cA_R(F)=\R[\rho,\nu,A,\mathsf{B},W].
\]
Its five generators are algebraically independent, and its spherical
quotient dimension is four. Their Laplacians and gradient products are
\begin{equation}\label{eq:nonradial-laplacians}
 \begin{gathered}
 \Delta\rho=6d+4,\qquad\Delta\nu=6d,\qquad
 \Delta A=\Delta \mathsf{B}=0,\\
 \Delta W=16d\rho+\tfrac{16}{3}(3d+2)\nu.
 \end{gathered}
\end{equation}
and, using symmetry, the following relations:
\begin{align}
 \Gamma(\rho,\rho)&=4\rho,&\Gamma(\rho,\nu)&=0,
       &\Gamma(\nu,\nu)&=4\nu,\notag\\
 \Gamma(\rho,A)&=6A,&\Gamma(\nu,A)&=0,
       &\Gamma(\rho,\mathsf{B})&=2\mathsf{B},\notag\\
 \Gamma(\nu,\mathsf{B})&=4\mathsf{B},&\Gamma(\rho,W)&=4W,
       &\Gamma(\nu,W)&=4W.\label{eq:nonradial-quadratic-pairings}
\end{align}
\begin{equation}\label{eq:nonradial-closure}
\begin{gathered}
 \begin{aligned}
 \Gamma(A,A)&=\tfrac43\rho^2,
 &\qquad\Gamma(\mathsf{B},\mathsf{B})&=\tfrac43\nu^2+W,\\
 \Gamma(A,\mathsf{B})&=\tfrac43\rho\nu-\tfrac12W,
 &\qquad\Gamma(A,W)&=16\nu A-\tfrac{16}{3}\rho \mathsf{B},
 \end{aligned}\\
 \Gamma(\mathsf{B},W)=16\rho \mathsf{B}-16\nu A+\tfrac{32}{3}\nu \mathsf{B},\\
 \Gamma(W,W)=16(\rho+\nu)W+\tfrac{64}{3}\mathsf{B}^2-64A\mathsf{B}.
\end{gathered}
\end{equation}
\end{proposition}

\begin{proof}
Put \(Q=uu^*\) and \(m_j=u^*Z^ju\). The trace-free characteristic
identity and the relevant gradients are
\begin{gather*}
 Z^3=\rho Z+AI,\qquad
 \nabla_ZA=2(Z^2-\tfrac23\rho I),\\
 \nabla_Z\mathsf{B}=-2(Q-\tfrac13\nu I),\qquad \nabla_u\mathsf{B}=-2Zu,\\
 \nabla_ZW=8(ZQ+QZ-\tfrac23m_1I),\qquad \nabla_uW=8Z^2u.
\end{gather*}
Their pairings give \eqref{eq:nonradial-closure}, including
\[
 |\nabla_ZW|^2=16\nu W+\tfrac{64}{3}\mathsf{B}^2,\qquad
 |\nabla_uW|^2=64m_4=16\rho W-64A\mathsf{B}.
\]
An orthonormal basis \(E_\alpha\) of \(\operatorname{Herm}_3(\mathbb K)_0\),
with the indicated trace metric, satisfies
\(\sum_\alpha E_\alpha^2=\frac23(3d+2)I\).
Taking diagonal and off-diagonal basis elements proves this identity;
it gives the last formula in \eqref{eq:nonradial-laplacians}.
The other Laplacians follow directly by taking traces. The chain
rule therefore proves that the displayed algebra is Laplacian.

For \(F=F_{-1/3}\) and \(S=|\nabla F|^2\), the table gives
\begin{equation}\label{eq:nonradial-gradient-square}
 S=\tfrac43(\rho-\nu/3)^2+\tfrac49W,\qquad
 \Delta S=\tfrac{32}{27}\bigl((15d+18)\rho+(3d+2)\nu\bigr).
\end{equation}
It also gives
\[
 \Gamma(F,S)=16(\rho+7\nu/27)F.
\]
Thus \(\cL(F)=-qF\), where \(q=8(\rho+7\nu/27)\) is positive
and nonradial. We check the remaining hypotheses needed for
\cref{thm:nonradial-five-generators}.

The Cartan cubic \(A\) is irreducible by its harmonic eikonal identity
\(|\nabla A|^2=4\rho^2/3\). Indeed, if a harmonic cubic
\(P\) in \(n>2\) variables satisfying
\(|\nabla P|^2=\kappa|x|^4\), \(\kappa>0\), had a real linear
factor, write it as
\(P(t,y)=t(at^2+t\ell(y)+Q(y))\).
Restriction to \(t=0\) gives \(Q=b|y|^2\), \(b^2=\kappa\).
Harmonicity gives \(\ell=0\) and \(6a+2b(n-1)=0\).
The coefficient of \(t^2|y|^2\) in the eikonal equation then
gives \(6ab+4b^2=2\kappa\), forcing \(n=2\), a contradiction.
Any real factorization of \(F\) would have a linear factor;
restriction to \(u=0\) would factor \(A\). Thus \(F\) is
irreducible.

If \(D_vF\equiv0\) for a constant vector \(v\), differentiating
\(\cL(F)=-qF\) gives \((D_vq)F=0\). Since \(q\) is positive
definite, this forces \(v=0\). The point \((Z,u)=(\operatorname{diag}(1,-1,0),0)\) is a
regular zero, and the two distinct eigenvalues of \(q\) establish
nonradiality.

The coordinate dictionary above now transfers
\cref{thm:nonradial-five-generators}, with \(m=3d\), to these
models, proving the smallest-algebra identity, algebraic independence,
and spherical quotient dimension four.
\end{proof}

\begin{theorem}[Minimal members of the Cartan deformation]
\label{thm:nonradial-models}
For \(d=1,2,4\), every cubic \(F_t\) in
\eqref{eq:nonradial-pencil} is harmonic. Its regular zero set is
minimal if and only if
\[
 t\in\{0,1,-\tfrac13\}.
\]
The members \(F_0=A\) and \(F_1=N\) are, respectively, a cylindrical
extension of a Cartan cubic and the original Cartan cubic.
The cubic \(F=F_{-1/3}\) is irreducible over \(\R\), has no
cylindrical directions, and satisfies
\begin{equation}\label{eq:nonradial-minimal-equation}
 \Delta F=0,\qquad
 \Gamma(F,S)=2qF,\qquad
 \cL(F)=-qF,\qquad q=8(\rho+7\nu/27).
\end{equation}
Its regular zero set is a minimal hypersurface in ambient dimension
\(8,14,26\), respectively. Its quadratic weight has two distinct
positive eigenvalues, with ratio \(7/27\).
\end{theorem}

\begin{proof}
Harmonicity was proved above. The closure table gives
\begin{equation}\label{eq:nonradial-pencil-gradient}
 S_t:=|\nabla F_t|^2
   =\tfrac43(\rho+t\nu)^2+t(t-1)W.
\end{equation}
Pairing its gradient with \(\nabla F_t\) yields
\begin{equation}\label{eq:nonradial-pencil-minimality}
 \cL(F_t)=-8\bigl(\rho+t^2(2-t)\nu\bigr)F_t
 -\tfrac83t(t-1)(3t+1)(\rho+t\nu)\mathsf B.
\end{equation}
This proves minimality for the three stated values and gives
\eqref{eq:nonradial-minimal-equation} at \(t=-1/3\).
For necessity, take \(Z=\operatorname{diag}(1,1,-2)\) and, for
\(t\ne0\), choose
\[
 u=t^{-1/2}e_3\quad(t>0),\qquad
 u=\sqrt{-2/t}\,e_1\quad(t<0).
\]
Here \(e_i\) are the standard coordinate vectors of \(\mathbb K^3\).
Then \(A=-2\), \(\mathsf B=2/t\), and \(F_t=0\), whereas
\(\nabla_uF_t=-2tZu\ne0\). Moreover, \(\rho+t\nu\) equals
\(4\) or \(1\), respectively. Thus
\eqref{eq:nonradial-pencil-minimality} cannot vanish at this regular
zero unless \(t=1\) or \(t=-1/3\).

The remaining assertions for \(F=F_{-1/3}\) were established in
\cref{prop:nonradial-algebra} while verifying the hypotheses of the
intrinsic five-generator theorem.
\end{proof}

These cones are not congruent to radial cubic cones. Indeed, if
\(P\circ Q\) has the same zero set as \(F\), with \(Q\) orthogonal
and \(P\) a radial cubic, the smooth-patch divisibility fact proved
in \cref{prop:minimizing-irreducibility} gives \(P\circ Q=cF\).
Here \(F\) is irreducible and has regular real zeros by
\cref{lem:square-free-critical-set}. Orthogonal equivariance of
\(\cL\) and \(\cL(cF)=c^3\cL(F)\) would then make \(c^2q\)
a multiple of \(R\), contradicting
\eqref{eq:nonradial-minimal-equation}.

For the rest of this section, put \(F=F_{-1/3}\) and
\(S=|\nabla F|^2\). The generator degrees give
\[
 (\cA_R(F))_2=\operatorname{span}\{\rho,\nu\},\qquad
 (\cA_R(F))_3=\operatorname{span}\{A,\mathsf{B}\}.
\]
Thus the cubic part is two-dimensional, unlike the hypothesis of
\cref{thm:cubic-one-dimensional-degree-three}.

\subsubsection{Singularities and the two unstable models}
\label{sec:nonradial-singularities}

\begin{lemma}\label{lem:nonradial-singularities}
The singular set of \(\{F=0\}\), equivalently the critical set of \(F\), is
\begin{equation}\label{eq:nonradial-critical-set}
 \Sing\{F=0\}=\Sigma=\{(Z,u):\nu=3\rho,\ Zu=0\}.
\end{equation}
The punctured set \(\Sigma\setminus\{0\}\) is smooth of
dimension \(4d\). At
\[
 c=(\operatorname{diag}(1,-1,0),(0,0,\sqrt3)),
\]
the Hessian of \(F\) has eigenvalues \(8/3,-8/3,0\), with
multiplicities \(d+1,d+1,4d\), respectively, and its kernel is
\(T_c\Sigma\). Every nonzero point of \(\Sigma\) is obtained
from \(c\) by a Hermitian change of basis and a positive dilation.
The zero sets, rescaled about such a point, converge smoothly on
compact subsets of the regular part of
\begin{equation}\label{eq:nonradial-transverse-cone}
 \R^{4d}\times\{(v_+,v_-)\in\R^{d+1}\oplus\R^{d+1}:
                           |v_+|=|v_-|\}.
\end{equation}
\end{lemma}

\begin{proof}
Equation \eqref{eq:nonradial-gradient-square} proves the description
of \(\Sigma\). For a nonzero \(u\), the restriction of \(Z\)
to \(u^\perp\) is a trace-free Hermitian \(2\times2\) matrix
of a fixed positive norm. These matrices form a sphere of dimension
\(d\), so \(\Sigma\setminus\{0\}\) is a smooth sphere bundle
over \(\mathbb K^3\setminus\{0\}\), of dimension \(4d\).
Its eigenvalues are \((\lambda,-\lambda,0)\), \(\lambda>0\),
which also proves the assertion about changes of basis and dilations.

At \(c\), let the diagonal variation of \(Z\) be
\((\xi,\eta,-\xi-\eta)\), let \(h_i\) be the variations of
\(u_i\), and use \(a,b,c'\) for the off-diagonal variations
in the matrix displayed above. The quadratic term of \(F\) is
\[
 Q_2=\xi^2-\eta^2-\tfrac2{\sqrt3}(\xi+\eta)
             \operatorname{Re}h_3
       +|b+\bar h_1/\sqrt3|^2-|a-h_2/\sqrt3|^2.
\]
The first three terms become \(8st'/3\) under the orthonormal
changes
\[
 s=\tfrac{\sqrt3}{2}(\xi+\eta),\qquad
 t=\tfrac12(\xi-\eta),\qquad
 t'=\tfrac12(\sqrt3t-\operatorname{Re}h_3).
\]
Completing to orthonormal coordinates gives
\(Q_2=\frac43(|v_+|^2-|v_-|^2)\) and the stated spectrum.
Differentiating \(\nabla F=0\) along \(\Sigma\), and comparing
dimensions, identifies the kernel. Finally,
\(r^{-2}F(c+rx)=Q_2(x)+rF(x)\), so
\cref{sec:tangent-instability-transfer} gives the tangent and smooth
regular graphical convergence.
The transverse cone is not a hyperplane, so every nonzero point of
\(\Sigma\) is geometrically singular; the vertex is singular by homogeneity.
\end{proof}

By \cref{prop:harmonic-positive-multiplier-obstruction}, a smooth
nonnegative subcalibrating multiplier must vanish on
\(\Sigma\setminus\{0\}\); in particular, radial powers fail.

\begin{proposition}\label{prop:nonradial-unstable}
The nonradial cones in ambient dimensions \(8\) and \(14\)
are unstable on their regular parts and are not area minimizing.
\end{proposition}

\begin{proof}
The transverse cone in \eqref{eq:nonradial-transverse-cone} has
dimension \(2d+1\) and squared second fundamental form
\(2d/r^2\). Its radial second variation, up to a positive factor,
is
\[
 \int_0^\infty\left(|\phi'|^2-\frac{2d}{r^2}\phi^2\right)
                                                      r^{2d}\,dr.
\]
The sharp Hardy constant \((2d-1)^2/4\) is smaller than \(2d\)
for \(d=1,2\), so a compactly supported test function on a radial
annulus gives a negative value. \Cref{lem:tangent-instability-transfer},
using the regular graphical convergence in \cref{lem:nonradial-singularities},
proves instability of the original cone.
\end{proof}

The boundary formula suggests a gradient factor.
Put \(K=\Delta S\) and \(T=|\nabla S|^2\). For \(H=F\sqrt S\),
the boundary speed is \(w=S^{-1}\).
On \(C_{\reg}\), the closure table gives
\(T/S=\frac{64}{27}(9\rho+\nu)\). Using
\eqref{eq:nonradial-gradient-square} and
\eqref{eq:nonradial-minimal-equation} in
\cref{prop:harmonic-power-Jacobi} yields
\begin{equation}\label{eq:nonradial-Jacobi-trace}
 -\frac{J_C(S^{-1})}{S^{-1}}
 =\frac{K/2-2q-T/S}{S}
 =\frac{16(d-3)(5\rho+\nu)}{9S}.
\end{equation}
For \(d=1,2\), this candidate fails the boundary condition. For \(d=4\),
its trace is strictly positive; we now prove the global subcalibration sign.

\subsubsection{The minimizing nonradial cone in dimension twenty-six}
\label{sec:nonradial-certificate}

\begin{theorem}\label{thm:nonradial-minimizing}
For \(d=4\), both homogeneous potentials
\[
 H=F\sqrt S,\qquad G=F\sqrt{RS}
\]
generate two-sided subcalibrations for \(E=\{F<0\}\).
Their divergence signs are strict on \(\{F\ne0\}\).
Consequently \(E\) is locally perimeter minimizing and
\(\partial\llbracket E\rrbracket\) is strictly area minimizing,
with support exactly \(\{F=0\}\subset\R^{26}\).
\end{theorem}

\begin{proof}
We first prove the assertion for the degree-five potential \(H\).
The closure table gives
\begin{equation}\label{eq:nonradial-T}
 T=\tfrac{64}{27}(9\rho+\nu)S-\tfrac{1024}{81}F\mathsf{B}.
\end{equation}
On \(S>0\), differentiation yields
\begin{align}\label{eq:nonradial-H-derivatives}
 \Delta H&=\frac{F}{\sqrt S}
                  \left(2q+\tfrac12K-\frac{T}{4S}\right),\notag\\
 |\nabla H|^2&=S^2+2qF^2+\frac{F^2T}{4S}\ge S^2.
\end{align}
The resulting defect is
\begin{align}\label{eq:nonradial-intermediate}
 \frac{\cL(H)}{S^{3/2}F}
 ={}&\tfrac12K-2q-\frac TS-\frac{F\Gamma(F,q)}S\notag\\
 &+\frac{F^2}{S^2}
       \left(qK+2q^2-\tfrac12\Gamma(S,q)\right)
       -\frac{F\Gamma(F,T)}{8S^2}\notag\\
 &+\frac{F^2}{8S^3}
                    \left(TK-\tfrac12\Gamma(S,T)\right).
\end{align}
The quotient is initially written for \(F\ne0\); multiplying
through gives an identity on all of \(\{S>0\}\).
Substituting the table gives
\begin{equation}\label{eq:nonradial-certificate}
 \cL(H)=\frac{16F}{729S^{3/2}}\mathcal N,
\end{equation}
where
\begin{align}\label{eq:nonradial-N}
 \mathcal N={}&(405\rho+81\nu)S^3+144F\mathsf{B}S^2\notag\\
 &+F^2S(46692\rho^2+20040\rho\nu+2212\nu^2-4239S)\notag\\
 &-(4992\rho+896\nu)F^3\mathsf{B}-256\nu F^4.
\end{align}

To prove its sign, define
\[
\begin{aligned}
 S_* &=\tfrac43\rho^2+\tfrac{40}{27}\rho\nu+\tfrac4{27}\nu^2,\\
 K_0&=4992\rho+896\nu,\\
 H_0&=\tfrac89(36639\rho^2+12634\rho\nu+1555\nu^2).
\end{aligned}
\]
The invariant image satisfies
\begin{equation}\label{eq:nonradial-bounds}
 S\le S_*,\qquad 9F^2\le RS,\qquad
 16\mathsf{B}^2\le9\nu S,\qquad
 F\mathsf{B}\le\tfrac3{16}(9\rho+\nu)S.
\end{equation}
Indeed, the trace-free eigenvalues of \(Z\) obey
\(\lambda_i^2\le4\rho/3\), whence
\(W\le16\rho\nu/3\) and the first inequality follows from
\eqref{eq:nonradial-gradient-square}. The second is Euler's
identity and Cauchy--Schwarz. The third follows from
\(\mathsf{B}^2\le\nu u^*Z^2u=\nu W/4\le9\nu S/16\).
The last is \(T\ge0\) in \eqref{eq:nonradial-T}.
There is the exact decomposition
\begin{align}\label{eq:nonradial-positive-certificate}
 \mathcal N={}&405\rho S^3+36S^2(F+2\mathsf{B})^2
                      +S^2(81\nu S-144\mathsf{B}^2)+F^2SH_0\notag\\
 &+K_0F^2\left(\tfrac3{16}(9\rho+\nu)S-F\mathsf{B}\right)\notag\\
 &+256\nu F^2\left(\tfrac19RS-F^2\right)
                      +4275F^2S(S_*-S).
\end{align}
Expansion recovers \eqref{eq:nonradial-N}; every summand is
nonnegative by \eqref{eq:nonradial-bounds}. Since $F\ne0$ implies
$\rho>0$ and $S>0$, the first summand proves the required strict
sign for \(\cL(H)\) off \(\{F=0\}\).

For the degree-six potential \(G=\sqrt R\,H\), we use
\eqref{eq:nonradial-H-derivatives} and
\eqref{eq:nonradial-T}, which give
\[
 \Delta H=\frac F{\sqrt S}
 \left(\frac{512}{9}\rho+\frac{320}{27}\nu
                         +\frac{256F\mathsf{B}}{81S}\right).
\]
The second and third bounds in \eqref{eq:nonradial-bounds} imply
\[
 \frac{|F\mathsf{B}|}{S}\le\frac14\sqrt{\nu(\rho+\nu)}
                   \le\frac{\rho+2\nu}{8}.
\]
Thus the parenthesis is at least
\((4576\rho+896\nu)/81>0\) for \(R>0\), and
\(H\Delta H\ge0\). On \(\{S>0\}\), the degree-five
potential \(H\) is smooth, satisfies \(H\cL(H)\ge0\), and
has nonvanishing gradient by \eqref{eq:nonradial-H-derivatives}.
\Cref{prop:cubic-radial-lift}, with degree \(5\) and ambient
dimension \(26\ge2\cdot5+2\), therefore gives
\(G\cL(G)\ge0\), strictly where \(F\ne0\), as well as
\(\nabla G\ne0\) on \(\{S>0\}\).

The exceptional set has codimension ten by
\cref{lem:nonradial-singularities}. Both potentials extend continuously
by zero across it with the signs of \(F\), and their normalized gradients
agree with the outward normal at regular zeros. Thus
\cref{thm:removable-subcalibration} applies; harmonicity gives the full
zero-set support. For both fields the strict sign on the two nonempty
phases gives \eqref{eq:two-phase-flux}, so
\cref{thm:strict-minimality-flux} proves strict area minimality.
\end{proof}

The degree-five potential \(H\) also proves strict stability directly.
Indeed, the cone has dimension \(n=25\), and its subcalibration sign
gives a nonnegative Jacobi boundary defect. Hence
\cref{thm:ground-state-gap} yields
\[
 Q_C(u)\geq\left(5-\frac{25}{2}\right)^2\int_Cr^{-2}u^2
       =\frac{225}{4}\int_Cr^{-2}u^2,
 \qquad u\in C_c^\infty(C_{\reg}\setminus\{0\}).
\]

\subsection{Cubic targets}

The common cubic potential gives the following pullback construction,
including for targets with nonisolated singularities.

\begin{corollary}[Pullback of the common cubic potential]
\label{cor:common-cubic-pullback}
Let \(f\not\equiv0\) be a harmonic cubic on \(\R^m\), and put
\(q=|y|^2\), \(S=|\nabla_yf|^2\), and \(H=f\sqrt{qS}\).
Suppose, on \(S>0\), that
\[
 f\Gamma_y(f,S)\geq0,\qquad H\cL_y(H)\geq0.
\]
For every quadratic Clifford map \(P:\R^N\to\R^m\) with
\(N\geq32\), set \(F=f(P)\), \(Q=q(P)\), and \(T=S(P)\).
Then
\[
 G=R^2F\sqrt{QT}
\]
is a subcalibration potential for \(\{F<0\}\). Its boundary current
is locally mass minimizing and strictly area minimizing, with support
\(\{F=0\}\).
The defining polynomial has degree six and the potential has
homogeneity degree sixteen.
\end{corollary}

\begin{proof}
Put \(H_0=f\sqrt S\). Bochner's identity and \(\Delta_y f=0\) give
\[
 H_0\Delta_yH_0
 =f^2\left(|\Hess f|^2-
       \frac{|\Hess f(\nabla f)|^2}{S}\right)
       +f\Gamma_y(f,S)\geq0.
\]
Since \(H_0\) has degree five, differentiating \(H=\sqrt q\,H_0\)
yields
\[
 H\Delta_yH=qH_0\Delta_yH_0+(m+9)H_0^2\geq0.
\]
Thus the ratio condition holds with \(M=0\). Euler's identity gives
\(\nabla H\ne0\) where \(H\ne0\); at a zero with \(S>0\),
\(\nabla H=\sqrt{qS}\nabla f\ne0\).
The exceptional set \(Z=\{\nabla f=0\}\) is conical and algebraic,
and Euler's identity gives \(Z\subset\{f=0\}\).
Its codimension is at least two by harmonicity and
\cref{lem:square-free-critical-set}.
Now apply \cref{thm:clifford-area-transfer} with \(\ell=6\),
\(a=2\), and \(M=0\): its bound is \(N\geq1+24+4+3=32\).
Harmonicity gives both signs near every target zero and hence the
support and strict-minimality assertions.
\end{proof}

The first hypothesis holds for the radial cubic models, where
\(\Gamma(f,S)\) is a positive multiple of \(qf\), and for the
nonradial models of \cref{thm:nonradial-models}, whose quadratic
minimality weight is positive. The second holds in the minimizing ranges of
\cref{thm:cubic-unified-sextic,thm:nonradial-minimizing}.
Since \(q,f,S\in\cA_q(f)\), this potential is a function of the
smallest target algebra. For an essential target with \(r\) minimal
homogeneous generators and a pullback with \(Q<R^2\) somewhere,
\cref{thm:smallest-clifford-pullback} gives exactly \(r+1\) source
generators. Thus the three-, four-, and five-generator cubic models
yield four-, five-, and six-generator pullback algebras whenever the
two displayed inequalities hold.

For the minimizing nonradial target in \cref{thm:nonradial-minimizing},
the degree-five potential \(H_0=f\sqrt S\) already satisfies
\(H_0\cL(H_0)\ge0\), and the Bochner identity above gives
\(H_0\Delta H_0\ge0\). Thus \cref{thm:clifford-area-transfer}, with
\(\ell=5\), \(a=1\), and \(M=0\), gives the degree-twelve potential
\(G=RF\sqrt T\) for every such Clifford map with \(N\ge28\).

For Cartan targets, a weighted inequality gives the lower-degree
pullback \(G=H(P)\).

\subsubsection{Cartan targets with a polynomial pullback potential}

Cartan's harmonic cubics on \(\R^m\), for \(m=14,26\), can be
normalized so that
\begin{equation}\label{eq:Cartan-target-normalization}
 \Delta f=0,\qquad |\nabla f|^2=9q^2,\qquad |f|^2\leq q^3,
 \qquad q=|y|^2.
\end{equation}
Their multiplicities are four and eight, respectively; see
\cite[Theorems~2.5 and~4.1]{Chi20}.  The algebra \(\R[q,f]\) is
Laplacian, with \(\Gamma(q,f)=6f\).

\begin{theorem}[Cartan cubic pullbacks]
\label{thm:cartan-pullback}
Let \(f\) be either normalized Cartan cubic in
\eqref{eq:Cartan-target-normalization}, and let
\(P:\R^N\to\R^m\) be a quadratic Clifford map.  Then
\[
 F=f\circ P,\qquad Q=|P|^2,\qquad G=QF
\]
give a perimeter-minimizing phase \(\{F<0\}\).
Its boundary current \(C\) is strictly area minimizing and is supported
on \(\{F=0\}\). It is strictly stable, with the explicit bound
\begin{equation}\label{eq:Cartan-pullback-Hardy}
 Q_C(u)\geq
 \left[\left(10-\frac{N-1}{2}\right)^2+8m-106\right]
 \int_Cr^{-2}u^2,
 \qquad u\in C_c^\infty(C_{\reg}\setminus\{0\}).
\end{equation}
The polynomial \(F\) has degree six and the
subcalibration potential \(G\) has degree ten.  The construction
exists for \((m,N)=(14,256)\) and \((26,8192)\), and for direct
sums of either Clifford system.
\end{theorem}

\begin{proof}
Use the degree-five target potential \(H=qf\) from
\cref{ex:harmonic-isoparametric-radial}. The target closure gives
\begin{align}
 |\nabla H|^2&=9q^4+16qf^2,\qquad
 \Delta H=2(m+6)f,\label{eq:Cartan-target-gradient}\\
 \cL(H)&=9(2m-24)q^4f+16(2m+1)qf^3.
 \label{eq:Cartan-target-defect}
\end{align}
Substituting the target closure in
\(\cL(H)=|\nabla H|^2\Delta H-\Gamma(H,|\nabla H|^2)/2\) gives
the second identity. Hence
\begin{equation}\label{eq:Cartan-target-weighted-defect}
 2q\cL(H)-5H|\nabla H|^2
 =q^2f\bigl(9(4m-53)q^3+16(4m-3)f^2\bigr).
\end{equation}
Both coefficients in parentheses are positive for \(m=14,26\).
After multiplication by \(H\), this is precisely
\eqref{eq:pullback-target-sign} for \(\kappa=1/2\).

The smallest source algebra is \(\cA_R(F)=\R[R,Q,F]\) by
\cref{thm:smallest-clifford-pullback}. Indeed, the target algebra is
\(\cA_q(f)=\R[q,f]\), and \(|\nabla f|^2=9q^2\) excludes a
nonzero cylindrical direction. For example,
\[
 \Delta Q=8mR,\quad \Delta F=0,\quad
 \Gamma(Q,F)=24RF,\quad |\nabla F|^2=36RQ^2.
\]
The target inequality therefore gives the required source sign by
\cref{prop:pullback-subcalibration}; the expanded gradient and defect
are recorded in \cref{app:Cartan-source-certificate}.

By \eqref{eq:Cartan-target-gradient}, \(\nabla H\ne0\) off the
target origin. The pullback results
\cref{prop:pullback-subcalibration,lem:Clifford-algebra-pullback}
therefore give comparison across \(P^{-1}(0)\), of codimension \(m\)
away from the origin. Surjectivity makes \(F\ne0\); harmonicity
identifies the full support and gives both phases by
\cref{rem:polynomial-finite-perimeter}. The strict sign in
\eqref{eq:Cartan-pullback-defect} then invokes
\cref{cor:polynomial-strict-minimality}.

For the improved Hardy bound, compute the boundary term. On the regular cone, \(a=Q\) and \(|\nabla F|^2=36RQ^2\), so
\eqref{eq:boundary-potential-algebraic} gives
\[
 \beta_G=\left.\frac{R\cK_F(Q)}{Q^3|\nabla F|^2}\right|_C
 =(8m-96)\frac{R^2}{Q}-10\geq8m-106>0.
\]
Apply \cref{thm:ground-state-gap} with \(\ell=10\) to obtain
\eqref{eq:Cartan-pullback-Hardy}.

Finally, the standard Clifford module dimensions give systems with
fourteen matrices on \(\R^{256}\) and twenty-six matrices on
\(\R^{8192}\); see \cite[Table~4.1]{RiedlerThesis25}.  Direct sums preserve
the Clifford relations and all the identities used above.
\end{proof}

\section[Spectral Algebras and Positivity]{Spectral Laplacian Algebras and a Uniform Positivity Criterion}
\label{sec:spectral-positivity}

From the spectral gradient-pairing table below and one Laplacian
coefficient bound, \cref{thm:spectral-gradient-subcalibration} constructs
\(G=F|\nabla F|\) for octonionic products and Pfaffians
(\cref{sec:norm-products-pfaffians}); complex determinants require
another multiplier (\cref{sec:complex-determinants}).

\subsection{Common gradient pairings and the smallest algebra}

Fix \(d\ge2\).  For \(r=(r_1,\ldots,r_d)\), let \(e_k(r)\)
be the elementary symmetric polynomial of degree \(k\), with
\(e_0=1\) and \(e_k=0\) for \(k<0\) or \(k>d\).
Consider real homogeneous polynomials \(F,E_1,\ldots,E_d\) on
\(\R^N\), of degrees \(d,2,\ldots,2d\), such that
\(F\not\equiv0\), \(E_1=R=|x|^2\), and at each point
\begin{equation}\label{eq:spectral-nonnegative-roots}
 (E_1(x),\ldots,E_d(x))=(e_1(r),\ldots,e_d(r))
 \quad\text{for some }r_1,\ldots,r_d\ge0.
\end{equation}
No differentiable choice or ordering of these roots is required.
Assume the identities
\begin{align}
 \Delta F&=0,& \Gamma(F,F)&=E_{d-1},
 \label{eq:spectral-F-pairings}\\
 \Gamma(F,E_k)&=2(d-k+1)E_{k-1}F,
 &&1\le k\le d,
 \label{eq:spectral-mixed-pairings}\\
 \Gamma(E_i,E_j)&=
 4\sum_{\alpha=1}^d r_\alpha
           \partial_\alpha e_i(r)\partial_\alpha e_j(r).
 \label{eq:spectral-root-pairings}
\end{align}
The right-hand side of \eqref{eq:spectral-root-pairings} is a
symmetric polynomial in the roots.  Its explicit polynomial form is
\begin{equation}\label{eq:spectral-polynomial-pairings}
 \Gamma(E_i,E_j)=4\sum_{\ell=0}^{i-1}
 (j-i+1+2\ell)E_{i-1-\ell}E_{j+\ell},
 \qquad 1\le i\le j\le d.
\end{equation}
To verify this formula, let \(p(t)=\prod_\alpha(1+r_\alpha t)\).
The generating function for one quarter of the pairings is
\[
 \sum_\alpha r_\alpha
       \partial_\alpha p(s)\partial_\alpha p(t)
 =\frac{st}{s-t}
       \bigl(sp'(s)p(t)-tp'(t)p(s)\bigr).
\]
The numerator vanishes at \(s=t\).  Polynomial division and
comparison of the coefficients of \(s^it^j\) give
\eqref{eq:spectral-polynomial-pairings}.

\begin{proposition}[Finite closure and exact generation]
\label{prop:spectral-smallest-algebra}
Suppose \eqref{eq:spectral-F-pairings}--\eqref{eq:spectral-root-pairings}
hold and
\begin{equation}\label{eq:spectral-Laplacian-drift}
 \Delta E_k=b_kE_{k-1}\qquad(1\le k\le d)
\end{equation}
for real constants \(b_k\).  Then
\(\R[E_1,\ldots,E_d,F]\) is a Laplacian algebra.
If \(d\ge3\) and \(b_k\ne0\) for \(2\le k\le d-1\), then
\begin{equation}\label{eq:spectral-smallest-algebra}
 \cA_R(F)=\R[E_1,\ldots,E_d,F].
\end{equation}
For \(d=2\), one instead has \(\cA_R(F)=\R[R,F]\).
If the \(d+1\) generators in \eqref{eq:spectral-smallest-algebra}
have independent differentials at some point, the spherical quotient
has dimension \(d\).
\end{proposition}

\begin{proof}
The generator criterion \cref{prop:MR-generator-test} gives closure.
Conversely, \(E_{d-1}=\Gamma(F,F)\in\cA_R(F)\), and successive
Laplacians recover \(E_{d-2},\ldots,E_1\) by
\eqref{eq:spectral-Laplacian-drift}.  For \(d\ge3\),
\eqref{eq:spectral-polynomial-pairings} gives
\begin{equation}\label{eq:norm-product-top-recovery}
 \Gamma(E_2,E_{d-1})
 =4\bigl((d-2)E_1E_{d-1}+dE_d\bigr),
\end{equation}
which recovers \(E_d\).  This proves equality of the algebras.
For \(d=2\), the identities
\(\Gamma(F,F)=R\), \(\Gamma(R,F)=4F\),
\(\Gamma(R,R)=4R\), \(\Delta F=0\), and \(\Delta R=2N\)
already close on \(R,F\).  Finally, independence gives quotient
dimension \(d+1\) in the Euclidean space; fixing the nonzero radial
generator \(E_1\) decreases this by one.
\end{proof}

The positivity theorem below requires only a lower bound for
\(\Delta E_{d-1}\), rather than the full drift formula
\eqref{eq:spectral-Laplacian-drift}.

\subsection{Two elementary inequalities}

Set
\begin{equation}\label{eq:spectral-SBCD}
 S=E_{d-1},\qquad B=E_{d-2},\qquad
 C=E_{d-3},\qquad D=E_d.
\end{equation}
In this subsection the letter \(C\) denotes this invariant, not the
zero cone.  All formulas will be used on \(S>0\), where
\(B>0\) and every \(E_k\) with \(0\le k\le d-1\) is positive.

\begin{lemma}[A Newton inequality and an antitone coefficient]
\label{lem:spectral-Newton-monotonicity}
For nonnegative roots and \(d\ge3\),
\begin{equation}\label{eq:spectral-Newton}
 BS\ge\frac{3d}{d-2}CD.
\end{equation}
For every \(d\ge2\), the function
\begin{equation}\label{eq:spectral-Q}
 Q_* =\frac{5B}{S}+\frac{3CD}{S^2}
\end{equation}
is positive and nonincreasing in each root \(r_i\) wherever
\(S>0\).
\end{lemma}

\begin{proof}
First suppose all roots are positive and put \(t_i=r_i^{-1}\).
Writing \(p_k=e_k(t_1,\ldots,t_d)\), one has
\(S=Dp_1\), \(B=Dp_2\), and \(C=Dp_3\).
For each triple of nonnegative numbers, the sum of the six monomials
\(t_i^2t_j\), with distinct indices in the triple, is at least
\(6t_it_jt_k\), by the arithmetic--geometric mean inequality.
Summing over triples counts each ordered pair \(d-2\) times.
Consequently
\[
 p_1p_2=\sum_{i\ne j}t_i^2t_j+3p_3
 \ge\frac{3d}{d-2}p_3,
\]
which proves \eqref{eq:spectral-Newton}; continuity includes zero roots.

In the reciprocal variables, \(Q_* =5p_2/p_1+3p_3/p_1^2\).
Fix \(t_i=z\), and let \(u,v,w\) be the elementary symmetric
polynomials of orders \(1,2,3\) in the remaining variables.
Then \(p_1=z+u\), \(p_2=zu+v\), and \(p_3=zv+w\), so
\begin{equation}\label{eq:spectral-Q-derivative}
 \frac{\partial Q_*}{\partial z}
 =\frac{z(5u^2-8v)+5u^3-2uv-6w}{(z+u)^3}\ge0.
\end{equation}
Indeed \(u^2\ge2v\) and \(u^3\ge6w\); the numerator is at least
\(zu^2+3u^3\).  Since \(\partial_{r_i}=-t_i^2\partial_{t_i}\),
this proves the asserted sign.  Formula \eqref{eq:spectral-Q}
is smooth in the original roots wherever \(S>0\), so its derivative
sign extends to zero roots by continuity.  Positivity follows from
\(B>0\).  For \(d=2\), the same argument has \(v=w=0\), or
one can use \(Q_*=5/E_1\) directly.
\end{proof}

\subsection{Subcalibration existence from a Laplacian bound}

\begin{theorem}[Spectral subcalibrations]
\label{thm:spectral-gradient-subcalibration}
Assume \eqref{eq:spectral-nonnegative-roots}--\eqref{eq:spectral-root-pairings}.
Suppose on \(U=\{S>0\}\) that
\begin{equation}\label{eq:spectral-threshold}
 \Delta S\ge\beta B,\qquad
 \beta\ge24-\frac{16}{d}.
\end{equation}
Then \(N\ge4d\), and there is a potential \(G\), homogeneous of
degree \(2d-1\),
with nonzero gradient on \(U\) and satisfying \(F\cL(G)\ge0\).
It is given explicitly by
\begin{equation}\label{eq:spectral-gradient-potential}
 G=F\sqrt S=F|\nabla F|.
\end{equation}
Writing \(X=\nabla G/|\nabla G|\), one moreover has
\begin{equation}\label{eq:spectral-strict-divergence}
 \operatorname{sgn}(F)\diver X>0
 \qquad\text{on }\{F\ne0\}.
\end{equation}
The field \(X\) is a two-sided subcalibration of \(\{F<0\}\)
across \(Z=\{S=0\}\), and the perimeter and current conclusions of
\cref{thm:removable-subcalibration} hold. The cone
\(\Sigma=\{F=0\}\) is strictly area minimizing and strictly
stable, with
\begin{equation}\label{eq:spectral-Hardy-gap}
 Q_\Sigma(u)\ge\frac{(N-4d+1)^2}{4}
       \int_\Sigma r^{-2}u^2,
 \qquad u\in C_c^\infty(\Sigma_{\reg}\setminus\{0\}).
\end{equation}
The constant in this estimate is at least \(1/4\).
\end{theorem}

\begin{proof}
For positive roots, the decreasing ratios of the normalized elementary symmetric
functions \(e_j/\binom dj\), as given by Newton's inequalities
\cite[Section~3.1]{BorceaBrandenLiggett09}, imply
\begin{equation}\label{eq:spectral-radial-Newton}
 R B\ge\frac{d(d-1)}2S.
\end{equation}
This also holds at zero roots by continuity, and is an equality for
\(d=2\). Since \(S=|\nabla F|^2\) is a nonzero polynomial,
\(U\) is dense, so \eqref{eq:spectral-threshold} extends to all
points by continuity. Integrating on the unit sphere and using the
homogeneous degree \(2d-2\) of \(S\) gives
\[
 (2d-2)(N+2d-4)\int_{\Sph^{N-1}}S
 =\int_{\Sph^{N-1}}\Delta S
 \ge(12d-8)(d-1)\int_{\Sph^{N-1}}S.
\]
The integral of \(S\) is positive. Division proves \(N\ge4d\).

The mixed pairing at \(k=d-1\) gives
\(\Gamma(F,S)=4BF\).  Thus the regular zero set is minimal and
\eqref{eq:harmonic-minimal-weight} holds with \(q=2B\).
By \eqref{eq:spectral-polynomial-pairings}, \(T=\Gamma(S,S)=4(BS+3CD)\).
Take \(h=\tfrac12\log S\), and write
\(\delta=\beta/4\).  The auxiliary function in
\eqref{eq:harmonic-gradient-multiplier-data} is exactly \(Q_*\).
Moreover,
\begin{align}
 \Delta h&\ge2(\delta-1)\frac BS-6\frac{CD}{S^2},
 \label{eq:spectral-log-Laplacian}\\
 D_h&\ge2(\delta-4)B-12\frac{CD}{S}.
 \label{eq:spectral-boundary-bound}
\end{align}
For \(d\ge3\), \cref{lem:spectral-Newton-monotonicity} implies
\(BS\ge3CD\); hence
\(\Delta h\ge2(\delta-2)B/S>0\).  The sharper constant gives
\[
 D_h\ge2\left(\delta-6+\frac4d\right)B\ge0.
\]
For \(d=2\), \(C=0\) and \(\delta\ge4\), so
\(\Delta h\ge2(\delta-1)/S>0\) and \(D_h\ge0\) follow
directly from \eqref{eq:spectral-log-Laplacian}--\eqref{eq:spectral-boundary-bound}.

For every smooth symmetric function \(u(r)\)
expressed in the generators, the mixed pairings give
\begin{equation}\label{eq:spectral-root-derivation}
 \Gamma(F,u)=2F\sum_i\partial_{r_i}u,
\end{equation}
because \(\sum_i\partial_{r_i}e_k=(d-k+1)e_{k-1}\).
The root pairing similarly gives
\[
 \Gamma(h,Q_*)=
 \frac2S\sum_i r_i(\partial_{r_i}S)(\partial_{r_i}Q_*)\le0,
 \qquad
 F\Gamma(F,Q_*)=2F^2\sum_i\partial_{r_i}Q_*\le0.
\]
Here \(\partial_{r_i}S\ge0\), and the other derivative has the
opposite sign by \cref{lem:spectral-Newton-monotonicity}.
The invariant identities extend to repeated and zero roots, verifying
\cref{thm:harmonic-monotonicity-criterion}. In particular,
\begin{equation}\label{eq:spectral-gradient-norm}
 |\nabla G|^2=S^2+5BF^2+\frac{3CD}{S}F^2\ge S^2>0.
\end{equation}
Since \(Q_*>0\) and \(\Delta h>0\), the term
\(F^2Q_*\Delta h\) in \eqref{eq:harmonic-log-defect} is strictly
positive when \(F\ne0\), while all the other terms in that bracket
are nonnegative. Thus \(F\cL(G)>0\) there. Euler's identity shows
that \(F\ne0\) implies \(S>0\), so this proves
\eqref{eq:spectral-strict-divergence} on both entire phases.

Harmonicity makes \(Z\) removable by
\cref{lem:square-free-critical-set}; the monotonicity theorem gives
comparison with full support by \cref{rem:polynomial-finite-perimeter}.
Both phases are nonempty, and \eqref{eq:spectral-strict-divergence}
verifies the flux hypothesis of \cref{thm:strict-minimality-flux}.
Homogeneity therefore gives strict area minimality.

For strict stability, we use the boundary Jacobi estimate.
At a smooth point of the full harmonic
nodal hypersurface, the
two local phases have opposite signs, and the Hopf boundary lemma
gives \(\nabla F\ne0\). Thus the regular cone lies in \(U\), where
\eqref{eq:harmonic-log-Jacobi} gives
\(V_G=D_h/S\ge0\). Apply \cref{thm:ground-state-gap} with
\(n=N-1\) and \(\ell=2d-1\) to obtain
\eqref{eq:spectral-Hardy-gap}, since
\[
 \left(\ell-\frac n2\right)^2
 =\frac{(N-4d+1)^2}{4}.
\]
The bound \(N\ge4d\) makes this constant at least \(1/4\).
\end{proof}

\begin{remark}[An explicit coefficient identity]
\label{rem:spectral-explicit-certificate}
If \(\Delta S=4\delta B\) exactly, put
\(A=E_{d-4}\) and \(J=E_{d-5}\), retaining the convention for
negative indices.  The same closure table gives the checkable identity
\begin{equation}\label{eq:spectral-explicit-certificate}
 \cL(F\sqrt S)=\frac{F}{2S^{3/2}}
 \left\{4S^2\bigl[(\delta-4)BS-6CD\bigr]+F^2\mathscr P_\delta\right\},
\end{equation}
where
\[
 \mathscr P_\delta=(20\delta+10)B^2S+(12\delta-18)BCD
                    -56CS^2-82ASD-30JD^2.
\]
\Cref{thm:spectral-gradient-subcalibration} proves the sign of this
identity by monotonicity, uniformly in \(d\).
\end{remark}

\subsection{The threshold and the role of the Laplacian}

The threshold is sharp for \(G=F\sqrt S\) when
\(\Delta S=4\delta B\) and the invariant image contains
\(r_1=\cdots=r_d=1\) with arbitrarily small nonzero \(F\).
Indeed, at these roots,
\begin{equation}\label{eq:spectral-balanced-defect}
 D_h=d(d-1)\left(\delta-6+\frac4d\right).
\end{equation}
For \(\delta<6-4/d\), the multiplier sign fails near \(F=0\).
Complex determinants realize these values with \(\delta=4\), so
\(F\sqrt S\) fails in every order \(d\ge3\). Their product multiplier
in \cref{sec:complex-determinants} succeeds: the threshold is not
necessary for area minimization.

The next proposition compares reduced operators with different
Laplacian coefficients; their ambient spaces may differ.

\begin{proposition}[Monotonicity under an increase of the drift]
\label{prop:spectral-drift-comparison}
Suppose two reduced operators have the same gradient pairing
\(\widehat\Gamma\), annihilate \(F\), and have generator drifts
\(\ell_i=\widehat\Delta E_i\) and
\(\widetilde\ell_i=\widetilde\Delta E_i\).  These are the full
drift functions, including any factors \(E_{i-1}\).
For a positive smooth multiplier
\(a=a(E_1,\ldots,E_d)\), put \(G=aF\) and
\(W=\widehat\Gamma(G,G)\).  Their curvature operators satisfy
\begin{equation}\label{eq:spectral-drift-comparison}
 G\widetilde{\cL}(G)-G\cL(G)
 =F^2aW\sum_i(\widetilde\ell_i-\ell_i)\partial_{E_i}a.
\end{equation}
On any common invariant domain where the reference certificate holds,
a nonnegative sign certificate is preserved if
\(\widetilde\ell_i\ge\ell_i\) for every generator used by \(a\)
and \(\partial_{E_i}a\ge0\).
\end{proposition}

\begin{proof}
The common gradient pairing gives the same \(W\) and the same
term \(-\widehat\Gamma(G,W)/2\).  The difference of the two
Laplacians on \(G\) is
\(F\sum_i(\widetilde\ell_i-\ell_i)\partial_{E_i}a\).
Multiplication by \(GW\) proves the formula.
\end{proof}

For example, consider drifts of the form
\begin{equation}\label{eq:spectral-two-parameter-drift}
 \Delta E_k=2j[\alpha+\eta(j-1)]E_{k-1},
 \qquad j=d-k+1.
\end{equation}
The simple criterion requires \(\alpha+\eta\ge6-4/d\).
The determinant drift \((\alpha,\eta)=(2,2)\) instead admits
\(a=(E_1\cdots E_{d-1})^{1/2}\) by
\cref{sec:complex-determinants}. Its nonnegative partial derivatives
allow \cref{prop:spectral-drift-comparison} to transfer the certificate
whenever
\begin{equation}\label{eq:spectral-determinant-drift-range}
 \alpha+\eta\ge4,\qquad
 \alpha+(d-1)\eta\ge2d.
\end{equation}
Indeed these endpoint conditions are equivalent to
\(\alpha+\eta(j-1)\ge2j\) for \(2\le j\le d\).
To apply this comparison geometrically, the field must be defined off
a removable exceptional set.

\section{Normed Products and Complex Pfaffians}
\label{sec:norm-products-pfaffians}

We apply \cref{thm:spectral-gradient-subcalibration} to octonionic products
and complex Pfaffians in every integer degree.
The multiplication maps over normed division algebras are harmonic morphisms; see
\cite[Section~4.5.2]{RiedlerThesis25}. The minimality of the real-part
Pfaffian cones also follows from the holomorphic construction of
Hoppe--Tkachev \cite[Corollary~5.2 and Remark~5.5]{HoppeTkachev19}.
Their smallest algebras have the spectral table, with Laplacian
coefficients \(32\) and \(24\), respectively. Both exceed
\(24-16/d\), giving the potential \(F|\nabla F|\).

\subsection{Products in a normed division algebra}

Let \(\mathbb K=\mathbb H\) or \(\mathbb O\), put
\(\delta=\dim_{\R}\mathbb K\), and fix an integer \(d\ge2\).
Choose a full parenthesization of the ordered product of \(d\)
variables, and write
\begin{equation}\label{eq:norm-product-definition}
 P:\mathbb K^d\longrightarrow\mathbb K,
 \qquad P(x_1,\ldots,x_d)=x_1\cdots x_d,
 \qquad F=\operatorname{Re}P.
\end{equation}
The fixed parentheses allow for the nonassociativity of
\(\mathbb O\). Set
\begin{equation}\label{eq:norm-product-generators}
 r_a=|x_a|^2,\qquad E_k=e_k(r_1,\ldots,r_d),\qquad
 R=E_1=\sum_a|x_a|^2,
\end{equation}
where \(e_k\) denotes the \(k\)-th elementary symmetric polynomial.
We use \(E_0=1\) and \(E_k=0\) for \(k>d\).

\begin{proposition}[Differential closure of normed products]
\label{prop:norm-product-closure}
For every parenthesization in \eqref{eq:norm-product-definition},
the polynomial \(F\) is real irreducible, homogeneous of degree \(d\), and
satisfies
\begin{equation}\label{eq:norm-product-F-closure}
\begin{gathered}
 \Delta F=0,\qquad \Gamma(F,F)=E_{d-1},\\
 \Gamma(F,E_k)=2(d-k+1)E_{k-1}F,
\end{gathered}
\end{equation}
\begin{equation}\label{eq:norm-product-E-closure}
\begin{aligned}
 \Delta E_k&=2\delta(d-k+1)E_{k-1},\\
 \Gamma(E_i,E_j)&=4\sum_{a=1}^d r_a
       \frac{\partial e_i}{\partial r_a}
       \frac{\partial e_j}{\partial r_a}.
\end{aligned}
\end{equation}
Consequently
\begin{equation}\label{eq:norm-product-minimality}
 \cL(F)=-2E_{d-2}F.
\end{equation}
The polynomial critical set is
\begin{equation}\label{eq:norm-product-critical-set}
 Z_d=\{\nabla F=0\}
 =\bigcup_{1\le a<b\le d}\{x_a=x_b=0\},
\end{equation}
and has real codimension \(2\delta\).
\end{proposition}

\begin{proof}
Multiplication in \(\mathbb K\) is real bilinear and satisfies
\(|uv|=|u||v|\); see
\cite[Chapter~1, Appendix, pp.~21--22]{Shapiro2000}.
A parenthesized product is therefore real
multilinear, and induction on its multiplication tree gives
\[
 |P(x)|^2=\prod_{a=1}^d r_a.
\]
With all variables except \(x_a\) fixed, the derivative
\(A_a=D_{x_a}P:\mathbb K\to\mathbb K\) is a linear map satisfying
\[
 |A_av|^2=\left(\prod_{b\ne a}r_b\right)|v|^2.
\]
By polarization, \(A_a^*A_a=(\prod_{b\ne a}r_b)I\), the usual
characterization of a linear similarity
\cite[Chapter~1, Lemma~1.3(1), pp.~15--16]{Shapiro2000}.
The source and target have the same dimension, so also
\(A_aA_a^*=(\prod_{b\ne a}r_b)I\); this remains true when the
coefficient is zero.  Since \(F=\ip{P}{1}\), it follows that
\[
 |\nabla_{x_a}F|^2=|A_a^*1|^2=\prod_{b\ne a}r_b.
\]
Summing over \(a\) proves \(\Gamma(F,F)=E_{d-1}\).
Every second derivative within a single variable block vanishes,
so \(\Delta F=0\).  The choice \(x_1=\cdots=x_d=1\) shows that
\(F\) is nonzero.

For irreducibility, a factorization \(F=AB\) must be homogeneous
in every variable block, by the lowest- and highest-degree argument
in \cref{lem:square-free-critical-set}. Since each block degree of
\(F\) is one, \(A\) and \(B\) depend on disjoint groups of blocks.
If both are nonconstant, retain one variable from each group and set
all remaining variables equal to \(1\). Every parenthesization then
reduces \(F\) to \(\operatorname{Re}(uv)=\langle u,\bar v\rangle\),
a nondegenerate bilinear form of rank \(\delta>1\)
\cite[Chapter~1, Appendix, pp.~21--22]{Shapiro2000}.
The specialized factorization would express it as a product of two
linear forms in separate variables, of rank one, a contradiction.

Euler's identity in each variable block gives
\(\Gamma(F,r_a)=2F\).  Hence
\[
 \Gamma(F,E_k)=2F\sum_a\frac{\partial e_k}{\partial r_a}
 =2(d-k+1)E_{k-1}F.
\]
Moreover, \(\Delta r_a=2\delta\),
\(\Gamma(r_a,r_b)=4r_a\delta_{ab}\), and each \(e_k\) is affine
in each individual \(r_a\).  These facts prove
\eqref{eq:norm-product-E-closure} by the chain rule.
Harmonicity and the identity for \(\Gamma(F,E_{d-1})\) now give
\eqref{eq:norm-product-minimality}.

Finally, the nonnegative sum
\(E_{d-1}=\sum_a\prod_{b\ne a}r_b\) vanishes precisely when at
least two blocks vanish.  Each largest linear component of this
set has dimension \(\delta(d-2)\), proving the codimension claim.
\end{proof}

The vector-valued map itself satisfies
\begin{equation}\label{eq:norm-product-harmonic-morphism}
 \Delta P=0,\qquad DP\,DP^*=E_{d-1}I_{\delta},\qquad
 DP(\nabla E_{d-1})=4E_{d-2}P.
\end{equation}
Indeed, the first two identities were established blockwise in the
proof, while multilinearity gives \(DP(\nabla r_a)=2P\), whence
the third.  Thus the target of the pullback defining our hypersurface
is the hyperplane \(\{\operatorname{Re}y=0\}\subset\mathbb K\).
The umbilic condition in \eqref{eq:norm-product-harmonic-morphism}
places its regular pullback within Riedler's umbilic framework
\cite[Theorem~5.1.4]{RiedlerThesis25}.  The direct identity
\eqref{eq:norm-product-minimality} additionally covers regular zeros
lying over \(P=0\), for instance points with exactly one zero block.

\begin{proposition}[Exact algebra and quotient image]
\label{prop:norm-product-exact-algebra}
For \(d\ge3\),
\begin{equation}\label{eq:norm-product-smallest-algebra}
 \cA_R(F)=\R[E_1,\ldots,E_d,F],
\end{equation}
and its \(d+1\) displayed generators are algebraically independent.
Their joint image is exactly
\begin{equation}\label{eq:norm-product-quotient-image}
 \left\{(e_1(r),\ldots,e_d(r),f):
             r\in[0,\infty)^d,\quad f^2\le e_d(r)\right\}.
\end{equation}
The spherical quotient dimension is \(d\).
For \(d=2\), the smaller algebra \(\cA_R(F)=\R[R,F]\) is already
closed and its spherical quotient dimension is one.
\end{proposition}

\begin{proof}
Apply \cref{prop:spectral-smallest-algebra} with
\(r_a=|x_a|^2\) and \(b_k=2\delta(d-k+1)>0\) from
\cref{prop:norm-product-closure}.

The norm identity implies \(F^2\le E_d\).
Conversely, choose an imaginary unit \(u\in\mathbb K\), and put
\[
 x_1=\sqrt{r_1}(\cos\theta+u\sin\theta),\qquad
 x_a=\sqrt{r_a}\quad(a\ge2).
\]
All bracketings then give \(F=\sqrt{e_d(r)}\cos\theta\).
This realizes every value in \eqref{eq:norm-product-quotient-image},
including its boundary.
At distinct positive \(r_a\), the Jacobian of
\((e_1,\ldots,e_d)\) is a nonzero Vandermonde determinant.
Taking \(\sin\theta\ne0\) adds a phase direction changing \(F\)
while fixing all \(E_k\).  The joint map thus has rank \(d+1\)
on a nonempty open set, which proves algebraic independence.
Fixing \(E_1=1\) lowers the rank by one.  In the quadratic case,
\(R,F\) have independent differentials at generic points and
give the asserted one-dimensional spherical quotient.
\end{proof}

In particular, different octonionic parenthesizations have the same
abstract differential closure and the same quotient image.  These
statements do not identify the corresponding embedded hypersurfaces
up to orthogonal congruence.

\begin{theorem}[Octonionic product cones in every degree]
\label{thm:octonionic-product-minimizing}
For every \(d\ge2\) and every fixed parenthesization, let
\(F=\operatorname{Re}(x_1\cdots x_d)\) on \(\mathbb O^d\).
Then the cone \(C_d=\{F=0\}\subset\R^{8d}\) is the support of the
locally mass-minimizing boundary \(\partial\llbracket\{F<0\}\rrbracket\).
A global two-sided subcalibration off the set \(Z_d\) in
\eqref{eq:norm-product-critical-set} is
\begin{equation}\label{eq:octonionic-product-potential}
 G=F\sqrt{E_{d-1}}=F|\nabla F|,
 \qquad X=\frac{\nabla G}{|\nabla G|}.
\end{equation}
Here \(Z_d\) has real codimension sixteen; it contains nonzero
points whenever \(d\ge3\). The cone is strictly area minimizing
and strictly stable, with
\begin{equation}\label{eq:octonionic-product-Hardy-gap}
 Q_{C_d}(u)\ge\frac{(4d+1)^2}{4}\int_{C_d}r^{-2}u^2,
 \qquad u\in C_c^\infty((C_d)_{\reg}\setminus\{0\}).
\end{equation}
\end{theorem}

\begin{proof}
By \cref{prop:norm-product-closure}, \(\Delta E_{d-1}=32E_{d-2}\)
and \(Z_d\) has codimension sixteen. Since \(32>24-16/d\),
\cref{thm:spectral-gradient-subcalibration}, with \(N=8d\), gives
all assertions, including \eqref{eq:octonionic-product-Hardy-gap}.
\end{proof}

\begin{remark}[The quaternionic closure]
\label{rem:quaternionic-product-threshold}
Quaternionic products have the same algebra and minimality identities,
but \(\Delta S=16B\). The potential \(F\sqrt S\) works for \(d=2\)
and fails for \(d\ge3\): at equal positive roots its boundary
coefficient is \((-4+8/d)B<0\), and the phase construction above
realizes arbitrarily small nonzero \(F\) with those roots fixed.
\end{remark}

\subsection{Complex Pfaffians and their minor norms}

Let \(W_d\) be the space of complex skew-symmetric \(2d\)-by-\(2d\)
matrices, with the Euclidean metric
\begin{equation}\label{eq:pfaffian-metric}
 |Z|^2=\sum_{i<j}|z_{ij}|^2=\tfrac12\tr(ZZ^*),
 \qquad \dim_{\R}W_d=2d(2d-1).
\end{equation}
The independent coordinates are the complex entries \(z_{ij}\),
\(i<j\).  For \(d\ge2\), set
\begin{equation}\label{eq:pfaffian-generators}
 f(Z)=\operatorname{Pf}(Z),\qquad F=\operatorname{Re}f,
 \qquad E_k(Z)=\sum_{\substack{I\subset\{1,\ldots,2d\}\\|I|=2k}}
                 |\operatorname{Pf}(Z_I)|^2.
\end{equation}
Here \(Z_I\) is the principal submatrix on \(I\), and
\(\operatorname{Pf}(Z_\varnothing)=1\).
In particular, \(E_1=R=|Z|^2\) and \(E_d=|f|^2\).

\begin{proposition}[Pfaffian differential closure]
\label{prop:pfaffian-closure}
For \(F,E_k\) defined by \eqref{eq:pfaffian-generators}, \(F\) is
real irreducible and
\begin{equation}\label{eq:pfaffian-F-closure}
\begin{gathered}
 \Delta F=0,\qquad\Gamma(F,F)=E_{d-1},\\
 \Gamma(F,E_k)=2(d-k+1)E_{k-1}F,
\end{gathered}
\end{equation}
\begin{equation}\label{eq:pfaffian-Laplacian}
 \Delta E_k=4(d-k+1)(2d-2k+1)E_{k-1}.
\end{equation}
At every point there exist nonnegative numbers \(r_1,\ldots,r_d\)
such that \(E_k=e_k(r)\), and
\begin{equation}\label{eq:pfaffian-Gamma}
 \Gamma(E_i,E_j)=4\sum_{a=1}^d r_a
       \frac{\partial e_i}{\partial r_a}
       \frac{\partial e_j}{\partial r_a}.
\end{equation}
In particular, \(\cL(F)=-2E_{d-2}F\).
For \(d\ge3\),
\begin{equation}\label{eq:pfaffian-smallest-algebra}
 \cA_R(F)=\R[E_1,\ldots,E_d,F],
\end{equation}
with algebraically independent generators and spherical quotient
dimension \(d\).  For \(d=2\), \(\cA_R(F)=\R[R,F]\).
\end{proposition}

\begin{proof}
Complexification gives
\(2F=\operatorname{Pf}(Z)+\operatorname{Pf}(W)\) for independent
skew-symmetric matrices \(Z,W\). It is linear in \(z_{12}\), with
coefficient \(\operatorname{Pf}(Z_{\{3,\ldots,2d\}})\).
That coefficient is coprime to the constant term, since the latter
contains \(w_{12}w_{34}\cdots w_{2d-1,2d}\) with coefficient one.
Gauss' lemma proves irreducibility
\cite[Chapter~VIII, Section~5, pp.~395--398]{Knapp16}.

Write \(\omega=\sum_{i<j}z_{ij}e_i\wedge e_j\) in an orthonormal
basis of \(\mathbb C^{2d}\).  The coefficient of \(e_I\) in
\(\omega^k/k!\) is \(\operatorname{Pf}(Z_I)\).  Consequently
\begin{equation}\label{eq:pfaffian-exterior-norm}
 E_k=\left|\frac{\omega^k}{k!}\right|^2,
\end{equation}
which proves that all \(E_k\) are invariant under unitary congruence
\(Z\mapsto UZU^T\).
The skew-symmetric unitary normal form is
\[
 Z=U\operatorname{diag}(\sigma_1J,\ldots,\sigma_dJ)U^T,
 \qquad \sigma_a\ge0,\qquad
 J=\begin{pmatrix}0&1\\-1&0\end{pmatrix}.
\]
For a direct proof, let \(T(v)=Z\bar v\).  Skew-symmetry gives
\(T^2=-ZZ^*\) and
\(\ip{v}{Tw}=-\ip{w}{Tv}\), where the Hermitian inner product is
conjugate linear in its first argument.  Choose a unit eigenvector
\(v\) of \(ZZ^*\) with positive eigenvalue \(\sigma^2\), and put
\(w=-Tv/\sigma\).  These identities imply that \(v,w\) are
orthonormal, \(Tv=-\sigma w\), and \(Tw=\sigma v\).
Their orthogonal complement is \(T\)-invariant.  Induction gives
the displayed normal form, with zero blocks at the last step.
At this normal form only
principal minors which are unions of complete two-by-two blocks
can have nonzero Pfaffian.  Thus \(E_k=e_k(\sigma_1^2,\ldots,
\sigma_d^2)\).

To compute first derivatives, more generally use the block diagonal
matrix with complex entries \(z_aJ\).  A nonzero principal Pfaffian
is a product of its block entries; its derivative in any off-block
coordinate vanishes.  A principal Pfaffian which vanishes contributes
zero to the first derivative of its squared modulus.  Therefore the
gradient of \(E_k\) has only block components at this matrix, and
\[
 \frac{\partial E_k}{\partial z_a}
   =\bar z_a\,e_{k-1}(r_1,\ldots,\widehat r_a,\ldots,r_d),
 \qquad r_a=|z_a|^2.
\]
The real Euclidean inner product of these gradients gives
\eqref{eq:pfaffian-Gamma}.  Unitary congruence preserves the metric
\eqref{eq:pfaffian-metric}, so this calculation holds everywhere.

The polynomial \(f\) is holomorphic, hence \(\Delta F=0\).
Differentiating its Pfaffian expansion gives
\[
 \frac{\partial f}{\partial z_{ij}}
   =\pm\operatorname{Pf}(Z_{\{1,\ldots,2d\}\setminus\{i,j\}}).
\]
For a holomorphic polynomial on complex Euclidean space,
\[
 |\nabla\operatorname{Re}f|^2=\sum_{i<j}|\partial f/\partial z_{ij}|^2.
\]
This proves \(\Gamma(F,F)=E_{d-1}\).
At the block diagonal matrix, complex bilinear extension of
\(\Gamma\) gives
\[
 \Gamma(f,E_k)
 =2\sum_a\frac{\partial f}{\partial z_a}
                \frac{\partial E_k}{\partial\bar z_a}
 =2(d-k+1)E_{k-1}f.
\]
Unitary congruence preserves every \(E_k\), and \(f(UZU^T)=\det(U)f(Z)\).
Hence the identity is covariant under this action.
Taking real parts proves the remaining identity in
\eqref{eq:pfaffian-F-closure}; polynomial continuity includes all
degenerate normal forms.

For the Laplacian, apply
\(\Delta|p|^2=4\sum_{i<j}|\partial p/\partial z_{ij}|^2\)
to each holomorphic principal Pfaffian in \eqref{eq:pfaffian-generators}.
Each derivative deletes a pair of indices.  A fixed set of \(2k-2\)
indices is obtained in exactly
\(\binom{2d-2k+2}{2}=(d-k+1)(2d-2k+1)\) ways, proving
\eqref{eq:pfaffian-Laplacian}.
The minimality identity follows from
\(-\tfrac12\Gamma(F,E_{d-1})=-2E_{d-2}F\).

Now apply \cref{prop:spectral-smallest-algebra} with
\(r_a=\sigma_a^2\) and \(b_k=4(d-k+1)(2d-2k+1)>0\).
To verify independence, take distinct positive \(r_a\), with
\(z_1=\sqrt{r_1}e^{i\theta}\) and \(z_a=\sqrt{r_a}\) for \(a\ge2\).
Then \(E_k=e_k(r)\) and \(F=\sqrt{e_d(r)}\cos\theta\).
The Vandermonde and phase argument in
\cref{prop:norm-product-exact-algebra} applies without change.
It also proves that the joint image is exactly
\eqref{eq:norm-product-quotient-image}.
\end{proof}

\begin{theorem}[Area minimization for complex Pfaffians]
\label{thm:pfaffian-minimizing}
For every \(d\ge2\), the cone
\[
 C_d=\{\operatorname{Re}\operatorname{Pf}(Z)=0\}
 \subset W_d\cong\R^{2d(2d-1)}
\]
is the support of the locally mass-minimizing boundary of
\(\{\operatorname{Re}\operatorname{Pf}(Z)<0\}\).
A two-sided subcalibration is the normalized gradient of
\begin{equation}\label{eq:pfaffian-potential}
 G=F\sqrt{E_{d-1}}=F|\nabla F|,
\end{equation}
defined off
\begin{equation}\label{eq:pfaffian-critical-set}
 Z_d=\{Z:\operatorname{rank}_{\mathbb C}Z\le2d-4\}
     =\{\nabla F=0\}.
\end{equation}
This exceptional set has real codimension twelve.
The cone is strictly area minimizing and strictly stable, with
\begin{equation}\label{eq:pfaffian-Hardy-gap}
 Q_{C_d}(u)\ge\frac{(4d^2-6d+1)^2}{4}\int_{C_d}r^{-2}u^2,
 \qquad u\in C_c^\infty((C_d)_{\reg}\setminus\{0\}).
\end{equation}
\end{theorem}

\begin{proof}
By \cref{prop:pfaffian-closure}, the spectral table holds with \(S=E_{d-1},\; B=E_{d-2},\; \Delta S=24B\),
and \(24\ge24-16/d\).

The normal form shows that \(S=0\) precisely when at most
\(d-2\) of the \(\sigma_a\) are nonzero, proving
\eqref{eq:pfaffian-critical-set}. The stratum of
complex skew-symmetric \(n\)-by-\(n\) matrices of rank \(2q\)
has complex dimension \(2q(n-2q)+\binom{2q}{2}\).
Indeed, choose its kernel, of dimension \(n-2q\), and then choose
a nondegenerate alternating form on the \(2q\)-dimensional quotient.
Its complex codimension is therefore
\(\binom{n-2q}{2}\).
For \(n=2d\) and \(q=d-2\), this is six; lower rank strata have
larger codimension.  Hence \(Z_d\) is a closed algebraic set of
real codimension twelve and has locally zero
\(\cH^{2d(2d-1)-1}\)-measure.

\Cref{thm:spectral-gradient-subcalibration}, with \(N=2d(2d-1)\),
gives the minimizing assertions and \eqref{eq:pfaffian-Hardy-gap}.
Here \(4d^2-6d+1>0\) for \(d\ge2\).
\end{proof}

\section[Complex Determinants]{Real Parts of Complex Determinants in All Orders}
\label{sec:complex-determinants}

The determinant has the gradient table of \cref{sec:spectral-positivity},
but \(\Delta E_{m-1}=16E_{m-2}\) excludes \(\sqrt S\) in orders
at least three; see \eqref{eq:spectral-balanced-defect}.
We prove \cref{main:complex-determinants} using the product of proper
minor norms and two coefficient inequalities uniform in matrix order.
We use the Frobenius Euclidean metric
\(\langle U,V\rangle=\operatorname{Re}\tr(UV^*)\) on
\(M_m(\mathbb C)\), so that \(|U|^2=\sum_{i,j}|U_{ij}|^2\), and write
\[
 f(Z)=\det_{\mathbb C}Z,\qquad F=\operatorname{Re}f,
 \qquad C_m=\{F=0\},\qquad E^-_m=\{F<0\}.
\]
For $0\leq k\leq m$, let
\begin{equation}\label{det:minor-norms}
 E_0=1,\qquad
 E_k(Z)=\sum_{|I|=|J|=k}|\det Z_{I,J}|^2.
\end{equation}
Let $x_1,\ldots,x_m\ge0$ be the squared singular values of $Z$,
and retain the elementary symmetric polynomials $e_k$ from
\cref{sec:spectral-positivity}. Cauchy--Binet gives
\[
 E_k(Z)=\sum_{|I|=k}\det(ZZ^*)_{I,I},\qquad
 \sum_{k=0}^m E_k(Z)t^k
 =\det(\Id+tZZ^*)=\prod_{i=1}^m(1+tx_i).
\]
Comparing coefficients yields $E_k=e_k(x_1,\ldots,x_m)$;
in particular, $E_1=|Z|^2$ and $E_m=|f|^2$.
The minor invariants $E_1,\ldots,E_{m-1}$ are positive at rank
at least $m-1$. Their product vanishes precisely at rank at most
$m-2$, the singular set of the cone, and is positive along its
regular part.
We take the product candidate
\begin{equation}\label{det:potential}
 a=\left(\prod_{k=1}^{m-1}E_k\right)^{1/2},\qquad G=aF.
\end{equation}
Logarithmic derivatives add over the minor orders.
\Cref{det:divergence-formula} reduces the defect sign to a boundary
coefficient controlled by Newton inequalities and a second coefficient
controlled by a covariance decomposition.
The homogeneous degree of $G$ is $m(m+1)/2$.
The case $m=1$ is a hyperplane; the coefficient calculations assume $m\geq2$.

For $m=2,3$, the potential is
\[
 G_2=F_2E_1^{1/2},\qquad G_3=F_3(E_1E_2)^{1/2}.
\]
The first multiplier is the radius $E_1^{1/2}=|Z|$; the second also
includes the norm $E_2^{1/2}$ of the second-order minors and agrees
with the \(d=2\) Jordan potential in \cref{thm:exc-jordan-minimizing}.
The low-order geometry and determinant phase are discussed in
\cref{det:low-orders-phase}; the coefficient signs are proved in
\cref{det:first-section,det:second-section}.

\subsection{Invariant calculus and the multiplier identity}

\begin{lemma}\label{det:basic-identities}
For $1\leq j,k\leq m$, one has
\begin{align}
 \Delta F&=0,&\Gamma(F,F)&=E_{m-1},\label{det:basic-one}\\
 \Gamma(F,E_k)&=2(m-k+1)E_{k-1}F,
 &\Delta E_k&=4(m-k+1)^2E_{k-1},\label{det:basic-two}\\
 \Gamma(E_j,E_k)
 &=4\sum_{i=1}^m x_i
       \frac{\partial e_j}{\partial x_i}
       \frac{\partial e_k}{\partial x_i}.
 \label{det:basic-three}
\end{align}
\end{lemma}

\begin{proof}
The real part of a holomorphic polynomial is harmonic.  The complex
first derivatives of $f$ are its cofactors, so their squared norms sum
to $E_{m-1}$.  This proves \eqref{det:basic-one}.

For the mixed identity, use complex matrix coordinates $z_\alpha$ and
observe that
\[
 \Gamma(F,E_k)
 =2\operatorname{Re}\sum_\alpha
       \frac{\partial f}{\partial z_\alpha}
       \frac{\partial E_k}{\partial\overline z_\alpha}.
\]
The complex expression in this sum transforms by the same determinant
phase as $f$ under the left and right unitary actions.  Thus singular
value decomposition reduces the identity to a diagonal matrix, with its
determinant phase retained.  At such a matrix,
\[
 \sum_\alpha
       \frac{\partial f}{\partial z_\alpha}
       \frac{\partial E_k}{\partial\overline z_\alpha}
 =f\sum_i e_{k-1}(x_1,\ldots,\widehat x_i,\ldots,x_m)
 =(m-k+1)E_{k-1}f.
\]
The polynomial identity extends to singular matrices by continuity.

For a holomorphic polynomial $q$,
$\Delta|q|^2=4\sum_\alpha|\partial q/\partial z_\alpha|^2$.
Apply this to the minors in \eqref{det:minor-norms}.  Each fixed
$(k-1)$-minor occurs $(m-k+1)^2$ times, which gives the Laplacian
formula.

The same unitary invariance reduces \eqref{det:basic-three} to
$D=\operatorname{diag}(\sigma_1,\ldots,\sigma_m)$, where
$x_i=\sigma_i^2$. At $D$, nonprincipal minors vanish, and
\[
 \mathrm d(\det Z_{I,I})_D[H]
 =\sum_{i\in I}
       \left(\prod_{\ell\in I\setminus\{i\}}\sigma_\ell\right)H_{ii}.
\]
Differentiating their squared moduli and summing gives
\[
 \mathrm dE_k(D)[H]
 =2\sum_i\sigma_i\frac{\partial e_k}{\partial x_i}
                                     \operatorname{Re}H_{ii}.
\]
Thus all other first derivatives vanish and
\[
 \nabla E_k(D)=\operatorname{diag}\left(
             2\sigma_i\frac{\partial e_k}{\partial x_i}\right)_{i=1}^m.
\]
Pairing these gradients proves \eqref{det:basic-three}, including
repeated or zero singular values.
\end{proof}

Apply \cref{prop:spectral-smallest-algebra} with \(d=m\), roots
\(x_i\), and \(b_k=4(m-k+1)^2>0\) from
\cref{det:basic-identities}. The algebra is Laplacian and, for \(m\ge3\),
\begin{equation}\label{det:smallest-algebra}
 \cA_R(F)=\R[E_1,\ldots,E_m,F].
\end{equation}
For \(m=2\), the same proposition gives
\(\cA_R(F)=\R[R,F]\).

The joint image of \((E_1,\ldots,E_m,F)\) is exactly
\begin{equation}\label{det:quotient-image}
 \left\{(e_1(x),\ldots,e_m(x),t):
       x\in[0,\infty)^m,\quad t\in\R,\quad t^2\le e_m(x)\right\}.
\end{equation}
Indeed, \(F^2\le|\det Z|^2=E_m\), and the diagonal matrices
\[
 Z=\operatorname{diag}\left(\sqrt{x_1}e^{i\vartheta},
                   \sqrt{x_2},\ldots,\sqrt{x_m}\right)
\]
have \(E_k=e_k(x)\) and \(F=\sqrt{e_m(x)}\cos\vartheta\).
The phase and Vandermonde argument of
\cref{prop:norm-product-exact-algebra} therefore gives the full
image, including zero spectra, and generic rank \(m+1\).
Thus for \(m\ge3\) the generators are independent and the spherical
quotient dimension is \(m\); for \(m=2\), the algebra \(\R[R,F]\)
has spherical quotient dimension one.

The identities \eqref{det:basic-one}--\eqref{det:basic-two} also give
\begin{equation}\label{det:mean-curvature}
 \cL F
 =-\tfrac12\Gamma(F,E_{m-1})=-2E_{m-2}F.
\end{equation}
Thus the regular part of $C_m$ is minimal, as in
\cite[Corollary~5.2 and Example~5.4]{HoppeTkachev19}.

\begin{example}[Generator degrees for harmonic determinant cones]
\label{ex:determinant-generator-degrees}
For \(m\ge3\), the polynomial \(F=\operatorname{Re}\det Z\) is
an irreducible homogeneous harmonic equation of degree \(m\) on
\(\R^{2m^2}\), with nonempty regular minimal zero set. Its
Laplacian algebra has \(m+1\) algebraically independent generators,
and the least possible maximal degree of homogeneous generators
is \(2m\).

Harmonicity and minimality follow from \eqref{det:basic-one} and
\eqref{det:mean-curvature}. For irreducibility, complexify the real
coordinates to independent matrices \(Z,W\), so that
\(2F=\det Z+\det W\). As a polynomial in \(z_{11}\), this is
\(Dz_{11}+C\), where \(D\) is the corresponding cofactor. Any
factor of \(D\) is independent of \(W\), whereas \(C\) has
coefficient \(1\) on \(w_{11}\cdots w_{mm}\). Thus \(D\) and
\(C\) are coprime, so \(Dz_{11}+C\) is primitive and linear over
the fraction field. Gauss' lemma proves irreducibility
\cite[Chapter~VIII, Section~5, pp.~395--398]{Knapp16}.
The zero \(Z=\operatorname{diag}(i,1,\ldots,1)\) is regular by
\eqref{det:basic-one}.

By \eqref{det:smallest-algebra},
\[
 \cA_R(F)=\R[E_1,\ldots,E_m,F],\qquad \deg E_k=2k.
\]
Algebraic independence gives the lower bound \(m+1\) for the
number of generators. Every homogeneous element of
degree less than \(2m\) belongs to the proper subalgebra
\(\R[E_1,\ldots,E_{m-1},F]\), and such elements cannot generate
\(E_m\). Conversely, the displayed generators have maximal
degree \(2m\). Thus the least possible maximal degree is exactly
\(2m\), and the uniform bound in
\cref{prop:uniform-finite-generation} must satisfy
\(B(2m^2,m)\ge2m\).
\end{example}

The singular set has
real codimension eight and is given by
\begin{equation}\label{det:singular-set}
 \mathcal S_m=\{Z:\operatorname{rank}_{\mathbb C}Z\leq m-2\}
 =\{\nabla F=0\}.
\end{equation}
The cofactor identity gives the second equality. To see that every
point of this set is geometrically singular, take a matrix of rank
$q\leq m-2$. Schur-complement coordinates, absorbing a nonvanishing
determinant factor into one row, identify its zero set locally with
a smooth factor times $\{\operatorname{Re}\det W=0\}$ at $W=0$,
where $W\in M_{m-q}(\mathbb C)$. Since $m-q\geq2$, this transverse
cone contains every complex rank-one line. Their real tangent
directions span the whole matrix space, so the zero set cannot be
a smooth hypersurface there.
The rank-\(q\) stratum has real dimension
\(2q(2m-q)\), so the largest stratum in \(\mathcal S_m\)
has dimension \(2m^2-8\).
The multiplier in \eqref{det:potential} is smooth and strictly positive
on $M_m(\mathbb C)\setminus\mathcal S_m$: at every matrix of rank at
least $m-1$, all $E_1,\ldots,E_{m-1}$ are positive. This includes the
regular zeros of $F$, even when the matrix is not invertible.
On $\mathcal S_m$ the displayed product gives $a=0$ and $G=0$.
For \(m\ge3\), this set contains nonzero points, where
\cref{prop:harmonic-positive-multiplier-obstruction} excludes multipliers
extending as strictly positive \(C^2\) functions. For \(m=2\),
\(\mathcal S_2=\{0\}\).

\begin{proposition}[Tangent cones along the determinant rank strata]
\label{prop:determinant-tangent-strata}
Let $p\in C_m$ have complex rank $q\leq m-2$, and put $r=m-q$.
The rescalings of $C_m$ about $p$ have the unique tangent cone
\begin{equation}\label{eq:determinant-rank-tangent}
 \R^{2(m^2-r^2)}\times C_r,
 \qquad C_r=\{\operatorname{Re}\det W=0:W\in M_r(\mathbb C)\},
\end{equation}
up to an orthogonal transformation. The corresponding oriented
boundary currents and their associated varifolds also converge,
with multiplicity one. The first nonzero Taylor term of $F_m$ at
$p$ has degree $r$.
\end{proposition}

\begin{proof}
The assertion at $p=0$ is homogeneity. Otherwise, unitary left and
right changes of coordinates put $p$ in the form
$\operatorname{diag}(A_0,0)$, where $A_0$ is a positive diagonal
$q$-by-$q$ matrix. These changes are Euclidean isometries; they
replace $F_m$ by $\operatorname{Re}(\chi\det Z)$ for some
$|\chi|=1$. Write nearby matrices as
\[
 Z=\begin{pmatrix}A&B\\ C&D\end{pmatrix},\qquad
 V=D-CA^{-1}B.
\]
Block elimination, with determinant one, gives
$\det Z=(\det A)\det V$. Put
\[
 \zeta(A)=\frac{\chi\det A}{|\det A|},\qquad
 U(A)=\operatorname{diag}(\zeta(A),1,\ldots,1),\qquad
 W=U(A)V.
\]
Then
\[
 \operatorname{Re}(\chi\det Z)=|\det A|\operatorname{Re}\det W.
\]
Thus the smooth local coordinate map \(\Psi(Z)=(A-A_0,B,C,W)\)
identifies both the zero set and its negative phase with a product
of a smooth coordinate space and the corresponding determinant
zero set or phase. At $p$, the differential of the Schur complement
is $dV=dD$, and the derivative of $U(A)$ contributes nothing
because $V(p)=0$. Consequently $D\Psi_p$ is the identity on the
free blocks and the unitary row transformation $D\mapsto U(A_0)D$
on the last block. In particular, it is an isometry.

For $s\downarrow0$, the maps $s^{-1}\Psi(p+sH)$ converge in
$C^1$ on compact sets to $D\Psi_p$. Since the target zero set and
phase are products of cones, this proves local Hausdorff and
local $L^1$ convergence of their inverse images. The same local
diffeomorphisms push forward the multiplicity-one boundary
current. Their $C^1$ convergence and the area formula give local
current and mass convergence, hence convergence of the associated
varifolds. This also proves uniqueness. Finally, the determinant
identity gives the leading term
\[
 \operatorname{Re}\bigl(\chi\det(A_0)\det H_{22}\bigr),
\]
which has degree $r$ and the zero set in
\eqref{eq:determinant-rank-tangent}.
\end{proof}

\begin{example}[Cubic transverse tangents at nonvertex points]
\label{ex:cubic-nonvertex-tangents}
For $F_4=\operatorname{Re}\det Z$ on $M_4(\mathbb C)$, take
$p=\operatorname{diag}(1,0,0,0)$. Then
\[
 \operatorname{Tan}_p C_4=\R^{14}\times
 \{\operatorname{Re}\det W=0:W\in M_3(\mathbb C)\}.
\]
The transverse equation is cubic. In contrast, the rank-two stratum
has quadratic transverse tangents. Both the original cone and the
cubic factor are area minimizing by \cref{main:complex-determinants}.
Their regular minimality is the determinant construction of
\cite[Corollary~5.2 and Example~5.4]{HoppeTkachev19}.

The normed-product family gives an equally direct example. Let
$P_d(x_1,\ldots,x_d)$ denote the left-associated octonionic product,
and let $\mathcal O_d=\{\operatorname{Re}P_d=0\}$. At
$p=(1,0,0,0)\in\mathcal O_4$, the coordinate change
\[
 (a,b,c,d)\longmapsto(a,(1+a)b,c,d)
\]
is a local diffeomorphism whose derivative at zero is the identity.
It transforms $\operatorname{Re}P_4(1+a,b,c,d)$ into
$\operatorname{Re}((b'c)d)$. Hence the unique set and boundary-current
tangent is $\R^8\times\mathcal O_3$. The same argument gives
$\R^8\times\mathcal O_r$ at $(1,0,\ldots,0)\in\mathcal O_{r+1}$
for every $r\geq2$. These cones are irreducible and area minimizing
by \cref{prop:norm-product-closure,thm:octonionic-product-minimizing}.

These are not cylindrical extensions of their defining equations.
For determinants, evaluation of directional derivatives on
rank-$(m-1)$ matrices detects every matrix entry of the direction.
For normed products, the directional derivative terms have distinct
block multidegrees, and $\operatorname{Re}(uv)$ is a nondegenerate
bilinear form. In both cases $D_vF\equiv0$ forces $v=0$.
The Euclidean factor appears in the tangent at the singular point;
the radial line splitting also follows from
\cref{sec:tangent-instability-transfer}.
\end{example}

For comparison, if a homogeneous cubic $F$ is defined on its
essential space, every nonzero critical point $p$ has
$\Hess F(p)\ne0$, since $D_pF(x)=\tfrac12\Hess F(p)[x,x]$
and essentiality excludes a nonzero direction with $D_pF\equiv0$.
Thus its first nonzero Taylor term is quadratic; for the harmonic
minimal case, see \cref{prop:cubic-general-singular-spectrum,sec:tangent-instability-transfer}.
The higher-order transverse cones above occur in defining degree four and higher.

We now specialize \cref{thm:exact-multiplier} to a positive smooth
invariant multiplier $a=a(E_1,\ldots,E_m)$ on the full-rank locus.
By \eqref{det:basic-two}, the chain rule gives
\[
 \Gamma(F,a)
 =F\sum_{k=1}^m2(m-k+1)E_{k-1}
                   \frac{\partial a}{\partial E_k}.
\]
Thus division by $F$ below gives smooth invariant functions.
Put $Q=E_{m-1}$, $P=a^{-2}|\nabla G|^2$, and
\begin{equation}\label{det:divergence-notation}
 c=\frac{\Gamma(F,a)}{aF},\qquad
 t=\frac{\Gamma(a,a)}{a^2},\qquad
 K=2c+t,\qquad \delta=\frac{\Delta a}{a}+c-t.
\end{equation}
For an invariant coefficient $q=q(E_1,\ldots,E_m)$, write
\begin{equation}\label{det:T-operator}
 \mathcal Tq=\frac{\Gamma(F,q)}F
 =\sum_{k=1}^m2(m-k+1)E_{k-1}\frac{\partial q}{\partial E_k}.
\end{equation}

\begin{lemma}[The multiplier identity]\label{det:divergence-formula}
With the notation just introduced, one has
\begin{align}
 P&=Q+F^2K,\label{det:P}\\
 \cL(G)&=a^3F(D_0+F^2D_1),
 \label{det:divergence}\\
 D_0&=Q\delta-2E_{m-2}-QK-\frac{\Gamma(a,Q)}{2a},
 \label{det:D0}\\
 D_1&=K\delta-\tfrac12\left(\mathcal TK+2cK+\frac{\Gamma(a,K)}a\right).
 \label{det:D1}
\end{align}
\end{lemma}

\begin{proof}
Apply \cref{lem:harmonic-log-multiplier} with \(h=\log a\).
Its notation becomes
\[
 \begin{gathered}
 S=Q,\quad q=2E_{m-2},\quad \tau=c,\quad H_h=t,\\
 Q_h=K,\quad \Delta h=\delta-c,\quad \kappa=\mathcal TK.
 \end{gathered}
\]
The last identity follows from \eqref{det:T-operator}; all these
coefficients are smooth across \(F=0\).
Substituting in \eqref{eq:harmonic-log-gradient} and
\eqref{eq:harmonic-log-two-coefficients}, with
\(\Gamma(h,Q)=\Gamma(a,Q)/a\) and
\(\Gamma(h,K)=\Gamma(a,K)/a\), gives the stated formulas.
\end{proof}

In the notation of \cref{thm:exact-multiplier}, the lemma states
\[
 \cK_F(a)=a^3(D_0+F^2D_1),\qquad
 G\cL(G)=aF^2\cK_F(a)=a^4F^2(D_0+F^2D_1).
\]
The coefficient $D_0$ determines the boundary Jacobi sign, while
$D_1>0$ supplies strict divergence off the cone once $D_0\geq0$.
We now take the uniform product multiplier
\begin{equation}\label{det:uniform-weight}
 a=\prod_{k=1}^{m-1}E_k^\lambda.
\end{equation}
The value $\lambda=1/2$ gives \eqref{det:potential}.

\subsection{From the coefficient signs to area minimization}
\label{det:geometric-proof}

We first deduce the geometric conclusions from the coefficient signs
proved in \cref{det:first-section,det:second-section}.

\begin{proof}[Proof of \cref{main:complex-determinants}]
The case $m=1$ is an oriented hyperplane.

\emph{The sign of $G\cL(G)$.}  For $m\geq2$, take
$\lambda=1/2$ in \eqref{det:uniform-weight}.  Then
$G=aF$ is the potential in \eqref{det:potential}, and
\cref{det:first-reduction,det:first-inequality,det:second-positive}
give $D_0,D_1\geq0$.  Therefore \eqref{det:divergence} yields
\begin{equation}\label{det:divergence-signs}
 G\cL(G)=a^4F^2(D_0+F^2D_1)\geq0
 \qquad\text{on the full-rank locus}.
\end{equation}

\emph{Extension and comparison.}
Off $\mathcal S_m$, the proper minor norms are positive, so $a,G$
are smooth and \eqref{det:divergence-signs} extends by continuity to
rank $m-1$. Positive homogeneity gives $\nabla G\ne0$ there by the
verification following \cref{cor:gradient-subcalibration}.
The continuous potential has the phases of $F$, and the exceptional
set has real codimension eight by \eqref{det:singular-set}.
Thus \cref{cor:gradient-subcalibration,thm:removable-subcalibration}
give perimeter and mass comparison; harmonicity gives full support
by \cref{rem:polynomial-finite-perimeter}.

\emph{Strict minimality.}  For $\lambda=1/2$, \cref{det:second-positive} gives $D_1>0$.
Hence $q_{X_G}=\operatorname{sgn}(G)\diver X_G>0$ at full-rank points
with $F\ne0$. Each phase contains such points, giving positive flux
integrals. Homogeneity and \cref{thm:strict-minimality-flux} prove
strict area minimality.

\emph{Stability.} Off \(\mathcal S_m\), all of
\(E_1,\ldots,E_{m-1}\) are positive, so the invariant expression
\eqref{det:D0} is smooth.
Thus \(D_0\geq0\) extends from full-rank matrices to rank \(m-1\)
by continuity. On the regular cone, the Jacobi potential in
\eqref{eq:boundary-potential-algebraic} is consequently
\[
 V_G=\frac{\cK_F(a)}{a^3|\nabla F|^2}=\frac{D_0}{E_{m-1}}\geq0.
\]
Apply \cref{thm:ground-state-gap} with
$n=2m^2-1$ and $\ell=m(m+1)/2$. This gives
\eqref{eq:main-determinant-stability}:
\[
 Q_{C_m}(u)\geq
 \left(\ell-\frac n2\right)^2\int_{C_m}r^{-2}u^2
 =\frac{(m^2-m-1)^2}{4}\int_{C_m}r^{-2}u^2.\qedhere
\]
\end{proof}

\begin{remark}[Low orders and determinant phase]\label{det:low-orders-phase}
For \(m=2\), the cone \(C_2\) is orthogonally congruent to the
classical Simons cone in \(\R^4\oplus\R^4\).
Indeed, for \(Z=\left(\begin{smallmatrix}z_1&z_2\\z_3&z_4\end{smallmatrix}\right)\),
the unitary change of coordinates
\[
 w=\frac1{\sqrt2}\bigl(z_1+z_4,\ i(z_1-z_4),\ i(z_2+z_3),\ z_2-z_3\bigr)
\]
gives \(\sum_{j=1}^4w_j^2=2\det Z\).
Writing \(w=u+iv\), with \(u,v\in\R^4\), we obtain
\(2F_2=|u|^2-|v|^2\). Thus this order recovers the standard
quadratic cone.

For every real $\theta$, the isometry $Z\mapsto e^{-i\theta/m}Z$
preserves all the $E_k$ and sends $F$ to
$\operatorname{Re}(e^{-i\theta}\det_{\mathbb C}Z)$.
Consequently the phase cone
$\{\operatorname{Re}(e^{-i\theta}\det_{\mathbb C}Z)=0\}$ has the
corresponding subcalibration and is locally area minimizing for every
$m\geq1$, with the same strict minimality and stability conclusions
when $m\geq2$. In particular, the real and imaginary determinant zero cones
are orthogonally congruent.
\end{remark}

\subsection{Nonnegativity of the first coefficient}
\label{det:first-section}

We prove $D_0\geq0$ for the choice $\lambda=1/2$ in
\eqref{det:uniform-weight}.  On the full-rank locus, set
$y_i=x_i^{-1}$. For the remainder of this section, an unadorned
$e_s$ denotes the elementary symmetric polynomial in $y_1,\ldots,y_m$.
For $1\leq s\leq m-1$, define
\begin{equation}\label{det:a-variables}
 a_{si}=\frac{e_s(y_1,\ldots,\widehat y_i,\ldots,y_m)}{e_s(y)},
 \qquad P_i=\sum_{s=1}^{m-1} a_{si},
\end{equation}
and introduce four scalar sums for the coefficient calculation:
\begin{equation}\label{det:A-V-Y}
 A=\sum_i y_i\sum_{s=1}^{m-1}(s a_{si}-a_{si}^2),\qquad
 B=\sum_i y_iP_i,\qquad
 V=\sum_i y_iP_i^2,\qquad Y=\sum_i y_i^2P_i.
\end{equation}

\begin{lemma}\label{det:first-reduction}
For the weight \eqref{det:uniform-weight},
\begin{equation}\label{det:first-reduction-equation}
 \frac{D_0}{Q}
 =4\lambda A-4\lambda^2V
       -2\frac{e_2}{e_1}+2\lambda\frac{Y}{e_1}.
\end{equation}
Consequently, at $\lambda=1/2$, the inequality $D_0\geq0$ is equivalent
to
\begin{equation}\label{det:first-homogeneous}
 2A+\frac{Y}{e_1}-V-2\frac{e_2}{e_1}\geq0.
\end{equation}
\end{lemma}

\begin{proof}
We first explain the reciprocal substitution. Put $p=\prod_i x_i$.
For a $k$-element index set $I$,
$\prod_{i\in I}x_i=p\prod_{j\notin I}y_j$. Summing over $I$ gives
\[
 E_k(x)=p\,e_{m-k}(y),\qquad
 \frac{E_{k-1}(x)}{E_k(x)}
 =\frac{e_{m-k+1}(y)}{e_{m-k}(y)}.
\]
Write $y_{\widehat i}$ for the tuple with $y_i$ omitted. Since \(\partial_{y_i}e_{s+1}(y)=e_s(y_{\widehat i})\),
Euler's identity for the homogeneous polynomial $e_{s+1}$ gives
\[
 \sum_i y_i e_s(y_{\widehat i})=(s+1)e_{s+1}(y),\qquad
 \sum_i y_i a_{si}=(s+1)\frac{e_{s+1}}{e_s}.
\]
For $s=m-k$, applying the complement identity in the $m-1$ variables
with $x_i$ omitted also gives
\[
 e_{k-1}(x_{\widehat i})=p\,y_i e_s(y_{\widehat i}),\qquad
 \frac{e_{k-1}(x_{\widehat i})}{E_k(x)}=y_i a_{si}.
\]
Consequently, the gradient formula \eqref{det:basic-three} yields
\[
 \frac{\Gamma(E_k,E_k)}{E_k^2}
 =4\sum_i x_i\left(\frac{e_{k-1}(x_{\widehat i})}{E_k(x)}\right)^2
 =4\sum_i y_i a_{si}^2,
\]
where $x_i y_i^2=y_i$.
Put
$h=\log a=\lambda\sum_{k=1}^{m-1}\log E_k$.
The chain rule in \eqref{eq:log-multiplier-derivatives} gives
\[
 \frac{\nabla a}{a}=\lambda\sum_{k=1}^{m-1}\frac{\nabla E_k}{E_k},
 \qquad
 \frac{\Delta a}{a}-t
 =\lambda\sum_{k=1}^{m-1}
 \left(\frac{\Delta E_k}{E_k}
       -\frac{\Gamma(E_k,E_k)}{E_k^2}\right).
\]
Using \cref{det:basic-identities} and the identities just proved, we
obtain the quantities defined in \eqref{det:divergence-notation}:
\begin{equation}\label{det:radial-identities}
 \begin{gathered}
 c=2\lambda B,\qquad t=4\lambda^2V,\qquad
 \frac{\Delta a}{a}-t=4\lambda(A+B),\\
 \delta=4\lambda A+6\lambda B,\qquad
 K=4\lambda B+4\lambda^2V.
 \end{gathered}
\end{equation}
Finally,
\[
 \frac{E_{m-2}}Q=\frac{e_2}{e_1},\qquad
 \frac{\Gamma(a,Q)}{aQ}=4\lambda\left(B-\frac{Y}{e_1}\right).
\]
Insert these identities in \eqref{det:D0}.  The three terms involving
$B$ cancel, leaving \eqref{det:first-reduction-equation}.
\end{proof}

\begin{theorem}[First-coefficient inequality]\label{det:first-inequality}
Let $m\geq2$ and let $y_1,\ldots,y_m$ be positive.
Then the quantities in \eqref{det:a-variables}--\eqref{det:A-V-Y}
satisfy \eqref{det:first-homogeneous}.
\end{theorem}

We prove the theorem by a recurrence for symmetric polynomials.  Put
\begin{equation}\label{det:cleared-polynomial}
 D_m^{\mathrm{sym}}=\prod_{s=1}^{m-1}e_s,\qquad
 \mathcal C_m=e_1(D_m^{\mathrm{sym}})^2
       \left(2A+\frac{Y}{e_1}-V-2\frac{e_2}{e_1}\right).
\end{equation}
This is a symmetric polynomial, viewed as a polynomial in
$e_1,\ldots,e_m$. In the first two orders,
\[
 \mathcal C_2=0,\qquad
 \mathcal C_3=6e_1e_2e_3(e_1^2-2e_2).
\]
Thus the first coefficient vanishes identically for $m=2$; for
$m=3$, the normalization $e_1(y)=1$ gives
\[
 2A+Y-V-2e_2
 =6e_3\frac{\sum_i y_i^2}{e_2}>0.
\]
For $N\geq m$, let $\mathcal C_m^{[N]}$ denote
this same formal polynomial evaluated at the elementary symmetric
functions of $N$ positive variables. Its coefficients and summation
limits remain those defining \(\mathcal C_m\); they are not recomputed
by replacing \(m\) with \(N\) in \eqref{det:cleared-polynomial}.
In the induction, \(N\) is fixed and \(r\) increases, so
\(\mathcal C_r^{[N]}\) is evaluated on a fixed \(N\)-tuple.
The formal recurrence below remains valid under these evaluations.

\begin{lemma}[Recurrence for the symmetric polynomials]\label{det:stable-recurrence}
One has $\mathcal C_2=0$.  For every $r\geq2$,
\begin{equation}\label{det:C-recurrence}
 \mathcal C_{r+1}=e_r^2\mathcal C_r+\mathcal R_r.
\end{equation}
With
\[
 \mathcal P_r=e_1^3e_re_{r+1}\prod_{j=2}^{r-1}e_j^2,
\]
the remainder is given by
\begin{align}
 \frac{\mathcal R_r}{\mathcal P_r}
 ={}&r(r+1)
 -\sum_{p=1}^{\lfloor r/2\rfloor}
       (r+1)(2+\mathbf1_{p=1}+\mathbf1_{2p=r})
       \frac{e_r}{e_pe_{r-p}}\notag\\
 &-\sum_{d=1}^{\lfloor(r-1)/2\rfloor}
       \sum_{k=1}^{r-2d}
       (r-k+1)(2+\mathbf1_{k=r-2d})
       \frac{e_ke_r}{e_{k+d}e_{r-d}}.
 \label{det:R-recurrence}
\end{align}
Every denominator on the right cancels against $\mathcal P_r$.
\end{lemma}

The proof is given in \cref{app:determinant-details}.

Newton's inequalities bound the sums in \eqref{det:R-recurrence}
by the following binomial sums, whose total is at most \(r(r+1)\).
This will give \(\mathcal R_r^{[N]}\geq0\).

\Needspace{10\baselineskip}
\begin{lemma}[Scalar binomial inequality]\label{det:binomial-budget}
For $r\geq2$, set
\begin{align*}
 I_r&=\sum_{p=1}^{\lfloor r/2\rfloor}
       \frac{(r+1)(2+\mathbf1_{p=1}+\mathbf1_{2p=r})}{\binom rp},\\
 J_r&=\sum_{d=1}^{\lfloor(r-1)/2\rfloor}\sum_{k=1}^{r-2d}
       (r-k+1)(2+\mathbf1_{k=r-2d})
       \frac{\binom{k+d}{d}}{\binom rd}.
\end{align*}
Then
\begin{equation}\label{det:scalar-budget}
 I_r+J_r\leq r(r+1).
\end{equation}
\end{lemma}

The proof is given in \cref{app:determinant-details}.

\begin{proof}[Proof of \cref{det:first-inequality}]
Fix $N\geq r+1$ positive variables.  Newton's inequalities state that
the normalized sequence $\overline e_j=e_j/\binom Nj$ satisfies
\(\overline e_j^{\,2}\geq\overline e_{j-1}\overline e_{j+1}\).
This is its log concavity.  Hence
\begin{align}
 \frac{e_r}{e_pe_{r-p}}
 &\leq\frac{\binom Nr}{\binom Np\binom N{r-p}}
 \leq\frac1{\binom rp},\label{det:newton-one}\\
 \frac{e_ke_r}{e_{k+d}e_{r-d}}
 &\leq\frac{\binom Nk\binom Nr}
           {\binom N{k+d}\binom N{r-d}}
 \leq\frac{\binom{k+d}{d}}{\binom rd}.
 \label{det:newton-two}
\end{align}
For the last inequality, the quotient of the middle expression by the
last expression is
\[
 \prod_{j=1}^d\frac{N-r+j}{N-k-d+j}\leq1,
\]
since $k+d\leq r$.  Thus \cref{det:binomial-budget} and
\eqref{det:R-recurrence} give $\mathcal R_r^{[N]}\geq0$ for every
$N\geq r+1$.

We induct in $r$ with the assertion
$\mathcal C_r^{[N]}\geq0$ for every $N\geq r$.  The base is
$\mathcal C_2=0$, and \eqref{det:C-recurrence} proves the next step.
Taking $N=m$ and dividing by the positive factor in
\eqref{det:cleared-polynomial} proves
\eqref{det:first-homogeneous}.  Equivalently, under $e_1(y)=1$, this is
$2A+Y-V-2e_2\geq0$.
\end{proof}

\subsection{Nonnegativity of the second coefficient}
\label{det:second-section}

We prove $D_1\geq0$ by a covariance formula valid for every
$\lambda\geq0$, including $\lambda=1/2$. We retain the variables
$y_i=x_i^{-1}$.
\begin{lemma}\label{det:second-positive}
For the uniform weight \eqref{det:uniform-weight}, $D_1\geq0$ on the
full-rank locus whenever $\lambda\geq0$. If $\lambda>0$, then $D_1>0$.
\end{lemma}

\begin{proof}
To differentiate the ratios $a_{si}$, we use a finite probability space.
For each $1\leq s\leq m-1$, consider the probability measure on
$s$-element subsets of $\{1,\ldots,m\}$ given by
\[
 \mathbb P_s(S)=\frac{\prod_{i\in S}y_i}{e_s(y)}.
\]
This is the conditional Bernoulli distribution of
\cite[pp.~876--877, equations~(3) and~(6); Section~3.2, p.~883]{ChenLiu97}, with weights
\(y_i\) and fixed sample size \(s\).
For random variables \(\xi,\eta\) on this finite probability space,
their covariance is
\[
 \operatorname{Cov}_s(\xi,\eta)
 =\mathbb E_s(\xi\eta)-\mathbb E_s(\xi)\mathbb E_s(\eta).
\]
Let \(\xi_i(S)\) be the indicator of \(i\in S\), equal to one
when \(i\) belongs to \(S\) and zero otherwise.
Define
\[
 \pi_{si}=\mathbb P_s(i\in S)=1-a_{si},\qquad
 w_{sij}=\pi_{si}\pi_{sj}-\mathbb P_s(i,j\in S)\quad(i\ne j).
\]
Thus \(w_{sij}=-\operatorname{Cov}_s(\xi_i,\xi_j)\).
If $e_q^{ij}$ denotes the elementary symmetric polynomial with $y_i,y_j$
omitted, then
\begin{equation}\label{det:covariance-positive}
 w_{sij}=\frac{y_iy_j}{e_s^2}
       \left[(e_{s-1}^{ij})^2-e_{s-2}^{ij}e_s^{ij}\right]\geq0.
\end{equation}
This is Newton's inequality, with the convention $e_{-1}^{ij}=0$.
It is also the Rayleigh inequality for \(e_s\), expressing pairwise
negative correlation; see
\cite[Definitions~2.3 and~2.5]{BorceaBrandenLiggett09}.
Since every subset has exactly \(s\) elements,
\(\sum_i\xi_i=s\).  This gives
\begin{equation}\label{det:covariance-rowsum}
 \sum_{j\ne i}w_{sij}=\pi_{si}(1-\pi_{si})=\pi_{si}a_{si}.
\end{equation}
Indeed, the sum of the covariances of the indicator of $i\in S$ with
all the indicators is its covariance with the constant $|S|=s$.

In addition to \eqref{det:A-V-Y}, put
\[
 L=\sum_i y_i\sum_{s=1}^{m-1}s a_{si},\qquad Z_2=\sum_i y_i^2P_i^2,
 \qquad U=\sum_i y_i^2P_i^3,
\]
and define
\begin{align}
 H_s&=\sum_i y_i^2a_{si}
       -\sum_{i<j}w_{sij}(y_i-y_j)^2,
 \label{det:H-positive}\\
 \mathcal Q_1&=\sum_{s,i<j}w_{sij}
       \left[B(y_i+y_j)-(y_i-y_j)(y_iP_i-y_jP_j)\right],
 \label{det:Q1}\\
 \mathcal Q_2&=\sum_{s,i<j}w_{sij}
       \left[V(y_i+y_j)-(y_iP_i-y_jP_j)^2\right].
 \label{det:Q2}
\end{align}

We claim the decomposition
\begin{equation}\label{det:D1-positive-formula}
 D_1=\lambda R_1+\lambda^2R_2+\lambda^3R_3,
\end{equation}
where
\begin{align}
 R_1&=4\sum_sH_s,\label{det:R1}\\
 R_2&=16LB+12Z_2+16\mathcal Q_1,\label{det:R2}\\
 R_3&=16LV+8U+16\mathcal Q_2.\label{det:R3}
\end{align}
We derive it first, then prove the coefficient signs.
Write
$\mathscr D_j=y_j\partial_{y_j}$.  Differentiating the probabilities gives
\begin{equation}\label{det:probability-derivative}
 \mathscr D_j a_{si}
 =\begin{cases}
   w_{sij},&i\ne j,\\
   -\sum_{\ell\ne i}w_{si\ell},&i=j.
  \end{cases}
\end{equation}
Indeed, direct differentiation gives
\(\mathscr D_j\mathbb P_s(S)=(\xi_j(S)-\pi_{sj})\mathbb P_s(S)\).
Thus \(-\mathscr D_j a_{si}=\operatorname{Cov}_s(\xi_i,\xi_j)\),
the usual covariance identity in the logarithmic parameters
\(\log y_j\).
Consequently,
\begin{align}
 \mathscr D_jB
 &=y_jP_j+\sum_s\sum_{i\ne j}w_{sij}(y_i-y_j),
 \label{det:DB}\\
 \mathscr D_jV
 &=y_jP_j^2+2\sum_s\sum_{i\ne j}w_{sij}(y_iP_i-y_jP_j).
 \label{det:DV}
\end{align}
For any smooth invariant function $q$ written in the $y$ variables,
the change of variables gives
$\partial_{x_j}q=-y_j\mathscr D_jq$.
The product multiplier \eqref{det:uniform-weight} satisfies
$a^{-1}\partial_{x_j}a=\lambda y_jP_j$, by the complement identity
in the proof of \cref{det:first-reduction}.
Thus \eqref{det:T-operator} and the gradient pairing
\eqref{det:basic-three} give
\begin{equation}\label{det:radial-operators}
 \mathcal Tq=-2\sum_j y_j\mathscr D_jq,\qquad
 \frac{\Gamma(a,q)}a=-4\lambda\sum_j y_jP_j\mathscr D_jq.
\end{equation}
Insert \eqref{det:radial-identities} and
\eqref{det:radial-operators} in \eqref{det:D1}.  The coefficients of
$\lambda,\lambda^2,\lambda^3$ are, respectively,
\begin{align*}
 R_1&=4\sum_jy_j\mathscr D_jB,\\
 R_2&=16B(A+B)+4\sum_jy_j\mathscr D_jV
                       +8\sum_jy_jP_j\mathscr D_jB,\\
 R_3&=16V(A+B)+8\sum_jy_jP_j\mathscr D_jV.
\end{align*}
Pairing the terms with indices $i,j$ in
\eqref{det:DB}--\eqref{det:DV} gives
\begin{align*}
 \sum_jy_j\mathscr D_jB&=\sum_sH_s,\\
 \sum_jy_j\mathscr D_jV
 &=Z_2-2\sum_{s,i<j}w_{sij}(y_i-y_j)(y_iP_i-y_jP_j),\\
 \sum_jy_jP_j\mathscr D_jB
 &=Z_2-\sum_{s,i<j}w_{sij}(y_i-y_j)(y_iP_i-y_jP_j),\\
 \sum_jy_jP_j\mathscr D_jV
 &=U-2\sum_{s,i<j}w_{sij}(y_iP_i-y_jP_j)^2.
\end{align*}
The row-sum identity \eqref{det:covariance-rowsum} also gives
\begin{equation}\label{det:AplusB}
 A+B=L+\sum_{s,i<j}w_{sij}(y_i+y_j).
\end{equation}
These identities yield \eqref{det:R1}--\eqref{det:R3}.

To prove the signs, use \eqref{det:covariance-rowsum} to obtain
\[
 \sum_{i<j}w_{sij}(y_i-y_j)^2
 \leq\sum_i y_i^2\pi_{si}a_{si}
 \leq\sum_i y_i^2a_{si}.
\]
Thus $H_s\geq0$.  Also,
\[
 |(y_i-y_j)(y_iP_i-y_jP_j)|
 \leq(y_i+y_j)(y_iP_i+y_jP_j)
 \leq B(y_i+y_j),
\]
so $\mathcal Q_1\geq0$.  Cauchy--Schwarz gives
\[
 (y_iP_i-y_jP_j)^2
 \leq(y_i+y_j)(y_iP_i^2+y_jP_j^2)
 \leq V(y_i+y_j),
\]
which proves $\mathcal Q_2\geq0$.  The other terms in
\eqref{det:R2}--\eqref{det:R3} are nonnegative by definition.
Finally, $P_i>0$ and $y_i>0$ imply $Z_2>0$.
Thus $R_2\geq12Z_2$ and $D_1\geq12\lambda^2Z_2>0$ when $\lambda>0$.
\end{proof}

Despite a common gradient-pairing table, the different Laplacian
coefficients of the spectral families in
\cref{sec:norm-products-pfaffians} and the determinants
lead to the multipliers $\sqrt S$ and the product of proper minor norms,
respectively. Algebraic closure and spectral inequalities certify both
in every order. For determinants, $m+1$ independent generators when
$m\ge3$ give unbounded defining degree and generator count.

\appendix
\section[Power Multipliers]{Power Multipliers and Radial Perturbations}
\label{app:power-identities}

All identities involving negative or nonintegral powers are used on
\(R,S>0\).  In the power calculations, \(F\) is homogeneous of
degree \(e\geq1\) and \(\cL(F)=F\lambda_F\), with \(\lambda_F\)
smooth on the calculation domain.  It is a polynomial whenever the
divisibility hypothesis of \cref{thm:exact-multiplier} holds.

\subsection{Notation and power factors}

Consider the power family
\begin{equation}\label{eq:power-family}
 G_{\alpha,\gamma}=R^\alpha S^\gamma F,
 \qquad R=|x|^2,\quad S=|\nabla F|^2,
\end{equation}
with real exponents.
Write \(e=\deg F\), and use the contractions
\[
 U=\Gamma(F,S),\quad V=\Gamma(S,S),\quad
 D=\Delta F,\quad B=\Delta S.
\]
Here \(U,V,D,B\) are ambient functions; in a finite Laplacian closure
they are functions of the generators.  Define
\begin{align}
 \mathcal V_\gamma
 &=S+\frac{2\gamma FU}{S}+\frac{\gamma^2F^2V}{S^2},
 \label{eq:gradient-power-V}\\
 \mathcal D_\gamma
 &=D+\frac{2\gamma U}{S}
       +\frac{\gamma FB}{S}
       +\frac{\gamma(\gamma-1)FV}{S^2}.
 \label{eq:gradient-power-D}
\end{align}
Let \(\mathcal Q_\gamma\) be the explicit expression in
\eqref{eq:gradient-power-Q}.  It satisfies
\(\cL(S^\gamma F)=S^{3\gamma}F\mathcal Q_\gamma\).

\subsection{Gradient and radial identities}

\begin{proposition}[The gradient-power defect]
\label{prop:gradient-power-identity}
For \(H=S^\gamma F\),
\[
 |\nabla H|^2=S^{2\gamma}\mathcal V_\gamma,\qquad
 \Delta H=S^\gamma\mathcal D_\gamma,\qquad
 \cL(H)=S^{3\gamma}F\mathcal Q_\gamma,
\]
where \(\mathcal V_\gamma,\mathcal D_\gamma\) are given in
\eqref{eq:gradient-power-V}--\eqref{eq:gradient-power-D}, and
\begin{align}
 \mathcal Q_\gamma={}&\lambda_F+\gamma B
 +\frac{2\gamma UD-\gamma\Gamma(F,U)
                 -(\gamma^2+\tfrac32\gamma)V}{S}
 +\frac{\gamma(\gamma+1)U^2}{S^2}\notag\\
 &+\frac{\gamma^2F}{S^2}
       \left(2UB+VD-\Gamma(S,U)-\tfrac12\Gamma(F,V)\right)\notag\\
 &+\frac{\gamma^3F^2}{S^3}
       \left(VB-\tfrac12\Gamma(S,V)\right).
 \label{eq:gradient-power-Q}
\end{align}
At every regular zero of \(F\), one has \(S>0\). Under the assumed
smoothness of \(\lambda_F\), the displayed formula is therefore
well-defined there; no division by \(F\) remains.
\end{proposition}

\begin{proof}
The gradient and Laplacian follow from the product rule.
For the defect, insert \(a=S^\gamma\) in
\cref{thm:exact-multiplier}.  The derivative identities are
\[
 \nabla a=\gamma S^{\gamma-1}\nabla S,\qquad
 \Hess a=\gamma S^{\gamma-1}\Hess S
       +\gamma(\gamma-1)S^{\gamma-2}\,dS\otimes dS.
\]
Differentiating \(S=|\nabla F|^2\), \(U=\Gamma(F,S)\), and
\(V=|\nabla S|^2\) gives
\begin{align*}
 \Hess F(\nabla F,\nabla S)&=V/2,\\
 \Hess S(\nabla F,\nabla F)&=\Gamma(F,U)-V/2,\\
 \Hess F(\nabla S,\nabla S)&=\Gamma(S,U)-\Gamma(F,V)/2,\\
 \Hess S(\nabla F,\nabla S)&=\Gamma(F,V)/2,\\
 \Hess S(\nabla S,\nabla S)&=\Gamma(S,V)/2.
\end{align*}
Substitution and collection of the powers of \(F\) yield
\eqref{eq:gradient-power-Q}.
\end{proof}

\begin{proposition}[The radial master identity]
\label{prop:radial-master-identity}
Let \(H\) be smooth and homogeneous of degree \(m\), and set
\(S_H=|\nabla H|^2\), \(D_H=\Delta H\), and
\(c_\alpha=4\alpha(m+\alpha)\).  Then
\begin{align}
 |\nabla(R^\alpha H)|^2
 &=R^{2\alpha}\left(S_H+\frac{c_\alpha H^2}{R}\right),
 \label{eq:radial-master-gradient}\\
 \cL(R^\alpha H)
 &=R^{3\alpha}\Biggl[\cL(H)
  +\frac{c_\alpha H^2D_H}{R}
  +\frac{\bigl(2\alpha(N-1)-c_\alpha\bigr)HS_H}{R}\notag\\
 &\hspace{31mm}
  +\frac{c_\alpha\bigl(m+2\alpha(N-1)\bigr)H^3}{R^2}
  \Biggr].
 \label{eq:radial-master-L}
\end{align}
\end{proposition}

\begin{proof}
Euler's identity gives \(\Gamma(R,H)=2mH\).  Hence
\[
 \Delta(R^\alpha H)
 =R^\alpha\left[D_H+
    \frac{2\alpha(2m+N+2\alpha-2)H}{R}\right],
\]
and the product rule gives \eqref{eq:radial-master-gradient}.
Writing \(W=S_H+c_\alpha H^2/R\), of degree \(2m-2\), yields
\[
 R^{-3\alpha}\cL(R^\alpha H)
 =WD_H+\frac{2\alpha(N-1)HW}{R}-\tfrac12\Gamma(H,W).
\]
Use
\(\Gamma(H,H^2/R)=2HS_H/R-2mH^3/R^2\)
to expand the last term.  This proves \eqref{eq:radial-master-L}.
\end{proof}

\subsection{The combined multiplier}

\begin{proposition}[The combined power multiplier]
\label{prop:combined-power-multiplier}
Suppose \(\cL(F)=F\lambda_F\).  Set
\[
 m=e+2\gamma(e-1),\quad \ell=m+2\alpha,\quad
 c_\alpha=4\alpha(m+\alpha),\quad
 b_\alpha=2\alpha(N-1)-c_\alpha.
\]
For \(G=G_{\alpha,\gamma}\), one has
\begin{align}
 |\nabla G|^2
 &=R^{2\alpha}S^{2\gamma}\widehat V_{\alpha,\gamma},
 &\widehat V_{\alpha,\gamma}
 &=\mathcal V_\gamma+\frac{c_\alpha F^2}{R},
 \label{eq:combined-power-V}\\
 \cL(G)
 &=R^{3\alpha}S^{3\gamma}F\widehat Q_{\alpha,\gamma},
 \label{eq:combined-power-L}
\end{align}
where
\begin{equation}\label{eq:combined-power-Q}
 \widehat Q_{\alpha,\gamma}
 =\mathcal Q_\gamma+\frac{c_\alpha F\mathcal D_\gamma}{R}
 +\frac{b_\alpha\mathcal V_\gamma}{R}
 +\frac{c_\alpha\bigl(m+2\alpha(N-1)\bigr)F^2}{R^2}.
\end{equation}
Thus \(\widehat V_{\alpha,\gamma}>0\) and
\(\widehat Q_{\alpha,\gamma}\geq0\) give the normalized-gradient
sign on the region where they hold.  If that region is one phase,
they give its one-sided sign.  Comparison across the omitted set
requires the hypotheses on the exceptional set and boundary flux in
\cref{thm:paired-subcalibration}.
\end{proposition}

\begin{proof}
Apply \cref{prop:radial-master-identity} to \(H=S^\gamma F\), of
degree \(m\), using the three identities of
\cref{prop:gradient-power-identity}. This gives
\eqref{eq:combined-power-V}--\eqref{eq:combined-power-Q}.
Since \(R^\alpha S^\gamma>0\), \(G\cL(G)\) has the sign of
\(F^2\widehat Q_{\alpha,\gamma}\).
\end{proof}

\subsection{Radial perturbations giving strict divergence estimates}

\begin{proposition}[Moving the degree toward \(n/2\)]
\label{prop:radial-strictification}
Let \(G\in C^\infty(\R^{n+1}\setminus\{0\})\) be homogeneous of degree
\(\ell\neq n/2\), with \(\nabla G\neq0\) throughout the punctured space.
Suppose \(C=\{G=0\}\) has a smooth compact minimal link,
\(G\) takes both signs, and
\[
 \operatorname{sgn}(G)\diver X_G>0\quad\text{off }C,
 \qquad\beta_G\geq0\quad\text{on }C\cap\Sph^n,
\]
where \(\beta_G\) is defined in \eqref{eq:boundary-Jacobi-potential}.
Put \(\sigma=\operatorname{sgn}(\ell-n/2)\).  For all sufficiently small
\(\eps>0\), the potential
\begin{equation}\label{eq:strictifying-potential}
 G_\eps=r^{-\sigma\eps}G,\qquad \ell_\eps=\ell-\sigma\eps,
\end{equation}
has nonvanishing gradient on the punctured space and generates a global
subcalibration satisfying
\begin{equation}\label{eq:strictified-distance-bound}
 \operatorname{sgn}(G_\eps)\diver X_{G_\eps}
 \geq c_\eps\frac{\dist(x,C)}{|x|^2}
 \qquad(x\neq0)
\end{equation}
for a constant \(c_\eps>0\).  Writing \(\beta_\eps=\beta_{G_\eps}\), its boundary coefficient is
\begin{equation}\label{eq:strictified-boundary}
 \beta_\eps=\beta_G+\left(\ell-\frac n2\right)^2
                 -\left(\ell_\eps-\frac n2\right)^2
       =\beta_G+2\left|\ell-\frac n2\right|\eps-\eps^2.
\end{equation}
The new potential is generally not polynomial. If \(G\) is polynomial
and \(\eps\) is positive rational, then \(G_\eps\) is semialgebraic
on \(r>0\): the radial factor is a rational power of the positive
polynomial \(R\), hence is described by a polynomial equation and
the choice of its positive root.
\end{proposition}

\begin{proof}
Nonvanishing of \(\nabla G_\eps\) for small \(\eps\) follows from smooth
parameter dependence and compactness of the unit sphere.
On \(C\), its reciprocal gradient length is
\(w_\eps=r^{\sigma\eps}w=r^{1-\ell_\eps}\phi\).
For \(w=r^\alpha\phi\), separation of variables gives
\[
 -r^2\frac{J_Cw}{w}
 =\frac{(-\Delta_\Sigma-|A_\Sigma|^2)\phi}{\phi}
                  -\alpha(\alpha+n-2).
\]
Use \(\alpha=1-\ell\) and \(\alpha_\eps=1-\ell_\eps\) to obtain
\eqref{eq:strictified-boundary}.

For the global sign, let \(s\) be signed spherical distance to the link
in a fixed tubular neighborhood, positive on \(G>0\).  On \(r=1\), put
\(f_\eps=\diver X_{G_\eps}\).
Minimality gives \(f_\eps=0\) on the link.  The quotient
\(h_\eps=f_\eps/s\) extends smoothly across it, with
\(h_\eps|_\Sigma=\beta_\eps\).
By \eqref{eq:strictified-boundary},
\(\partial_\eps h_\eps|_{\eps=0,\Sigma}=2|\ell-n/2|>0\).
Compactness and continuity give a smaller fixed tube and a parameter
interval on which \(\partial_\eps h_\eps\geq|\ell-n/2|\).
Since \(h_0\geq0\), integration in the parameter yields
\(h_\eps\geq\eps|\ell-n/2|\) throughout that tube.
On the complement of a still smaller tube the original strict sign has a
positive minimum, which persists for small \(\eps\).
Spherical normal distance is comparable to \(\dist(\theta,C)\) near the
link; away from it the latter is bounded.  This proves the uniform
unit-sphere estimate, and homogeneity gives
\eqref{eq:strictified-distance-bound}.
\end{proof}

For a singular link,
\eqref{eq:strictified-boundary} still holds on the regular part.
The global sign requires additional estimates near the singular strata.

\subsection{Radial multiplication preserving the divergence sign}

\begin{proposition}[Radial multiplication of a homogeneous potential]
\label{prop:cubic-radial-lift}
Let \(h\) be smooth and homogeneous of degree \(d>0\) on a conical
open subset of \(\R^N\setminus\{0\}\). If
\[
 h\cL(h)\ge0,\qquad h\Delta h\ge0,\qquad N\ge2d+2,
\]
then \(g=\sqrt R\,h\) satisfies \(g\cL(g)\ge0\), strictly where
\(h\ne0\). If \(\nabla h\ne0\) throughout the domain, then
\(\nabla g\ne0\) there as well.
\end{proposition}

\begin{proof}
The radial master identity, \cref{prop:radial-master-identity} with
\(\alpha=1/2\), reads
\begin{align}
 \cL(\sqrt R\,h)=R^{3/2}\Biggl[\cL(h)
 &+(2d+1)\frac{h^2\Delta h}{R}
 +(N-2d-2)\frac{h|\nabla h|^2}{R}\notag\\
 &+(2d+1)(d+N-1)\frac{h^3}{R^2}\Biggr].
 \label{eq:cubic-radial-lift}
\end{align}
Multiplication by \(g=\sqrt R\,h\) makes every term nonnegative;
the last one is positive when \(h\ne0\). The same master identity
gives
\[
 |\nabla g|^2=R|\nabla h|^2+(2d+1)h^2,
\]
proving the final assertion.
\end{proof}

\section[Odd-Degree Rigidity and Cubic Structure]{Odd-Degree Rigidity, Cubic Structure and Local Geometry}
\label{app:cubic-structure}

\subsection{Odd-degree rigidity by second-order expansions}
\label{sec:odd-rigidity-proof}
We prove \cref{thm:odd} and determine the normalized cubic closure.

\subsubsection{The closure identities and the degree}
Throughout this subsection, \(e\geq3\) is odd and
\(S=|\nabla F|^2\). Counting homogeneous monomials in the
generators of degrees \(2,e,2e-2\) gives
\begin{equation}\label{eq:oddtable}
\begin{aligned}
 \Delta F&=0,&\Delta S&=\alpha R^{e-2},\\
 \Gamma(F,S)&=\beta R^{e-2}F,&
 \Gamma(S,S)&=\gamma R^{2e-3}
       +\delta R^{e-3}F^2+\varepsilon R^{e-2}S.
\end{aligned}
\end{equation}
Bochner's identity gives \(2|\Hess F|^2=\alpha R^{e-2}\),
with \(\alpha>0\).

Let \(p\ne0\) be a critical point. Euler's identity gives
\(F(p)=0\), and \(\nabla S(p)=0\), so \(\gamma=0\).
Put \(H=\Hess F(p)\) and \(R_0=|p|^2\). The quadratic terms
at \(p\) in the last two identities of \eqref{eq:oddtable} give
\[
 H^3=\frac\beta4R_0^{e-2}H,\qquad
 H^4=\frac\varepsilon4R_0^{e-2}H^2.
\]
Since \(H\ne0\), these identities imply \(\beta>0\) and
\(\varepsilon=\beta\).

At a positive maximum of \(F\) on the unit sphere,
\(\nabla F=eFx\) and \(S=e^2F^2>0\). Moreover,
\(\nabla S=2\Hess F\,\nabla F\) and the differentiated Euler
identity \(\Hess F\,x=(e-1)\nabla F\) give
\(\nabla S=2(e-1)Sx\). Substitution gives
\[
 S=\frac{\beta}{2e(e-1)},\qquad
 4(e-1)^2S^2=\delta\frac{S}{e^2}+\beta S,
\]
whence \(\delta=e(e-2)\beta\). Multiplying \(F\) by
\(2/\sqrt\beta\) and reusing \(F,S,\alpha\) for the normalized
quantities, we obtain
\begin{equation}\label{eq:normalizedodd}
\begin{aligned}
 \Delta F&=0,&\Delta S&=\alpha R^{e-2},\\
 \Gamma(F,S)&=4R^{e-2}F,&
 \Gamma(S,S)&=4R^{e-2}S+4e(e-2)R^{e-3}F^2.
\end{aligned}
\end{equation}

Write \(f=F|_{\Sph^{N-1}}\) and \(s=S|_{\Sph^{N-1}}\).
Removing the radial derivatives from \eqref{eq:normalizedodd} gives
\begin{equation}\label{eq:odd-spherical-products}
\begin{aligned}
 |\nabla_{\Sph}f|^2&=s-e^2f^2,\\
 \langle\nabla_{\Sph}f,\nabla_{\Sph}s\rangle
   &=f\bigl(4-2e(e-1)s\bigr),\\
 |\nabla_{\Sph}s|^2&=4s+4e(e-2)f^2-4(e-1)^2s^2.
\end{aligned}
\end{equation}
Let \(a>0\) be the maximum of \(f\), attained at \(x_0\).
Both spherical gradients vanish at \(x_0\), and
\[
 s(x_0)=\frac{2}{e(e-1)},\qquad
 a^2=\frac{2}{e^3(e-1)}.
\]
Let \(B=\Hess_{\Sph}f(x_0)\) and
\(C=\Hess_{\Sph}s(x_0)\). Comparing second-order terms in
the three identities of \eqref{eq:odd-spherical-products},
in geodesic normal coordinates at \(x_0\), gives
\begin{equation}\label{eq:odd-maximum-hessians}
\begin{aligned}
 C&=2B(B+e^2aI),\\
 BC&=-e(e-1)aC,\\
 C^2&=\left(\frac8e-6\right)C+4e(e-2)aB.
\end{aligned}
\end{equation}
The first equation shows that \(B\) and \(C\) commute.
The first two equations restrict every eigenvalue of \(B\) to
\(0,-e^2a,-e(e-1)a\). The value \(-e^2a\) would give a zero
eigenvalue of \(C\) but a nonzero right-hand side in the last
equation, so it is impossible. Since
\(\tr B=-e(N+e-2)a<0\), some eigenvalue of \(B\) is nonzero.
On the corresponding eigenvector, \(B=-e(e-1)a\) and
\(C=-4/e\). The last equation now gives
\[
 0=\frac{16}{e^2}
   -\left(\frac8e-6\right)\left(-\frac4e\right)
   +4e^2(e-1)(e-2)a^2
   =\frac{8(e-2)(e-3)}{e^2}.
\]
Thus \(e=3\), proving \cref{thm:odd}.

\subsubsection{The remaining cubic coefficient}
For \(e=3\), the spherical Laplacians are
\[
 \Delta_{\Sph}f=-3(N+1)f,\qquad
 \Delta_{\Sph}s=\alpha-4(N+2)s.
\]
Let brackets denote normalized spherical integration, and set
\[
 M_0=\langle f^2\rangle>0,\quad M_1=\langle s\rangle,
 \quad M_2=\langle s^2\rangle,\quad M_3=\langle f^2s\rangle.
\]
Using \eqref{eq:odd-spherical-products} with \(e=3\), integration
by parts gives
\begin{align}
 M_1&=3(N+4)M_0,\label{eq:m1}\\
 (\alpha+8)M_0&=4(N+8)M_3,\label{eq:m3}\\
 M_2+4M_0&=3(N+8)M_3,\label{eq:m2}\\
 4(N+6)M_2&=(\alpha+4)M_1+12M_0.\label{eq:m4}
\end{align}
These follow, respectively, by integrating
\(\Delta_{\Sph}(f^2)\), by testing the equation for
\(\Delta_{\Sph}s\) against \(f^2\), by testing the equation for
\(\Delta_{\Sph}f\) against \(fs\), and by integrating
\(\Delta_{\Sph}(s^2)\).
The middle two identities give
\(4M_2=(3\alpha+8)M_0\). Substituting this and
\eqref{eq:m1} into \eqref{eq:m4}, and cancelling \(M_0>0\), yields \((N+6)(3\alpha+8)=3(N+4)(\alpha+4)+12\).
Consequently,
\begin{equation}\label{eq:alphadim}
 \alpha=\frac{2(N+3)}3.
\end{equation}

\subsection{Odd gradient-square closure with trivial critical set}
\label{sec:odd-isolated-proof}

\begin{proof}[Proof of \cref{thm:odd-isolated-closure}]
The count in \eqref{eq:oddtable} applies without a nonzero critical
point. We retain its constants
\(\alpha,\beta,\gamma,\delta,\varepsilon\), without imposing the
normalization \eqref{eq:normalizedodd}. Write
\(f=F|_{\Sph^{N-1}}\), \(s=S|_{\Sph^{N-1}}\), and put
\[
 K=2e(e-1),\qquad L=4(e-1)^2,\qquad
 D=2(e-1)(N+2e-4).
\]
The spherical identities are
\begin{equation}\label{eq:odd-isolated-spherical}
\begin{aligned}
 |\nabla f|^2&=s-e^2f^2,&
 \langle\nabla f,\nabla s\rangle&=f(\beta-Ks),\\
 |\nabla s|^2&=Q(s)+\delta f^2,&
 Q(t)&=\gamma+\varepsilon t-Lt^2,\\
 \Delta f&=-e(N+e-2)f,&
 \Delta s&=\alpha-Ds.
\end{aligned}
\end{equation}
Here and throughout this proof, gradients, Hessians, and Laplacians
without subscripts are spherical. These formulas follow by subtracting
the radial derivatives from \eqref{eq:oddtable}; the Laplacian formula
uses \(\Delta(r^jv)=r^{j-2}[\Delta_{\Sph}v+j(j+N-2)v]\).

\emph{Step 1: the upper bound for \(s\).}
The assumption on the critical set gives \(s>0\). At a critical
point of \(f\), Euler's identity gives
\(\nabla_{\R^N}F=efx\), and differentiation of Euler's identity
gives \(\nabla_{\R^N}S=2(e-1)Sx\). Since \(s=e^2f^2>0\)
there, the identity for \(\Gamma(F,S)\) yields \(\beta=2e^3(e-1)f^2>0\).
Rescale \(F\) so that \(\beta=1\), and write
\[
 s_0=\frac1K,\qquad a^2=\frac{s_0}{e^2}.
\]
Thus the only critical values of \(f\) are \(a\) and \(-a\),
and \(s=s_0\) at every such point. Oddness ensures that both values
occur and that \(-a\leq f\leq a\).

We claim that \(s\leq s_0\) everywhere. Otherwise follow the complete
gradient flow of \(f\) from a point where \(g=s-s_0>0\). Along
this trajectory, with parameter \(t\),
\[
 \dot f=s-e^2f^2,\qquad \dot g=-Kfg.
\]
The second equation preserves \(g>0\). The function \(f\) increases
and is bounded, and every accumulation point is critical: indeed,
\(\int_0^\infty|\nabla f|^2\,dt<\infty\), while
\(|\nabla f|^2\) is uniformly continuous along a smooth flow on the
compact sphere. Since the initial value is larger than \(-a\), it
follows that \(f\to a\) and \(g\to0\). Direct differentiation
of the two flow equations shows that
\[
 \bigl(f^2-a^2-g/e\bigr)g^{-e/(e-1)}
\]
is constant. Consequently
\[
 f^2-a^2=\frac ge+c_1g^{e/(e-1)}.
\]
The right-hand side is positive for sufficiently small positive \(g\),
contradicting \(|f|\leq a\). This proves the claim.

\emph{Step 2: the minimum and maximum Hessians.}
Suppose that \(s\) is nonconstant. Its minimum \(m\) satisfies
\(0<m<s_0\). At a minimum point, the second identity in
\eqref{eq:odd-isolated-spherical} gives \(f=0\), so
\(v=\nabla f\) has \(|v|^2=m\). Let \(B_m=\Hess s\) there.
Differentiating that identity once, and the identity for
\(|\nabla s|^2\) twice, gives
\begin{equation}\label{eq:odd-isolated-minimum}
 B_mv=\eta v,\qquad
 B_m^2=\frac{Q'(m)}2B_m+\delta v\otimes v,
 \qquad \eta=1-Km>0.
\end{equation}
In particular,
\begin{equation}\label{eq:odd-isolated-minimum-scalar}
 Q(m)=0,\qquad
 \eta^2=\frac{Q'(m)}2\eta+\delta m.
\end{equation}
For example, the second matrix identity follows because the Hessian
of \(|\nabla s|^2\) at a critical point of \(s\) is \(2B_m^2\),
whereas the Hessian of \(\delta f^2\) at \(f=0\) is
\(2\delta v\otimes v\).

At a point where \(f=a\), put \(B=\Hess f\) and
\(C=\Hess s\). Differentiating
\eqref{eq:odd-isolated-spherical} twice gives
\begin{equation}\label{eq:odd-isolated-maximum}
 C=2B(B+e^2aI),\qquad
 BC=-e(e-1)aC,\qquad
 C^2=\left(\frac\varepsilon2-Ls_0\right)C+\delta aB.
\end{equation}
In the middle identity we used the first one to commute \(B\) and
\(C\). It follows that the eigenvalues of \(B\) belong to
\[
 0,\quad -e^2a,\quad -e(e-1)a,
\]
and the corresponding eigenvalues of \(C\) are
\(0,0,-1/e\).

\emph{Step 3: exclusion of a nonconstant \(s\).}
If \(\delta=0\), the equations at the two extrema give
\(Q(m)=Q(s_0)=0\). Hence \(Q(t)=-L(t-m)(t-s_0)\).
The scalar identity in \eqref{eq:odd-isolated-minimum-scalar},
together with \(\eta>0\), gives
\[
 K(s_0-m)=\eta=\frac{Q'(m)}2=\frac L2(s_0-m).
\]
This is impossible because \(K=2e(e-1)\ne2(e-1)^2=L/2\).

Now suppose \(\delta\ne0\). In
\eqref{eq:odd-isolated-maximum}, the eigenvalue \(-e^2a\) of
\(B\) is impossible, because its \(C\)-eigenvalue is zero.
Since \(\tr B=-e(N+e-2)a<0\), the eigenvalue
\(-e(e-1)a\) occurs. Substitution in the last identity yields
\begin{equation}\label{eq:odd-isolated-delta}
 \delta=4e-6-e\varepsilon.
\end{equation}
The value of the gradient-square identity at this same point is
\begin{equation}\label{eq:odd-isolated-q-at-maximum}
 Q(s_0)+\delta s_0/e^2=0.
\end{equation}
Set \(x=m/s_0\in(0,1)\). Eliminating \(\gamma\) with
\eqref{eq:odd-isolated-q-at-maximum}, and substituting
\eqref{eq:odd-isolated-delta}, the two scalar equations
\eqref{eq:odd-isolated-minimum-scalar}, after multiplication by
their nonzero denominators, become
\begin{align*}
 0={}&e\varepsilon[ex-(e-1)]
       +2(e^2-3e+3)-2e(e-1)x^2,\\
 0={}&e\varepsilon[ex-(e-1)]
       +2e(e-1)+(10-8e)x-2(e-1)(e-2)x^2.
\end{align*}
Their difference is
\[
 2(x-1)\bigl(2(e-1)x-(2e-3)\bigr)=0.
\]
Since \(x<1\), this gives
\begin{equation}\label{eq:odd-isolated-forced-coefficients}
 \begin{aligned}
 x&=\frac{2e-3}{2(e-1)},&
 \varepsilon&=\frac{4e^2-15e+12}{e(e-2)},\\
 \delta&=\frac e{e-2},&
 Q'(m)&=-\frac1{e-2},&
 \eta&=\frac1{2(e-1)}.
 \end{aligned}
\end{equation}
The nonnegative Hessian \(B_m\) preserves \(v^\perp\). By
\eqref{eq:odd-isolated-minimum}, its eigenvalues on \(v^\perp\)
are either zero or \(Q'(m)/2<0\), so they are all zero. Thus
\(\Delta s=\eta\) at the minimum.

At the maximum point of \(f\), let \(r\) be the multiplicity
of \(-e(e-1)a\) in \(B\). Taking traces gives
\[
 r=\frac{N+e-2}{e-1},\qquad \Delta s=-\frac re.
\]
On the other hand, subtracting the two values of
\(\Delta s=\alpha-Ds\) gives
\[
 \frac1{2(e-1)}+\frac re
 =D(s_0-m)=\frac{N+2e-4}{2e(e-1)}.
\]
Consequently \(r=(N+e-4)/(2(e-1))\), in contradiction with
the preceding trace formula. Hence \(s\) is constant, proving
\(S=cR^{e-1}\) with \(c>0\).

\emph{Step 4: the harmonic eikonal classification.}
Rescale once more so that \(S=e^2R^{e-1}\). Tkachev's polynomial
eikonal theorem \cite[Theorem~1]{TkachevIsoparametric14} gives, up to
sign,
\[
 f=T_j(g),\qquad e=jp,\qquad p\in\{1,2,3,4,6\},
\]
where \(g\) is the spherical restriction of a Cartan--M\"unzner
polynomial of degree \(p\). Since \(e\) is odd, \(p\in\{1,3\}\)
and that polynomial is harmonic. In particular,
\[
 |\nabla g|^2=p^2(1-g^2),\qquad
 \Delta g=-p(p+N-2)g.
\]
The Chebyshev equation
\((1-t^2)T_j''-tT_j'+j^2T_j=0\) and the chain rule give
\[
 \Delta f=-e^2T_j(g)-p(N-2)gT_j'(g).
\]
Comparing this with harmonicity of \(F\), namely
\(\Delta f=-e(e+N-2)f\), and using \(N>2\), yields
\[
 tT_j'(t)=jT_j(t)\qquad(-1<t<1).
\]
This identity would make \(T_j\) a monomial. For \(j>1\), its
coefficient of \(t^{j-2}\) is nonzero, so \(j=1\). Thus \(e=p=3\).
Cartan's classification of harmonic cubic eikonal polynomials
\cite{Cartan39}, also recovered in
\cite[Theorem~1.1]{TkachevCartan10}, gives the four stated
dimensions and models.
\end{proof}

\subsection{Spectral exclusions for the weight operator}
\label{app:nr-spectral-exclusions}

In the notation of
\cref{subsec:nr-maximal-idempotents,subsec:nr-spectral-rigidity},
$A$ is an exact Euclidean Hsiang algebra, $c$ is its maximal
unit idempotent, $E_r=\ker(L_c-r\Id)$, and
$T=p(L_c)-6P_c$, where $p(t)=4t^3+4t^2-t$.
In particular, $Z=E_{-1}$, $H=E_{-1/2}$, $W=E_{1/2}$, and
$B=E_T(1)=\R c\oplus Z\oplus H\oplus W$.
We use \cref{lem:nr-fusion-square,prop:nr-spectral-interval}.

\subsubsection{The negative Peirce spaces}

\begin{proof}[Proof of \cref{lem:nr-negative-spaces}]
\emph{(i)} For \(h\in H\), fusion gives
\(h^2=-|h|^2c/2+q(h)\), with \(q(h)\in H\).
The norm identity at \(r=-1/2\) gives
\(|q(h)|^2=|h|^4/2\). Thus
\(h\diamond k=\sqrt2\,\pi_H(hk)\) makes \(H\) a Euclidean
eikonal algebra. A positive spherical maximum for its cubic gives a
unit idempotent \(a\): the multiplier equation is
\(a\diamond a=Ma\), and \(|a\diamond a|=1\) forces \(M=1\).
In the original algebra,
\[
 a^2=-\frac12c+\frac1{\sqrt2}a,\qquad
 e=-\frac13c+\frac{2\sqrt2}{3}a
\]
gives a unit idempotent \(e\) of \(A\).
\Cref{lem:nr-idempotent-quotient}, applied to \(c\) and \(e\),
yields \([T,L_a]=0\).

Fix \(U_m=E_T(m)\) with \(m\ne1\).
A nonzero block of \(L_a|_{U_m}\) from \(E_r\) to \(E_s\) requires
\[
 4(r^2+rs+s^2)-3=0,\qquad p(r)=p(s)=m.
\]
If \(r=s\), then \(r=\pm1/2\), giving \(m=1\).
If \(r\ne s\), the divided difference of \(p\), together with
this equality, gives \(r+s=-1/2\). The third root of
\(p(t)-m\) is then \(-1/2\), again giving \(m=1\).
Thus \(L_a|_{U_m}=0\). For \(u\in U_m\cap E_r\), it follows
that \(L_eu=-ru/3\). Since \(u\perp e\), the quotient formula
at \(e\) gives
\[
 0=p(r)-p(-r/3)=\frac4{27}r(28r^2+24r-9).
\]
The quadratic factor has one root below \(-1\) and the other in
\((0,1/3)\), both excluded by \eqref{eq:nr-Peirce-interval}.
Hence \(r=m=0\).

\emph{(ii)} For a unit \(z\in Z\), fusion and invariance give \(z^2=-c\)
and \(cz=-z\). The vectors \(e_\pm=(-c\pm\sqrt3z)/2\)
are unit idempotents.
Together with \(c\), they form a triality system in the terminology of
\cite[Definition~6.8.1]{NadirashviliTkachevVladut14}.
Thus \cref{lem:nr-idempotent-quotient} gives \([T,L_z]=0\).
Fix \(U_m=E_T(m)\), \(m\ne1\), and put
\(X=L_c|_{U_m}\), \(Y=L_z|_{U_m}\).
For a nonzero block of \(Y\) between distinct roots \(r,s\) of
\(p(t)=m\), fusion and the divided difference of \(p\) require
\[
 4(r^2+rs+s^2)-2(r+s)-1=0,\qquad
 4(r^2+rs+s^2)+4(r+s)-1=0.
\]
Thus \(r+s=0\) and \(\{r,s\}=\{1/2,-1/2\}\), which would
give \(m=1\). Hence \(X,Y\) commute and have a common
orthonormal eigenbasis.

On a common eigenvector with eigenvalues \(r,y\), the quotient
formula at \(c,e_+,e_-\) gives
\[
 p(r)=p\left(\frac{-r+\sqrt3y}{2}\right)
     =p\left(\frac{-r-\sqrt3y}{2}\right)=m.
\]
These three arguments sum to zero, whereas the three roots of
\(p(t)-m\) sum to \(-1\). Therefore two of the arguments coincide.
If \(y=0\), then
\[
 0=p(r)-p(-r/2)=\frac32r(3r-1)(r+1),
\]
so \(r=0\) or \(r=1/3\); the root \(-1\) would give \(m=1\).
If \(y\ne0\), one of the other arguments equals \(r\), and the
remaining argument is \(-2r\). Thus \(y^2=3r^2\) and
\[
 0=p(r)-p(-2r)=3r(6r+1)(2r-1).
\]
Here \(r\ne0\), and \(r=1/2\) again gives \(m=1\). Hence
\(r=-1/6\), \(m=7/27\), and \(y^2=1/12\).
This proves the spectral assertion and the actions on \(P,Q\).
Scaling \(z\) proves \eqref{eq:nr-U-peirce}.
\end{proof}

\subsubsection{Excluding a one-dimensional weight-one space}

\begin{proof}[Proof of \cref{lem:nr-not-one-dimensional}]
Suppose \(E_T(1)=\R c\). On $A_{\mathrm{ess}}$ the Peirce
eigenvalues \(-1,-1/2,1/2,0\) are absent. Fix another eigenvalue \(r\ne1/3\).
The right side of \eqref{eq:nr-square-norm} is nonzero for every
nonzero \(u\in E_r\), so \(u^2\) has a nonzero component in
\(c^\perp\).
There is at most one root of \(Q_r\) in \([-1,1/2]\).
Indeed, if \(r\le-3/4\), then \(Q_r\) is strictly decreasing
there. If \(-3/4<r\le-1/6\), then \(Q_r(1/2)<0\),
so \(1/2\) lies strictly between its two roots. If
\(1/3<r<1/2\), then \(Q_r(-1)<0\), and \(-1\) lies
strictly between the roots. These cases cover every remaining \(r\).
The nonzero square component must therefore belong to a single
eigenspace \(E_s\). This root is simple: in the decreasing case,
\(s<1/2\) gives \(Q'_r(s)<0\), and the other cases have two
distinct roots. Thus the other root lies outside the interval.
In particular, the endpoint \(r=r_0\)
cannot occur: its unique root in the interval is \(1/2\), which is
absent. We have
\begin{equation}\label{eq:nr-unique-square-output}
 E_r^2\subset\R c\oplus E_s,\qquad
 u^2=r|u|^2c+q(u),\qquad 0\ne q(u)\in E_s.
\end{equation}
The equality \(r=s\) would imply
\(\mathscr K(r,r,r)=3(4r-1)(2r+1)=0\), contrary to
\eqref{eq:nr-Peirce-interval} and the absence of \(E_{-1/2}\).

Let \(d=-2r-s-1/2\) be the other root of \(Q_r\).
Symmetry of \(\mathscr K\) gives
\[
 \mathscr K(r,r,t)=4(t-s)(t-d),\qquad
 \mathscr K(r,s,t)=4(t-r)(t-d).
\]
Since \(d\) is inadmissible and \(r\ne s\), fusion and
orthogonality imply \(E_rE_s\subset E_r\). Hence
\[
 \ip{u^2}{u}=0,\qquad u^3\in E_r,\qquad
 u^4\in\R c\oplus E_s.
\]
The full vector identity, evaluated at \(u\), now implies
\((u^2)^2\in\R c\oplus E_s\).
Expanding \eqref{eq:nr-unique-square-output} shows that
\(q(u)^2\in\R c\oplus E_s\).
But \(\mathscr K(s,s,s)\ne0\), since \(s=-1/2,1/4\) are
excluded. Its \(E_s\) component vanishes, and therefore
\(q(u)^2\in\R c\).
Applying \eqref{eq:nr-square-norm} to this nonzero vector gives
\(s^2(3s-1)(s+1)=0\). Since \(s=0,-1\) are absent, \(s=1/3\).

Every eigenvalue \(r\ne1/3\) thus satisfies
\[
 \mathscr K(r,r,1/3)=0
 \quad\Longleftrightarrow\quad 108r^2+60r-17=0.
\]
The positive root \((-5+2\sqrt{19})/18\) belongs to
\((0,1/3)\) and is excluded. Only
\(r_*=(-5-2\sqrt{19})/18\) remains. If
\(a=\dim E_{r_*}\) and \(b=\dim E_{1/3}\), exactness gives
\[
 0=\tr L_c=1+\frac{-5-2\sqrt{19}}{18}a+\frac13b.
\]
For \(a=0\) this is positive, while for \(a>0\) its irrational
part cannot cancel. Both cases are impossible.
\end{proof}

\subsection{The trace argument for exactness of the core}
\label{app:nr-core-exactness-proof}

Here $A=B\oplus U$ is the nonradial grading of
\cref{prop:nr-eikonal-core}, $B$ is eikonal, and $M_b=L_b|_U$.
We use the representation identities of
\cref{prop:nr-representation-identities} and the square-norm identity
$|u^2|^2=|u|^4/9$ from \cref{lem:nr-U-cubic}.

\begin{proof}[Proof of \cref{prop:nr-core-exactness}]
Put $n=\dim B$, $m=\dim U$, and let
$\tau=\sum_{\alpha=1}^n e_\alpha^2$ for an orthonormal basis
$(e_\alpha)$ of $B$.
Thus $\ip\tau b=\tr_B L_b$.
Exactness of $A$ and the grading give
\begin{equation}\label{eq:nr-core-trace-vector}
 \tr_U M_b=-\ip\tau b,\qquad
 \Delta_U(u^2)=-2\tau.
\end{equation}
Summing \eqref{eq:nr-representation-quadratic} over the basis yields
\[
 K:=\sum_{\alpha=1}^n M_{e_\alpha}^2
     =\frac n{18}\Id_U+\frac16M_\tau.
\]
Apply $\Delta_U$ to $|u^2|^2=|u|^4/9$. Using
\eqref{eq:nr-core-trace-vector}, this gives
\[
 8\ip{Ku}u-4\ip\tau{u^2}
       =\frac{4(m+2)}9|u|^2.
\]
Consequently
\begin{equation}\label{eq:nr-core-scalar-trace-action}
 M_\tau=\kappa\Id_U,\qquad
 \kappa=\frac{n-m-2}{6},\qquad
 m\kappa=-|\tau|^2.
\end{equation}
The last equality is its trace, using
\eqref{eq:nr-core-trace-vector}.

Suppose $\tau\ne0$, so $\kappa<0$.
The coefficient of $t^2$ in \eqref{eq:nr-representation-cubic}
at $b+t\tau$ is
\begin{equation}\label{eq:nr-core-trace-dichotomy}
 3\bigl(36\kappa^2-|\tau|^2\bigr)M_b
      =\ip{\tau^2+6\kappa\tau}b\Id_U.
\end{equation}
If $36\kappa^2\ne|\tau|^2$, every $M_b$ is scalar:
$M_b=\ell(b)\Id_U$ for a linear form $\ell$,
which is nonzero because $\ell(\tau)=\kappa$.
The cubic representation identity then gives, with $r=|b|^2$ and
$s=|u|^2$,
\[
 \ip{b^2}b=108\ell(b)^3-9r\ell(b),\qquad
 \ip{bu}u=\ell(b)s.
\]
The grading therefore factors the original cubic as
\[
 F(b+u)=\frac{\ell(b)}2\bigl(36\ell(b)^2-3r+s\bigr),
\]
contrary to irreducibility.

It remains that $36\kappa^2=|\tau|^2$.
From \cref{eq:nr-core-scalar-trace-action,eq:nr-core-trace-dichotomy},
we obtain
\[
 \kappa=-\frac m{36},\qquad
 \tau^2=-6\kappa\tau.
\]
Set $k=-6\kappa=m/6>0$ and $e=\tau/k$.
Then $e$ is a unit idempotent of $B$ and $M_e=-\Id_U/6$.
Linearizing the eikonal identity at $e$ gives the orthogonal
decomposition
\[
 B=\R e\oplus Z'\oplus W',\qquad
 L_e|_{Z'}=-\Id,\qquad L_e|_{W'}=\tfrac12\Id.
\]
For $\eta\in W'$, the polarized quadratic identity gives
$M_\eta/2=3\{M_e,M_\eta\}=-M_\eta$, hence $M_\eta=0$.
Taking the trace of the quadratic identity at $\eta$ now gives
\[
 -\frac m3|\eta|^2
 =\tr_U M_{\eta^2}
 =-\ip\tau{\eta^2}
 =-\frac k2|\eta|^2
 =-\frac m{12}|\eta|^2.
\]
Thus $W'=0$. Since $\dim B\ge2$ by $Z\ne0$ in the original
Peirce decomposition, $\dim Z'\ge1$. But then
\[
 0<k=\ip\tau e=\tr_B L_e=1-\dim Z'\le0,
\]
a contradiction.
Therefore $\tau=0$, proving both trace assertions;
\eqref{eq:nr-core-scalar-trace-action} then gives $n=m+2$.
\end{proof}

\subsection{Local geometry near the cubic critical set}
\label{sec:cubic-local-geometry}

The local geometry in \cref{prop:cubic-singular-normalization} and
\eqref{eq:tube} follows from \eqref{eq:cubic-normalized-closure}
and the Hessian eigenvalues established in that proposition.

\subsubsection{The critical set}
Fix \(p\in Z^*\).  Translate \(p\) to the coordinate origin and
write orthonormal coordinates as \((x,y)\), with
\(x\in\R^{2k}\) in the nonzero Hessian eigenspaces and
\(y\in\R^{N-2k}\) in its kernel.  In these translated coordinates,
\(R\) is still the original squared radius, so its constant term
is \(|p|^2>0\).  The matrix \(F_{xx}(0,0)\) is invertible, while
\(F_{xy}(0,0)=F_{yy}(0,0)=0\).

The analytic implicit function theorem solves \(F_x=0\) locally
and uniquely as \(x=\psi(y)\), with
\(\psi(0)=0\) and \(\psi'(0)=0\).  Set
\(g(y)=F(\psi(y),y)\).  Since \(F(p)=0\), \(\nabla F(p)=0\),
and \(F_{yy}(p)=0\), the function \(g\) vanishes to order at
least three.  We claim that it vanishes identically.

Suppose instead that its first nonzero homogeneous term is \(g_j\),
of degree \(j\geq3\).  On the graph of \(\psi\), the identity
\(F_x=0\) gives
\[
 \nabla F=(0,\nabla g),\qquad S=|\nabla g|^2,\qquad F=g.
\]
Differentiating \(F_x(\psi(y),y)=0\) and \(g_y=F_y\) gives
\[
 F_{xy}=-F_{xx}\psi',\qquad
 F_{yy}=g_{yy}+(\psi')^tF_{xx}\psi'.
\]
Consequently,
\[
 F_{xy}=O(|y|),\qquad
 F_{yy}=O(|y|^{j-2})+O(|y|^2),\qquad
 \nabla g=O(|y|^{j-1}).
\]
Since \(j\geq3\), these estimates imply
\(\nabla S=2\Hess F\nabla F=O(|y|^j)\) on the graph.
The last identity in \eqref{eq:cubic-normalized-closure} therefore
becomes
\[
 |\nabla S|^2=4R|\nabla g|^2+12g^2.
\]
Its left side vanishes to order at least \(2j\), but its right
side has the nonzero term
\(4|p|^2|\nabla g_j|^2\) of degree \(2j-2\).
This contradiction proves \(g\equiv0\), by analyticity.

Every point of the graph \(x=\psi(y)\) is now critical, and every
nearby critical point lies on that graph.  Thus \(Z^*\) is
real analytic of dimension \(N-2k\).  Its tangent space is the
Hessian kernel: differentiating \(\nabla F=0\) along the graph
gives inclusion in that kernel, and the dimensions agree.

\subsubsection{Local coordinates and the link volume}
Subtract the graph by setting \(x=\psi(y)+\xi\).  The resulting
function \(\widetilde F(\xi,y)\) and its first \(\xi\)-derivatives
vanish at \(\xi=0\).  Taylor's integral formula gives
\[
 \widetilde F(\xi,y)=\xi^t A(\xi,y)\xi,\qquad
 A(\xi,y)=\int_0^1(1-t)\,
 F_{xx}(\psi(y)+t\xi,y)\,dt.
\]
The matrix \(A\) is analytic and symmetric.  After shrinking the
neighborhood it is invertible with \(k\) positive and \(k\)
negative eigenvalues.  Successive completion of squares gives
an analytic invertible matrix \(B(\xi,y)\) such that
\[
 A=B^tDB,\qquad D=\operatorname{diag}(\Id_k,-\Id_k).
\]
Indeed, at the coordinate origin \(A\) is diagonal with nonzero
entries.  All pivots remain nonzero on a sufficiently small
neighborhood, and the positive square roots used to normalize their
absolute values are analytic there.  The map
\((\xi,y)\mapsto(B(\xi,y)\xi,y)\) has invertible derivative at
the origin.  The analytic inverse function theorem therefore gives
coordinates \((u,v,y)\) in which
\[
 F=|u|^2-|v|^2,\qquad Z^*=\{u=v=0\}.
\]
The coordinate derivatives and inverse derivatives are bounded on smaller
neighborhoods, so Euclidean lengths and submanifold areas are
comparable to those in the coordinate model.

The radial vector lies in \(T_pZ^*\), because \(Z^*\) is a cone.
It is transverse to the unit sphere.  Hence
\(M=Z^*\cap\Sph^{N-1}\) is a smooth compact submanifold of
dimension \(N-2k-1\).  On the sphere, \(f=F|_{\Sph^{N-1}}\)
and its first derivatives vanish along \(M\).  Its Hessian has
the same \(2k\) nonzero transverse eigenvalues as the ambient
Hessian; the removed radial direction belongs to the kernel.
Applying the preceding Taylor formula and completion of squares
in a sphere chart gives local coordinates \((u,v,y)\), now with
\(y\in\R^{N-2k-1}\), for which
\[
 f=|u|^2-|v|^2,\qquad M=\{u=v=0\}.
\]

Put \(t=(|u|^2+|v|^2)^{1/2}\).  The smooth positive-definite
spherical metric and the bounded coordinate derivatives give
\(\rho\asymp t\), where \(\rho\) is spherical distance to \(M\).
They also give \(|\nabla_{\Sph}f|^2\asymp t^2\).  Since
\(s=|\nabla_{\Sph}f|^2+9f^2\) and \(|f|\leq t^2\), it follows
that \(s\asymp\rho^2\).

In these coordinates the regular zero link is the product, in a
bounded chart, of the \(y\)-variables with
\[
 \{(u,v):|u|=|v|>0\}\subset\R^{2k}.
\]
This quadratic cone has dimension \(2k-1\).  Its link is the
compact product of two spheres of radius \(1/\sqrt2\), and its
area inside a ball of radius \(a\) is a constant times
\(a^{2k-1}\).  Bounded coordinate distortion thus bounds the
area in each chart with \(\rho<\eps\) by
\(C\eps^{2k-1}\).  A finite cover of \(M\) gives
\[
 \cH^{N-2}\bigl(\Sigma_{\reg}\cap\{\rho<\eps\}\bigr)
 \leq C\eps^{2k-1},
\]
which completes the proof of \eqref{eq:tube}.

\section[Positivity and Obstruction Certificates]{Exact Positivity and Obstruction Certificates}
\label{app:exact-certificates}

\subsection{Exact positivity certificates}
\label{subsec:legacy-exact-certificates}

Our use of Bernstein coefficients to certify polynomial nonnegativity
was inspired by \cite[Section~6.3 and Appendix~A]{LienTsai26}.
For a polynomial \(P\) of degree at most \(m\) on \([a,b]\), put
\(z=a+(b-a)t\) and write \(P(a+(b-a)t)=\sum_{j=0}^m c_jt^j\).
Its Bernstein expansion is
\begin{equation}\label{eq:Bernstein-expansion}
 P(a+(b-a)t)=\sum_{i=0}^m b_i\binom mi t^i(1-t)^{m-i},
\end{equation}
where
\begin{equation}\label{eq:Bernstein-coefficients}
 b_i=\sum_{j=0}^i c_j\frac{\binom ij}{\binom mj}.
\end{equation}
Indeed,
\[
 \sum_{i=j}^m\frac{\binom ij}{\binom mj}\binom mi t^i(1-t)^{m-i}
 =t^j\sum_{i=j}^m\binom{m-j}{i-j}t^{i-j}(1-t)^{m-i}=t^j.
\]
The Bernstein basis functions are nonnegative and sum to one on
\([0,1]\), so \(P\geq\min_i b_i\) on \([a,b]\).

\begin{example}[An exact check for the \((4,1,7)\) correction]
\label{ex:Bernstein-17}
For the polynomial \(P\) in \eqref{eq:P-17}, \eqref{eq:Bernstein-coefficients}, at degree five on each interval,
gives the following minimum coefficients:
\[
\begin{array}{c|ccccc}
 \text{interval}
 &[-7/4,-1]&[-1,-1/2]&[-1/2,0]&[0,1/8]&[1/8,1/4]\\ \hline
 \min b_i
 &372925/1024&33735/32&42&10973/1280&1099085/32768
\end{array}
\]
All five numbers are positive.  These intervals cover the entire
realizable interval, so this proves the positivity assertion in
\cref{prop:17-correction}.
\end{example}

The accompanying file \nolinkurl{LSC-INTERVAL-POSITIVITY.wl} checks these
positivity certificates with exact rational arithmetic. The related
isoparametric identities are checked in \nolinkurl{LSC-ISOPARAMETRIC.wl}.

\subsection{The Clifford power-multiplier obstruction}
\label{app:Clifford-power-obstruction}

Use the Clifford closure variables from \cref{sec:invariants}, with
\((m,q)=(4,3)\), \(J=|\nabla S|^2\), and \(K=\Delta S=64R\).
The instability theorem already rules out these multipliers;
the calculation below supplies the explicit wrong-sign points described
in \cref{prop:Clifford-power-obstruction}.

\begin{proof}[Construction of the wrong-sign points]
Put \(a=R^\alpha S^\gamma\). The multiplier identity in
\cref{prop:combined-power-multiplier}, together with
\cref{prop:first-order}, gives
\[
 \cL(G_{\alpha,\gamma})=a^3FQ_{\alpha,\gamma},
\]
where \(Q_{\alpha,\gamma}\) extends smoothly across \(F=0\) on
\(\{R>0,S>0\}\). We need only its restriction to
\(R=1\), \(F=0\), \(T=u^2\). Writing \(u=\lambda\),
\(v=1-\lambda\), set
\[
 \begin{aligned}
 s_\lambda&=\lambda(4-3\lambda),\\
 j_\lambda&=16\bigl(4\lambda(1-\lambda)
                      +(4-3\lambda)\lambda^2\bigr).
 \end{aligned}
\]
These are respectively \(S,J\) at the stated boundary values.
The closure table gives
\(\Gamma(F,\Gamma(F,S))=16RS+96F^2\).
Using this identity and \(K=\Delta S=64\) in
\eqref{eq:gradient-power-Q}--\eqref{eq:combined-power-Q} gives
\begin{equation}\label{eq:Clifford-boundary-coefficient}
 \begin{aligned}
 B_{\alpha,\gamma}(\lambda)
 &=48\gamma-8-\gamma(\gamma+\tfrac32)\frac{j_\lambda}{s_\lambda}\\
 &\quad+2\alpha(5-8\gamma-2\alpha)s_\lambda.
 \end{aligned}
\end{equation}
In particular,
\begin{align}
 \lim_{\lambda\downarrow0}B_{\alpha,\gamma}(\lambda)
 &=-8(\gamma-1)(2\gamma-1),
 \label{eq:obstruction-lambda0}\\
 B_{\alpha,\gamma}\left(\frac23\right)
 &=-\frac83\left[
 2\alpha^2+(8\gamma-5)\alpha+8\gamma^2-6\gamma+3
 \right].
 \label{eq:obstruction-lambda23}
\end{align}
If the required sign \(Q_{\alpha,\gamma}\geq0\) held, then
\eqref{eq:obstruction-lambda0} would force
\begin{equation}\label{eq:gamma-window}
 \frac12\leq\gamma\leq1.
\end{equation}
But the bracket in \eqref{eq:obstruction-lambda23} satisfies
\begin{align}
 &2\alpha^2+(8\gamma-5)\alpha+8\gamma^2-6\gamma+3\notag\\
 &\qquad
 =2\left(\alpha+\frac{8\gamma-5}{4}\right)^2
  +4\gamma-\frac18
 \geq\frac{15}{8}
 \quad\text{under \eqref{eq:gamma-window}}.
 \label{eq:obstruction-square}
\end{align}
Therefore the value in \eqref{eq:obstruction-lambda23} is at most
$-5$.

To realize and perturb these boundary values, put the first two matrices in
\eqref{eq:standard-first-two}.  For prescribed
$\lambda\in(0,1)$, set $u=\lambda$, $v=1-\lambda$, choose
$x\in\R^4$ with $|x|^2=u/2$, and let
\begin{equation}\label{eq:actual-y}
 y=(x,x).
\end{equation}
Then
\[
 \tau(y)=(0,u,0,\ldots,0),\qquad T=u^2.
\]
For $0<\eps<\sqrt v$, set
\[
 z_\eps=(\sqrt{v-\eps^2},\eps,0,\ldots,0).
\]
Now \(R=1\), \(T=u^2\), and
\begin{equation}\label{eq:actual-omega}
 F(y,z_\eps)=u\eps\longrightarrow0,
 \qquad
 Q_{\alpha,\gamma}(y,z_\eps)
   \longrightarrow B_{\alpha,\gamma}(\lambda).
\end{equation}
For fixed \(\alpha,\gamma\), the product rule on
\(\{F=0,R>0,S>0\}\) gives
\[
 |\nabla(R^\alpha S^\gamma F)|^2
 =R^{2\alpha}S^{2\gamma+1}>0.
\]
Hence the gradient remains nonzero for sufficiently small \(\eps\).
If \(\gamma\notin[1/2,1]\), first fix a sufficiently small
\(\lambda>0\) using \eqref{eq:obstruction-lambda0}; otherwise take
\(\lambda=2/3\).  The function \(Q_{\alpha,\gamma}\) remains negative
for all sufficiently small \(\eps>0\),
so the normalized gradient is defined and has the wrong divergence sign
at the nearby points.
\end{proof}

\subsection{The cubic certificate without harmonicity}
\label{app:nonharmonic-cubic-certificate}

For a cubic \(F\), put \(R=|x|^2\), \(S=|\nabla F|^2\).
We specialize \cref{app:power-identities} to \(G=F\sqrt{RS}\),
retaining the quadratic weight and trace terms.

For \(\cL(F)=-qF\), put
\[
 t=\Delta F,\qquad H_2=|\Hess F|^2,\qquad K=\Delta S,
 \qquad T=\Gamma(S,S).
\]
Then
\begin{equation}\label{eq:cubic-nonharmonic-basic}
 \Gamma(F,S)=2St+2qF,\qquad K=2H_2+2\Gamma(F,t).
\end{equation}
In the associated metrized algebra this means
\[
 t(x)=\tr L_x,\qquad H_2(x)=\tr(L_x^2),\qquad
 K(x)=2\tr(L_x^2)+\langle x^2,b\rangle,
 \quad \langle b,x\rangle=\tr L_x.
\]
The fixed vector \(b\) is the trace vector.

\paragraph*{Polynomial coefficients}

Define
\begin{align}
 B_0={}&16S^2\bigl[S(H_2-2q+3t^2)-T\bigr],
 \label{eq:cubic-nonharmonic-B0}\\
 B_1={}&-16S^2\Gamma(F,q)-8S^2\Gamma(S,t)+16KS^2t\notag\\
       &+80S^2qt-4STt-2S\Gamma(F,T),
 \label{eq:cubic-nonharmonic-B1}\\
 B_2={}&-8S\Gamma(S,q)+16KqS+2KT+32q^2S-\Gamma(S,T).
 \label{eq:cubic-nonharmonic-B2}
\end{align}
Then put
\begin{equation}\label{eq:cubic-nonharmonic-coefficients}
\begin{aligned}
 C_0={}&R^2B_0+16(N-12)RS^4,\\
 C_1={}&R^2B_1+16(2N+9)RS^3t,\\
 C_2={}&R^2B_2
       +R\bigl[88KS^2+32(N-1)qS^2+4(N-23)ST\bigr]\\
       &+176(N+4)S^3.
\end{aligned}
\end{equation}
The degrees of \((B_0,B_1,B_2)\) are \((14,11,8)\), and those
of \((C_0,C_1,C_2)\) are \((18,15,12)\). In particular, all terms
of \(C_0+FC_1+F^2C_2\) have degree eighteen.

\paragraph*{The conditional certificate}

\begin{proposition}[A polynomial certificate without harmonicity]
\label{thm:cubic-nonharmonic-certificate}
Let \(F\) be a nonzero square-free homogeneous cubic on \(\R^N\),
\(N\ge2\), satisfying \(\cL(F)=-qF\) for a homogeneous quadratic
polynomial \(q\). Let \(C_0,C_1,C_2\) be the explicit homogeneous
polynomials in \eqref{eq:cubic-nonharmonic-coefficients}. If
\begin{equation}\label{eq:cubic-nonharmonic-PSD}
 C_0\ge0,\qquad C_2\ge0,\qquad C_1^2\le4C_0C_2
\end{equation}
at every point, then \(G=F\sqrt{RS}\) generates a two-sided
subcalibration off \(Z=\{\nabla F=0\}\). Consequently
\(E=\{F<0\}\) is locally perimeter minimizing and
\(\partial\llbracket E\rrbracket\) is locally mass minimizing.
No harmonicity or radiality assumption is required.
\end{proposition}

\begin{proof}
The derivation below gives
\begin{equation}\label{eq:cubic-nonharmonic-defect}
 \cL(G)=\frac{F}{16\sqrt R\,S^{3/2}}
                   (C_0+FC_1+F^2C_2)\qquad(S>0).
\end{equation}
The positive semidefinite matrix in \eqref{eq:cubic-nonharmonic-PSD},
evaluated on \((1,F)\), gives \(F\cL(G)\ge0\).
The potential is continuous with the phases of \(F\); its positive
multiplier and degree six give \(\nabla G\ne0\) off \(Z\) by
\cref{cor:gradient-subcalibration}. Square-freeness makes \(Z\)
removable by \cref{lem:square-free-critical-set}, so
\cref{thm:removable-subcalibration} applies.
\end{proof}

The leading coefficient is
\begin{equation}\label{eq:cubic-nonharmonic-leading-H}
 C_0=16RS^2\mathscr H_{\rm gen},\qquad
 \mathscr H_{\rm gen}
 =(N-12)S^2+RS(H_2-2q+3t^2)-RT.
\end{equation}
When \(t=0\), \(q=2R\), and \(T=4RS+12D\), this is precisely
\(\mathscr H\) in \eqref{eq:cubic-unified-H}. For harmonic models,
the transport identities of \cref{thm:cubic-unified-sextic} give sharper
conditions than \eqref{eq:cubic-nonharmonic-PSD}.

The minimizing current has support \(\{F=0\}\) if both phases have
positive volume in every ball centered at a zero; square-freeness
alone does not ensure this.

\paragraph*{Derivation of the defect}

First let \(H=F\sqrt S\), of degree five. Product differentiation gives
\begin{align}
 Q_H:=|\nabla H|^2
   &=S^2+2FSt+F^2\left(2q+\frac{T}{4S}\right),
 \label{eq:cubic-nonharmonic-H-gradient}\\
 \Delta H
   &=3\sqrt S\,t+
     \frac{F}{\sqrt S}\left(2q+\frac K2-\frac{T}{4S}\right).
 \label{eq:cubic-nonharmonic-H-laplacian}
\end{align}
To compute \(\cL(H)\), regard
\(Q_H\) as a function of \(F,S,t,q,T\). Its partial derivatives are
\[
\begin{aligned}
 (Q_H)_F&=2St+4qF+\frac{FT}{2S},&
 (Q_H)_S&=2S+2Ft-\frac{F^2T}{4S^2},\\
 (Q_H)_t&=2FS,&
 (Q_H)_q&=2F^2,&
 (Q_H)_T&=\frac{F^2}{4S}.
\end{aligned}
\]
Use the chain rule with
\(\Gamma(F,F)=S\), \(\Gamma(F,S)=2St+2qF\),
\(\Gamma(S,S)=T\), and
\(\Gamma(F,t)=K/2-H_2\). Since
\[
 \Gamma(H,Q_H)=\sqrt S\,\Gamma(F,Q_H)
                    +\frac{F}{2\sqrt S}\Gamma(S,Q_H),
\]
substitution into \(\cL(H)=Q_H\Delta H-\Gamma(H,Q_H)/2\) and
collection of the displayed powers of \(F\) yield
\begin{equation}\label{eq:cubic-nonharmonic-H-defect}
 \cL(H)=\frac{F}{16S^{3/2}}(B_0+FB_1+F^2B_2).
\end{equation}
The term independent of these displayed factors uses the Bochner
relation \(K/2-\Gamma(F,t)=H_2\); this accounts for the \(H_2\) term
in \eqref{eq:cubic-nonharmonic-B0}.

Finally \(G=\sqrt R\,H\). The radial identity
\eqref{eq:cubic-radial-lift}, with \(d=5\), gives
\[
 \cL(G)=R^{3/2}\left[
 \cL(H)+\frac{11H^2\Delta H}{R}
       +\frac{(N-12)H|\nabla H|^2}{R}
       +\frac{11(N+4)H^3}{R^2}\right].
\]
Inserting \eqref{eq:cubic-nonharmonic-H-gradient}--\eqref{eq:cubic-nonharmonic-H-defect}
and factoring out \(F/(16\sqrt R\,S^{3/2})\) gives
\eqref{eq:cubic-nonharmonic-coefficients} and
\eqref{eq:cubic-nonharmonic-defect}, with polynomial extension across
the zeros of \(F\).

\subsection{Quartic potential identities}
\label{app:quartic-potential-details}

For the normalization of \cref{ex:Riedler-correct-power}, take
$G=RQ^{3/4}F$ and use the differential table
\eqref{eq:quartic-ambient-closure}.
The quotient chain rule gives
\begin{align}
 |\nabla G|^2
 &=RQ^{1/2}\bigl(16R^2Q^2+33R^2F^2+32QF^2\bigr),
 \label{eq:Riedler-correct-gradient}\\
 \frac{\cL(G)}{F}
 &=RQ^{1/4}\Bigl[
 16(12k-47)R^4Q^2+16(2N-41)R^2Q^3\notag\\
 &\qquad
 +F^2\bigl((396k+66)R^4+(384k+66N-681)R^2Q\notag\\
 &\hspace{35mm} +(64N+160)Q^2\bigr)\Bigr].
 \label{eq:Riedler-correct-defect}
\end{align}
Every coefficient is positive for $N\geq32$ and $k\geq4$.
Equivalently, apply \cref{prop:combined-power-multiplier} with
\eqref{eq:Riedler-correct-potential}.

On the regular zero set, these formulas yield
\eqref{eq:Riedler-positive-boundary}.

\paragraph*{The $(3,7,7)$ specialization}
The differential table and certificate used for
\cref{thm:weighted-Clifford-quartic} are
\cref{eq:balanced-quartic-table,eq:balanced-quartic-certificate}
at $N=32$ and $\nu=7$:
\begin{equation}\label{eq:weighted-Clifford-table}
 \begin{aligned}
 \Delta(R,W,F)&=(64,28R,0),&\quad \Gamma(F,W)&=4RF,\\
 \Gamma(F,F)&=\Gamma(W,W)=4RW.
 \end{aligned}
\end{equation}
\begin{equation}\label{eq:weighted-Clifford-certificate}
 \cL(WF)=16WF\left[(3R^2-2W)W^2+(21R^2-6W)F^2\right].
\end{equation}
The certificate file is
\nolinkurl{LSC-WEIGHTED-CLIFFORD-QUARTIC.wl}.

\subsubsection{Trace reductions for rectangular quartics}
\label{app:rectangular-trace-details}

We complete the trace calculations for
\cref{prop:rectangular-minimal-selection}, using
$B=X^{\mathsf T}X$, $P_j=\tr(B^j)$, and $R=P_1$.
In \eqref{eq:rectangular-trace-differentiation}, summing coordinate
derivatives of $XB$ gives $(p+3+1)P_1$. For $XB^2$, the derivative
of the first factor contributes $pP_2$; the two derivatives of the
first $B$ contribute $P_2$ and $P_1^2$, and those of the second
contribute $P_2$ and $3P_2$. Thus
\[
 \Delta P_2=4(p+4)P_1,\qquad
 \Delta P_3=6\bigl((p+5)P_2+P_1^2\bigr).
\]
These identities hold without any regular-eigenvalue assumption.

For the three eigenvalues of $B$, put
\[
 e_2=(R^2-P_2)/2,\qquad
 e_3=(P_3-RP_2+Re_2)/3.
\]
Their characteristic equation gives
$P_j=RP_{j-1}-e_2P_{j-2}+e_3P_{j-3}$ for $j\ge4$.
Together with \eqref{eq:rectangular-S-traces}, this yields
\begin{align*}
 P_2&=(F+R^2)/2,\\
 P_3&=(S+32RF+16R^3)/64,\\
 P_4&=RS/48+5R^2F/12+R^4/8+F^2/8,\\
 P_5&=5R^2S/256+5R^3F/16+R^5/16
                 +5RF^2/24+5FS/768.
\end{align*}
Differentiation of $S$ gives
\[
 \nabla S=X\bigl(384B^2-256RB+(32R^2-64F)\Id\bigr).
\]
Pairing this with $\nabla F=4X(2B-R\Id)$ and with itself, and then
substituting the displayed power traces, proves the last two entries
of \eqref{eq:rectangular-quartic-gamma}.

\subsection{Polynomial cubic alternatives and pullback calculations}
\label{app:application-details}
\label{app:Cartan-source-certificate}

We detail the polynomial potentials of
\cref{sec:cubic-algebraic-existence}, Cartan pullbacks of
\cref{thm:cartan-pullback}, and unequal Lawson pullback of
\cref{ex:lawson-pullback}.

\subsubsection*{The degree-seven defect}
\begin{proof}[Proof of \cref{lem:cubic-SF-defect}]
The product rule gives \eqref{eq:cubic-SF-general-gradient}.
\Cref{lem:cubic-radial-transports} and Euler's identities give
\[
 \begin{aligned}
 \Gamma(G,R)&=14SF,& \Gamma(G,F)&=S^2+4RF^2,\\
 \Gamma(G,S)&=8RSF+12DF,& \Gamma(G,D)&=2S^2F+8RDF.
 \end{aligned}
\]
Differentiating the displayed polynomial \(Q\) in \(R,F,S,D\)
and substituting in \(\cL(G)=Q\Delta G-\Gamma(G,Q)/2\)
proves \eqref{eq:cubic-SF-general-defect}.
\end{proof}

\subsubsection*{The positive decomposition for the polynomial certificate}
Under the hypotheses of \cref{thm:cubic-algebraic-trace-certificate},
the multiplication dictionary and Bochner identity give
\begin{equation}\label{eq:cubic-gradient-identities-only}
 \Delta F=0,\qquad
 \Gamma(F,S)=4RF,\qquad
 \Gamma(S,S)=4RS+12F^2,
\end{equation}
and
\begin{equation}\label{eq:cubic-variable-trace}
 K:=\Delta S=2\tr(L_x^2)\ge16R.
\end{equation}
The trace bound gives $F\not\equiv0$, and
\cref{lem:cubic-unified-bounds} gives
\begin{equation}\label{eq:cubic-gradient-only-S-bound}
 0\le S\le R^2/3.
\end{equation}

Set \(D=F^2\). Then
\(\Gamma(F,D)=2SF\) and \(\Gamma(S,D)=8RD\), so
\cref{lem:cubic-SF-defect} gives
\begin{equation}\label{eq:cubic-variable-trace-gradient}
 Q:=|\nabla G|^2=S^3+12RSF^2+12F^4,
 \qquad \Delta G=(K+8R)F.
\end{equation}
The defect \eqref{eq:cubic-SF-general-defect} specializes to
\begin{align}
 \frac{\cL(G)}F
 &=(K-16R)S^3+F^2S(12RK-126S)
                       +12F^4(K-6R)\nonumber\\
 &=(K-16R)Q+150R^2SF^2
       +126F^2S\left(\frac{R^2}3-S\right)+120RF^4.
 \label{eq:cubic-variable-trace-positive-decomposition}
\end{align}
The quotient is understood by polynomial extension across $F=0$.
Every term on the last line is nonnegative by
\eqref{eq:cubic-variable-trace} and
\eqref{eq:cubic-gradient-only-S-bound}. If $F\ne0$, Euler's
identity implies $S>0$ and $R>0$, so the term $150R^2SF^2$
is strictly positive. Thus $F\cL(G)>0$ away from $F=0$.

The geometric comparison is proved in
\cref{thm:cubic-algebraic-trace-certificate}.

\paragraph{The Cartan source certificate.}

For \(F,Q,G\) in \cref{thm:cartan-pullback}, the pullback identities give
\begin{align}
 |\nabla G|^2&=4RQ(9Q^3+16F^2),
 \label{eq:Cartan-pullback-gradient}\\
 \cL(G)&=8QF\Bigl[
 9\bigl((4m-48)R^2-5Q\bigr)Q^3\notag\\
 &\hspace{31mm}
 +16\bigl((4m+2)R^2-5Q\bigr)F^2\Bigr].
 \label{eq:Cartan-pullback-defect}
\end{align}
Since \(Q\leq R^2\), the bracket is positive when \(Q>0\).
Thus \(G\cL(G)\geq0\) throughout the regular set.

\paragraph{The unequal Lawson pullback.}
For the target data $T,f,H$ in \cref{ex:lawson-pullback},
\[
 \Delta f=2,\qquad |\nabla f|^2=80T+4f,\qquad
 \Gamma(T,f)=4f.
\]
In particular $\cL(f)=-160f$, so its regular zero set is minimal.
The product rule gives
\begin{align*}
 |\nabla H|^2&=4T(20T^2+Tf+3f^2),\\
 \cL(H)&=8Tf(60T^2+7Tf+33f^2),\\
 2T\cL(H)-4H|\nabla H|^2
 &=32T^2f(20T^2+3Tf+15f^2).
\end{align*}
The last quadratic factor is positive for $T>0$: its discriminant
as a polynomial in $f$ is $-1191T^2$ and its leading coefficient
is positive. Hence \eqref{eq:pullback-target-sign} holds with
$\ell=4$ and $\kappa=1/2$, strictly off $f=0$.
For the source weights $(22,5,6)$, one has $N=64$ and $h=11$.
Substitution in \eqref{eq:bihomogeneous-Hardy} gives
\[
 Q_C(u)\geq\frac{2273}{4}\int_Cr^{-2}u^2,
 \qquad u\in C_c^\infty(C_{\reg}\setminus\{0\}),
 \quad C=\{F=0\},
\]
since $(64-17)^2/4+8\cdot11-72=2273/4$.

\subsection{Recurrence and scalar bounds for determinant potentials}
\label{app:determinant-details}

We prove the algebraic lemmas for the all-degree induction in
\cref{det:first-section}, using \(\mathcal C_m\), \(\mathcal R_r\), and
the binomial sums of
\eqref{det:cleared-polynomial}--\eqref{det:scalar-budget}.

\subsubsection*{Coefficient extraction for the recurrence}

\begin{proof}[Proof of \cref{det:stable-recurrence}]
For $q\geq1$, write
\[
 H_{st}^{(q)}=\sum_i y_i^q e_s(\widehat y_i)e_t(\widehat y_i).
\]
Here and below $e_j=0$ for $j<0$.  Direct substitution in
\eqref{det:cleared-polynomial} gives
\begin{align}
 \mathcal C_m=e_1(D_m^{\mathrm{sym}})^2\biggl[
 &2\sum_{s=1}^{m-1}s(s+1)\frac{e_{s+1}}{e_s}
 -2\sum_{s=1}^{m-1}\frac{H_{ss}^{(1)}}{e_s^2}
 -\sum_{s,t=1}^{m-1}\frac{H_{st}^{(1)}}{e_se_t}\notag\\
 &+\frac1{e_1}\sum_{s=1}^{m-1}\frac{H_{s0}^{(2)}}{e_s}
 -\frac{2e_2}{e_1}\biggr].
 \label{det:C-expanded}
\end{align}
This also shows that the denominators cancel.  The case $m=2$ gives
$\mathcal C_2=0$.

We calculate the coefficient of $e_{r+1}$ in
\eqref{det:C-expanded} for $m=r+1$.  Let
$\mathscr E(z)=\sum_{j=0}^{r+1}e_jz^j=\prod_i(1+y_i z)$.
Summing \(H_{st}^{(1)}z^sw^t\) gives the formal power series identity
\begin{align}
 \sum_{s,t\geq0}H_{st}^{(1)}z^sw^t
 &=\mathscr E(z)\mathscr E(w)
       \sum_i\frac{y_i}{(1+y_i z)(1+y_i w)}\notag\\
 &=\frac{z\mathscr E'(z)\mathscr E(w)
          -w\mathscr E'(w)\mathscr E(z)}{z-w}.
 \label{det:H-generating}
\end{align}
Extracting the coefficient of $z^sw^t$ in
\eqref{det:H-generating} gives, for $0\leq s,t\leq r$,
\[
 H_{st}^{(1)}=\sum_{j=0}^{\min(s,t)}
       (s+t+1-2j)e_j e_{s+t+1-j},
\]
where $e_j=0$ for $j>r+1$ in this calculation with $r+1$ variables.
Since $j\leq r$ in every summand, this expression is affine in
$e_{r+1}$. For $1\leq s,t\leq r$, differentiation with respect to
that formal variable therefore yields
\begin{equation}\label{det:H-coefficient}
 \frac{\partial H_{st}^{(1)}}{\partial e_{r+1}}
 =\begin{cases}
   (2r+1-s-t)e_{s+t-r},&s+t\geq r,\\
   0,&s+t<r.
  \end{cases}
\end{equation}
Similarly,
\[
 \sum_{s\geq0}H_{s0}^{(2)}z^s
 =\frac{e_1\mathscr E(z)-\mathscr E'(z)}z,
 \qquad
 H_{s0}^{(2)}=e_1e_{s+1}-(s+2)e_{s+2}.
\]
It follows that
\begin{equation}\label{det:H2-coefficient}
 \frac{\partial H_{s0}^{(2)}}{\partial e_{r+1}}
 =\begin{cases}
   0,&s\leq r-2,\\
   -(r+1),&s=r-1,\\
   e_1,&s=r.
  \end{cases}
\end{equation}
These formulas show that $\mathcal C_{r+1}$ depends affinely on
$e_{r+1}$.  Setting $e_{r+1}=0$ leaves $e_r^2\mathcal C_r$.
Indeed, append a zero variable to $r$ positive variables.  The terms
with $s=r$ vanish at every nonzero coordinate, and every term from the
zero coordinate has a factor $y_i$ or $y_i^2$.  This proves the claimed
restriction as a polynomial identity.

Since the ratio of $e_1(D_{r+1}^{\mathrm{sym}})^2$ to
$\mathcal P_r/e_{r+1}$ is $e_r$, equations
\eqref{det:H-coefficient}--\eqref{det:H2-coefficient} give
\begin{align}
 \frac{\mathcal R_r}{\mathcal P_r}
 ={}&2r(r+1)
 -2e_r\sum_{\substack{1\leq s\leq r\\2s\geq r}}
       \frac{(2r+1-2s)e_{2s-r}}{e_s^2}\notag\\
 &-e_r\sum_{\substack{1\leq s,t\leq r\\s+t\geq r}}
       \frac{(2r+1-s-t)e_{s+t-r}}{e_se_t}
 +1-(r+1)\frac{e_r}{e_1e_{r-1}}.
 \label{det:R-before-grouping}
\end{align}
The diagonal term $s=r$ contributes $-2$.  In the double sum, the
terms with $s=r$ or $t=r$ contribute $-(r^2+r-1)$.  Together with
$2r(r+1)+1$, these terms give the constant $r(r+1)$.

The remaining terms with $s+t=r$ have the form
$e_r/(e_pe_{r-p})$.  The double sum contributes the factor
$2-\mathbf1_{2p=r}$, and the diagonal sum contributes
$2\mathbf1_{2p=r}$.  The last term of
\eqref{det:R-before-grouping} adds $\mathbf1_{p=1}$.
Their total coefficient is the first coefficient in
\eqref{det:R-recurrence}.

For the terms with $s+t=r+k>r$ and $s,t<r$, choose $s\leq t$ and
write $s=k+d$, $t=r-d$.  Then
$1\leq d\leq\lfloor(r-1)/2\rfloor$ and $1\leq k\leq r-2d$.
The two orders in the double sum and the diagonal contribution give
$2+\mathbf1_{k=r-2d}$, while
$2r+1-s-t=r-k+1$.  This is the second sum in
\eqref{det:R-recurrence}, and proves the recurrence.
\end{proof}

\subsubsection*{The scalar binomial estimate}

\begin{proof}[Proof of \cref{det:binomial-budget}]
Let $S_{r,d}$ denote the inner sum in $J_r$, and put
\[
 t_{r,d}=\frac{\binom{r-d}{d}}{\binom rd},\qquad
 t_{r,d}=0\quad\text{if }2d>r.
\]
For $\ell=r-2d$, the identities
\[
 \sum_{k=1}^{\ell}\binom{k+d}{d}
     =\binom{r-d+1}{d+1}-1,\qquad
 \sum_{k=1}^{\ell}k\binom{k+d}{d}
     =(d+1)\binom{r-d+1}{d+2}
\]
give
\begin{align}
 S_{r,d}={}&\frac2{\binom rd}
  \left[(r+1)\left(\binom{r-d+1}{d+1}-1\right)
       -(d+1)\binom{r-d+1}{d+2}\right]\notag\\
 &+(2d+1)t_{r,d}.
 \label{det:Srd}
\end{align}
Also,
\[
 \frac{t_{r,d+1}}{t_{r,d}}
 =\frac{(r-2d)(r-2d-1)}{(r-d)^2}.
\]
We claim that
\begin{equation}\label{det:telescoping}
 S_{r,d}\leq\frac{2r^3}{r-1}
  \left(\frac{t_{r,d}}{d+1}-\frac{t_{r,d+1}}{d+2}\right).
\end{equation}
To check the sign explicitly, divide \eqref{det:Srd} by $t_{r,d}$.
After omitting the negative term
$-2(r+1)/\binom{r-d}{d}$, the result is
\[
 B_{r,d}=\frac{2(r-d+1)(r+2d^2+3d+2)}{(d+1)(d+2)}+2d+1.
\]
Put $u=d-1\geq0$ and $v=r-2d-1\geq0$.  Over the positive denominator
$(r-1)(d+1)(d+2)(r-d)^2$, the numerator of
\[
 \frac{2r^3}{r-1}\left(\frac1{d+1}
    -\frac{(r-2d)(r-2d-1)}{(d+2)(r-d)^2}\right)-B_{r,d}
\]
is
\begin{align*}
 {}&2uv^4+(4u^3+25u^2+29u)v^3\\
 &+(16u^4+102u^3+206u^2+140u+14)v^2\\
 &+(18u^5+137u^4+385u^3+484u^2+258u+44)v\\
 &+4u^6+38u^5+140u^4+252u^3+230u^2+104u+24.
\end{align*}
Every coefficient is nonnegative.  This proves
\eqref{det:telescoping}.  Summation in $d$, with
$t_{r,1}=(r-1)/r$, gives
\begin{equation}\label{det:J-bound}
 J_r\leq\frac{2r^3}{r-1}\frac{t_{r,1}}2=r^2.
\end{equation}

For the remaining sum, direct calculation gives
\[
 (I_2,I_3,I_4)=\left(6,4,\frac{25}{4}\right),\qquad
 (J_2,J_3,J_4)=\left(0,6,\frac{43}{4}\right).
\]
Thus \eqref{det:scalar-budget} holds for $r\leq4$.
Also, $I_5=24/5\leq5$ and $I_6=329/60\leq6$.  For $r\geq7$,
$\binom rp\geq\binom r2$ when $2\leq p\leq r/2$, so
\[
 I_r\leq\frac{3(r+1)}r
       +\frac{3(r+1)(r-2)}{r(r-1)}\leq r.
\]
Combine this with \eqref{det:J-bound} to finish the proof.
\end{proof}

\section[Weight Realizability]{Weight Realizability for Orthogonal Quartics}
\label{app:quartic-weight-realizability}

The necessity in \cref{prop:quartic-weight-realizability} follows
from the quartic isoparametric classification and these compatibility lemmas.

\begin{lemma}[A block of weight one]
\label{lem:quartic-weight-one}
Let $E$ be a nonzero homogeneous quartic on $\R^N$, with $R=|x|^2$,
satisfying $|\nabla E|^2=16RE$ and $\Delta E=8R$.
Then $E=(x^TPx)^2$ for a symmetric matrix $P$ with
$P^2=\Id$ and $\tr P=0$. In particular, its quadratic square root
is nondegenerate.
\end{lemma}

\begin{proof}
The norm identity gives $E\geq0$. Set $q=\sqrt E$ on a connected
component of $\{E>0\}$. The equations give
\[
 |\nabla q|^2=4R,\qquad \Delta q=0,\qquad
 |\operatorname{Hess}q|^2=\tfrac12\Delta|\nabla q|^2=4N.
\]
Since all second derivatives of $q$ are harmonic, applying the
Laplacian once more yields
$0=\Delta|\operatorname{Hess}q|^2=2|\nabla^3q|^2$.
Thus $q$ is locally a polynomial of degree at most two. Homogeneity
removes its affine terms, and polynomial continuation gives globally
$E=(x^TPx)^2$, where $P=P^T$, $P^2=\Id$, and $\tr P=0$.
In particular, this quadratic form is nondegenerate.
\end{proof}

\begin{lemma}[Joining compatible Clifford blocks]
\label{lem:quartic-Clifford-joining}
Let $(P_\alpha)$ and $(Q_\beta)$ be two nonempty families of symmetric
Clifford matrices on the same Euclidean space; a one-matrix family
means a symmetric involution. If
\[
 A=\sum_\alpha(x^TP_\alpha x)^2,\qquad
 B=\sum_\beta(x^TQ_\beta x)^2,\qquad A+B\leq R^2,
\]
then their union is a Clifford system.
\end{lemma}

\begin{proof}
Fix $P=P_\alpha$ and $Q=Q_\beta$.
On either eigenspace of $P$, $(x^TPx)^2=R^2$, so $x^TQx=0$ there.
Polarization shows that the two diagonal blocks of $Q$ relative to
these eigenspaces vanish. Hence $PQ+QP=0$ for every cross pair.
For the eigenspace description of Clifford systems, see
\cite[Section~3.9, pp.~164--165]{CecilRyan15}.
\end{proof}

\paragraph{Classification input.}
For a triple satisfying \eqref{eq:orthogonal-quartic-data}, put
$\ell=r_0+r_1+r_2-1$, so the ambient dimension is $2\ell$.
For $r_i\ge2$, \cref{rem:compatible-quartic-CM} gives
\[
 \Phi_i=R^2-2E_i,\qquad
 |\nabla\Phi_i|^2=16R^3,\qquad
 \Delta\Phi_i=8(\ell+1-2r_i)R.
\]
The nonzero blocks exclude $\Phi_i=\pm R^2$. By the eikonal classification
\cite[Theorem~1 and the discussion of $m=4$]{TkachevIsoparametric14},
any other nongenuine possibility
has, up to sign and orthogonal coordinates, the form
$R^2-8uv$, where $u=|x'|^2$, $v=|x''|^2$, and $u+v=R$.
Its Laplacian is radial only when the two block dimensions both equal
$\ell$. The corresponding $E_i$ is then $(u-v)^2$ or $4uv$, of
weight $1$ or $\ell$, respectively. Both are excluded by
$2\le r_i\le\ell-1$.
The ordered multiplicities of $\Phi_i$ are consequently
$(r_i-1,\ell-r_i)$.
The completed classification says that, up to sign and orthogonal
coordinates, every such polynomial is OT--FKM, apart from the
homogeneous unordered multiplicities $(2,2)$ and $(4,5)$;
see \cite{ChiFourIII13,ChiFourIV20} and the summary in
\cite[Section~1]{GeTangSOS23}.

\begin{lemma}[Exclusion of the exceptional triple]
\label{lem:quartic-exclude-155}
There is no triple satisfying \eqref{eq:orthogonal-quartic-data}
with weights $(1,5,5)$.
\end{lemma}

\begin{proof}
Suppose such a triple exists. Then $\ell=10$. By the classification
input above, a weight-five block has multiplicities $(4,5)$ and is
the exceptional homogeneous example \cite{ChiFourIII13}.
On $W=\Lambda^2\mathbb C^5$, let $P_1,\ldots,P_5$ be the
holomorphic quadratic Pl\"ucker coordinates of $Z\wedge Z/2$.
The polynomial in \cite[Section~3, equation~(3.2)]{GeTangSOS23} is
\[
 \Psi=R^2-8\sum_{j=1}^5|P_j|^2,
 \qquad (m_+,m_-)=(5,4).
\]
Thus our block is $E_5=(R^2+\Psi)/2$. By
\cref{lem:quartic-weight-one}, the weight-one block is $q^2$ for
a nondegenerate real quadratic form $q$ satisfying
\[
 q^2\le R^2-E_5=4\sum_{j=1}^5|P_j|^2.
\]
Consequently $q$ vanishes on the cone
$\mathcal D=\{u\wedge v:u,v\in\mathbb C^5\}=\{P_1=\cdots=P_5=0\}$.
The equality follows from the skew normal form of a complex
two-vector.

We verify directly that every real quadratic form vanishing on
$\mathcal D$ is degenerate. Decompose it as
$q(Z)=\Re h(Z)+K(Z,Z)$, with $h$ holomorphic quadratic and $K$
Hermitian. Replacing $Z\in\mathcal D$ by $e^{it}Z$ shows that
$h$ and $K(Z,Z)$ vanish separately on $\mathcal D$.
For fixed $u$, Hermitian polarization on $u\wedge\mathbb C^5$ gives
$K(u\wedge v,u\wedge w)=0$.
Replace $u$ by $u+ta$, $t\in\mathbb C$; the coefficients of $t$
and $\bar t$ give $K(u\wedge v,a\wedge w)=0$. Since decomposable
vectors span $W$, this proves $K=0$.
Let $B$ polarize $h$. Then $B(u\wedge v,u\wedge w)=0$.
The tensor $T(u,v,w,t)=B(u\wedge v,w\wedge t)$ is alternating
within each pair; replacing $u$ by $u+w$ in the preceding identity
also makes it alternating in its first and third entries. Hence
$T$ is a four-form, so $h(Z)=\omega(Z\wedge Z)$ for some
$\omega\in\Lambda^4(\mathbb C^5)^*$.
Every nonzero such $\omega$ is $\iota_v\mathrm{vol}$ for a nonzero
vector $v$. In a basis with $v=e_5$, $h$ ignores the component
$e_5\wedge\operatorname{span}_{\mathbb C}(e_1,\ldots,e_4)$.
This gives a real subspace of dimension eight in the radical of
$\Re h$. The zero form is also degenerate. This contradicts the
nondegeneracy of $q$, excluding $(1,5,5)$.
\end{proof}

\begin{proof}[Proof of necessity in \cref{prop:quartic-weight-realizability}]
Put $\ell=r_0+r_1+r_2-1$ and reorder the positive integer weights
as $a\le b\le c$. Since $0\le E_i\le R^2$, the preceding
classification applies. Suppose $\delta(a+b-1)\nmid\ell$.
The function $\delta$ is nondecreasing and takes powers of two as
values. Since $\ell-r_i\ge\ell-c=a+b-1$, it follows that
\[
 \delta(\ell-r_i)\nmid\ell\qquad(i=0,1,2).
\]
Call a block exceptional if its genuine Cartan--M\"unzner polynomial
has unordered multiplicities $(2,2)$ or $(4,5)$.
Every nonexceptional block must itself be a Clifford square sum with
$r_i$ terms: weight one is covered by \cref{lem:quartic-weight-one},
and the classification input above leaves only two OT--FKM
orientations in higher weight. The opposite orientation would give
\[
 E_i=R^2-\sum_{\alpha=1}^{\ell-r_i+1}(x^TP_\alpha x)^2,
\]
requiring $\delta(\ell-r_i)\mid\ell$ by
\cite[Section~3.5]{FKM81}.
If two blocks are nonexceptional, their sum is at most $R^2$;
\cref{lem:quartic-Clifford-joining} produces
$r_i+r_j=\ell+1-r_k$ Clifford matrices,
again forcing $\delta(\ell-r_k)\mid\ell$, a contradiction.
There must therefore be at least two exceptional blocks.
For multiplicities $(2,2)$, one has $\ell=5$, $r_i=3$, and two
blocks leave zero third weight. For $(4,5)$, one has $\ell=10$,
$r_i\in\{5,6\}$, leaving only $(1,5,5)$, excluded by
\cref{lem:quartic-exclude-155}. The different ambient dimensions
prevent mixing the two types. This proves the required divisibility.
\end{proof}

\subsection{The active-sum-nine check}

For \cref{cor:quartic-nine-weights}, the active pairs, up to exchange,
are $(2,7)$, $(3,6)$, and $(4,5)$. The table records
$\delta(a+b-1)$ after ordering the three weights as $a\leq b\leq c$.
None of its entries divides the corresponding $\ell$.
\begin{table}[htbp]
\centering
\caption{The excluded weights $1\leq r_0\leq7$ when $r_1+r_2=9$.}
\label{tab:quartic-nine-low-weights}
\[
\begin{array}{c|c|ccc}
 r_0&\ell=r_0+8&(2,7)&(3,6)&(4,5)\\ \hline
 1&9 &2&4&4\\
 2&10&4&4&8\\
 3&11&4&8&8\\
 4&12&8&8&8\\
 5&13&8&8&8\\
 6&14&8&8&8\\
 7&15&8&8&8
\end{array}
\]
\end{table}

\FloatBarrier
\section[Supplementary Details]{Supplementary Algebraic and Geometric Details}
\label{app:editorial-foundations}

\subsection{Polarization and metrized operations in arbitrary degree}
\label{sec:higher-metrized-operations}

For a homogeneous polynomial $F$ of degree $d\ge3$, its scalar
polarization is the symmetric $d$-linear form
$A_F=D^dF/d!$, characterized by $A_F(x,\ldots,x)=F(x)$
\cite[Section~2.6.4, equation~(2.6.5)]{Landsberg2012}.
Here $D^dF$ is the constant $d$th differential. Its evaluation is
given by the polarization formula
\[
 D^dF[x_1,\ldots,x_d]
 =\sum_{I\subseteq\{1,\ldots,d\}}(-1)^{d-|I|}
     F\Bigl(\sum_{i\in I}x_i\Bigr).
\]
The alternating sum cancels terms missing some vector argument;
homogeneity leaves the mixed derivative of order $d$.

Using the Euclidean inner product, define the symmetric
$(d-1)$-linear map $\mu:V^{d-1}\to V$ by
\[
 \langle\mu(x_1,\ldots,x_{d-1}),x_d\rangle
       =dA_F(x_1,\ldots,x_d)
       =\frac{1}{(d-1)!}D^dF[x_1,\ldots,x_d].
\]
The complete symmetry of the right side is the metric compatibility
condition. Since $DF(x)[v]=dA_F(x,\ldots,x,v)$, this normalization
makes $\mu(x,\ldots,x)=\nabla F(x)$. For $d=3$, the product in
\eqref{eq:cubic-polarized-product} satisfies
$\langle xy,z\rangle=6A_F(x,y,z)$, so $xy=2\mu(x,y)$;
its multiplication operator is consequently twice the one used here.
Put $P(x)=\mu(x,\ldots,x)$ and
$L_xv=\mu(x,\ldots,x,v)$, with $d-2$ copies of $x$ in the latter.
Polarization and Euler's identity give
\begin{equation}\label{eq:higher-metrized-dictionary}
 F=\frac1d\langle P,x\rangle,\qquad
 \nabla F=P,\qquad \Hess F=(d-1)L_x,\qquad
 \Delta F=(d-1)\tr L_x.
\end{equation}
Every $L_x$ is self-adjoint. With $S=|P|^2$,
\begin{equation}\label{eq:higher-metrized-minimality}
 \cL(F)=(d-1)\bigl(S\tr L_x-\langle L_xP,P\rangle\bigr).
\end{equation}
In the harmonic case, $\tr L_x=0$ for all $x$, and Bochner's formula
and differentiation of $S$ give
\begin{equation}\label{eq:higher-metrized-Bochner}
 \Delta S=2(d-1)^2\tr(L_x^2),\qquad
 \Gamma(S,S)=4(d-1)^2|L_xP|^2.
\end{equation}
Indeed, $\nabla S=2(d-1)L_xP$ and
$\Delta|\nabla F|^2=2|\Hess F|^2$ when $\Delta F=0$.
Thus $\cL(F)=-qF$ is equivalent to
\[
 d(d-1)\langle L_xP,P\rangle=q\langle P,x\rangle
\]
in this harmonic normalization. The weight $q$ has degree $2d-4$
and need not be radial.

The polarization formula is an equivalent algebraic description of
the polynomial, and the contractions above express the minimal surface
equation in terms of a symmetric metrized operation. In higher degree
this can organize the identities needed to compute
$q,S,\Delta S,\Gamma(S,S)$ in
\cref{sec:harmonic-multiplier-criterion}. It becomes useful for a sign
test when additional structure, such as Clifford relations, norm
composition \cite[Chapter~1, Appendix, pp.~21--22]{Shapiro2000}, or
spectral minor identities, makes these contractions tractable.
Polarization alone supplies neither a classification nor the required
sign inequalities. For $d>3$ the operation is $(d-1)$-linear, so the
binary-algebra methods for cubics do not transfer automatically.

\subsection{Exceptional sets in the algebraic constructions}
\label{app:exceptional-set-summary}

For each cited potential \(G=aF\), the multiplier is smooth and
positive off the indicated set, and \(G\) extends continuously with
the signs of \(F\). Homogeneity gives the gradient check following
\cref{cor:gradient-subcalibration}.

\begin{center}
\small
\begin{tabular}{@{}lll@{}}
\toprule
Construction & Exceptional set \(Z\) & Ambient codimension\\
\midrule
\Cref{thm:iso-polynomial-table} & \(\{0\}\) & \(N\)\\
\Cref{thm:cubic-rank-two-existence} & \(\{S=0\}\), \(S=|\nabla F|^2\) & at least \(2k\)\\
\Cref{thm:exc-jordan-minimizing} & \(\{z^\#=0\}\) & \(2(d+2)\)\\
\Cref{thm:nonradial-minimizing}, \(d=4\) & \(\{S=0\}\) & \(10\)\\
\Cref{thm:exc-g2} & \(\{\operatorname{rank}[a\ b\ c]\leq1\}\) & \(12\)\\
\Cref{ex:Riedler-correct-power} & \(\{Q=0\}\) & at least \(2k\)\\
\Cref{ex:lawson-pullback} & \(P^{-1}(0)\) & at least \(11\)\\
\Cref{thm:cartan-pullback} & \(P^{-1}(0)\) & at least \(m\in\{14,26\}\)\\
\Cref{thm:Hopf-pairing} & \(\{|x|\,|y|=0\}\) & \(16\)\\
\Cref{thm:weighted-Clifford-quartic} & \(\{W=0\}\) & \(14\)\\
\Cref{thm:octonionic-product-minimizing} & at least two zero factors & \(16\)\\
\Cref{thm:pfaffian-minimizing} & \(\{\operatorname{rank}_{\mathbb C}Z\leq2d-4\}\) & \(12\)\\
\Cref{main:complex-determinants}, \(m\geq2\) & \(\{\operatorname{rank}_{\mathbb C}Z\leq m-2\}\) & \(8\)\\
\bottomrule
\end{tabular}
\end{center}

In the minimizing ranges, every listed set satisfies \eqref{eq:H-null}.
Together with the cited sign certificates, this verifies
\cref{cor:gradient-subcalibration}.

\subsection{Calibrated complex cones with zero stability gap}
\label{sec:prelim-complex-gap}

Dimler--Lee \cite[Theorem~4.2]{DimlerLee24} construct calibrated
complex hypersurface cones with zero Hardy gap when degree equals
complex dimension. The following extends this to all larger degrees.

\begin{proposition}[Area-minimizing complex cones without strict stability]
\label{prop:complex-zero-gap}
Let \(k\geq2\), and let \(f\) be a homogeneous holomorphic polynomial
of degree \(d\geq k\) on \(\mathbb C^{k+1}\), with \(df(z)\ne0\) for
\(z\ne0\).  Then \(C=f^{-1}(0)\), with its complex orientation, is
area minimizing of real dimension \(2k\), and its Hardy stability gap
is zero:
\begin{equation}\label{eq:complex-zero-gap}
 \inf_{0\ne V\in C_c^\infty(NC_{\reg})}
 \frac{Q_C^\perp(V)}{\displaystyle\int_Cr^{-2}|V|^2}=0.
\end{equation}
Thus \(C\) is area minimizing but not strictly stable in the Hardy sense.
\end{proposition}

\begin{proof}
The K\"ahler calibration \(\omega_0^k/k!\), where
\(\omega_0=\sum_{j=1}^{k+1}dx_j\wedge dy_j\), calibrates the complex
current \(C\) \cite{HarveyLawson82a}.  Thus
\eqref{eq:calibration-Stokes-comparison} gives area minimization and
\(Q_C^\perp\geq0\).

Choose \(g=\ell^{d-k}\), with \(\ell\) a linear form nonzero at
some point of \(C\setminus\{0\}\), and take \(g=1\) if \(d=k\).
Thus \(g\) has degree \(d-k\) and does not vanish identically on \(C\).
Writing \(f=u+iv\), holomorphicity gives \(\nabla v=J\nabla u\):
the gradients are orthogonal with equal length. Hence
\begin{equation}\label{eq:complex-critical-Jacobi-field}
 V_g=\frac{\operatorname{Re}g\,\nabla u+
                 \operatorname{Im}g\,\nabla v}{|\nabla u|^2}
 \quad\text{on }C_{\reg}
\end{equation}
is normal and satisfies \(df(V_g)=g\).
Near each compact subset of \(C_{\reg}\), the holomorphic equations
\(f-tg=0\) form a smooth family of complex, hence minimal,
hypersurfaces for small real \(t\). Differentiating the equation and
minimality at \(t=0\) gives normal velocity \(V_g\) and
\(J_C^\perp V_g=0\).

Since \(\deg g=d-k\) and \(\deg\nabla u=d-1\),
\[
 V_g(r\theta)=r^{1-k}W(\theta),\qquad
 \theta\in\Lambda=C\cap\Sph^{2k+1},
\]
where \(W\) is smooth and not identically zero on the compact smooth
link. The degree \(1-k=-(\dim_{\R}C-2)/2\) is the critical Hardy degree.
Multiplying \(J_C^\perp V_g=0\) by \(\chi^2V_g\) and integrating
by parts in \eqref{eq:normal-Jacobi-definition} yields
\begin{equation}\label{eq:normal-Jacobi-cutoff}
 Q_C^\perp(\chi V_g)=\int_C|\nabla_C\chi|^2|V_g|^2
\end{equation}
for every smooth compactly supported scalar cutoff on \(C_{\reg}\).
Take a nonzero \(\zeta\in C_c^\infty((-2,-1))\) and, for \(L>1\), set
\(\chi_L(r)=\zeta(\log r/L)\).  Writing
\(b=\int_\Lambda|W|^2>0\) and using
\(d\mu_C=r^{2k-1}\,dr\,d\mu_\Lambda\), we obtain
\[
 Q_C^\perp(\chi_LV_g)=\frac bL\int_\R|\zeta'|^2,
 \qquad
 \int_Cr^{-2}|\chi_LV_g|^2=bL\int_\R|\zeta|^2.
\]
The quotient is
\(L^{-2}\int|\zeta'|^2/\int|\zeta|^2\), which tends to zero.
Together with stability, this proves \eqref{eq:complex-zero-gap}.
\end{proof}

\section{Computational Verification}
\label{app:computational-verification}

The accompanying exact-arithmetic package and its execution record are archived
in the fixed release
\begin{center}
\href{https://github.com/HongbinCui/laplacian-subcalibrations/releases/tag/verification-20261013-v2}{\texttt{HongbinCui/laplacian-subcalibrations: verification-20261013-v2}}.
\end{center}
The same verification files are supplied in the ancillary archive
\texttt{verification-package.zip}.
Using \cref{thm:finite-reduction}, the programs check identities and
exact sign certificates. The index records tested orders and distinguishes checks conditional on
prescribed closure tables from checks in Euclidean coordinates.
The rational Bernstein certificate in \cref{app:exact-certificates}
proves positivity on its entire interval. The determinant recurrence
and bounds are proved in arbitrary order in
\cref{sec:complex-determinants}; bounded-order computational checks
only supplement those proofs.

The Wolfram Language suite requires Mathematica~12.0 or later and runs
by evaluating the input cell in \texttt{StartHere.nb}, including closure
identities, orthogonal quartic families, the complex Gram sextic, and
determinant initial-form checks. Each run records its date, Mathematica
version, operating system, source hashes, and individual results in a
new directory. The README supplies reproduction instructions;
\texttt{CertificateIndex.md} gives each unit's scope and statement/equation
LaTeX labels. Its printed equation numbers and page locations refer to
the bundled V7 manuscript snapshot; this article retains the
mathematical labels with a different layout.

The author ran the complete suite in Wolfram Mathematica~12.0.0
(Windows 64-bit) on October~12, 2026: all 26 units and 7,134 checks
passed. The accompanying results directory records the execution,
source hashes, and individual results.

\section*{Acknowledgments}

The author is grateful to Professors Xiaoxiang Jiao and Chia-Kuei Peng
for inspiring discussions on calibrated geometry and isoparametric
geometry many years ago. He is especially grateful to Professor
Xiaowei Xu for encouraging him many years ago to study the
area-minimizing property of minimal cubic cones, and for bringing
the papers by Peng and Xiao \cite{PengXiao93} and by Hoppe and
Tkachev \cite{HoppeTkachev19} to his attention. He also thanks
Professor Quo-Shin Chi for summer lectures and discussions on
isoparametric geometry at Fujian Normal University, and for his
engaging account of the history of OT--FKM isoparametric hypersurfaces.
The author thanks Fujian Normal University for its hospitality.
He also thanks an anonymous reader for constructive suggestions on an earlier
version, especially on the self-contained exposition and the presentation
of the underlying algebraic structures.

\bigskip
\noindent\textbf{AI disclosure.} ChatGPT (OpenAI) was used to assist
with language editing, notation and organization, literature searches,
mathematical exploration, the preparation of computer code for symbolic
and algebraic computations, and AI-assisted figure generation based on
the author's sketches and instructions. The author ran the supplied verification code
in Wolfram Mathematica~12.0.0 (Windows 64-bit) on October~12, 2026;
the complete execution record accompanies the source files.
The author takes full responsibility for the mathematical content and
computational conclusions.

\end{document}